\documentclass{memo-l}

\usepackage{amssymb,bm,mathrsfs,xcolor}
\usepackage[all]{xy}
\usepackage{hyperref}

\newtheorem{theorem}{Theorem}[section]
\newtheorem{proposition}[theorem]{Proposition}
\newtheorem{lemma}[theorem]{Lemma}
\newtheorem{corollary}[theorem]{Corollary}
\newtheorem{conjecture}[theorem]{Conjecture}

\newtheorem{bigtheorem}{Theorem}
\newtheorem{bigconjecture}[bigtheorem]{Conjecture}

\newtheorem{Atheorem}{Theorem}[chapter]
\newtheorem{Acorollary}[Atheorem]{Corollary}

\theoremstyle{definition}
\newtheorem{definition}[theorem]{Definition}

\theoremstyle{remark}
\newtheorem{remark}[theorem]{Remark}
\newtheorem{Aremark}[Atheorem]{Remark}

\numberwithin{section}{chapter}
\numberwithin{equation}{section}

\newcommand{\map}[1]{\xrightarrow{#1}}
\newcommand{\mil}{\varprojlim}

\newcommand{\iso}{\cong}
\newcommand{\define}{\stackrel{\mathrm{def}}{=}}

\newcommand{\kk}{{\bm{k}}}
\newcommand{\inv}{\operatorname{inv}}

\newcommand{\Gal}{\operatorname{Gal}}
\newcommand{\Hom}{\operatorname{Hom}}
\newcommand{\Aut}{\operatorname{Aut}}
\newcommand{\End}{\operatorname{End}}
\newcommand{\Spec}{\operatorname{Spec}}
\newcommand{\Q}{\mathbb{Q}}
\newcommand{\Z}{\mathbb{Z}}
\newcommand{\R}{\mathbb{R}}
\newcommand{\C}{\mathbb{C}}
\newcommand{\F}{\mathbb{F}}
\newcommand{\A}{\mathbb{A}}
\renewcommand{\H}{\mathcal{H}}
\newcommand{\calS}{\mathcal{S}}
\newcommand{\calY}{\mathcal{Y}}
\newcommand{\calD}{\mathcal{D}}
\newcommand{\calH}{\mathcal{H}}
\newcommand{\calZ}{\mathcal{Z}}
\newcommand{\frakg}{\mathfrak{g}}
\newcommand{\frakz}{\mathfrak{z}}

\newcommand{\co}{\mathcal O}
\newcommand{\alg}{\mathrm{alg}}

\newcommand{\Lie}{\operatorname{Lie}}

\newcommand{\GS}{\mathrm{GS}}
\newcommand{\BKK}{\mathrm{pre}}
\newcommand{\CH}{\operatorname{CH}}
\newcommand{\chern}{\operatorname{\mathrm{c}_1}}
\newcommand{\Pic}{\operatorname{Pic}}
\newcommand{\vol}{\operatorname{vol}}
\newcommand{\Kra}{\mathrm{Kra}}
\newcommand{\Pap}{\mathrm{Pap}}

\newcommand{\nonexc}{\mathrm{nexc}}
\newcommand{\taut}{\mathcal{L}}
\newcommand{\tautmod}{\mathcal{K}}

\newcommand{\zxz}[4]{\begin{pmatrix} #1 & #2 \\ #3 & #4 \end{pmatrix}}
\newcommand{\leg}[2]{\left( \frac{#1}{#2} \right)}
\newcommand{\kzxz}[4]{\left(\begin{smallmatrix} #1 & #2 \\ #3 & #4\end{smallmatrix}\right) }

\newcommand{\eps}{\varepsilon}
\newcommand{\norm}{\operatorname{N}}
\newcommand{\tr}{\operatorname{tr}}

\newcommand{\sgn}{\operatorname{sgn}}
\newcommand{\bs}{\backslash}

\newcommand{\Orth}{\operatorname{O}}
\newcommand{\Uni}{\operatorname{U}}

\newcommand{\GL}{\operatorname{GL}}
\newcommand{\SO}{\operatorname{SO}}
\newcommand{\SL}{\operatorname{SL}}

\newcommand{\dv}{\operatorname{div}}

\newcommand{\Div}{\operatorname{Div}}

\newcommand{\reg}{\operatorname{reg}}

\begin{document}

\frontmatter

\title[Arithmetic volumes of unitary Shimura varieties]{Arithmetic volumes of unitary Shimura varieties}


\author[Jan H.~Bruinier]{Jan Hendrik Bruinier}
\address{Fachbereich Mathematik, Technische Universit\"at Darmstadt, Schlossgartenstrasse 7, D--64289 Darmstadt, Germany}
\email{bruinier@mathematik.tu-darmstadt.de}
\thanks{J.B. was  supported in part by the LOEWE research unit USAG and by the
DFG Collaborative Research Centre TRR 326 ``Geometry and Arithmetic of Uniformized Structures'', project number 444845124.}

\author[Benjamin Howard]{Benjamin Howard}\address{Department of Mathematics, Boston College, 140 Commonwealth Ave, Chestnut Hill, MA 02467, USA}
\email{howardbe@bc.edu}
\thanks{B.H. was supported in part by NSF grants DMS-1801905 and DMS-2101636.}

\date{}




\begin{abstract}
The integral model of a $\mathrm{GU}(n-1,1)$ Shimura variety  carries a universal abelian scheme over it,
and the  dual top exterior power of its Lie algebra carries a natural hermitian metric.  We express the arithmetic volume of this metrized line bundle, defined as  an iterated self-intersection in the arithmetic Chow ring, in terms of logarithmic derivatives of Dirichlet $L$-functions. We also determine the arithmetic volumes of Kudla-Rapoport divisors and relate them to coefficients of Eisenstein series. 
\end{abstract}

\maketitle

\tableofcontents


\mainmatter


\chapter{Introduction}





The explicit calculation of arithmetic volumes of Shimura varieties began with the work of K\"uhn \cite{kuhn} and Bost (independently, in unpublished work).  Over the moduli stack $\mathcal{X} \to \Spec(\Z)$ of elliptic curves there is a  line bundle of weight one modular forms, characterized as the dual  Lie algebra of the universal elliptic curve.  
This line bundle carries a natural hermitian metric, and so determines a class in the codimension one arithmetic Chow group of $\mathcal{X}$  of Gillet-Soul\'e (as extended by Burgos-Kramer-K\"uhn).  Working on a suitable compactification, K\"uhn and Bost computed the self-intersection multiplicity of this hermitian line bundle, and gave a simple formula for it in terms of the logarithmic derivative of the Riemann zeta function at $s=-1$.

More generally, if $\mathcal{X}$ is an  integral model of a PEL type (or even Hodge type) Shimura variety, then $\mathcal{X}$ carries a \emph{Hodge bundle}: the dual of the top exterior power of the Lie algebra of the universal abelian scheme over $\mathcal{X}$.  The Hodge bundle again carries a natural hermitian metric, and one can ask if its arithmetic volume, defined as the  $\mathrm{dim}(\mathcal{X})$-fold self-intersection in the arithmetic Chow group, is again related to logarithmic derivatives of Dirichlet $L$-functions.

In some instances this is known. 
 If $\mathcal{X}$ is the integral model of a quaternionic Shimura curve over $\Q$, the volume was computed by Kudla-Rapoport-Yang \cite{KRY}.   Hilbert modular surfaces and the Siegel threefold were considered by Bruinier-Burgos-K\"uhn \cite{BBK} and Jung-von Pippich \cite{JvP}, respectively.
 The arithmetic volumes of $\mathrm{GSpin}(n,2)$   Shimura varieties, which include all of the above examples as special cases, were computed (up to a $\Q$-linear combination of logarithms of certain bad primes)  by H\"ormann \cite{hormann}.
Most recently, Yuan \cite{Yuan}  extended the above mentioned result of Kudla-Rapoport-Yang 
(using completely different methods)
 to all  quaternionic Shimura curves over totally real fields.

In \cite{FS}  Freixas and Sankaran study arithmetic intersections on a Shimura curve and on a twisted Hilbert modular surface over the same real quadratic base field. They relate the corresponding arithmetic volumes via the arithmetic Riemann-Roch Theorem and the Jacquet-Langlands correspondence.
Maillot and R\"ossler \cite{MR} employ the equivariant Grothendieck-Riemann-Roch Theorem in Arakelov geometry to derive general conjectures relating Chern classes of certain hermitian vector bundles in Arakelov geometry to logarithmic derivatives of Artin $L$-functions of number fields. See in particular Section~8 of \cite{MR} for a discussion of the relationship with some of the above results on arithmetic volumes of Shimura varieties.
Finally, we mention the work of Andreasson and Berman, who study arithmetic varieties for which the log canonical bundle (or its dual) with respect to a suitable divisor is relatively ample \cite{AB}. Under some stability condition, they equip  this bundle with a K\"ahler-Einstein metric and study various properties of the corresponding arithmetic volume. In the special case when the arithmetic variety is $\mathbb{P}^1_\Z$ they use their results to obtain different proofs of some of the above results for Shimura curves of genus zero, see Section 9 of \cite{AB}.

The present paper deals with the case of $\mathrm{GU}(n-1,1)$ Shimura varieties.
 The overall strategy is closely related to that of \cite{BBK} and \cite{hormann}, insofar as it ultimately rests on the use of Borcherds products.  However,  Borcherds products will make no explicit appearance in this paper;  all the arithmetic information we need from the theory has already been extracted  in \cite{BHKRY-1}, and is recalled here in \S \ref{ss:BHKRYmodularity}.


\section{The volume formulas}


Let $\kk\subset \C$ be an imaginary quadratic field of odd discriminant 
\[
-D = \mathrm{disc}(\kk).
\]
Given an integer $n\ge 1$, there is a regular  Deligne-Mumford stack 
\[
\mathcal{M}_{(n-1,1)} \to \Spec(\co_\kk),
\]
 flat of relative dimension $n-1$,  
 parametrizing principally polarized abelian schemes $A$, endowed with an action  $\co_\kk\to \End(A)$,  and extra data encoding a signature $(n-1,1)$ condition on $\Lie(A)$.
See \S \ref{ss:basic moduli} for details.

This stack admits a decomposition 
\[
\mathcal{M}_{(n-1,1)}  = \bigsqcup_W \mathcal{M}_W
\]
into open and closed substacks indexed by strict similarity classes of relevant $\kk$-hermitian spaces $(W,h)$ of signature $(n-1,1)$,  in such a way that the generic fiber of $\mathcal{M}_W$ is a   Shimura variety  for the unitary similitude group $\mathrm{GU}(W)$.
Here \emph{relevant} means  that $W$ contains an $\co_\kk$-lattice $\mathfrak{a} \subset W$ that is self-dual, in the sense of \eqref{self-dual lattice}.
The notion of strict similarity is defined in \S \ref{ss:basic moduli}, but the important fact is that there is a dichotomy between the cases of $n$ odd and $n$ even.   
If $n$ is odd the  relevant $W$ form a single strict similarity class, and the disjoint union has a single term.
If $n$ is even then strict similarity is equivalent to isometry, and the number of terms in the disjoint union is $2^{o(D) -1}$, where
\begin{equation*}
o(D) = \# \{ \mbox{prime divisors of }D\}.
\end{equation*}

Fix a relevant $W$ as above, and denote by 
$
\pi: A \to \mathcal{M}_W
$
the restriction to $\mathcal{M}_W$ of the universal abelian scheme over $\mathcal{M}_{(n-1,1)}$.
Its relative dimension is $n=\operatorname{dim}(A)$, and  its \emph{metrized Hodge bundle}
\[
\widehat{\omega}^\mathrm{Hdg}_{A / \mathcal{M}_W} \in \widehat{\Pic}(\mathcal{M}_W)
\]
is the line bundle
$
\omega^\mathrm{Hdg}_{A / \mathcal{M}_W} =  \pi_* \Omega^{\operatorname{dim}(A)}_{A / \mathcal{M}_W}  
$
  endowed with the  hermitian metric 
\begin{equation}\label{hodge metric}
\| s_z \|^2 =   \left|\frac{1}{ (2\pi i )^{\operatorname{dim}(A) }  }   \int_{A_z (\C)} s_z \wedge \overline{s}_z  \right|
\end{equation}
for any  $z\in \mathcal{M}_W(\C)$ and 
$
s_z\in \omega^\mathrm{Hdg}_{A / \mathcal{M}_W , z} \iso H^0( A_z , \Omega^{\operatorname{dim}(A_z) }_{A_z / \C } ) .  
$

The stack $\mathcal{M}_W$ has a  canonical toroidal compactification $\bar{\mathcal{M}}_W$,  with boundary a Cartier divisor smooth over $\co_\kk$. 
 Although the Hodge bundle has a distinguished extension to the  compactification, the metric on it does not extend smoothly.  Instead, the metric is pre-log singular along the boundary,  in the sense of  Definition 1.20 of \cite{BBK}.
The pre-log singular conditions  allow us to view  
\begin{equation}\label{introHodge}
\widehat{\omega}^\mathrm{Hdg}_{A / \mathcal{M}_W} \in   \widehat{\CH}^1( \bar{\mathcal{M}}_W ,\mathscr{D}_\BKK) 
\end{equation}
as a class in the codimension one arithmetic Chow group of  Burgos-Kramer-K\"uhn \cite{BKK}.  

These arithmetic Chow groups come with intersection  pairings, and distinguished linear functionals in codimensions $n-1$ and $n$.  The first of these linear functionals is the  \emph{complex degree}, denoted
\[
\deg_\C : \widehat{\CH}^{n-1}( \bar{\mathcal{M}}_W ,\mathscr{D}_\BKK)  \to \Q,
\]
defined by first sending an arithmetic cycle class to its image in the Chow group $\CH^{n-1}( \bar{\mathcal{M}}_{W/\C})$ of $0$-cycles on the complex fiber, and then taking the degree in the usual sense.
The second is the more subtle \emph{arithmetic degree}, denoted
\[
\widehat{\deg} :   \widehat{\CH}^n( \bar{\mathcal{M}}_W ,\mathscr{D}_\BKK)  \to \R. 
\]
We define the \emph{complex volume}
\[
\vol_\C (  \widehat{\omega}^\mathrm{Hdg}_{A / \mathcal{M}_W} )  
= 
\deg_\C \big( 
\underbrace{ \widehat{\omega}^\mathrm{Hdg}_{A / \mathcal{M}_W} \cdots \widehat{\omega}^\mathrm{Hdg}_{A / \mathcal{M}_W} }_{n-1 \ \mathrm{times} }  
\big) 
\]
and the
 \emph{arithmetic volume}
\[
\widehat{\vol} (  \widehat{\omega}^\mathrm{Hdg}_{A / \mathcal{M}_W} )  
= 
\widehat{\deg}\big( 
\underbrace{ \widehat{\omega}^\mathrm{Hdg}_{A / \mathcal{M}_W} \cdots \widehat{\omega}^\mathrm{Hdg}_{A / \mathcal{M}_W} }_{n \ \mathrm{times} }  
\big) 
\]
 as the complex and arithmetic degrees of the $(n-1)$-fold and $n$-fold iterated intersections, respectively.

Our main result is the calculation of these complex and arithmetic volumes in terms of Dirichlet $L$-functions.
To state it, for any place $\ell \le \infty$ denote by    
\[
\inv_\ell(W)= (\det W,-D)_\ell  \in \{ \pm 1\}
\]
 the local invariant of  $W$, where the right hand side is the usual Hilbert symbol.
Our assumption that $W$ contains a self-dual $\co_\kk$-lattice is equivalent to the condition that  $\mathrm{inv}_\ell(W) =1$ for all finite primes $\ell\nmid D$.  As  $\inv_\infty(W)=-1$, and the product of all local invariants is $1$, we obtain
\begin{equation}\label{signs 1}
\prod_{\ell\mid D}  \mathrm{inv}_\ell(W) 
=-1 .
\end{equation}

The quadratic Dirichlet character determined by  $\kk/\Q$ is denoted
\[
\eps:(\Z/D\Z)^\times \to \{ \pm 1\}.
\]
For an integer $k \ge 1$ set
\begin{equation}\label{a_k}
\mathbf{a}_k(s) =   \frac { D^{k/2} \Gamma ( s + k )  L(2s+k,\eps^k) }{2^k \pi^{s+k} } ,
\end{equation}
where, if $k$ is even,  we understand $L(s,\eps^k)=\zeta(s)$.    Define
\begin{align}
\mathbf{A}_W (s)  & =  \mathbf{a}_1(s) \cdots \mathbf{a}_n(s)  \nonumber \\
& \quad \times 
\begin{cases}
 \prod_{\ell \mid D} \big(   1+ \leg{-1}{\ell}^{\frac{n}{2} }  \mathrm{inv}_\ell(W) \ell^{ -s- \frac{n}{2} }   \big) & \mbox{if $n$ is even} \\
 1 & \mbox{if $n$ is odd.}
\end{cases}
\label{A_V}
\end{align}
If $n$ is odd there is no dependence on $W$ beyond its dimension.
We note that all $\mathbf{a}_{k}(0)$  are positive rational numbers, and hence so is $\mathbf{A}_W(0)$.

The following is our main result.  It  appears in the text as Theorem \ref{thm:final hodge}.

\begin{bigtheorem}\label{thm:intro main}
The metrized Hodge bundle has  complex volume
\[
\vol_\C (  \widehat{\omega}^\mathrm{Hdg}_{A / \mathcal{M}_W} ) 
 =
 \mathbf{A}_W (0) \cdot
 \begin{cases}
 2^{n-1  }    & \mbox{if $n$ is odd } \\
2^{n-o(D)  }   & \mbox{if $n$ is even}
\end{cases}
\]
and arithmetic volume
\[
\widehat{\vol}  ( \widehat{\omega}^\mathrm{Hdg}_{A/ \mathcal{M}_W  } )
= 
  \left(  2 \frac{  \mathbf{A}_W'(0) }{ \mathbf{A}_W(0) }    - nC_0(n)  
  + \log(D)  \right)  \cdot     \vol_\C (  \widehat{\omega}^\mathrm{Hdg}_{A / \mathcal{M}_W} ) ,
\]
where we have set
\[
C_0 (n) = 2  \log\left( \frac{4\pi e^\gamma}{ \sqrt{D} } \right)     +     (n-4)  \left(   \frac{L'(0,\eps)}{L(0,\eps)}  +  \frac{\log(D)}{2}    \right) .
\]
Here $\gamma=-\Gamma'(1)$ is the Euler-Mascheroni constant.
\end{bigtheorem}

\begin{remark}\label{rem:gross motive}
Gross  \cite{gross-motive} associates a motive to a  reductive group over a  field.
The  $L$-function of the motive associated to $\mathrm{U}(W)$ is essentially  $\mathbf{a}_1(s) \cdots \mathbf{a}_n(s)$, which  perhaps provides some explanation for why this particular product appears in the statement of  Theorem \ref{thm:intro main}.
\end{remark}

\begin{remark}
There is a slightly different way to express the meromorphic function $\mathbf{A}_W(s)$ as a product.
Define  
$
\mathbf{b}_{W,1}(s)  =\mathbf{a}_1(s).
$
For  $k \ge 2$  even,  set
\[
\mathbf{b}_{W,k}(s)  
= \mathbf{a}_k(s) 
\prod_{\ell \mid D} \left(   1+ \leg{-1}{\ell}^{ \frac{k}{2} } \mathrm{inv}_\ell(W) \ell^{ -s- \frac{k}{2} }   \right).
\]
 For  $k \ge 3$ odd,  set
\[
\mathbf{b}_{W,k}(s)  
= \mathbf{a}_k(s) 
\prod_{\ell \mid D} \left(  1+ \leg{-1}{\ell}^{ \frac{k-1}{2} } \mathrm{inv}_\ell(W) \ell^{ -s +\frac{1-k}{2}}  \right)^{-1}.
\]
When  $k>1$ we have
$\mathbf{b}_{W,k}(s)  \mathbf{b}_{W,k+1}(s) =  \mathbf{a}_{k}(s)  \mathbf{a}_{k+1} (s)$,
which implies  that 
\[
\mathbf{A}_W(s) = \mathbf{b}_{W,1}(s) \cdots \mathbf{b}_{W,n}(s)  .
\]
\end{remark}


\section{The generating series of divisors}
\label{ss:intro generating}


Most of the calculations that go into the proof of Theorem \ref{thm:intro main}  are not carried out on the $\mathrm{GU}(W)$ Shimura variety $\mathcal{M}_W$,  but rather on a slightly different Shimura variety.  

Let us again fix a relevant $\kk$-hermitian space $V$ of signature $(n-1,1)$, and now assume $n\ge 2$.
We may present this hermitian space in the form 
\[
V \iso  \Hom_\kk (W_0,W),
\]
where  $W_0$ and $W$ are  relevant $\kk$-hermitian space of signature $(1,0)$ and $(n-1,1)$, respectively (the right hand side carries a natural $\kk$-hermitian form, and here we require that the isomorphism is an isometry).
  Following \cite{KRunitaryII}, we associate to $V$ an open and closed substack
\[
\mathcal{S}_V \subset \mathcal{M}_{W_0} \times_{\Spec(\co_\kk)} \mathcal{M}_W .
\]
As explained in \S 2 of \cite{BHKRY-1},  the generic fiber of $\mathcal{S}_V$ is a Shimura variety for the subgroup $G\subset \mathrm{GU}(W_0) \times \mathrm{GU}(W)$ consisting of pairs for which the similitude factors are equal, and there is a natural surjection $G \to \mathrm{U}(V)$.   
This allows one to  view the moduli space $\mathcal{S}_V$ as an approximation to the  Shimura variety for the unitary group $\mathrm{U}(V)$, which is awkward to work with directly because it is  of abelian type but not PEL type.

In \S \ref{ss:special shimura} we recall that the Shimura variety $\mathcal{S}_V$ carries a family of Kudla-Rapoport divisors 
\[
\mathcal{Z}_V(m) \to \mathcal{S}_V
\]
indexed by positive $m\in \Z$.  It also carries a distinguished line bundle $\mathcal{L}_V$, called the \emph{tautological bundle} or the \emph{line bundle of weight one modular forms}, and an \emph{exceptional divisor} $\mathrm{Exc}_V \subset  \mathcal{S}_V$ that is proper and supported in characteristics $\ell \mid D$.

The main result of \cite{BHKRY-1} says that the Kudla-Rapoport divisors can be enhanced  to arithmetic divisors
\[
\widehat{\mathcal{Z}}_V^\mathrm{tot} (m) = (  \bar{\mathcal{Z}}_V(m) + \mathcal{B}_V(m) , \Phi_V(m) )
 \in  \widehat{\CH}^1(\bar{\mathcal{S}}_V , \mathscr{D}_\BKK)
\]
(for $m>0$) in such a way that the formal generating series 
\[
 \widehat{\phi}_V (\tau) =
\sum_{m\ge 0}  \widehat{\mathcal{Z}}_V^\mathrm{tot} (m)  \cdot  q^m  \in \widehat{\CH}^1( \bar{\mathcal{S}}_V ,\mathscr{D}_\BKK ) [[q]] 
\]
is, at least when $n\ge 2$, a modular form.  More precisely, it is the $q$-expansion of an element 
\[
 \widehat{\phi}_V (\tau) \in
M_n(\Gamma_0(D) , \eps^n) \otimes \widehat{\CH}^1( \bar{\mathcal{S}}_V ,\mathscr{D}_\BKK ).
\]
Here $\bar{\mathcal{S}}_V$ is the canonical toroidal compactification of $\mathcal{S}_V$,
$\bar{\mathcal{Z}}_V(m) \subset \bar{\mathcal{S}}_V$ is the Zariski closure of $\mathcal{Z}_V(m)$, 
$\mathcal{B}_V(m)$ is an explicit linear combination of boundary components, and $\Phi_V(m)$ is a Green function for $ \bar{\mathcal{Z}}_V(m) + \mathcal{B}_V(m)$ constructed using the theory of  regularized theta lifts.  The constant term is defined by first endowing $\taut_V$ with a particular pre-log singular hermitian metric, and then setting
\[
\widehat{\mathcal{Z}}_V^\mathrm{tot} (0) =  - \widehat{\taut}_V + ( \mathrm{Exc}_V  , -\log(D))  .
\]
The second term on the right is the exceptional divisor endowed with the constant Green function $-\log(D)$.

Given that the above generating series is modular, the arithmetic divisor
\begin{equation}\label{introK}
\widehat{\tautmod}_V \define 2 \widehat{\taut}_V - (\mathrm{Exc}_V,0) \in 
\widehat{\CH}^1( \bar{\mathcal{S}}_V ,\mathscr{D}_\BKK )
\end{equation}
determines a modular form
\begin{equation}\label{intro second form}
\widehat{\deg} ( \widehat{\phi}_V (\tau)  \cdot \widehat{\tautmod}_V^{n-1} ) =
\sum_{m\ge 0}   \widehat{\deg} ( \widehat{\mathcal{Z}}_V^\mathrm{tot} (m)  \cdot \widehat{\tautmod}_V^{n-1} ) \cdot  q^m  .
\end{equation}
Here $\widehat{\tautmod}_V^{n-1}$ is the $(n-1)$-fold iterated arithmetic intersection.
The peculiar choice of \eqref{introK} is motivated by the fact 
(Theorem \ref{thm:taut-hodge compare}) that it has trivial arithmetic intersection with every irreducible component of  the exceptional divisor $\mathrm{Exc}_V$, which simplifies the calculation of the coefficients of \eqref{intro second form}.

\begin{bigconjecture}\label{intro conjecture}
If $n>2$, we have the equality of modular forms
\begin{align}\lefteqn{
\widehat{\deg}(  \widehat{\phi}_V (\tau) \cdot \widehat{\tautmod}_V^{n-1} )  } \nonumber  \\
  & \stackrel{?}{=}
  - 
   \mathrm{vol}_\C (  \widehat{\tautmod}_V )
 \sum_{r\mid D}\left(  \frac{\mathbf{A}'_V (0) }{\mathbf{A}_V (0)}     +  \frac{3}{2}  \log(D) -\log(r)   \right)    \gamma_r E_r  (\tau)  ,\label{intro degree series} 
  \end{align}
 in which 
 \[
 E_r(\tau) \in M_n(\Gamma_0(D) , \eps^n)
 \]
  is the Eisenstein series of Proposition \ref{prop:eisenstein formulas}, and $\gamma_r$ is the $4^\mathrm{th}$ root of unity (depending on $V$) defined in Proposition \ref{prop:multi eisenstein}.
\end{bigconjecture}

Conjecture \ref{intro conjecture} is stated in an equivalent form as Conjecture \ref{conj:arithmetic degrees}; the equivalence of the two statements follows from Corollary \ref{cor:arithmetic eis}.
Because the coefficients of the Eisenstein series $E_r(\tau)$ are given by explicit formulas, 
Conjecture \ref{intro conjecture} provides us with (conjectural) explicit formulas for the arithmetic intersection multiplicities appearing as the coefficients in \eqref{intro second form}.  
These can be found in \S \ref{ss:volume corollaries},  especially \eqref{degree conjecture 2}.

The following is our main evidence for Conjecture \ref{intro conjecture}.

\begin{bigtheorem}\label{thm:intro many coefficients}
If  $n> 2$, the  modular forms on either side of \eqref{intro degree series} have the same  $m^\mathrm{th}$ Fourier coefficient  for all  positive $m\in \Z$ with $\eps(m)=1$.  Moreover, their difference is a cusp form.
\end{bigtheorem}

The final claim of Theorem \ref{thm:intro many coefficients} allows us to equate the constant terms on both sides of  \eqref{intro degree series}, yielding
\[
\widehat{\deg}( \widehat{\mathcal{Z}}_V^\mathrm{tot} (0)  \cdot \widehat{\tautmod}_V^{n-1} )
=
 - 
   \mathrm{vol}_\C (  \widehat{\tautmod}_V )
\left(  \frac{\mathbf{A}'_V (0) }{\mathbf{A}_V (0)}     +  \frac{3}{2}  \log(D)    \right)  .
\]
This equality is equivalent to the second claim in the following theorem, which appears in the text as 
Theorems \ref{thm:degree induction} and \ref{thm:K volume}.

\begin{bigtheorem}\label{thm:intro K volume}
The arithmetic divisor  \eqref{introK} has complex volume 
\[
\vol_\C(  \widehat{\tautmod}_V )  = 
\frac{|\mathrm{CL} (\kk)|}{    2^{o(D) - n }    |\co_\kk^\times|     } \cdot  \mathbf{A}_V (0) ,
\]
and  arithmetic volume 
\[
 \widehat{\vol}( \widehat{\tautmod}_V)  
=  \left(  2\frac{\mathbf{A}'_V (0) }{\mathbf{A}_V (0)}     +   \log(D)   \right)
 \vol_\C(  \widehat{\tautmod}_V).
\]
\end{bigtheorem}

In practice we do not deduce the second claim of Theorem \ref{thm:intro K volume} from the cuspidality claim of Theorem \ref{thm:intro many coefficients}.  Instead, the proofs of the two statements are intertwined in a rather complicated way.  

In the text, Theorem \ref{thm:intro main} is deduced from Theorem \ref{thm:intro K volume} using the relations between 
  \eqref{introHodge} and  \eqref{introK}  found in   Theorem \ref{thm:taut-hodge compare}.
This latter theorem also allows us to deduce formulas for the complex and arithmetic volumes of $\widehat{\taut}$ from Theorem \ref{thm:intro K volume}.   
These are found in  Theorem \ref{thm:degree induction} and Corollary \ref{cor:other volumes}, but the formula for the arithmetic volume is somewhat complicated by correction terms coming from the components of the exceptional divisor $\mathrm{Exc}_V \subset \mathcal{S}_V$.

We end this section with some  remarks about the restriction $\eps(m)=1$ in Theorem \ref{thm:intro many coefficients}.  In the proof of that theorem, we first prove the equality for all primes $p$ satisfying $\eps(p)=1$.  The equality for all $m$ with $\eps(m)=1$ then follows from  a result of Bella\"{i}che, found in Appendix \ref{appendix:bel}.

When $\eps(p)=1$,  the strategy for the calculation of the coefficients of \eqref{intro second form} is to first relate the Kudla-Rapoport divisor $\mathcal{Z}_V(p)$ to a Shimura variety  $\mathcal{S}_{V^\flat}$ in one dimension lower, where $V^\flat \subset V$ is the orthogonal complement to a vector of hermitian norm $p$.  This is done in  Chapter \ref{s:KR divisors}.
The assumption  $\eps(p)=1$ guarantees that the hermitian space $V^\flat$ is relevant (i.e.~contains a self-dual lattice), just as $V$ was assumed to be, and so by  induction one knows the analogue of Theorem \ref{thm:intro K volume} with $V$ replaced by $V^\flat$.  
This formula for the arithmetic volume of $\widehat{\tautmod}_{V^\flat}$ on $\bar{\mathcal{S}}_{V^\flat}$  then gives a partial formula (Lemma \ref{lem:second volume lemma}) for the $p^\mathrm{th}$ coefficient in \eqref{intro second form}.  Combining this partial formula with the modularity of the generating series $\widehat{\phi}_V(\tau)$ then proves Theorems \ref{thm:intro many coefficients} and \ref{thm:intro K volume} simultaneously.

In contrast, when $\eps(p)=-1$ the Kudla-Rapoport divisor $\mathcal{Z}_V(p)$ should be related to the Shimura variety defined by a  hermitian space $V^\flat$ that is not relevant in the sense above.
Thus  carrying out the above inductive procedure when $\eps(p)=-1$ would require working with Shimura varieties whose moduli interpretation involves non-principal polarizations, but the results of  \cite{BHKRY-1}, and in particular the  modularity of $\widehat{\phi}_V(\tau)$, are not known in this generality.

In summary, an extension of  the modularity results of \cite{BHKRY-1} to Shimura  varieties associated to non-relevant hermitian spaces seems to be the missing ingredient needed to   remove the restriction $\eps(m)=1$ in Theorem \ref{thm:intro many coefficients}.


\section{Connections with the Kudla-Rapoport conjecture}

Theorem \ref{thm:intro K volume} is closely related to a conjecture of Kudla-Rapoport \cite{KRunitaryII} on the degrees of arithmetic  $0$-cycles on integral models of unitary Shimura varieties.  
Briefly, the Shimura variety $\mathcal{S}_V$ is  endowed not just with the family of Kudla-Rapoport divisors of the previous section, but with special cycles in all codimensions.
In particular, it carries a family of arithmetic $0$-cycles $\widehat{\mathcal{Z}}_V(T)$  indexed by $n\times n$ hermitian matrices $T$.  
The conjecture of Kudla-Rapoport predicts that the arithmetic degrees of these $0$-cycles should agree with the Fourier coefficients of the derivative of an Eisenstein series on the quasi-split unitary group $\Uni(n,n)$.

For those $0$-cycles  with $\det(T)\neq 0$, the Kudla-Rapoport conjecture is now a theorem of Li-Zhang \cite{LZ} and He-Li-Shi-Yang \cite{HLSY}, who prove  the unramified and ramified cases, respectively.
The relevance of Theorem \ref{thm:intro K volume} is to the degenerate  cases in which $\det(T)=0$, and especially to the most degenerate case $T=0$.  
In this case the associated arithmetic $0$-cycle is, up to some correction factors at the primes $p\mid D$,  the $n$-fold iterated intersection 
of  the  hermitian line bundle $\widehat{\tautmod}_V$ (or $\widehat{\taut}_V$) on $\mathcal{S}_V$ appearing above.
In other words, the Kudla-Rapoport conjecture predicts that the arithmetic volume of $\widehat{\tautmod}_V$ is essentially the constant term of the derivative of an Eisenstein series.

We expect that the  proof of Theorem \ref{thm:intro main} makes accessible more degenerate cases of the Kudla-Rapoport conjecture,  beyond the most degenerate case  $T=0$. For example,  consider  a matrix of the form
\[
T = \begin{pmatrix}
T' &  \\
 & 0_{n-d}
\end{pmatrix}
\]
with $T'$ a nonsingular  $d\times d$ hermitian matrix with $m=\det(T')$ a squarefree product of primes split in $\kk$.
One could hope to relate the arithmetic degree of $\widehat{\mathcal{Z}}_V(T)$   to the arithmetic intersection number
\[
\widehat{\deg}(   \widehat{\mathcal{Z}}^\mathrm{tot}_{V^\flat}(m) \cdot \widehat{\tautmod}_{V^\flat}^{n-d} ) 
\]
appearing in \eqref{intro second form},  where $V^\flat$ is a hermitian space of signature $(n-d,1)$.
Theorem \ref{thm:intro many coefficients} might then be used to related these arithmetic degrees to degenerate coefficients of a $\mathrm{U}(n,n)$ Eisenstein series.

We hope to return to these questions in future work.  


\section{Outline of the paper}


Fix a  self-dual $\co_\kk$-lattice $L\subset V$ in a $\kk$-hermitian space of signature $(n-1,1)$.

We construct in  Chapter \ref{s:eisenstein}  an Eisenstein series  $E_L(\tau,s,n)$ of weight $n$, valued in a finite dimensional representation of $\SL_2(\Z)$.
Its Fourier coefficients  can be expressed in terms of local representation densities of the lattice $L$, which we then compute in order to make the  coefficients completely explicit.   
Evaluating at the non-central point $s_0=(n-1)/2$ yields a holomorphic Eisenstein series, which we express in terms of the Eisenstein series $E_r(\tau)$ appearing in Conjecture \ref{intro conjecture}.

In Chapter \ref{s:green integrals}  we introduce a complex Shimura variety $\mathrm{Sh}_K(H,\mathcal{D})$ associated to the unitary group $H=\Uni(V)$.  It carries a family of special divisors $Z(m)$ indexed by positive $m\in \Z$, and a metrized line bundle $\widehat{\taut}$.  
To a harmonic Maass form $f$ of weight $2-n$ we attach a divisor $Z(f)$, defined as a linear combination of $Z(m)$'s.  
We then construct a Green function $\Phi(f)$ for $Z(f)$ using the machinery of regularized theta lifts, and show in  Theorem \ref{thm:int} that the normalized integral
\[
 \mathrm{vol}_\C (\widehat{\taut})^{-1} \int_{\mathrm{Sh}_K(H,\mathcal{D}) } \Phi(f) \chern(\widehat{\taut})^{n-1} 
\]
can be expressed in terms of the Fourier coefficients of the derivative of $E_L(\tau,s,n)$ at $s=s_0$.  
Here $c_1$ denotes the first Chern form of a metrized line bundle, as in Remark \ref{rem:chern form}.
The proof follows and generalizes ideas of Kudla \cite{Ku:Integrals}, who computed similar integrals on orthogonal Shimura varieties.

In Chapter \ref{s:integral models} we work with $\kk$-hermitian spaces $W_0$ and $W$ of signatures $(1,0)$ and $(n-1,1)$.   The Shimura varieties associated to their similitude groups  have moduli interpretations, which  can be used to construct regular integral models $\mathcal{M}_{W_0}$ and $\mathcal{M}_W$.
We recall  the definition of these integral models, their toroidal compactifications,  and the structure of their reductions at $p\mid D$.  Then we recall those aspects of the Gillet-Soul\'e  arithmetic intersection theory (as extended by Burgos-Kramer-K\"uhn \cite{BKK}) that are needed in the sequel.

In Chapter \ref{s:special shimura bundles} we choose $W_0$ and $W$ so that 
$
V \iso \Hom_\kk(W_0,W) 
$
as hermitian spaces.   This allows us to define the open and closed substack
\[
\mathcal{S}_V \subset \mathcal{M}_{W_0} \times_{\co_\kk} \mathcal{M}_W
\]
appearing in \S \ref{ss:intro generating}.  It is related to the unitary Shimura variety  of Chapter \ref{s:green integrals}  by a finite cover
\[
\mathcal{S}_V(\C) \to \mathrm{Sh}_K(H,\mathcal{D}).
\]
We  define  the  Kudla-Rapoport divisors  on $\mathcal{S}_V$, and the  metrized line bundle $\widehat{\taut}_V$.
The universal abelian scheme over $\mathcal{M}_W$ pulls back to an abelian scheme $A \to \mathcal{S}_V$, which determines a metrized Hodge bundle on $\mathcal{S}_V$.   
The main result of Chapter \ref{s:special shimura bundles} is Theorem \ref{thm:taut-hodge compare}, which explains the precise connection between $\widehat{\taut}_V$ and the metrized Hodge bundle.
Going back and forth between these  line bundles and \eqref{introK}, which for our purposes is the best behaved of the three, is essential to our methods.

In Chapter \ref{s:KR divisors} we prove Theorems \ref{thm:height descent 1} and \ref{thm:height descent 2}, which are essential to the induction arguments used to prove  Theorems \ref{thm:intro many coefficients} and \ref{thm:intro K volume}.   
The idea is that if $p$ is a prime split in $\kk$, the Kudla-Rapoport divisor $\mathcal{Z}_V(p)$ should be closely related to a Shimura variety $\mathcal{S}_{V^\flat}$ defined in the same way as $\mathcal{S}_V$, but with $\mathrm{dim}(V^\flat)= \mathrm{dim}(V)-1$.   
This is true, up to some error terms coming from divisors supported in characteristics dividing $pD$.
Using such a relation, we are able to express  the height of the divisor $\mathcal{Z}_V(p)$ with respect to $\widehat{\tautmod}_V$ in terms of the  arithmetic volume of the metrized line bundle $\widehat{\tautmod}_{V^\flat}$ on the lower-dimensional Shimura variety $\mathcal{S}_{V^\flat}$.  
Because of the many subtleties in the moduli problem defining the integral models, this is perhaps the most technical part of this work.

In Chapter \ref{s:borcherds} we recall the modularity results of \cite{BHKRY-1}.
We also prove Theorem \ref{thm:no boundary heights}, which can be summarized as saying that the boundary of the toroidal compactification $\bar{\mathcal{S}}_V$ can, for our purposes, be ignored.

In Chapter \ref{s:volumes}, we put everything together to prove the main results.


\section{Notation and conventions}


Throughout the paper,  $\kk \subset \C$ is a quadratic imaginary field  of discriminant $-D$ whose associated  associated quadratic character is denoted 
\[
\eps: (\Z/D\Z)^\times \to \{\pm 1\}.
\]

The symbol   $V$   always denotes a $\kk$-hermitian space of signature $(n-1,1)$ with $n\ge 1$, and the hermitian form is denoted by $\langle - , - \rangle$.
Hermitian forms are always linear in the first variable and conjugate-linear in the second, and are assumed to be nondegenerate.

Beginning in Chapter \ref{s:integral models}, and continuing for the rest of the paper, we assume that $D$ is odd and  the hermitian space admits an $\co_\kk$-lattice $L\subset V$ that is \emph{self-dual}, in the sense that 
\begin{equation}\label{self-dual lattice}
L = \{ x \in V :  \langle  x , L \rangle \subset \co_\kk \}.
\end{equation}
These restrictions are not imposed in  Chapter \ref{s:eisenstein} and Chapter \ref{s:green integrals} unless stated explicitly.

The term \emph{stack}  means separated Deligne-Mumford stack.


\section{Acknowledgements}


We thank Joel Bella\"{i}che for his help with Appendix \ref{appendix:bel},
and Dick Gross for pointing out Remark \ref{rem:gross motive}.
We also thank the anonymous referee for a careful reading of the manuscript, and for many helpful comments.


\chapter{Eisenstein series and theta functions}
\label{s:eisenstein}


Fix  a $\kk$-hermitian space $V$ of signature $(n-1,1)$ with $n \ge 1$, and denote the hermitian form by $\langle-,-\rangle$.
After reviewing some basics of  Eisenstein series, theta functions, and the Siegel-Weil formula,
we attach a particular Eisenstein series to a lattice $L \subset V$, and  express  its Fourier coefficients in terms of representation densities.
We then  compute these representation densities in the case of a self-dual lattice, and so obtain explicit formulas for the Fourier coefficients; these formulas have a different shape depending on whether $n$ is even  or odd.  


\section{A seesaw dual reductive pair}
\label{ss:seesaw}


Let $V_0$ be the unique symplectic space over $\Q$ of dimension $2$, and denote the symplectic form by $\langle x,y\rangle_0$. 
Set $V_{0\kk}=V_0\otimes_\Q\kk$,  and extend the symplectic form $\kk$-linearly in the first argument and $\kk$-conjugate-linearly in the second argument. This defines a skew-hermitian form on $V_{0\kk}$.

The   $\Q$-vector space underlying $V$ carries a $\Q$-bilinear form
\begin{equation}\label{Q bilinear}
[  x,y ] =\tr_{\kk/\Q} \langle x, y\rangle
\end{equation}
of signature $(2n-2,2)$, with associated quadratic form
\begin{equation}\label{Q quadratic}
Q(x) =  \langle x, x\rangle .
\end{equation}
This  data determine  a seesaw dual reductive pair 
\begin{align*}
\xymatrix{
G=\Uni(V_{0\kk}) \ar@{-}[d] \ar@{-}[dr]& H' =\Orth(V)  \ar@{-}[d]\\
G' =\operatorname{Sp}(V_0) \ar@{-}[ur]& H=\Uni(V).
}
\end{align*}
In particular, there are compatible Weil representations of $G(\A)\times H(\A)$ and $G'(\A)\times H'(\A)$ on the space $S(V(\A))$ of Schwartz-Bruhat functions on $V(\A)$.  This will be made explicit in \S \ref{ss:siegel-weil}.

Let $e,f$ be a basis of $V_0$ with Gram matrix
$
\kzxz{0}{1}{-1}{0}, 
$
and use this to identify 
\[
G \iso  \left\{ g\in \operatorname{Res}_{\kk/\Q} \GL_2  : g\zxz{0}{1}{-1}{0} {}^t\bar g= \zxz{0}{1}{-1}{0}\right\} 
\]
and $G' \iso \SL_2$.
Let $P=NM$ be the parabolic subgroup of $G$ with Levi factor 
\[
M=\left\{m(a)= \zxz{a}{0}{0}{\bar a^{-1}}  :  a\in\operatorname{Res}_{\kk/\Q} \mathbb{G}_m\right\}
\]
and unipotent radical 
\[
N=\left\{ n(b)=\zxz{1}{b}{0}{1}  :  b\in \mathbb{G}_{a} \right\}.
\]

For a place $p\leq \infty$ of $\Q$, let $K_p \subset G(\Q_p)$ be the maximal compact subgroup defined by
\[
K_p= \begin{cases} G(\Q_p) \cap \GL_2(\co_{\kk,p}),&\text{if $p<\infty$,}\\
G(\R) \cap \Uni(2,\R),&\text{if $p=\infty$,}
\end{cases}
\]
and put $K=\prod_{p\leq\infty} K_p$. 
If we let  $\kk^1$ be the torus of norm one elements in  $\kk^\times=\mathrm{Res}_{\kk/\Q}\mathbb{G}_m$, there is a natural homomorphism
\[
G' \times\kk^1\map{ (g,a)\mapsto  g \cdot m(a)  }   G ,
\]
and the image of $\kk^1$ is the center of $G$.
This homomorphism    induces an isomorphism
\[
\SO(2,\R)\times \Uni(1,\R) /\{\pm 1\}\cong K_\infty.
\]

Let  $\H$ be the complex upper half-plane. 
The first isomorphism  in 
\[
\H\iso \SL_2(\R)/\SO(2,\R)\iso G(\R)/K_\infty .
\]
sends   $\tau=u+iv\in \H $ to the element 
\begin{equation}\label{gtau}
  g_\tau =  n(u)m(v^{1/2})  \in \SL_2(\R) 
\end{equation}
satisfying $g_\tau \cdot  i =\tau$.  The second is induced $\SL_2(\R) = G'(\R) \subset G(\R)$.


\section{Eisenstein series} 


Given an $s\in \C$ and a character $\chi : \A_\kk^\times/\kk^\times \to \C^\times$,  let 
\[
I(s,\chi)=
\operatorname{Ind}_{P(\A)}^{G(\A)}(\chi |\cdot|_{\A_\kk}^s)
\]
 be the induced representation, realized on the space of smooth $K$-finite functions $\Phi$ on $G(\A)$  satisfying
\[
\Phi(n(b)m(a)g,s)= \chi(a) |a|_{\A_\kk}^{s+\frac{1}{2}}\Phi(g,s)
\] 
for $b\in \A$ and  $a\in \A_\kk^\times$. 
Here  $|a|_{\A_\kk}$ denotes the norm on $\A_\kk^\times$. 
In particular, at the archimedian place we have the normalized absolute value $|a|_{\infty}=a\bar a$. 

Recall that a section  of $I(s,\chi)$ is called \emph{standard} if its restriction to $K$ is independent on $s$.
For a standard section $\Phi$ of $I(s,\chi)$, the associated Siegel Eisenstein series 
\[
E(g,s,\Phi) = \sum_{\gamma\in P(\Q)\backslash G(\Q)} \Phi(\gamma g, s)
\]
converges absolutely for $\operatorname{Re}(s)> 1/2$, and  defines an automorphic form on $G(\A)$.
The Eisenstein series has  meromorphic continuation to all $s\in \C$.
We now describe this Eisenstein series in more classical terms, by restricting it to the subgroup $\SL_2\iso G'\subset G$.

%
%

Assume that $\Phi(s)=\Phi_\infty(s)\otimes \Phi_f(s)$  is a factorizable standard section  with $\Phi_f=\otimes_{p<\infty} \Phi_p(s)$, and that $\Phi_\infty$ is a normalized standard section of weight $\ell\in \Z$.
In other words, $\Phi_\infty(1,s)=1$, and for all 
\[
 k=\zxz{a}{b}{-b}{a} \in \SO(2,\R) \subset K_\infty
\]
we have $\Phi_\infty(gk,s)=\Phi_\infty(g,s)  \cdot \underline{k}^\ell$,  where 
\begin{equation}\label{underline k}
\underline{k} = a+ib \in \C^\times .
\end{equation}

Abbreviate $\Gamma=\SL_2(\Z) \subset G(\A)$,  and let $\Gamma_\infty=P(\Q)\cap \Gamma$ be the subgroup of   upper triangular matrices.

\begin{lemma}
\label{lem:eis1}
If $\Phi_\infty$ is a normalized standard section of weight $\ell\in \Z$ as above, then 
\begin{align*}
E(g_\tau,s,\Phi)
&= j(g_\tau,i)^{-\ell}  \sum_{\gamma\in \Gamma_\infty\backslash \Gamma}
 \operatorname{Im}(\gamma \tau)^{s+ \frac{1}{2}-\frac{\ell } {2 }  } \cdot j(\gamma,\tau)^{-\ell}\cdot \Phi_f(\gamma).
\end{align*}
Here $g_\tau \in \SL_2(\R)$ is as in \eqref{gtau}, and  $j(\gamma,\tau)=c\tau+d$  the usual automorphy factor
associated to   \[ \gamma= \begin{pmatrix} a& b \\ c & d\end{pmatrix} \in \SL_2(\R). \]
\end{lemma}

\begin{proof}
For $\gamma \in \Gamma$ we factor
$\gamma g_\tau = n(\beta)m(\alpha) k$
with $\beta\in \R$, $\alpha\in \R^+$, and $k\in \SO(2,\R)$. Then a computation shows that 
\begin{align*}
\alpha = \frac{v^{1/2}}{|c\tau +d|}, \qquad 
\underline k = \frac{c\bar \tau +d}{|c\tau +d|},
\end{align*}
and hence
\[
\Phi_\infty(\gamma g_\tau,s) = 
v^{\ell/2}\cdot \frac{v^{s+ \frac{1}{2}-\frac{\ell}{2}  }  }  {|c\tau+d|^{2s+1-\ell}} \cdot (c\tau+d)^{-\ell}.
\]

The  inclusions $\Gamma\subset \SL_2(\Q)\subset G(\Q)$ induce  bijections 
\[
\Gamma_\infty\backslash \Gamma \iso  (P(\Q)\cap \SL_2(\Q))\backslash \SL_2(\Q)   \iso P(\Q)\backslash G(\Q),
\]
and hence 
\begin{align*}
E(g_\tau,s,\Phi) 
& = \sum_{\gamma\in \Gamma_\infty\backslash \Gamma} \Phi_\infty(\gamma g_\tau,s)\Phi_f(\gamma) \\
&= 
v^{\ell/2} \sum_{\gamma\in \Gamma_\infty\backslash \Gamma}
 \operatorname{Im}(\gamma \tau)^{s+  \frac{ 1-\ell }{ 2 } } \cdot (c\tau+d)^{-\ell}\cdot \Phi_f(\gamma)\\
&=
j(g_\tau,i)^{-\ell}  \sum_{\gamma\in \Gamma_\infty\backslash \Gamma}
 \operatorname{Im}(\gamma \tau)^{s+  \frac{ 1-\ell }{ 2 } }  \cdot j(\gamma,\tau)^{-\ell}\cdot \Phi_f(\gamma),
\end{align*}
as desired.
\end{proof}


\section{The Siegel-Weil formula}
\label{ss:siegel-weil}


Write $\psi$ for the standard additive character of $\A/\Q$, satisfying  $\psi_\infty(x)=e^{2\pi i x}$. 
Let 
\[
\eps_\A : \A^\times /\Q^\times \to \{\pm 1\}
\]
 be the idele class character determined by the quadratic extension $\kk/\Q$, and fix a character  
 \[
  \chi :   \A_\kk^\times/ \kk^\times \to \C^\times
  \]
   such that $\chi |_{\A^\times}=\eps_\A^n$.

 As explained in \cite{HKS} and \cite{Ic2}, the choices of $\psi$ and $\chi$ determine a Weil representation 
$\omega=\omega_{\psi,\chi}$ of the group $G(\A)\times H(\A)$ on the space of Bruhat-Schwartz functions $S(V(\A))$. 
Recalling the $\Q$-bilinear form \eqref{Q bilinear} and associated quadratic form  \eqref{Q quadratic}, the action is given  by the formulas 
\begin{align*}
\omega(m(a))\varphi (x) &= \chi(a) |a|_{\A_\kk}^{n/2} \varphi(xa),\\
\omega(n(b))\varphi(x) & = \psi( bQ(x) )\varphi(x),\\
\omega(\zxz{0}{1}{-1}{0})\varphi(x) &= \int_{V(\A)} \varphi(y) \psi( [  x,y] )\,dy,\\
\omega(h)\varphi(x) & = \varphi(h^{-1} x)
\end{align*}
for $\varphi\in S(V(\A))$ and $x\in V(\A)$. Here $a\in \A_\kk^\times$, $b\in \A$, $h\in H(\A)$, and the Fourier transform is taken with respect to the self-dual measure $dy$ on $V(\A)$.
To a Schwartz function $\varphi$ there is an associated theta function
\begin{equation}\label{theta kernel}
\theta(g,h,\varphi)=  \sum_{x\in V(\Q)} \omega(g,h)\varphi(x).
\end{equation}

Abbreviating $s_0=(n-1)/2$, there is a $G(\A)$-intertwining operator 
\begin{align*}
\lambda:S(V(\A))\to I(s_0,\chi) 
\end{align*}
defined by $\lambda(\varphi)(g)= (\omega(g,1)\varphi)(0)$.
We extend $\lambda(\varphi)$ to a standard section of $I(s,\chi)$ by setting
\[
\lambda(\varphi)(g,s)= |a(g)|_{\A_\kk}^{s-s_0}(\omega(g,1)\varphi)(0).
\]
The following theorem is a case of the Siegel-Weil formula;  see \cite{We2} or  Theorem 1.1 of \cite{Ic2}.

\begin{theorem}
\label{thm:sw}
Assume that $n>2$,  or that $V$ is anisotropic. For $\varphi\in S(V(\A))$ and $g\in G(\A)$ we have 
\[
 E(g,\lambda(\varphi),s_0) =\kappa  \int_{H(\Q)\backslash H(\A)} \theta(g,h,\varphi)\, dh ,
\]
where  Haar measure  is normalized by  $\vol(H(\Q)\backslash H(\A))=1$, and 
\begin{equation}\label{eq:kappa}
\kappa = 
\begin{cases}
1 & \mbox{if }n>1 \\
2 & \mbox{if }n=1. 
\end{cases}
\end{equation}
\end{theorem}



\section{A special Eisenstein series}
\label{ss:lattice eisenstein}


We consider Eisenstein series constructed from  particular Schwartz functions on $V(\A)$. 
At the archimedian place we take the Gaussian associated to a majorant and at the non-archimedian places we take the characteristic function of a coset of an $\co_\kk$-lattice.  

Let $L \subset V$ be an $\co_\kk$-lattice on which the hermitian form is $\co_\kk$-valued.
Its dual lattice  under the $\Q$-bilinear form \eqref{Q bilinear} is  denoted
\begin{equation}\label{dual lattice}
L'  \define    \{ x\in V  :  [  x,L]  \subset \Z\} .
\end{equation}
  For $\mu\in L'/L$  let 
  \[
  \varphi_\mu=\operatorname{char}(\mu+\hat L)  \in S(V(\A_f))
  \]
   be the characteristic function of $\mu+\hat L\subset V(\A_f)$, and denote 
by $S_L$ the subspace of $S(V(\A_f))$ generated by all $\varphi_\mu$ with  $\mu\in L'/L$. 
Hence we may identify
\[
S_L\iso  \C[L'/L].
\]

The restriction to $\Gamma=\SL_2(\Z)$ of the Weil representation  of $\SL_2(\A_f)\subset G(\A_f)$ takes the subspace 
$S_L$ to itself, giving rise to a representation 
\begin{align}
\label{eq:ol}
\omega_L:\Gamma\to \Aut(S_L).
\end{align}
The corresponding dual representation   
\[
\omega_L^\vee:\Gamma\to \Aut(S_L^\vee)
\]
is given by $\omega_L^\vee(\gamma)(f)=f\circ \omega_L^{-1}(\gamma)$ for $f\in S_L^\vee$.
On the space $S_L$ we also have the conjugate representation $\bar\omega_L$ defined by
\[
\bar\omega_L(\gamma)(\varphi)= \overline{\omega_L(\gamma)(\bar\varphi)}
\]
for $\varphi\in S_L$. 
If we use the standard $\C$-bilinear pairing  
\begin{equation}\label{S_L pairing}
\Big\langle \sum_\mu a_\mu\varphi_\mu, \sum_\mu b_\mu\varphi_\mu \Big\rangle = \sum_\mu a_\mu b_\mu  
\end{equation}
to identify $S_L\iso S_L^\vee$ as $\C$-vector spaces, then $\bar{\omega}_L = \omega_L^\vee$, 
and this agrees with  the representation denoted $\rho_L$  in \cite{Bo1}, \cite{Br1}, and \cite{BY1}.

For $\ell\in \Z$ and $\tau\in \H$, we define an $S_L$-valued Eisenstein series 
\begin{align}
\label{eq:eisl}
E_L(\tau,s,\ell)= \sum_{\gamma\in \Gamma_\infty\backslash \Gamma}
 \operatorname{Im}(\gamma \tau)^{s+\frac{1-\ell}{2}}  (c\tau+d)^{-\ell}\cdot \bar\omega_L(\gamma)^{-1}\varphi_0
\end{align}
 of weight $\ell$ and representation $\bar \omega_L$.
Note that this Eisenstein series has a slightly different normalization than the Eisenstein series in equation (2.17) of \cite{BY1}, since it is adapted to the unitary setting.

Recall  the Maass lowering and raising operators in weight $\ell$  defined by
\[
L_\ell =-2iv^2\frac{\partial}{\partial \bar \tau}, \qquad 
R_\ell =2i\frac{\partial}{\partial \tau}+\ell v^{-1}.
\]
They lower (respectively raise) the weight of an automorphic form by $2$.
It is easily seen that
\begin{align}
\label{eq:leis}
L_\ell
E_L(\tau,s, \ell) &= \left(s+ \frac{1-\ell}{2} \right) E_L(\tau,s, \ell-2) \\
R_\ell E_L(\tau,s, \ell) &= \left(s+\frac{1+\ell}{2}\right) E_L(\tau,s,\ell+2). \nonumber
\end{align}
The Eisenstein series $E_L(\tau,s,\ell)$ is an eigenform of the hyperbolic Laplacian $\Delta_\ell$ of weight $\ell$ (normalized as in (3.1) of \cite{BF}) with eigenvalue $(\frac{\ell-1}{2})^2-s^2$.
The following is a consequence of \eqref{eq:leis}.

\begin{lemma}
\label{lem:eisl}
Assume that $n>2$, or that $V$ is anisotropic.
The  Eisenstein series $E_L(\tau,s_0,n)$  of weight $n=\mathrm{dim}(V)$ at   $s_0=(n-1)/2$ is  holomorphic  in $\tau$.
Its derivative 
\[
E_L'(\tau,s_0,n) \define \frac{\partial}{\partial s} E_L'(\tau,s,n)\big|_{s=s_0}
\]
 satisfies
\begin{align*}
L_n(E_L'(\tau,s_0,n)) & = E_L(\tau,s_0,n-2).
\end{align*}
\end{lemma}

\begin{remark}
The Eisenstein series $E_L(\tau,s,n-2)$ is coherent as it arises via the Siegel-Weil formula from the global hermitian space $V$. On the other hand, the Eisenstein series $E_L(\tau,s,n)$ is incoherent as it is associated to the incoherent collection of local hermitian spaces obtained by replacing $V_\infty$ by the positive definite hermitian space of dimension $n$ and keeping the non-archimedian part. In particular, at $s=0$, the center of symmetry of the functional equation, the Eisenstein series $E_L(\tau,s,n)$ vanishes identically. 
\end{remark}

\section{Coefficients of Eisenstein series}
\label{ss:coefficients}


When the weight of the  Eisenstein series \eqref{eq:eisl}  is $n=\mathrm{dim}(V)$, we denote its Fourier expansion by 
\begin{align}
\label{eq:fouriereis}
E_L(\tau,s,n)= \sum_{\mu\in L'/L}\sum_{m\in \Z+\langle \mu,\mu\rangle} B(m,\mu;s,v) q^m\varphi_\mu,
\end{align}
where $v = \mathrm{Im}(\tau)$ is the imaginary part of $\tau$.
We  summarize some well-known facts about the coefficients;  details can be found in \cite{BrKue} and \cite{KY}.

If $m\neq 0$, there is a factorization
\begin{align}
\label{eq:coeff}
B(m,\mu;s,v)= B(m,\mu,s)\cdot \mathcal{W}_m(s,v)
\end{align}
in which the non-archimedian contribution $B(m,\mu,s)$ is independent of $v$.
The second factor is the   archimedian Whittaker function 
\begin{align}
\label{eq:whittaker}
\mathcal{W}_m(s,v) = (4\pi |m| v)^{-n/2} e^{2\pi m v} \cdot W_{\sgn(m)\frac{n}{2}, s}(4\pi |m| v),
\end{align}
where  
\[
W_{\kappa,\mu}(z)= e^{-z/2}z^{1/2+\mu} U(\frac{1}{2}+\mu-\kappa, 1+2\mu,z)
\]
 denotes the classical confluent hypergeometric function as in \cite[Chapter 13]{AS}, and 
\begin{align}
\label{eq:uint}
U(a,b,z) = \frac{1}{\Gamma(a)}\int_0^\infty e^{-zt} t^{a-1} (1+t)^{b-a-1} \,dt
\end{align}
for $\mathrm{Re}(a)>0$ and $|\arg(z)|<\pi/2$.  

If $m=0$, the Fourier coefficient has the form 
\[
B(0,\mu;s,v)= \delta_{\mu,0}\cdot  v^{s+\frac{1}{2}-\frac{n}{2}}+ B(0,\mu,s)\cdot 
v^{-s+\frac{1}{2}-\frac{n}{2}},
\]
where $B(0,\mu,s)$ is independent of $v$, and $\delta_{\mu,0}$ is the Kronecker delta. 

\begin{remark}
Note that our normalization of the archimedian Whittaker function in \eqref{eq:whittaker} differs from the normalization in \cite[Proposition 2.3]{KY}. The normalizing factor $\Gamma(s+\frac{n}{2}+\frac{1}{2})^{-1}$ for $m>0$, and $\Gamma(s-\frac{n}{2}+\frac{1}{2})^{-1}$ for $m<0$, appears in loc.~cit. In our case this factor is included in the normalization of the $B(m,\mu,s)$, since this leads to slightly cleaner formulas in Theorem~\ref{thm:int}.
\end{remark}

The integral representation \eqref{eq:uint} can be used to  
describe the asymptotic behavior of the Whittaker function as $v\to \infty$, which is given by 
\begin{align*}
\mathcal{W}_m(s,v)= \begin{cases} 1+O(v^{-1}) &\mbox{if }m>0 \\
O(e^{-4\pi |m| v})  &\mbox{if }m<0, 
\end{cases}
\end{align*}
locally uniformly in $s$ and $m$. 

We will also require the following lemma on the value at $s_0=(n-1)/2$ of the Whittaker function and its derivative $\mathcal{W}_m'(s_0,v)=\frac{\partial}{\partial s}\mathcal{W}_m(s,v)\mid_{s=s_0}$.

\begin{lemma} 
\label{lem:wasymptotic}
The special value at $s_0$ of the Whittaker function is 
\begin{align*}
\mathcal{W}_m(s_0,v)= \begin{cases} 1 &\mbox{if } m>0 \\
\Gamma(1-n,4\pi |m|v) &\mbox{if }m<0 .
\end{cases}
\end{align*}
For $m>0$ we have 
\[
\mathcal{W}_m'(s_0,v)= 
\sum_{j=1}^{n-1} \binom{n-1}{j} \frac{\Gamma(j)}{(4\pi mv)^j}.
\]
In particular, 
\[
\mathcal{W}_m'(s_0,v)= O(v^{-1}), \quad \text{for $v\to \infty$.}
\]
\end{lemma}

\begin{proof}
We only prove this for $m>0$ and leave the similar case when $m<0$ to the reader.
According to \eqref{eq:whittaker} and \eqref{eq:uint}, we have for $m>0$ that 
\begin{align*}
\mathcal{W}_m(s,v) &=  \frac{(4\pi m v)^{s-s_0} }{\Gamma(s-s_0)} \cdot 
\int_0^\infty e^{-4\pi m v t} t^{s-s_0-1} (1+t)^{s+s_0} \, dt\\
&=  1
+\frac{(4\pi m v)^{s-s_0} }{\Gamma(s-s_0)} \cdot 
\int_0^\infty e^{-4\pi m v t} t^{s-s_0-1} ((1+t)^{s+s_0}-1) \,dt,
\end{align*}
where the latter integral converges absolutely for $\mathrm{Re}(s)>s_0-1$. We immediately see that $\mathcal{W}_m(s_0,v)=1$, and in addition that   
\begin{align*}
\mathcal{W}'_m(s_0,v) &= 
\int_0^\infty e^{-4\pi m v t} ((1+t)^{n-1}-1) \,\frac{dt}{t}\\
&= \sum_{j=1}^{n-1} \binom{n-1}{j} 
\int_0^\infty e^{-4\pi m v t} t^j \,\frac{dt}{t}\\
&= \sum_{j=1}^{n-1} \binom{n-1}{j} \frac{\Gamma(j)}{(4\pi mv)^j}.
\end{align*}
This gives the claimed formula and the bound for $v\to \infty$.
\end{proof}

We now turn to the explicit calculation of the factor $B(m,\mu,s)$ in \eqref{eq:coeff}, using  formulas from \cite{BrKue} and \cite{KY}. 
Fix $\mu \in L'/L$,  and let \[d_\mu=\min\{b\in \Z^+ :  b\mu=0\}\]  be the order of $\mu$.
For $m\in \Z+Q(\mu)$ and $a\in \Z$, we denote by $N_{m,\mu}(a)$ the modulo $a$ representation number
\begin{align*}
N_{m,\mu}(a)= \#\{ r\in L/aL:\; Q(r+\mu)\equiv m\pmod{a\Z}\}.
\end{align*}
For a prime $p$, set 
\begin{align}
\label{eq:lp}
L^{(p)}_{m,\mu}(p^{-s})= (1-p^{-s+2n-1}) \sum_{\nu=0}^\infty N_{m,\mu}(p^\nu) p^{-\nu s}
\end{align}
so that 
\[
\sum_{a=1}^\infty N_{m,\mu}(a) a^{- s} = \zeta(s-2n+1)\prod_p L^{(p)}_{m,\mu}(p^{-s}).
\]

\begin{proposition}
\label{prop:eis1}
For nonzero $m\in \Z+Q(\mu)$ we have 
\[
B(m,\mu,s)=-\frac{2^{n} \pi^{s+\frac{n}{2}+\frac{1}{2}}}{ \sqrt{|L'/L|}} \cdot \begin{cases}
\displaystyle
\frac{|m|^{s+\frac{n}{2}-\frac{1}{2}}}{\Gamma(s+\frac{n}{2}+\frac{1}{2})} 
\prod_{p} L^{(p)}_{m,\mu}(p^{-2s-n})
 & \mbox{if } m>0 \\[3ex]
\displaystyle
\frac{|m|^{s+\frac{n}{2}-\frac{1}{2}}}{\Gamma(s-\frac{n}{2}+\frac{1}{2})} 
\prod_{p} L^{(p)}_{m,\mu}(p^{-2s-n})
 & \mbox{if }m<0.
\end{cases}
\]
\end{proposition}

\begin{proof}
This is obtained from Proposition 3.2 of  \cite{BrKue} by noticing that our  Eisenstein series $E_L(\tau,s,n)$ is $1/2$  the Eisenstein series $E_0(\tau,s-\frac{n}{2}+\frac{1}{2})$ of \cite{BrKue} with $\kappa=n$ and $r=2n$.
\end{proof}

According to equation (3.20) of \cite{BrKue}, the Euler factors  $L^{(p)}_{m,\mu}(p^{-s})$ are polynomials in $p^{-s}$. 
Moreover, by (3.23) of \cite{BrKue}, if $p$ does not divide $d_\mu^2 m |L'/L|$ then
\[
L^{(p)}_{m,\mu}(p^{-s})=1-\chi_F(p) p^{n-1-s},
\]
where $\chi_F$ is the quadratic Dirichlet character associated to  the quadratic extension $\Q(\sqrt{(-1)^n |L'/L|})$. 
This implies that the Euler product in Proposition  \ref{prop:eis1} converges absolutely for $\operatorname{Re}(s)>0$ and that $B(m,\mu,s_0)$ is rational if  $n>1$ (so that $s_0=\frac{n-1}{2}$ lies in the region of convergence).
 In this case it is nonpositive for positive $m$ and vanishes for negative $m$.

Later we will require the following lemma on the behavior of the coefficients of the Eisenstein series in the limit $v\to \infty$.

\begin{lemma}
\label{lem:coeffasy}
Assume that $n>2$,  or that $V$ is anisotropic.
The coefficients $B'(m,\mu;s_0,v)$ of $E_L'(\tau,s_0,n)$ satisfy
\[
\lim_{v\to \infty}\big( B'(m,\mu;s_0,v ) -\kappa\delta_{m,0}\delta_{\mu,0}\log v\big)=\begin{cases}
B'(m,\mu,s_0) &\text{if $m>0$} \\
B'(0,\mu,s_0) &\text{if $m=0$ and $n=1$} \\
0 &\text{if $m=0$ and $n>1$} \\
0 &\text{if $m<0$.}
\end{cases}
\]
\end{lemma}

\begin{proof}
For $m\neq 0$ the assertion follows from Proposition \ref{prop:eis1} together with Lemma \ref{lem:wasymptotic} on the asymptotic behavior of the Whittaker function $\mathcal{W}_m(s,v)$ and its derivative at $s=s_0$.

For $m=0$, we consider the Laurent expansion of $B(0,\mu;s,v)$ at $s=s_0$, which is given by
\begin{align*}
&B(0,\mu;s,v)= \delta_{\mu,0}+ B(0,\mu,s_0)v^{-2s_0}\\
&\phantom{=}{}+\left( \delta_{\mu,0} \log(v) - B(0,\mu,s_0)\log(v) v^{-2s_0} + B'(0,\mu,s_0)v^{-2s_0}\right)(s-s_0) \\
&\phantom{=}{}+O((s-s_0)^2).
\end{align*}
If $n>1$, then the holomorphicity of $E_L(\tau,s_0,n)$ implies that $B(0,\mu,s_0)=0$. Since $2s_0>0$, the term involving $v^{-2s_0}$ vanishes in the limit $v\to \infty$,  and we obtain the assertion.

If $n=1$ then $s_0=0$ is the center of symmetry of the functional equation of the incoherent Eisenstein series, and hence  $E_L(\tau,0,n)$ vanishes. The vanishing of the constant Fourier coefficient  implies 
\[
B(0,\mu,0)=-\delta_{\mu,0},
\]
and consequently
\[
B'(0,\mu;0,v)= 2\delta_{\mu,0} \log(v) + B'(0,\mu,s_0).
\]
This concludes the proof of the lemma.
\end{proof}


\section{Self-dual lattices}
\label{ss:self-dual coefficients}


In this section  we assume  the $\co_\kk$-lattice $L\subset V$ is self-dual \eqref{self-dual lattice} under the hermitian form, and that $D=- \mathrm{disc}(\kk)$ is odd.
The first of these assumptions implies that   \eqref{dual lattice} satisfies  $L'= \mathfrak{d}_\kk^{-1} L$, where   $\mathfrak{d}_\kk \subset \co_\kk$ is the different.
We will determine the coefficient $B(m,\mu,s)$ from   \eqref{eq:coeff} more explicitly.

To compute the Euler factors $L^{(p)}_{m,\mu}(p^{-s})$ in Proposition \ref{prop:eis1}, we need to determine the representation numbers $N_{m,\mu}(p^\nu)$. 
We now  derive explicit formulas for these quantities  using finite Fourier transforms and formulas for certain lattice Gauss sums.   As we will only compute the coefficients for $\mu=0$, we abbreviate
$N_m(p^\nu)=N_{m,0}(p^\nu)$.

Fix an  $a\in \Z$.  For $c\in \Z^+$,  consider the Gauss sum
\[
G(a, c ) = \sum_{x\in \mathcal{O}_\kk/ c  \mathcal{O}_\kk} e\left(\frac{a\norm(x)}{c}\right).
\]
We are mainly interested in the case when $c=p^\nu$ is a prime power.

\begin{lemma}
\label{lem:gauss1}
Let $p^\nu$ be a prime power.
If we write  
$\gcd(a,p^\nu) =p^\alpha$ with $0\leq \alpha \leq \nu$ and $a=p^\alpha a'$ with $a'\in \Z$, then
\[
G(a,p^\nu)= \begin{cases}
p^{2\nu} &\text{if $\alpha=\nu$}\\
p^{\alpha+\nu}\leg{-D}{p}^{\nu-\alpha} &\text{if $\alpha <\nu$ and $p\nmid D$}\\
\epsilon_p p^{\alpha+\nu+\frac{1}{2}}\leg{a'}{p}\leg{D'}{p}^{\nu-\alpha-1} &\text{if $\alpha <\nu$ and $p\mid D$.}
\end{cases}
\]
Here, in the final case, we have set  $D'=D/p$ and 
\[
\epsilon_r = \begin{cases}
1 & \mbox{if } r\equiv 1\pmod{4} \\
i & \mbox{if } r\equiv 3 \pmod{4} .
\end{cases}
\]
for any odd positive integer $r$.
 \end{lemma}

\begin{proof}
It is easily seen that 
\[
G(a,p^\nu)= p^{2\alpha} G(a',p^{\nu-\alpha}).
\]
Hence we may assume that $a$ is coprime to $p$. 
Now, if $p$ is odd, the assertion can be easily deduced from the classical Gauss sum 
\[
\sum_{x\in \Z/c \Z} e\left(\frac{a x^2}{c}\right)=\epsilon_{c}\sqrt{c} \leg{a}{c}
\]
for odd positive $c$ with $\gcd(c,a)=1$.

For $p=2$ and $a$ odd we have
\begin{align}
\label{eq:gauss1}
G(a,2^\nu)= \sum_{x,y\in \Z/2^\nu\Z} e\left( a\frac{x^2+xy+\frac{1+D}{4}y^2}{2^\nu}\right),
\end{align}
where we have used that $D$ is odd.
If $\nu=1$ the asserted formula is easily checked. Hence, assume that $\nu\geq 2$. Then we use the classical Gauss sum
\[
\sum_{x\in\Z/2^\nu\Z}  e\left( a\frac{x^2+xy}{2^\nu}\right)
=\begin{cases}
\displaystyle 
0,& 2\nmid y,\\[1ex]
\displaystyle
e\left(-a\frac{y^2}{2^{\nu+2}}\right)\sum_{x\in \Z/2^\nu\Z}  e\left( a\frac{x^2}{2^\nu}\right), & 2\mid y,
\end{cases}
\]
to rewrite \eqref{eq:gauss1} as follows:
\begin{align*}
G(a,2^\nu)= \frac{1}{2}\sum_{x,y\in \Z/2^{\nu}\Z} e\left( a\frac{x^2+Dy^2}{2^\nu}\right).
\end{align*}
Inserting the evaluation of the  Gauss sum 
\[
\sum_{x\in\Z/2^\nu\Z}  e\left( a\frac{x^2}{2^\nu}\right)
= (1+i)\epsilon_a^{-1}2^{\nu/2} \leg{2^\nu}{a},
\]
we finally find that 
\begin{align*}
G(a,2^\nu)= 2^\nu \leg{2^\nu}{D}=2^\nu \leg{-D}{2}^\nu,
\end{align*}
concluding the proof of the lemma.
\end{proof}

\begin{lemma}
\label{lem:gauss2}
Let $p^\nu$ be a prime power. 
If we write  
$\gcd(a,p^\nu) =p^\alpha$ with $0\leq \alpha \leq \nu$ and $a=p^\alpha a'$ with $a'\in \Z$, then 
the Gauss sum
\begin{align*}
G_L(a,p^\nu)= \sum_{x\in L/p^\nu L} e\left(\frac{a\langle x,x\rangle }{p^\nu}\right)
\end{align*}
is given by
\[
G_L(a,p^\nu)= \begin{cases}
p^{2n\nu }&\text{if $\alpha=\nu$ }\\
p^{n\alpha+n\nu}\leg{-D}{p}^{n(\nu-\alpha)} &\text{if $\alpha <\nu$ and $p\nmid D$}\\
\epsilon_p^n \inv_p(V) p^{n\alpha+n\nu+\frac{n}{2}}\leg{a'}{p}^n\leg{D'}{p}^{n(\nu-\alpha-1)} &\text{if $\alpha <\nu$ and $p\mid D$,}
\end{cases}
\]
where $\epsilon_p$ and $D'$  have the same meaning as in Lemma \ref{lem:gauss1}.
\end{lemma}

\begin{proof}
The Gauss sum $G_L(a,p^\nu)$ only depends on $L_p=L\otimes_\Z \Z_p$. 
If we let  $b_1,\dots,b_n$ be an orthogonal  $\mathcal{O}_{\kk,p}$-module basis of $L_p$ (which exists because of our hypothesis that $D$ is odd), then 
\[
G_L(a,p^\nu) = G(a\langle b_1,b_1\rangle,p^\nu)\cdots G(a\langle b_n,b_n\rangle,p^\nu),
\] 
and  the claim follows from Lemma \ref{lem:gauss1}.  
Note that in the case $p\mid D$, standard formulas for the Hilbert symbol imply
\[
\inv_p(V) = \leg{\langle b_1,b_1\rangle}{p}\cdots \leg{\langle b_n,b_n\rangle}{p} . \qedhere
\]
\end{proof}

\begin{proposition}
\label{prop:lp1}
Assume that $n$ is even, and $m\in \Z$ is nonzero.
Fix a prime $p$, and  factor $m=p^\beta m'$ with  $\gcd(m',p)=1$. 
If $p\nmid D$ then
\[
L^{(p)}_{m}(p^{-s})= 
(1-p^{n-s-1})\sum_{\gamma=0}^{\beta} p^{(n-s)\gamma}.
\]
If $p\mid D$ then
\begin{align*}
L^{(p)}_{m}(p^{-s})  
&= 1-\leg{-1}{p}^{\frac{n}{2}} \inv_p(V)p^{\frac{n}{2}}p^{n-1-s} \\
& \quad + \leg{-1}{p}^{\frac{n}{2}} \inv_p(V)p^{\frac{n}{2}}\left(1 -  p^{n-1-s}\right)\sum_{\gamma=1}^{\beta} p^{(n-s)\gamma} .
\end{align*}
\end{proposition}

\begin{proof}
We compute the representation number $N_m(p^\nu)$ using the identity
\begin{align*}
N_m(p^\nu) &= \frac{1}{p^{\nu}} \sum_{a \in \Z/p^\nu\Z }  G_L(a,p^\nu) e\Big(-\frac{am}{p^\nu}\Big)\\
&=  \frac{1}{p^{\nu}} \sum_{\alpha=0}^\nu
 \sum_{ a' \in (\Z/ p^{\nu-\alpha} \Z)^\times}  G_L(p^\alpha a',p^\nu) e \Big(-\frac{a'm}{p^{\nu-\alpha}}  \Big).
\end{align*}
As $n$ is even,  Lemma \ref{lem:gauss2} implies
\[
G_L(a,p^\nu)= \begin{cases}
p^{2n\nu} &\text{if $\alpha=\nu$}\\
p^{n\alpha+n\nu} &\text{if $\alpha <\nu$ and $p\nmid D$ }\\
\leg{-1}{p}^{n/2} \inv_p(V) p^{n\alpha+n\nu+\frac{n}{2}} &\text{if $\alpha <\nu$ and $p\mid D$.}
\end{cases}
\]
If $p\nmid D$ we obtain
\begin{align*}
N_m(p^\nu) 
&= p^{n\nu-\nu}\sum_{\alpha=0}^{\nu}  p^{n\alpha}
  \sum_{a' \in  (\Z/ p^{\nu-\alpha} \Z)^\times } e \Big(-\frac{a'm}{p^{\nu-\alpha}} \Big)   \\
	&= p^{n\nu-\nu}\sum_{\alpha=0}^{\nu}  p^{n\alpha}
  \sum_{d\mid  \gcd(p^{\nu-\alpha},m)}   \mu\Big(\frac{p^{\nu-\alpha}}{d}\Big)  d   \\
&=p^{n\nu-\nu}\sum_{\alpha=0}^{\nu}  p^{n\alpha}
 \sum_{\gamma=0}^{\min(\nu-\alpha,\beta)}     \mu\big(p^{\nu-\alpha-\gamma}\big) p^\gamma   .
%
\end{align*}
Here we have used the evaluation of the Ramanujan sum. 
Consequently, the Euler factor \eqref{eq:lp} is
\begin{align*}
L^{(p)}_{m}(p^{-s})&= (1-p^{2n-s-1})
\sum_{\nu=0}^\infty N_{m}(p^\nu) p^{-\nu s}\\
&=(1-p^{2n-s-1})
\sum_{\gamma=0}^{\beta} \sum_{\alpha=0}^{\infty} \sum_{\nu=\alpha+\gamma}^\infty p^{n\alpha+\gamma} \mu(p^{\nu-\alpha-\gamma})p^{(n-1-s)\nu}\\
&=(1-p^{n-s-1})\sum_{\gamma=0}^{\beta} p^{(n-s)\gamma}.
\end{align*}

In the case $p\mid D$  we find
\begin{align*}
N_m(p^\nu) 
&= p^{2n\nu-\nu}+\leg{-1}{p}^{n/2} \inv_p(V)p^{n\nu-\nu+\frac{n}{2}}\sum_{\alpha=0}^{\nu-1}  p^{n\alpha}  
\sum_{a'\in (\Z / p^{\nu-\alpha} \Z)^\times } e \Big(-\frac{a'm}{p^{\nu-\alpha}} \Big) \\
&=
p^{2n\nu-\nu}+\leg{-1}{p}^{n/2} \inv_p(V)p^{n\nu-\nu+\frac{n}{2}}\sum_{\alpha=0}^{\nu-1}  p^{n\alpha} 
\sum_{d\mid \gcd(p^{\nu-\alpha},m)}   \mu \Big(\frac{p^{\nu-\alpha}}{d} \Big) d \\
&=
p^{2n\nu-\nu}+\leg{-1}{p}^{n/2} \inv_p(V)p^{n\nu-\nu+\frac{n}{2}}\sum_{\alpha=0}^{\nu-1}  p^{n\alpha} \sum_{\gamma=0}^{\min(\nu-\alpha,\beta)} \mu \big(p^{\nu-\alpha-\gamma}\big) p^\gamma.
\end{align*}
Consequently,  the Euler factor $L^{(p)}_{m}(p^{-s})$ of \eqref{eq:lp} is 
\begin{align*}\lefteqn{ 
(1-p^{2n-s-1})\sum_{\nu=0}^\infty N_{m}(p^\nu) p^{-\nu s} } \\
&=1+\leg{-1}{p}^{\frac{n}{2}} \inv_p(V)p^{\frac{n}{2}}(1-p^{2n-s-1})\sum_{\nu=1}^\infty p^{(n-1-s)\nu}
\sum_{\alpha=0}^{\nu-1}  \sum_{\gamma=0}^{\min(\nu-\alpha,\beta)}p^{n\alpha+\gamma}  \mu(p^{\nu-\alpha-\gamma}) \\
&=1+\leg{-1}{p}^{\frac{n}{2}} \inv_p(V)p^{\frac{n}{2}}(1-p^{2n-s-1})
\sum_{\gamma=0}^{\beta}\sum_{\alpha=0}^{\infty}\sum_{\substack{\nu\geq \alpha+\gamma\\ \nu\geq \alpha +1}}p^{(n-1-s)\nu} p^{n\alpha+\gamma}  \mu(p^{\nu-\alpha-\gamma}).
\end{align*}
The contribution of $\gamma=0$ to the latter sum is given by 
\begin{align*}
(1-p^{2n-s-1})\sum_{\alpha=0}^{\infty}\sum_{\substack{  \nu\geq \alpha +1}}p^{(n-1-s)\nu} p^{n\alpha}  \mu(p^{\nu-\alpha})=-p^{n-1-s}.
\end{align*}
The contribution of $\gamma\geq 1$ is equal to
\begin{align*}
&(1-p^{2n-s-1})\sum_{\gamma=1}^{\beta}\sum_{\alpha=0}^{\infty}\sum_{\substack{  \nu\geq \alpha +\gamma}}p^{n\alpha+\gamma} p^{(n-1-s)\nu}  \mu(p^{\nu-\alpha-\gamma})\\
&=(1-p^{2n-s-1})\sum_{\gamma=1}^{\beta}\sum_{\alpha=0}^{\infty} p^{n\alpha+\gamma} \left(p^{(n-1-s)(\alpha+\gamma)} -  p^{(n-1-s)(\alpha+\gamma+1)}\right)\\
&=\left(1 -  p^{n-1-s}\right)\sum_{\gamma=1}^{\beta} p^{(n-s)\gamma} .
\end{align*}
Putting the terms together, we find
\begin{align*}
L^{(p)}_{m}(p^{-s})&= 1-\leg{-1}{p}^{\frac{n}{2}} \inv_p(V)p^{\frac{n}{2}}p^{n-1-s}   \\
& \quad + \leg{-1}{p}^{\frac{n}{2}} \inv_p(V)p^{\frac{n}{2}}\left(1 -  p^{n-1-s}\right)\sum_{\gamma=1}^{\beta} p^{(n-s)\gamma} .
\end{align*}
This concludes the proof of the proposition.
\end{proof}

\begin{proposition}
\label{prop:lp2}
Assume $n$ is odd, and $m\in \Z$ is nonzero.
Fix a prime $p$, and factor $m=p^\beta m'$ with  $\gcd(m',p)=1$.
If $p\nmid D$ then
\[
L^{(p)}_{m}(p^{-s})= 
\left(1-\leg{-D}{p}p^{n-s-1}\right)\sum_{\gamma=0}^{\beta} \leg{-D}{p}^\gamma p^{(n-s)\gamma}.
\]
If $p\mid D$ then
\[
L^{(p)}_{m}(p^{-s})= 1+\inv_p(V)\leg{-1}{p}^{\frac{n-1}{2}}\leg{D'}{p}^\beta\leg{m'}{p} p^{
(\beta+3/2)n-1/2-(\beta+1)s},
\]
where we have set $D'=D/p$.
\end{proposition}

\begin{proof}
We argue in the same way as in the proof of Proposition \ref{prop:lp1}, using the identity
\begin{align*}
N_m(p^\nu) &= \frac{1}{p^{\nu}} \sum_{a \in \Z/ p^\nu\Z }  G_L(a,p^\nu) e \Big(-\frac{am}{p^\nu} \Big)
\end{align*}
and the value of the Gauss sum $G_L(a,p^\nu)$ computed in Lemma \ref{lem:gauss2}. 
If we write $a=p^\alpha a'$ with $a'$ coprime to $p$, we have 
\[
G_L(a,p^\nu)= \begin{cases}
p^{2n\nu}  &\text{if $\alpha=\nu$ }\\
p^{n\alpha+n\nu}\leg{-D}{p}^{\nu-\alpha}  &\text{if $\alpha <\nu$ and $p\nmid D$ }\\
\epsilon_p^n \inv_p(V) p^{n\alpha+n\nu+\frac{n}{2}}\leg{a'}{p}\leg{D'}{p}^{\nu-\alpha-1} &\text{if $\alpha <\nu$ and $p\mid D$.}
\end{cases}
\]

If $p\nmid D$ we obtain
\begin{align*}
N_m(p^\nu) &= p^{n\nu-\nu}\sum_{\alpha=0}^{\nu}  p^{n\alpha}\leg{-D}{p}^{\nu-\alpha}
  \sum_{a'  \in  (\Z /  p^{\nu-\alpha} \Z)^\times } e \Big(-\frac{a'm}{p^{\nu-\alpha}} \Big)\\
	&= p^{n\nu-\nu}\sum_{\alpha=0}^{\nu}  p^{n\alpha}\leg{-D}{p}^{\nu-\alpha}
  \sum_{d\mid  \gcd(p^{\nu-\alpha},m)}\mu \Big(\frac{p^{\nu-\alpha}}{d} \Big) d  \\
&=p^{n\nu-\nu}\sum_{\alpha=0}^{\nu}  p^{n\alpha} \leg{-D}{p}^{\nu-\alpha}\sum_{\gamma=0}^{\min(\nu-\alpha,\beta)} \mu\big(p^{\nu-\alpha-\gamma}\big) p^\gamma .
%
\end{align*}
Here we have used the evaluation of the Ramanujan sum. 
Hence,  the Euler factor \eqref{eq:lp} is
\begin{align*}
L^{(p)}_{m}(p^{-s})&= (1-p^{2n-s-1})
\sum_{\nu=0}^\infty N_{m}(p^\nu) p^{-\nu s}\\
&=(1-p^{2n-s-1})
\sum_{\gamma=0}^{\beta} \sum_{\alpha=0}^{\infty} \sum_{\nu=\alpha+\gamma}^\infty p^{n\alpha+\gamma} \leg{-D}{p}^{\nu-\alpha}\mu(p^{\nu-\alpha-\gamma})p^{(n-1-s)\nu}\\
&=\left(1-\leg{-D}{p}p^{n-s-1}\right)\sum_{\gamma=0}^{\beta} \leg{-D}{p}^\gamma p^{(n-s)\gamma}.
\end{align*}

When  $p\mid D$ we find
\begin{align*}
N_m(p^\nu)   = p^{2n\nu-\nu} 
&+\epsilon_p^n\inv_p(V)p^{n\nu-\nu+\frac{n}{2}}\\
&\phantom{+}{}\times\sum_{\alpha=0}^{\nu-1} \leg{D'}{p}^{\nu-\alpha-1}  p^{n\alpha} \sum_{a'  \in  (\Z /  p^{\nu-\alpha} \Z)^\times } \leg{a'}{p} e \Big(-\frac{a'm}{p^{\nu-\alpha}}  \Big).
\end{align*}
The latter Gauss sum is equal to $p^{\nu-\alpha-\frac{1}{2}}\epsilon_p \leg{-m'}{p}$ if $\beta=\nu-\alpha-1$ and is zero otherwise. Inserting this we find that $N_m(p^\nu)$ is equal to 
\begin{align*}
p^{2n\nu-\nu} 
& \cdot \begin{cases}
1+\inv_p(V)\leg{-1}{p}^{\frac{n+1}{2}}\leg{D'}{p}^\beta\leg{-m'}{p} p^{(1-n)(\beta+1/2)}  &\text{if $\nu>\beta$} \\[1ex]
1 &\text{if $\nu\leq \beta$,}
\end{cases}
\end{align*}
and hence $L^{(p)}_{m}(p^{-s})$ is equal to 
\begin{align*} \lefteqn{
(1-p^{2n-s-1})\sum_{\nu=0}^\infty N_{m}(p^\nu) p^{-\nu s} }\\
&=1+\inv_p(V)\leg{-1}{p}^{\frac{n+1}{2}}\leg{D'}{p}^\beta\leg{-m'}{p} p^{(1-n)(\beta+1/2)} p^{(2n-1-s)(\beta+1)}\\
&=1+\inv_p(V)\leg{-1}{p}^{\frac{n-1}{2}}\leg{D'}{p}^\beta\leg{m'}{p} p^{
(\beta+3/2)n-1/2-(\beta+1)s}.
\end{align*}
This concludes the proof of the proposition.
\end{proof}

Recall the function $\mathbf{a}_n(s)$ of \eqref{a_k}, and abbreviate $s_0=(n-1)/2$.
 We now express the coefficients $B(m,0,s)$ in terms of  
\[
\mathbf{a}_n(s-s_0) =  
 \frac { D^{n/2} \Gamma \left( s + \frac{n}{2} + \frac{1}{2}  \right)  L\left(2s +1,\eps^n\right) }   {2^n \pi^{  s + \frac{n}{2} + \frac{1}{2}  } } .
\]

\begin{corollary}
\label{cor:eis2}
Assume that $n\ge 2$ is even.
If $m>0$ then
\[
B(m,0,s)= -
\frac{  m^{s+s_0}  }{ \mathbf{a}_n(s-s_0)    } 
\cdot \prod_{p\nmid D}\sum_{\gamma=0}^{v_p(m)} p^{-2s\gamma}\cdot
 \prod_{p\mid D} \frac{L^{(p)}_{m}(p^{-2s-n})}{1-p^{-2s-1}}.
\]
If  $m>0$ is prime to $D$ then
\[
B(m,0,s)=-\frac{ m^{s+s_0}   \sigma_{-2s}(m) }{  \mathbf{a}_n(s-s_0)    } 
\cdot
 \prod_{p\mid D} \frac{1-\leg{-1}{p}^{\frac{n}{2}} \inv_p(V)p^{\frac{n}{2}-1-2s} }{1-p^{-2s-1}},
\]
where $\sigma_s(m)=\sum_{d\mid m} d^s$ is the usual divisor sum. 
\end{corollary}

\begin{proof}
Combine Propositions \ref{prop:eis1} and \ref{prop:lp1}.  
\end{proof}

\begin{corollary}
\label{cor:eisodd2}
Assume that $n\ge 1$ is odd. 
If $m>0$ then 
\begin{align*}
B(m,0,s)  =  -\frac{ m^{s+ s_0}}{  \mathbf{a}_n(s-s_0) }   
\cdot
 \prod_{p\nmid D}\sum_{\gamma=0}^{v_p(m)} \eps(p)^\gamma p^{-2s\gamma}\cdot
\prod_{p\mid D} L^{(p)}_{m}(p^{-2s-n}).
\end{align*}
If $m>0$ is  prime to $D$, then 
\begin{align*}
B(m,0,s)  =-\frac{ m^{s+ s_0}\sigma_{-2s,\eps}(m)}{ \mathbf{a}_n(s-s_0) }  
\cdot \prod_{p\mid D} 
\left( 1+   \leg{-1}{p}^{\frac{n-1}{2}} \!\leg{m}{p}   \inv_p(V)   p^{\frac{n-1}{2}-2s} \right),
\end{align*}
where $\sigma_{s,\eps}(m)=\sum_{d\mid m} \eps(d) d^s$ is  the usual divisor sum. 
\end{corollary}

\begin{proof}
Combine  Propositions   \ref{prop:eis1} and  \ref{prop:lp2}.
\end{proof}

\begin{remark}
If $n=1$, then 
$s_0=0$ is the center of symmetry of the functional equation of the incoherent Eisenstein series $E_L(\tau,s,1)$. Hence $E_L(\tau,s,1)$ vanishes identically (to odd order) at $s=0$. For all  positive $m$ this implies  that the coefficients $B(m,\mu,s_0)$ also vanish. Consequently, in view of 
Proposition \ref{prop:eis1}, there must be at least one prime $p<\infty$ for which the local factor $L^{(p)}_m(p^{-1})$ is zero. According to Proposition \ref{prop:lp2}, this happens exactly when 
\[
\inv_p(V)= -\begin{cases}\leg{-D}{p}^\beta,&\text{if $p\nmid D$,}\\[1ex]
\leg{D'}{p}^\beta\leg{m'}{p},&\text{if $p\mid D$,}
\end{cases}
\]
where we have used the notation of the cited proposition. In terms of the local Hilbert symbol this is equivalent to the condition $\inv_p(V)= -(-D,m)_p$, which means that $V_p$ does not represent $m$.

The derivatives $B'(m,0,s_0)$ of the coefficients can be computed as follows. 
Let $\mathbb{V}$ be the incoherent hermitian space over $\A_\kk$ whose non-archimedian contribution is given by $V\otimes_\Q \A_f$ and whose archimedian contribution is $V\otimes_\Q \R$ with the positive definite hermitian form $-\langle \cdot,\cdot \rangle$. Following \cite{Ku:Annals} we define the `Diff set' associated with $\mathbb{V}$ and $m\in \Q^\times$ by 
\[
\operatorname{Diff}(\mathbb{V},m)=\{ p\leq \infty \mid \;\text{$\mathbb{V}_p$ does not represent $m$}\}.
\]
This is a finite set of places of $\Q$ of odd cardinality. When $m>0$ it only consists of finite primes. The above argument shows that $L^{(p)}_m(p^{-1})=0$ if and only if $p\in \operatorname{Diff}(\mathbb{V},m)$. Hence $B'(m,0,s_0)\neq 0$ if and only if $ \operatorname{Diff}(\mathbb{V},m)$ consists of a single prime $q$. By Proposition \ref{prop:eis1} we find for $m\in \Z^+$ 
in this case that
\[
B'(m,0,0)= -\frac{2\pi}{\sqrt{D}}\cdot \prod_{\substack{\text{$p$ prime}\\ p\neq q}} L_m^{(p)} (p^{-1}) \cdot \frac{d}{ds} L_m^{(q)} (q^{-2s-1}) \big|_{s=0} .
\]
This expression  can be evaluated explicitly by means of Proposition \ref{prop:lp2}. For instance, when $(m,D)=1$ and $m=q^\beta m'$ with $m'$ coprime to $q$, we see that $\inv_q(V)=1$, the prime $q$ must be inert in $\kk$, and $\beta$ must be odd. Under these conditions we find that 
\[
B'(m,0,0)= -\frac{2^{o(D)+1}\pi\,\sigma_{0,\eps}(m')}{\sqrt{D}L(\eps,1)}  \left(\beta+1\right)\cdot \log(q),
\]
where $o(D)$ denotes the number of prime factors of $D$. 
\end{remark}

Regardless of whether $n$ is even or odd, we have the  following lower bound for the coefficients of the Eisenstein series $E_L(\tau,s_0,n)$.
 
\begin{corollary}
\label{cor:eisgrowth}
Assume that $n \geq 2$.   
There is a constant $C>0$, depending  only  on $L$,  such that 
\[
-B(m,0,s_0)> C\cdot m^{n-1}
\]
for all $m\in \Z^+$.
\end{corollary}

\begin{proof}
This follows  easily from   Corollary \ref{cor:eis2} and Corollary \ref{cor:eisodd2}.
\end{proof}

 
\section{Decomposing the Eisenstein series}
\label{ss:eis decomp}
 

As in the previous section,  we assume  the $\co_\kk$-lattice $L\subset V$ is self-dual  under the hermitian form, and that $D=- \mathrm{disc}(\kk)$ is odd.  We also assume that $V$ has dimension $n\ge 2$.  

Consider again  the  holomorphic Eisenstein series
\[
E_L(\tau,s_0,n) = \sum_{ m \ge 0}  \sum_{ \mu \in L'/L} B(m,\mu,s_0) q^m \varphi_\mu
\]
valued in $S_L=\C[L'/L]$  obtained by evaluating \eqref{eq:fouriereis} at $s_0=(n-1)/2$.
Applying the linear functional $S_L \to \C$ sending $\varphi_0\mapsto 1$ and $\varphi_\mu\mapsto 0$ for $\mu\neq 0$ yields an Eisenstein series
\[
\sum_{ m \ge 0 } B(m, 0 ,s_0)  \cdot q^m  \in M_n(\Gamma_0(D) , \eps^n ), 
\]
which can  be decomposed as a linear combination of  classical Eisenstein series indexed by the cusps
\[
\infty_r=  \frac{r}{D} \in \Gamma_0(D) \backslash \mathbb{P}^1(\Q)
\]
as $r\mid D$ varies.  In this section we make this decomposition explicit.

 For any divisor $r\mid D$ set $r'=D/r$.  Our assumption that $D$ is odd implies that the quadratic character $\eps : (\Z/D\Z)^\times \to \{ \pm 1\}$ determined by $\kk$ is 
 \[
 \eps(a) =  \leg{a}{D}.
 \]
Hence we may  factor  $\eps = \eps_r \cdot \eps_{r'}$ with
\[
\eps_r(a) = \left( \frac{a}{r} \right) \quad \mbox{and} \quad \eps_{r'} (a) = \left( \frac{a}{r'} \right).
\]
Define the quadratic Gauss sum
\[
\tau(\eps_r) =  \sum_{  a \in (\Z/r\Z)^\times } \eps_r(a) e^{2\pi i a/r} = \begin{cases}
\sqrt{r}  & \mbox{if } r\equiv 1\pmod{4} \\
i\sqrt{r} & \mbox{if } r \equiv 3\pmod{4},
\end{cases}
\]
and similarly with $r$ replaced by $r'$.

\begin{proposition}\label{prop:eisenstein formulas}
For every divisor  $r\mid D$ there is an Eisenstein series 
 \[
E_r (\tau)=  \sum_{m \ge 0} e_r(m) \cdot q^m \in M_n(\Gamma_0(D),\eps^n)
\]
whose Fourier coefficients are as  follows.
The constant term  is 
\[
e_r(0)= \begin{cases} 
 1 & \mbox{if }r=1 \\
 0 &  \mbox{otherwise.}
 \end{cases}
\]
If $n$ is even the coefficients indexed by $m>0$ are 
\[
e_r(m)   =   
  \frac{  r^{n/2} (-2\pi i )^n    }{ D^n  \Gamma(n)  L_D( n ,\eps^n ) } 
  \sum_{ \substack{   c\mid m \\ c>0  \\ \gcd( m /c, r) =1   } }
   c^{n-1}   \sum_{ d \mid \gcd( c , r'  ) } d \mu( r' / d) .
\]
If $n$ is odd the coefficients indexed by $m>0$ are 
\[
 e_r(m)   =    
 \eps_r(r')   \frac{ r^{n/2} (-2\pi i )^n  \tau(\eps_{r'})    }{  D^n   \Gamma(n) L_D( n ,\eps^n ) } 
  \sum_{ \substack{  c\mid m \\ c>0  \\ \gcd( m /c, r) =1   } }
 \eps_r( m/c)   \eps_{r'}(c) \cdot c^{n-1}.
\]
In both formulas $L_D(s,\eps^n)$ is the Dirichlet $L$-function with Euler factors at all $\ell\mid D$ removed.
\end{proposition}

 \begin{proof}
For each $s,t\in \Z/D\Z$ define an Eisenstein series
    \[
 G_{(s,t)}(z) = \sum_{  \substack{   c,d \in \Z  \\  (c,d) \neq (0,0) \\ (c,d)  \equiv (s,t)  \pmod{D}  } }    (   cz + d)^{-n}  .
 \]
 Theorem 7.1.3 of \cite{miyake} implies that
  \[
 E_r(z)   =
 \frac{   \eps^n_r(r')  }{  2 r^{n/2}  L_D( n ,\eps^n ) }   
 \sum_{ s,t  \in \Z/D\Z  }   \eps_r^n(  s )  \eps_{r'}^n(  t)   G_{ (  s,t   )  }  ( r' z )
 \]
has the desired Fourier expansion.
 \end{proof}


\begin{remark}\label{rem:each cusp}
The Eisenstein series of Proposition \ref{prop:eisenstein formulas} agree with the Eisenstein series denoted the same way in \S 4.4 of \cite{BHKRY-1}.  In particular, $E_r$ is nonvanishing at the cusp $\infty_r= r/D$, and vanishes at all other cusps.
\end{remark}

\begin{proposition}\label{prop:multi eisenstein}
We have the equality 
\[
\sum_{r\mid D} \gamma_r E_r(\tau)   = \sum_{m\ge 0} B(m,0,s_0)  \cdot q^m 
\]
of forms in $M_n(\Gamma_0(D) , \eps^n)$
Here for a prime $\ell \mid D$ we have set
 \[
\gamma_\ell =   (D,\ell)_\ell^n \cdot \mathrm{inv}_\ell(V)
\cdot \begin{cases}
1 & \mbox{if } \ell \equiv 1\pmod{4} \\
i^{-n} & \mbox{if } \ell \equiv 3 \pmod{4} ,
\end{cases}
\]
where    $(-,-)_\ell$ is the usual Hilbert symbol at $\ell$, and for a divisor $r\mid D$   we have set  
\[
\gamma_r = \prod_{\ell\mid r}\gamma_\ell .
\]
\end{proposition}

\begin{proof}
We use the $\C$-bilinear pairing $\langle \cdot,\cdot \rangle$ on $S_L$ to write
\[
\sum_{m\ge 0} B(m,0,s_0)  \cdot q^m =\langle E_L,\varphi_0\rangle.
\] 
On the other hand, using \eqref{eq:eisl} and the scalar Eisenstein series 
\[
E=\sum_{\gamma\in \Gamma_\infty\bs \Gamma_0(D)} \eps^n(d)\cdot (1\mid_n\gamma)
\]
of weight $n$ for the cusp $\infty$ in $M_n(\Gamma_0(D),\eps^n)$, we obtain 
\begin{align*}
E_L=E_L(\tau,s_0,n) &= \sum_{\gamma\in \Gamma_\infty\backslash \Gamma}
   (c\tau+d)^{-n}\cdot \bar\omega_L(\gamma)^{-1}\varphi_0\\
	&= \sum_{\gamma\in \Gamma_0(D)\backslash \Gamma} (E\mid_n\gamma) 
   \cdot \bar\omega_L(\gamma)^{-1}\varphi_0.
\end{align*}
Hence 
\begin{align*}
\langle E_L,\varphi_0\rangle = \sum_{\gamma\in \Gamma_0(D)\backslash \Gamma} (E\mid_n\gamma) 
   \cdot\langle  \bar\omega_L(\gamma)^{-1}\varphi_0,\varphi_0\rangle .
\end{align*}
We compute this finite sum by means of  the left coset decomposition 
\[
\Gamma = \bigsqcup_{r\mid D}\bigsqcup_{c\in \Z/r\Z} \Gamma_0(D)R_r \zxz{1}{c}{0}{1}  
,
\]
see e.g.~equation (8.2.1) in \cite{BHKRY-1}. Using the notation of \cite{BHKRY-1} we find
\begin{align*}
\langle E_L,\varphi_0\rangle &= \sum_{\gamma\in \Gamma_0(D)\backslash \Gamma} (E\mid_n\gamma) 
   \cdot\langle  \varphi_0,\omega_L(\gamma)\varphi_0\rangle \\
&= \sum_{r\mid D} \sum_{c\in \Z/r\Z}   (E\mid_n R_r \zxz{1}{c}{0}{1} )\cdot\langle  \varphi_0,\omega_L(R_r) \omega(\kzxz{1}{c}{0}{1})  \varphi_0\rangle\\
& =\sum_{r\mid D} \sum_{c\in \Z/r\Z}   (E\mid_n W_r \zxz{1}{c}{0}{r} )\cdot\langle  \varphi_0,\omega_L(R_r)   \varphi_0\rangle.
\end{align*}
In terms of the $U_r$ operators on $M_n(\Gamma_0(D),\eps^n)$, 
given by $\sum_{n} a(n)q^n\mid_n U_r = \sum_{n} a(rn)q^n$, we may rewrite this as 
\begin{align*}
\langle E_L,\varphi_0\rangle & = \sum_{r\mid D} \langle  \varphi_0,\omega_L(R_r)   \varphi_0\rangle\cdot r^{1-n/2}(E\mid_n W_r \mid_n U_r).
\end{align*}
By means of equation (8.2.3) of  \cite{BHKRY-1}, we obtain for the matrix coefficients
\begin{align*}
\langle  \varphi_0,\omega_L(R_r)   \varphi_0\rangle = \eps_{r'}^n(\alpha)\eps_r^n(-\beta)\gamma_r r^{-n/2} ,
\end{align*}
where we have written $rr'=D$. 
Inserting the definition 
\[
E_r= \eps_r^n(-\beta)\eps_{r'}^n(\alpha r) (E \mid_n W_r)
\]
of the Eisenstein series for the cusp $\infty_r= r/D$, we find 
\begin{align*}
\langle E_L,\varphi_0\rangle & = \sum_{r\mid D} \eps_{r'}^n( r) \gamma_r r^{1-n}(E_r\mid_n U_r) .
\end{align*}
It follows from the Fourier expansion of $E_r$ given in Proposition \ref{prop:eisenstein formulas} that 
\[
E_r\mid_n U_r = \eps_{r'}^n( r)r^{n-1}E_r.
\] 
Inserting this, we finally obtain 
\begin{align*}
\langle E_L,\varphi_0\rangle & = \sum_{r\mid D}  \gamma_r E_r ,
\end{align*}
concluding the proof of the proposition.
\end{proof}

The following fact about the coefficients of the Eisenstein series of Proposition \ref{prop:eisenstein formulas} will be needed much later, in the proof of Lemma \ref{lem:third volume lemma}.

\begin{proposition}\label{prop:eisenstein twist}
For any positive $m\in \Z$ with  $\gcd(m,D)=1$, we have
\begin{align*}
\sum_{r\mid D}  \gamma_r e_r(m) \log(D/r)
=
 \left(   \sum_{\ell \mid D}    \frac{  \log(\ell )  }{  1+\beta_\ell \cdot  \leg{m}{\ell}^n }   \right) \sum_{r\mid D}  \gamma_r e_r(m).
 \end{align*}
Here both sums over $r\mid D$ are over all positive divisors of $D$, the sum over $\ell \mid D$ is over only the prime divisors, $\gamma_r$ is as in Proposition \ref{prop:multi eisenstein}, and we abbreviate
\[
\beta_\ell  = (-1)^{n+1} \inv_\ell(V)  \cdot \begin{cases}
   \leg{-1}{\ell}^{\frac{n}{2} }        \ell^{ \frac{n}{2} }  &  \mbox{if $n$ is even} \\[2ex]
\leg{-1}{\ell}^{\frac{n-1}{2} }     \ell ^{ \frac{n-1}{2}  } &  \mbox{if $n$ is odd.}
\end{cases}.
\]
\end{proposition}

\begin{proof}
Directly from the formulas of Proposition \ref{prop:eisenstein formulas}, one can see that there is a factorization
 \[
\gamma_r e_r(m) = c(m)    f_r(m) 
\]
in which  $c(m)$ is some function independent of $r$, and 
\[
f_r(m) =  \leg{m}{r}^n  \cdot   \prod_{\ell \mid r} \beta_\ell.
\]
The relation $f_r(m) = \prod_{\ell \mid r} f_\ell(m)$ implies the combinatorial  identity
\[
  \sum_{r\mid D}   f_r(m) \log(D/r)     
 = 
  \left(   \sum_{\ell \mid D}    \frac{  \log(\ell )  }{1+f_\ell(m) }   \right)    \sum_{r \mid D} f_r(m)  ,
\]
and multiplying both sides by $c(m)$ proves the claim.
\end{proof}


\chapter{Automorphic Green functions}
\label{s:green integrals}


As in Chapter \ref{s:eisenstein}, we fix  a $\kk$-hermitian space $V$ of signature $(n-1,1)$ with $n \ge 1$.
We recall the family of special divisors on the complex Shimura variety associated to the unitary group  $\Uni(V)$, and derive a geometric variant of the Siegel-Weil formula.
Using this, we compute the integrals of  automorphic Green functions for the special divisors.


\section{A complex Shimura variety}
\label{ss:u(v) shimura}


The  fixed embedding $\kk \subset \C$  identifies $\kk\otimes_\Q\R \iso \C$, and  allows us to distinguish between the two  orthogonal idempotents $\epsilon, \overline{\epsilon} \in \kk\otimes_\Q \C$, which we label in such a way that 
 \[
( \alpha \otimes 1) \epsilon = (1\otimes \alpha) \epsilon ,\qquad 
(\alpha \otimes 1) \overline{\epsilon} =  (1\otimes \overline{\alpha}  )  \overline{\epsilon} 
 \]
for all $\alpha\in \kk$.

Viewing $V(\R)=V\otimes_\Q\R$  as $\C$-hermitian space of signature $(n-1,1)$,  define a hermitian symmetric domain\footnote{We take this opportunity to point out a misstatement in (2.1.2) of \cite{BHKRY-1}:  the word ``planes" should be replaced by ``lines" in the definition of the hermitian symmetric domain.}
\begin{equation}\label{hermitian model}
\mathcal{D} = \{  \mbox{negative definite complex lines } z \subset V(\R) \}.
\end{equation}
The natural map
 \[
 V(\R) \hookrightarrow V(\C) = \epsilon V(\C) \oplus \overline{\epsilon} V(\C) \map{\mathrm{proj}}  \epsilon V(\C)
 \]
 is a $\C$-linear isomorphism, and identifies 
 \begin{equation}\label{bilinear model}
\mathcal{D} = \{   z \in   \mathbb{P}( \epsilon V(\C))   : [ z , \overline{z} ]  <0 \}.
\end{equation}
 Here we have endowed  $V(\C) =V\otimes_\Q\C$ with the $\C$-bilinear extension of the $\Q$-bilinear form $[-,-]$ of \eqref{Q bilinear}, and $z\mapsto \overline{z}$ is the complex conjugation on the second factor in the tensor product.  
Under either interpretation, there is an evident action of the real points of 
\[
H=\Uni(V)
\]
 on $\mathcal{D}$ by holomorphic automorphisms.

Fix  an $\co_\kk$-lattice $L\subset V$ on which the hermitian form is $\co_\kk$-valued.
Choose a compact open subgroup   $K\subset H(\A_f)$  that stabilizes $L$, and hence acts on the finite dimensional vector space $S_L = \C [ L' / L]$ of \S \ref{ss:lattice eisenstein}.
This choice determines a complex Shimura variety
\begin{equation}
\label{eq:X}
\mathrm{Sh}_K(H,\mathcal{D})=H(\Q)\backslash \mathcal{D}\times H(\A_f)/K ,
\end{equation}
which we view as a complex orbifold of dimension $n-1$.
If we write $H(\A_f)= \bigsqcup_j H(\Q)h_j K$ as a finite disjoint union, then
\begin{equation}\label{X components}
\mathrm{Sh}_K(H,\mathcal{D}) =\bigsqcup_j  ( H(\Q)\cap K_j) \backslash \mathcal{D},
\end{equation}
where $K_j = h_j K  h_j^{-1}$ stabilizes the lattice $L_j = h_j L$.

The Shimura variety  carries a family of special divisors  
\[
Z(m) \in \mathrm{Div}(  \mathrm{Sh}_K(H,\mathcal{D}) )
\]
indexed by positive integers $m$, which 
we describe in terms of the decomposition \eqref{X components}.  
For each $x\in V$ with $\langle x,x\rangle >0$, define an analytic divisor
\[
\mathcal{D}(x) = \{ z\in \mathcal{D} :  x \perp z  \}
\]
where $\perp$ means orthogonal with respect to the hermitian or bilinear form, depending on whether the reader prefers the model \eqref{hermitian model} or \eqref{bilinear model}.
The pullback of  $Z(m)$ via  the uniformization
\[
\mathcal{D} \to ( H(\Q)\cap K_j) \backslash \mathcal{D} \subset \mathrm{Sh}_K(H,\mathcal{D})
\]
is the locally finite analytic divisor 
\[
\sum_{ \substack{ x \in  L_j  \\  \langle x, x\rangle =m  } } \mathcal{D}(x) \in \mathrm{Div}(\mathcal{D}).
\]

More generally,  given a positive $m\in \Q$ and a $K$-fixed function $\varphi \in S_L$,  there is a special divisor 
\[
Z(m,\varphi) \in \mathrm{Div}_\C(  \mathrm{Sh}_K(H,\mathcal{D}) )
\]
whose pullback via the above uniformization is 
\[
\sum_{ \mu \in L '/L  }  \varphi(\mu) 
\sum_{ \substack{ x \in \mu_j+L_j   \\  \langle x, x\rangle =m  } }  
\mathcal{D}(x) \in \mathrm{Div}_\C (\mathcal{D}),
\]
 where  $\mu_j$ is the image of $\mu$ under $h_j : L'/L \to L_j'/L_j$.
 Taking $\varphi=\varphi_0$ to be the characteristic function of $0\in L'/L$ recovers $Z(m)$.

Let $\taut$ be the tautological line bundle over $\mathcal{D}$ whose fiber over a point $z\in \mathcal{D}$ is the complex line $\taut_z=z \subset V(\R)$.  
Denote by  $\widehat{\taut}$ the tautological  line bundle endowed with the hermitian metric  $\|\cdot \|$ defined as follows:  if we use the interpretation \eqref{hermitian model} to view $z\in \mathcal{D}$ as a line in $V(\R)$, then 
\[
\| s\|^2 = - \frac{ \langle s,s\rangle}{4\pi e^\gamma}
\]
 for every $s\in \taut_z$.  If we instead use \eqref{bilinear model} to view $z$ as a line in $\epsilon V(\C)$, the metric is the same, but with $\langle-,-\rangle$ replaced by $[-,-]$.  
 The Chern form  $\chern(\widehat{\taut})$ is an $H(\R)$-invariant K\"ahler form on $\mathcal{D}$.
 The metrized tautological bundle descends to $\mathrm{Sh}_K(H,\mathcal{D})$, and its holomorphic sections are hermitian modular forms of weight $1$.  As in the introduction, we abbreviate
\[
\vol_\C( \widehat{\taut}) = \int_{\mathrm{Sh}_K(H,\mathcal{D})}  \chern(\widehat{\taut})^{n-1}.
\]

In the sections below we will recall the construction of Green functions for certain linear combinations of special  divisors,  and  compute their integrals with respect to the volume form $\chern(\widehat{\taut})^{n-1}$.

\begin{remark}\label{rem:strong complex volume}
Exactly as in the case of orthogonal Shimura varieties \cite{Ku:Integrals}, one can show that
\[
  \vol_\C( \widehat{\taut} )^{-1}    \int_{Z(m,\varphi)}  \chern(\widehat{\taut})^{n-2} = 
  -  \sum_{ \mu \in L'/L } \varphi(\mu) B(m,\mu,s_0).
\]
On the right hand side the notation is that of \S \ref{ss:coefficients}.
\end{remark}


\section{A geometric variant of the Siegel-Weil formula}


For $z\in \mathcal{D}$ and $x\in V(\R)$, let 
\[
\varphi_\infty(z,x)= \exp\big(-2\pi \langle x_{z^\perp}, x_{z^\perp}\rangle +2\pi   \langle x_{z}, x_{z}\rangle \big)
\]
be the Gaussian.  Here we are using the model \eqref{hermitian model}, so that $z$ determines an orthogonal decomposition
$
V(\R) = z^\perp \oplus z,
$
and the vectors $x_{z^\perp}$ and $x_z$ are the projections of $x$ to the two summands.

As in \S \ref{ss:seesaw}, let $G\iso \Uni(1,1)$ be the quasi-split unitary group over $\Q$, so that $\SL_2\iso G' \subset G$.
Fix a base point $z_0\in \mathcal{D}$ and consider the  Schwartz function 
\[
\varphi_\infty(x) \define \varphi_\infty(z_0,x) \in S(V(\R)).
\]
 It has weight $n-2$ under the action of the Weil representation (\S \ref{ss:siegel-weil}) of $\SL_2(\R)$.  In other words, using the notation of \eqref{underline k}, 
\begin{align*}
\omega(k)(\varphi_\infty)=\underline k^{n-2} \varphi_\infty
\end{align*}
for all $k\in \SO(2,\R)$.

A choice of   $\varphi_f \in S(V(\A_f))$  determines a Schwartz function $\varphi_\infty \otimes \varphi_f \in S(V(\A))$.
Applying the construction \eqref{theta kernel} yields a  corresponding theta function $\theta( g,h, \varphi_\infty \otimes \varphi_f)$ of the variables $(g,h) \in G(\A) \times H(\A)$.  Restricting the first variable to 
$g\in \SL_2(\R)  \subset G(\A)$, we obtain the  Siegel theta function 
\[
\theta(g  ,h,\varphi_\infty\otimes \varphi_f)
= \sum_{x\in V(\Q)} \omega(g,1) \varphi_\infty(z_0,h_\infty^{-1} x )\cdot \varphi_f(h_f^{-1}x),
\]
where $h=h_\infty h_f\in H(\A)$.

We may view this theta function as a function of the variables  
\[
(\tau, z , h_f) \in \H \times \mathcal{D}  \times H(\A_f)
\]
as follows. Let $g_\tau\in \SL_2(\R)$ be as in \eqref{gtau}, and  choose an  $h_\infty\in H(\R)$ with $h_\infty z_0= z$, 
so that $\varphi_\infty (z_0,h_\infty^{-1} x)= \varphi_\infty (z, x)$.  Now set
\begin{align}
\label{eq:siegel2}
\theta(\tau,z,h_f,\varphi_f)&=v^{1- \frac{n}{2} }
\theta(g_\tau,h_\infty h_f, \varphi_\infty\otimes\varphi_f)\\
\nonumber
&= v \sum_{x\in V(\Q)} e\big(\langle x_{z^\perp}, x_{z^\perp}\rangle \tau-  \langle x_{z}, x_{z}\rangle\bar\tau\big)\cdot \varphi_f(h_f^{-1}x).
\end{align}
In the variable $\tau$, we have the transformation law 
\begin{align}
\label{eq:thetatr}
\theta(\gamma\tau, z,h_f ,\varphi_f)=(c\tau +d)^{n-2}\theta(\tau,z,h_f, \omega(\gamma)^{-1}\varphi_f)
\end{align}
for $\gamma\in \Gamma$. 
In the variable $z$, we have the transformation law 
\begin{align}
\label{eq:thetatr2}
\theta(\tau,\delta z,\delta h_f ,\varphi_f)=\theta(\tau,z,h_f, \varphi_f).
\end{align}
for  $\delta \in H(\Q)$,

Now suppose that $\varphi_f \in S_L = \C[L'/L]$ is fixed by the action of $K$.
Using \eqref{eq:siegel2} and \eqref{eq:thetatr2}, we see that   the  Siegel theta function \eqref{eq:siegel2} 
descends to a function on  $ \mathcal{H}\times \mathrm{Sh}_K(H,\mathcal{D})$.    We are interested in the geometric theta integral
\begin{align}
\label{eq:intgeo}
\mathcal{I}(\tau,\varphi_f)
&=  \vol_\C( \widehat{\taut}  )^{-1} \int_{\mathrm{Sh}_K(H,\mathcal{D})} 
 \theta(\tau,z,h_f, \varphi_f)\,\chern(\widehat{\taut})^{n-1}
\end{align}
as a function of $\tau \in \H$.

%
%

\begin{proposition}
\label{lem:swgeo}
Assume that $n>2$ or that $V$ is anisotropic. 
Recalling the notation of \eqref{eq:kappa} and \eqref{eq:eisl},  we have 
\[
\kappa\cdot \mathcal{I}(\tau,\varphi_f) = 
\left\langle E_L(\tau, s_0,n-2),\varphi_f\right\rangle,
\]
where the pairing on the right is \eqref{S_L pairing}.
\end{proposition}

\begin{proof}
We rewrite the integral in \eqref{eq:intgeo} as an integral over the quotient $H(\Q)\backslash H(\A)$. 
To this end,  we normalize the Haar measure on the compact subgroup $H(\R)_{z_0}\subset H(\R)$ so that the volume is $1$. 
Then we normalize the Haar measure $dh_\infty$ on $H(\R)$ so that the induced measure on the quotient 
$H(\R)/H(\R)_{z_0}\cong \mathcal{D}$ agrees with the measure defined by $\chern(\widehat{\taut})^{n-1}$. 
Finally, we normalize the Haar measure $dh_f$ on $H(\A_f)$ so that $\vol(H(\Q)\backslash H(\A))=1$.
Setting
\[
\vol(K) = \int_K dh_f,
\]
we then  have the identities
\[
\vol_\C(  \widehat{\taut} )  = \int_{   H(\Q) \backslash H(\A) / H(\R)_{z_0} K  }  dh_\infty dh_f = \vol(K)^{-1}.
\]

Using \eqref{eq:siegel2} we obtain  
\begin{align*}
\mathcal{I}(\tau,\varphi) &= \vol(K)\int_{\mathrm{Sh}_K(H,\mathcal{D})} \theta(\tau,z, h_f,\varphi_f)\, \chern(\widehat{\taut})^{n-1}  \\
&= 
v^{1-\frac{n}{2}}\int_{H(\Q)\backslash H(\A)} \theta(g_\tau,h_\infty h_f,\varphi_\infty\otimes\varphi_f)\, dh_\infty dh_f,
\end{align*}
and applying the Siegel-Weil formula of  Theorem \ref{thm:sw} shows
\begin{align*}
\kappa\cdot \mathcal{I}(\tau,\varphi) &=v^{1-\frac{n}{2}}
 E(g_\tau,\lambda(\varphi_\infty\otimes \varphi_f),s_0).
\end{align*}

As $\varphi_\infty$ has weight $n-2$,  Lemma \ref{lem:eis1} implies
\begin{align*}\lefteqn{
E(g_\tau,\lambda(\varphi_\infty\otimes \varphi_f),s_0)  } \\
&=v^{\frac{n}{2}-1} \sum_{\gamma\in \Gamma_\infty\backslash \Gamma}
 \operatorname{Im}(\gamma \tau)^{s_0+\frac{1}{2}-\frac{n-2}{2}}  (c\tau+d)^{2-n}\cdot (\omega_L(\gamma)\varphi_f)(0) \\
\nonumber 
&=v^{\frac{n}{2}-1} \sum_{\gamma\in \Gamma_\infty\backslash \Gamma} 
\operatorname{Im}(\gamma \tau)^{s_0+\frac{1}{2}-\frac{n-2}{2}}  (c\tau+d)^{2-n}\cdot (\omega_L^\vee(\gamma)^{-1}\delta_0)(\varphi_f) .
\end{align*}
Here $\delta_0\in S_L^\vee$ denotes the functional  $\varphi_f\mapsto \varphi_f(0)$.
As the Weil representation is unitary,  we have 
\[
 (\omega_L^\vee(\gamma)^{-1}\delta_0)(\varphi_f) = \langle \bar \omega_L(\gamma)^{-1}\varphi_0, \varphi_f\rangle.
\]
To complete the proof, substitute this equality into the final expression above, and recall the definition of $E_L(\tau,s,n-2)$ from \eqref{eq:eisl}.
\end{proof}


\section{Automorphic Green functions}
\label{ss:green functions}



Let $\sigma$ be a finite dimensional representation of $\Gamma=\SL_2(\Z)$
on a complex vector space $V_\sigma$, and assume that $\sigma$  factors through a finite quotient.
In our applications, $\sigma$ will be  the Weil representation 
\[
\omega_L:\Gamma\to \Aut(S_L)
\]
of \eqref{eq:ol},  or its complex conjugate, or its dual.

For $k\in \Z$,  denote by $H_k(\sigma)$ the vector space of (weak) harmonic Maass forms of weight $k$ for the group $\Gamma$ with representation $\sigma$ as in \cite{BY1}. 
We write   
\[
S_k(\sigma) \subset  M_k(\sigma)   \subset M_k^!(\sigma)
\]
  for the spaces of cusp forms,  holomorphic modular forms, and weakly holomorphic modular forms, respectively.

A harmonic Maass form $f\in H_k(\sigma)$ has a Fourier expansion 
\begin{align}
\label{eq:fourierf}
f(\tau)=\sum_{\substack{m\in \Q\\ m\gg -\infty }} c^+(m) q^m
+\sum_{\substack{m\in \Q \\ m<0 } } c^-(m) \Gamma(1-k, 4\pi |m| v) q^m
\end{align}
with coefficients $c^\pm(m)\in V_\sigma$. 
Here $\tau=i+i v \in \H$,  $q=e^{2\pi i\tau}$, and 
\[
\Gamma(s,x)=\int_x^\infty e^{-t}t^{s-1}dt
\]
 is the incomplete gamma function. 
The coefficients are supported on rational numbers with uniformly bounded denominators. 
The first summand on the right hand side of \eqref{eq:fourierf} is denoted by $f^+$ and is called the holomorphic part of $f$, the second summand is denoted by $f^-$ and is called the non-holomorphic part. 
%

As in \cite{BF}, there is a  surjective conjugate-linear differential operator 
\begin{equation}\label{xi}
\xi_k: H_k(\omega_L)\to S_{2-k}(\bar\omega_L)
\end{equation}
defined by 
\[
\xi_k(f)(\tau)=2iv^{k} \overline{\frac{\partial f}{\partial \bar\tau}}.
\]
Its kernel is   $M^!_{k}(\omega_L)$.

Suppose  $f\in H_{2-n}(\omega_L)$ is $K$-fixed  with Fourier coefficients  $c^\pm(m)\in S_L$.
Define the special divisor associated to $f$ by 
\begin{align}
\label{eq:zf}
Z(f)= \sum_{m>0}\sum_{\mu \in L'/L} c^+(-m,\mu) Z(m, \varphi_\mu)\in \Div_\C(\mathrm{Sh}_K(H,\mathcal{D})).
\end{align}
Here $c^\pm(m,\mu)\in \C$ is the value of $c^{\pm}(m)$ at  $\mu\in L'/L$, and 
$\varphi_\mu$ is  the characteristic function of the coset $\mu+\hat L$.

We recall from \S 3  of  \cite{BHY} the construction of an automorphic Green function for $Z(f)$.  
Using \eqref{eq:siegel2}, define an $S_L$-valued Siegel theta function 
\begin{equation}\label{lattice theta}
\theta_L(\tau,z,h)=\sum_{\mu \in L'/L} \theta(\tau,z,h,\varphi_\mu)\varphi_\mu
\end{equation}
for $\tau \in \H$, $z\in \mathcal{D}$, and $h\in H(\A_f)$.
The transformation law \eqref{eq:thetatr} implies that it defines a non-holomorphic modular form for $\Gamma$ of weight $n-2$ with representation $\bar \omega_L$,  
and  so, using the notation \eqref{S_L pairing},  the pairing $\langle f, \theta_L(\tau,z,h)\rangle $ is  $\Gamma$-invariant as a  function of $\tau \in \H$.

Following \cite{Bo1} and \cite{BF}, we consider the regularized theta lift
\begin{align} 
\label{eq:AutoGreen}
\Phi(z,h,f)= \int_{\Gamma\backslash \H}^{\reg}
\langle f,\theta_L(\tau,z,h)\rangle \, d\mu(\tau)
\end{align}
of $f$, where $d\mu(\tau)= \frac{du\,dv}{v^2}$ is the invariant measure.
Here the integral is regularized as in \cite{Bo1} and \cite{BHY}.
The main properties of 
\[
\Phi(f) = \Phi(z,h,f) ,
\]
viewed as a function on $\mathrm{Sh}_K(H,\mathcal{D})$,  are summarized in the following theorem, see Theorem~3.3.1 and Proposition~3.3.4 of \cite{BHY}.

\begin{theorem}
The Green function $\Phi(z,h,f)$ is smooth on $\mathrm{Sh}_K(H,\mathcal{D})\smallsetminus Z(f)$, and  has logarithmic singularities along $Z(f)$.
It is integrable over $\mathrm{Sh}_K(H,\mathcal{D})$,  with respect to the invariant measure descended from $\mathcal{D}$, if $n>2$ or if $V$ is anisotropic. The corresponding current on smooth compactly supported $(n-1,n-1)$-forms on $\mathrm{Sh}_K(H,\mathcal{D})$ satisfies the Green equation
\[
dd^c [\Phi(f)] +\delta_{Z(f)} = [dd^c \Phi(f)].
\]
\end{theorem}



\section{Integrals of Green functions}


In this section we assume that either $n>2$, or that $n\ge 1$ and $V$ is anisotropic.
 Choose  an $S_L$-valued harmonic Maass form
 \[
 f\in H_{2-n}(\omega_L)
 \]
  fixed by the action of the compact open subgroup $K \subset H(\A_f)$ on $S_L$.

The main result of this section expresses the integral
\begin{align}
\label{eq:greenint}
\mathcal{I}(f) =  \vol_\C(\widehat{\taut})^{-1} \int_{\mathrm{Sh}_K(H,\mathcal{D})} \Phi(f)  \chern( \widehat{\taut})^{n-1}
\end{align}
 in terms of the derivatives  $B'(m,\mu,s_0)$ at  $s_0=(n-1)/2$  of the coefficients   \eqref{eq:coeff} of $E_L (\tau,s,n)$.
 The proof follows the argument of \cite{Ku:Integrals} in the orthogonal case, which is based on the Siegel-Weil formula.  In fact, we slightly generalize this argument by allowing harmonic Maass forms for the inputs of the regularized theta lift. 
Let  $\kappa \in \{ 1,2\}$ be as in  \eqref{eq:kappa}.

\begin{lemma}
\label{lem:int}
We have
\begin{align*}
\mathcal{I}(f) &=\lim_{T\to \infty}\left(\frac{1}{\kappa}
\int_{\mathcal{F}_T} \left\langle f , \,E_L(\tau,s_0,n-2)\right\rangle \, d\mu(\tau) -c^+(0,0) \log(T)\right).
\end{align*}
Here $\mathcal{F}_T=\{\tau\in \mathcal{F} :  \operatorname{Im}(\tau) \leq T\}$ is the truncation at height $T$ of the standard fundamental domain $\mathcal{F}$ for the action of $\Gamma$  on the complex upper half plane.
\end{lemma} 

\begin{proof}
Exactly as in  Proposition 2.5 of \cite{Ku:Integrals},  on $\mathrm{Sh}_K(H,\mathcal{D})\smallsetminus Z(f)$ we have the equality
\[
\Phi(z,h,f)= \lim_{T\to\infty}\left(\int_{\mathcal{F}_T} \langle f, \theta_L(\tau,z,h)\rangle \, d\mu(\tau) -c^+(0,0) \log (T) \right).
\]
The  integral $\mathcal{I}(f)$ is therefore given by the limit as $T\to \infty$ of 
\begin{align*}
  \vol_\C(\widehat{\taut})^{-1}
 \int_{\mathrm{Sh}_K(H,\mathcal{D})} 
   \int_{\mathcal{F}_T} \langle f, \theta_L(\tau,z,h)\rangle \, d\mu(\tau) \chern(\widehat{\taut})^{n-1} -c^+(0,0) \log(T) .
\end{align*}
As $\mathcal{F}_T$ is compact, we may interchange the order of integration.
Inserting \eqref{lattice theta} and  \eqref{eq:intgeo}, we find that $\mathcal{I}(f)$ is  the limit as $T\to \infty$ of 
\[
  \vol_\C(\widehat{\taut})^{-1}
\int_{\mathcal{F}_T} \left\langle f, \int_{\mathrm{Sh}_K(H,\mathcal{D})} 
\theta_L(\tau,z,h) \chern(\widehat{\taut})^{n-1}  \right\rangle \, d\mu(\tau) -c^+(0,0) \log(T)  ,
\]
which we rewrite as
\[
\int_{\mathcal{F}_T} \sum_{ \mu \in L' / L }  f(\tau)(\mu) \cdot \mathcal{I}(\tau,\varphi_\mu)  \, d\mu(\tau) -c^+(0,0) \log(T)  .
\]
Applying Proposition~\ref{lem:swgeo} completes the proof.
\end{proof}

\begin{theorem}
\label{thm:int}
 If $n>2$, or if $n=2$ and $V$ is anisotropic,  
 the integral \eqref{eq:greenint} satisfies
\[
\mathcal{I}(f)=\sum_{\mu\in L'/L}\sum_{\substack{m\in \Q \\ m>0}} c^+(-m,\mu)B'(m,\mu,s_0).
\]
If $n=1$,  the integral \eqref{eq:greenint} satisfies
\[
\mathcal{I}(f)=\frac{1}{2}\sum_{\mu\in L'/L}\sum_{\substack{m\in \Q  \\ m\geq 0}} c^+(-m,\mu)B'(m,\mu,s_0).
\]
\end{theorem}

\begin{proof}
Abbreviate
\[
A_T(f) =  \int_{\mathcal{F}_T} \left\langle f , \,E_L(\tau,s_0,n-2)\right\rangle \, d\mu(\tau),
\]
so that Lemma \ref{lem:int}  becomes
\begin{align}
\label{eq:if1}
\mathcal{I}(f) &=\lim_{T\to \infty}\left(\frac{A_T(f)}{\kappa}  -c^+(0,0) \log(T) \right).
\end{align}

Using Lemma \ref{lem:eisl} we see that
\[
E_L(\tau,s_0,n-2)\,d\mu(\tau)=\big(L_n E_L'(\tau,s_0,n)\big) \, d\mu(\tau) = -\bar\partial E_L'(\tau,s_0,n)\, d\tau ,
\]
and so 
\begin{align*}
A_T(f)
&=- \int_{\mathcal{F}_T} \left\langle f , \,\bar\partial E_L'(\tau,s_0,n)  \right\rangle d\tau\\
&=- \int_{\mathcal{F}_T} d\left(\left\langle f , \, E_L'(\tau,s_0,n)\,  \right\rangle  d\tau\right)+\int_{\mathcal{F}_T} \left\langle (\bar\partial f ), \, E_L'(\tau,s_0,n)  \right\rangle d\tau .
\end{align*}
Applying  Stokes' theorem to the first term,  and 
\[
(\bar \partial f) d\tau = -(L_{2-n}f) \,d\mu(\tau)
\]
 to the second,  we obtain
\begin{align*}
A_T(f)&=- \int_{\partial \mathcal{F}_T} \left\langle f , \, E_L'(\tau,s_0,n)\,  \right\rangle  d\tau - \int_{\mathcal{F}_T} \left\langle  L_{2-n} f , \, E_L'(\tau,s_0,n)  \right\rangle d\mu(\tau)\\
&=\int_{u=0}^1\left\langle f(u+iT ), \, E_L'(u+iT,s_0,n)\,  \right\rangle  du  \\
& \quad - \int_{\mathcal{F}_T} \left\langle  \overline{\xi_{2-n}(f)} , \, E_L'(\tau,s_0,n)  \right\rangle v^n d\mu(\tau),
\end{align*}
where $\xi_{2-n}$ is the differential operator \eqref{xi}.
Inserting the Fourier expansions \eqref{eq:fourierf} and \eqref{eq:fouriereis} of $f$ and the Eisenstein series yields
\begin{align*}
A_T(f) 
&=   \sum_{\substack{\mu\in L'/L\\ m\in \Q}} c^+(-m,\mu)B'(m,\mu;s_0,T)  \\
& \quad +\sum_{\substack{\mu\in L'/L\\ m\in \Q^+}} c^-(-m,\mu)\Gamma(n-1,4\pi |m| T) B'(m,\mu;s_0,T)  \\
&\quad - \int_{\mathcal{F}_T} \left\langle  \overline{\xi_{2-n}(f)} , \, E_L'(\tau,s_0,n)  \right\rangle v^n d\mu(\tau).
\end{align*}
The exponential decay of the incomplete gamma function and the polynomial growth of the coefficients $c^-(m,\mu)$ imply that the second term goes to zero in the limit $T\to \infty$. The third term converges to the Petersson scalar product of $\xi_{2-n}(f)$ and $E_L'(\tau,s_0,n)$. But since $ \xi_{2-n}(f)$ is cuspidal and hence orthogonal to Eisenstein series, this Petersson scalar product vanishes.
Inserting this into \eqref{eq:if1}, we find
\begin{align*}
\label{eq:if2}
\mathcal{I}(f) &=\lim_{T\to \infty}\bigg(\frac{1}{\kappa}
\sum_{\substack{\mu\in L'/L\\ m\in \Q  }} c^+(-m,\mu)B'(m,\mu;s_0,T)
-c^+(0,0) \log (T) \bigg).
\end{align*}

The exponential decay of the coefficients $B'(m,\mu;s_0,T)$ in \eqref{eq:coeff} for $m<0$ and the subexponential growth of the coefficients $c^+(-m,\mu)$ imply that the contribution of all $m<0$ in the above term vanishes. 
If $n>1$ we obtain by virtue of Lemma~\ref{lem:coeffasy} that 
\begin{align*}
\mathcal{I}(f)
&=\frac{1}{\kappa}
\sum_{\substack{\mu\in L'/L\\ m> 0}} c^+(-m,\mu)B'(m,\mu,s_0),
\end{align*}
as desired.
When $n=1$, again by Lemma~\ref{lem:coeffasy}, there is an additional  contribution to the sum from $m=0$.
\end{proof}


\chapter{Integral models and arithmetic intersection theory}
\label{s:integral models}


We recall the integral models of $\mathrm{GU}(n-1,1)$ Shimura varieties constructed by Pappas and Kr\"amer,  and the arithmetic intersection theory of Gillet-Soul\'e and Burgos-Kramer-K\"uhn on their toroidal compactifications.

From here until the end of the paper,  we assume $D=-\mathrm{disc}(\kk)$ is odd.


\section{Moduli problems}
\label{ss:basic moduli}


For any $n\ge 1$,  work of Pappas \cite{pappas} and Kr\"amer \cite{kramer}, as summarized in \S 2.3  of \cite{BHKRY-1},   provides us with a regular and flat  $\co_\kk$-stack
\[
\mathcal{M}_{(n-1,1)} \to \Spec(\co_\kk) 
\]
with reduced fibers.  For an  $\co_\kk$-scheme $S$, the objects of the groupoid $\mathcal{M}_{(n-1,1)}(S)$  are quadruples $(A,\iota,\psi,\mathcal{F})$ in which
\begin{itemize}
\item
$A \to S$ is an abelian scheme of dimension $n$,   
\item
$\iota : \co_\kk \to \End(A)$ is an $\co_\kk$-action, 
\item
 $\psi : A \to A^\vee$ is a principal polarization whose induceded Rosati involution   restricts to complex conjugation on the image of $\iota$, 
 \item
 $\mathcal{F} \subset \Lie(A)$ is  an $\co_\kk$-stable hyperplane\footnote{That is to say, an $\co_S$-module local direct summand     of rank $n-1$}
 satisfying  the signature $(n-1,1)$ condition of Kr\"amer \cite{kramer}: the actions of $\co_\kk$  on $\mathcal{F}$  and on $\Lie(A)/\mathcal{F}$ are through the structure morphism $i_S:\co_\kk \to \co_S$ and its complex conjugate, respectively. 
  \end{itemize}

\begin{remark}
The stack $\mathcal{M}_{(n-1,1)}$ is denoted $\mathcal{M}^\Kra_{(n-1,1)}$ in \cite{BHKRY-1}.
\end{remark}

\begin{remark}
One does not know a good theory of integral models $\mathcal{M}_{(n-1,1)}$ if $2$ is ramified in $\kk$, and 
this lack of knowledge is the main reason for restricting to quadratic imaginary fields of odd discriminant.
\end{remark}

\begin{definition}\label{def:relevant}
A finite dimensional $\kk$-hermitian space $W$ is \emph{relevant} if it admits a self-dual $\co_\kk$-lattice $\mathfrak{a} \subset W$.
Two relevant hermitian spaces $W$ and $W'$ are \emph{strictly similar} if there is a $\kk$-linear isomorphism $W\iso W'$ identifying the hermitian forms up to scaling by a positive rational number.
\end{definition}

\begin{proposition}\label{prop:jacobowitz}
Let $o(D)$ denote the number of prime divisors of $D$.
\begin{enumerate}
\item
Any two  self-dual $\co_\kk$-lattices in a relevant $\kk$-hermitian space  are isometric everywhere locally.

\item
There are $2^{o(D) -1}$ isomorphism classes of relevant hermitian spaces of any fixed signature $(r,s)$.
If $r+s$ is odd all  lie in the same  strict similarity class, but if $r+s$ is even all lie in different strict similarity classes.
\end{enumerate}
\end{proposition}

\begin{proof}
The first claim is a theorem of Jacobowitz \cite{jacobowitz}, which uses our assumption that $D$ is odd.  The second claim is Lemma 2.11 of  \cite{KRunitaryII}.
\end{proof}

As in Proposition 2.12(i) of \cite{KRunitaryII},  there is a decomposition
\begin{equation}\label{relevant decomp}
\mathcal{M}_{(n-1,1)} = \bigsqcup_W \mathcal{M}_W,
\end{equation}
where the disjoint union is over the strict similarity classes of relevant hermitian spaces $W$ of signature $(n-1,1)$,
and the generic fiber of each $\mathcal{M}_W$ is a Shimura variety of type $\mathrm{GU}(W)$.
When $n$ is odd,    Proposition \ref{prop:jacobowitz} implies that the disjoint union has a single term.
When $n$ is even, the disjoint union is over the $2^{o(D)-1}$ isomorphism classes of relevant $\kk$-hermitian spaces.

Denote by 
\[
\mathcal{M}_{(1,0)} \to \Spec(\co_\kk)
\]
 the moduli stack of elliptic curves with CM by $\co_\kk$.  More precisely, the functor of points assigns to an $\co_\kk$-scheme $S$ the groupoid $\mathcal{M}_{(1,0)}(S)$ whose objects are pairs $(A_0,\iota_0 )$ in which  
\begin{itemize}
\item
$A_0 \to S$ is an elliptic curve,   
\item
$\iota_0 : \co_\kk \to \End(A_0)$ is an $\co_\kk$-action satisfying the signature $(1,0)$ condition: the induced action of $\co_\kk$ on $\Lie(A_0)$  is through the  structure morphism $i_S: \co_\kk \to \co_S$.
\end{itemize}
 As $A_0$ is an elliptic curve, there is a  unique principal polarization
\begin{itemize}
\item
 $\psi_0 : A_0 \to A_0^\vee$,
 \end{itemize}
and the induced Rosati involution on $\End(A_0)$ restricts to complex conjugation on the image of $\iota_0$.  By Proposition 5.1 of \cite{KRYtiny} or  Proposition 2.1.2 of  \cite{howard-unitary-II},  $\mathcal{M}_{(1,0)}$ is  finite \'etale over $\co_\kk$.

\begin{remark}
For ease of notation, we usually write objects of $\mathcal{M}_{(1,0)}$ simply as $A_0$,  suppressing the remaining data from the notation.  Similarly, objects of $\mathcal{M}_{(n-1,1)}$ will usually be written simply as  $A$.
\end{remark}

When $n=1$, the data of $\psi$ and $\mathcal{F}$ in the moduli problem  $\mathcal{M}_{ (n-1,1) }$ are uniquely determined, and  $\mathcal{M}_{ (0,1) }$  classifies pairs $(A,\iota)$ over $\co_\kk$-schemes $S$ in which $A\to S$ is an elliptic curve and $\iota : \co_\kk \to \End(A)$ is an action such that the induced action of  $\co_\kk$ on the   $\co_S$-module  $\Lie(A)$ is through the conjugate of the structure morphism $\co_\kk \to \co_S$.  Replacing $\iota$ by its conjugate therefore defines a canonical isomorphism
\[
 \mathcal{M}_{ (0,1) } \iso \mathcal{M}_{(1,0)}.
\]

In particular, there is a decomposition of $\mathcal{M}_{(1,0)}$ analogous to \eqref{relevant decomp}, but this statement is vacuous, as the decomposition has a single term by Proposition \ref{prop:jacobowitz}.
If  $W_0$ denotes the unique, up to strict similarly, relevant $\kk$-hermitian space of signature $(1,0)$,  it will be convenient to set
\[
 \mathcal{M}_{W_0} = \mathcal{M}_{(1,0)} .
\]

Fix relevant $\kk$-hermitian spaces $(W_0,h_0)$ and $(W,h)$ of signatures $(1,0)$ and $(n-1,1)$, and 
self-dual $\co_\kk$-lattices $\mathfrak{a}_0 \subset W_0$ and $\mathfrak{a} \subset W$.
  The hermitian forms induce alternating $\Q$-bilinear forms
\begin{equation}\label{symplectic}
e_0 ( a, b ) = \mathrm{Tr}_{\kk/\Q} \left(  \frac{h_0(a,b) }{ \sqrt{-D} }  \right) ,\qquad e ( a , b ) = \mathrm{Tr}_{\kk/\Q} \left(  \frac{h(a,b) }{ \sqrt{-D} }  \right)
\end{equation}
on $W_0$ and $W$, and $\mathfrak{a}_0$ and $\mathfrak{a}$ are self-dual with respect to these forms.

\begin{definition}\label{def:kramer-pappas level}
Suppose we are given an integer $N\ge 1$ and an $\co_\kk$-scheme $S$ with $N\in \co_S^\times$.
A  \emph{level $N$-structure} on $A_0 \in \mathcal{M}_{W_0}(S)$ is a pair $(\eta_0, \xi_0)$ consisting of isomorphisms
\[
\eta_0: A_0[N] \iso  \underline{\mathfrak{a}_0 / N \mathfrak{a}_0} ,\qquad \xi_0 : \mu_N \iso \underline{\Z/N\Z} 
\]
of \'etale sheaves on $S$ that identify the Weil pairing on $A_0[N]$ induced by the principal polarization with the pairing on $\mathfrak{a}_0/N\mathfrak{a}_0$ induced by $e_0$.
Similarly, a  \emph{level $N$-structure} on $A \in \mathcal{M}_{W}(S)$ is a pair $(\eta, \xi)$ consisting of isomorphisms
\[
\eta: A[N] \iso  \underline{\mathfrak{a} / N \mathfrak{a}} ,\qquad \xi : \mu_N \iso \underline{\Z/N\Z} 
\]
 that identify the Weil pairing on $A[N]$ with the alternating pairing on $\mathfrak{a}/N\mathfrak{a}$.
\end{definition}

Over $\co_\kk[1/N]$, we obtain finite  \'etale covers
\begin{equation}\label{level covers}
\mathcal{M}_{W_0}(N) \to \mathcal{M}_{W_0/\co_\kk[1/N]} , \qquad 
\mathcal{M}_{W}(N) \to \mathcal{M}_{W/\co_\kk[1/N]} 
\end{equation}
by adding level structures to the moduli problems.   By the first claim of Proposition \ref{prop:jacobowitz}, these are independent of the choices of $\mathfrak{a}_0$ and $\mathfrak{a}$.


\section{The exceptional divisor}


Now assume $n\ge 2$.
We recall some more constructions from \S 2.3  of \cite{BHKRY-1}, following work of Pappas \cite{pappas} and Kr\"amer \cite{kramer}.

There is a normal and flat $\co_\kk$-stack
\[
\mathcal{M}^\Pap_{(n-1,1)} \to \Spec(\co_\kk)
\]
defined as the moduli space of triples  $(A,\iota,\psi)$ whose components are  as in the definition of $\mathcal{M}_{(n-1,1)}$.  However,   instead of including a hyperplane $\mathcal{F} \subset \Lie(A)$ as part of the moduli problem, we demand that the action of $\co_\kk$ on $\Lie(A)$ satisfy a Kottwitz-style signature $(n-1,1)$ condition, and the wedge conditions of Pappas. We refer the reader to  \S 2.3  of \cite{BHKRY-1} for the precise definitions.

%

Forgetting the hyperplane $\mathcal{F} \subset \Lie(A)$ defines a  morphism
\begin{equation}\label{kramer-pappas}
\mathcal{M}_{(n-1,1)} \to \mathcal{M}^\Pap_{(n-1,1)},
\end{equation}
which can be realized as a blow-up.  To interpret it as such, define the  \emph{singular locus} 
\begin{equation}\label{singular}
\mathrm{Sing}_{(n-1,1)} \subset \mathcal{M}^\Pap_{(n-1,1)},
\end{equation}
 as the reduced locus of nonsmooth points.
 It  is a  proper  $\co_\kk$-stack of  dimension $0$, supported in characteristics dividing $D$.

The morphism \eqref{kramer-pappas} can be identified with the blow-up of the singular locus, and so there is a canonical cartesian diagram
 \begin{equation}\label{basic exceptional}
\xymatrix{
{\mathrm{Exc}_{(n-1,1)} }  \ar[r] \ar[d]  &  { \mathcal{M}_{(n-1,1)} } \ar[d] \\
{\mathrm{Sing}_{(n-1,1)} } \ar[r]  & { \mathcal{M}^\Pap_{(n-1,1)} }
}
\end{equation}
in which the upper left corner is the  \emph{exceptional divisor}.
In particular,  \eqref{kramer-pappas} is relatively representable and projective, and restricts to an isomorphism
\[
\mathcal{M}_{(n-1,1)} \smallsetminus  \mathrm{Exc}_{(n-1,1)}
\iso \mathcal{M}^\Pap_{(n-1,1)} \smallsetminus \mathrm{Sing}_{(n-1,1)} .
\]

\begin{remark}\label{rem:singular description}
 If $\mathfrak{p} \subset \co_\kk$ is a prime  above $p\mid D$ with residue field $\F_\mathfrak{p}$,   the $\F_\mathfrak{p}^\alg$-points of the singular locus are those 
\[
A\in \mathcal{M}^\Pap_{(n-1,1)} ( \F_\mathfrak{p}^\alg  ) 
\]
 for which the action of $\co_\kk$ on the $\F_\mathfrak{p}^\alg$-vector space $\Lie(A)$ is through the reduction map  $\co_\kk \to \F_\mathfrak{p}^\alg$.
 At such a point,   \emph{any} hyperplane $\mathcal{F} \subset \Lie(A)$ is $\co_\kk$-stable and satisfies Kr\"amer's signature condition, and the fiber over $A$ of the left vertical morphism in \eqref{basic exceptional} is the projective space parametrizing all such $\mathcal{F}$.
\end{remark}


\section{Toroidal compactification}
\label{ss:toroidal}


We turn to the study of compactifications of $\mathcal{M}_{(n-1,1)}$.
When $n=1$ this stack is finite (hence proper)  over $\co_\kk$, and so we set 
\[
\bar{\mathcal{M}}_{(0,1)} = \mathcal{M}_{(0,1)}.
\] 
For the rest of this section, assume $n\ge 2$.

One finds in \S 2 of  \cite{howard-unitary-II} the construction of a toroidal compactification
\[
\mathcal{M}_{(n-1,1)}  \subset  \bar{\mathcal{M}}_{(n-1,1) },
\]
obtained by imitating the constructions of Faltings and Chai \cite{faltings-chai}.  If one works over $\co_\kk[1/D]$, this is a  very special case of the constructions of  Lan \cite{lan}. 
The compactification is regular, and is smooth outside  the exceptional divisor of \eqref{basic exceptional}.

 The universal abelian scheme  $A \to \mathcal{M}_{(n-1,1)}$ extends uniquely to a semi-abelian scheme 
 \[
\bar{A} \to \bar{\mathcal{M}}_{(n-1,1)}
\]
with $\co_\kk$-action.
 At a geometric point of the boundary, this semi-abelian scheme is an extension of an abelian variety  by a torus of rank two.   
 The universal hyperplane $\mathcal{F} \subset \Lie(A)$ over $\mathcal{M}_{(n-1,1)}$ extends uniquely to an $\co_\kk$-stable hyperplane
\begin{equation}\label{hyperplane extension}
\bar{\mathcal{F}} \subset \Lie(\bar{A}),
\end{equation}
which again  satisfies Kr\"amer's signature $(n-1,1)$ condition.

There is a similar toroidal compactification of $\mathcal{M}_{(n-1,1)}^\Pap$ (now only normal instead of regular), and  a Cartesian diagram
 \begin{equation}\label{basic exceptional compact}
\xymatrix{
{\mathrm{Exc}_{(n-1,1)} }  \ar[r] \ar[d]  &  { \bar{\mathcal{M}}_{(n-1,1)} } \ar[d] \\
{\mathrm{Sing}_{(n-1,1)} } \ar[r]  & { \bar{\mathcal{M}}^\Pap_{(n-1,1)} }
}
\end{equation}
extending  \eqref{basic exceptional}.
The vertical arrow on the right is the blow-up along the closed immersion in the bottom row.

\begin{remark}\label{rem:no cone}
For general Shimura varieties, a toroidal compactification depends on a choice of rational polyhedral cone decomposition of a convex cone sitting inside the real points of a $\Q$-vector space.  
For our $\mathrm{GU}(n-1,1)$ Shimura varieties, the vector space in question is one-dimensional.  
Hence the rational polyhedral cone decomposition is uniquely determined, and the toroidal compactifications are canonical.
\end{remark}

Let $\mathcal{A}_n$ be the moduli stack of principally polarized abelian schemes of dimension $n$.
There are canonical morphisms
\begin{equation}\label{to siegel}
\mathcal{M}_{(n-1,1)} \map{\eqref{kramer-pappas}} \mathcal{M}_{(n-1,1)}^\Pap \to \mathcal{A}_{n/\co_\kk},
\end{equation}
in which the first arrow is projective, and the second is  finite.

\begin{lemma}\label{lem:normalized compactification}
If $\bar{\mathcal{A}}_n$ is any choice of smooth Faltings-Chai toroidal compactification,  
$\bar{\mathcal{M}}^\Pap_{(n-1,1)}$ is canonically identified with the normalization of 
 \begin{equation}\label{pap-siegel}
 \mathcal{M}_{(n-1,1)}^\Pap \to \bar{\mathcal{A}}_{n/\co_\kk}.
 \end{equation}
 See \cite[\href{https://stacks.math.columbia.edu/tag/0BAK}{Tag 0BAK}]{stacks-project} for normalization.
\end{lemma} 

\begin{proof}
By  \cite{howard-unitary-II},  the universal abelian scheme $A \to \mathcal{M}^\Pap_{(n-1,1)}$  extends (necessarily uniquely) to a semi-abelian scheme $\bar{A}$ over the compactification.
This extension satisfies a universal  property: Suppose $S$ is an irreducible normal scheme with generic point $\eta \to  S$, and we are given a morphism 
\[
\eta \to  \mathcal{M}^\Pap_{(n-1,1)}. 
\]
If the pullback $A_\eta$  extends to a semi-abelian scheme over $S$,  then this extension is the pullback of 
$\bar{A}$ via  a (necessarily unique) morphism 
\[
S \to \bar{\mathcal{M}}^\Pap_{(n-1,1)}
\]
restricting to  the given morphism at the generic point.

This extension  property is analogous to Theorem IV.5.7(5) of \cite{faltings-chai}, but the statement is  simplified by  the uniqueness of the rational polyhedral cone decomposition of Remark \ref{rem:no cone}.
If one works over  $\co_\kk[1/D]$, the extension property is  a special case of  Theorem 6.4.1.1(6) of \cite{lan}.  The proof over $\co_\kk$ is exactly the same. 

On the other hand, Theorem 11.4 of \cite{lan-ramified} shows that the compactification of $\mathcal{M}^\mathrm{Pap}_{(n-1,1)}$ defined by normalization  also carries a semi-abelian scheme satisfying the same universal  property, and the lemma follows.
\end{proof}

Fix an $\mathcal{M}_W$ as in \eqref{relevant decomp}.  
We   need compactifications of the \'etale covers
\[
\mathcal{M}_{W}(N) \to \mathcal{M}_{W/\co_\kk[1/N]} 
\]
of \eqref{level covers}.
Rather than repeat the constructions of \cite{howard-unitary-II} with level structure, 
we will  appeal to the results of \cite{lan-ramified} and \cite{mp}, in which integral models and compactifications of general Hodge-type Shimura varieties are constructed  as normalizations of Faltings-Chai compactifications of Siegel moduli spaces.
These results cannot be applied directly to the covers above, as the failure of the first arrow in \eqref{to siegel} to be finite (it contracts the entire exceptional divisor to a $0$-dimensional substack) implies that Lemma \ref{lem:normalized compactification}  is false if $\mathcal{M}^\Pap_{(n-1,1)}$ is replaced by  $\mathcal{M}_{(n-1,1)}$.  Thus the proof of Proposition \ref{prop:full compactification} below requires the roundabout step of  first compactifying covers of  $\mathcal{M}^\Pap_{(n-1,1)}$.

For any $N\ge 1$,  define  $ \bar{\mathcal{M}}_W(N)$ as the  normalization of  
\[
\mathcal{M}_W(N)  \to \bar{\mathcal{M}}_{(n-1,1)/\co_\kk[1/N]}.
\]

\begin{proposition}\label{prop:full compactification}
The stack $\bar{\mathcal{M}}_W(N)$ is regular, flat, and proper over $\co_\kk[1/N]$, and is a projective scheme if $N\ge 3$. 
 It is smooth away from its exceptional divisor
 \[
 \mathrm{Exc}_{(n-1,1)} \times_{ \bar{\mathcal{M}}_{(n-1,1)} }  \bar{\mathcal{M}}_W(N).
 \]
 In particular, it is smooth  in a neighborhood of its boundary, which is a Cartier divisor  smooth over $\co_\kk[1/N]$.  
\end{proposition}

\begin{proof}

There is a decomposition
\[
\mathcal{M}^\Pap_{(n-1,1)} = \bigsqcup_W \mathcal{M}^\Pap_W,
\]
exactly as in \eqref{relevant decomp}.  We may add $N\ge 1$ level structure, exactly as in Definition \ref{def:kramer-pappas level},  to obtain a finite \'etale cover 
\[
\mathcal{M}^\Pap_W (N) \to \mathcal{M}^\Pap_{W/\co_\kk[1/N]}. 
\]
This cover has a  compactification $\bar{\mathcal{M}}^\Pap_W (N)$ defined as the normalization of 
\[
\mathcal{M}^\Pap_W (N) \to  \bar{\mathcal{M}}^\Pap_{(n-1,1) / \co_\kk[1/N] },
\]
which, by Lemma \ref{lem:normalized compactification},  is the same   as the normalization of 
\begin{equation}\label{pap-siegel level}
\mathcal{M}^\Pap_W (N) \to   \bar{\mathcal{A}}_{n/\co_\kk[1/N]}.
\end{equation}

The properness,  flatness, and normality of $\bar{\mathcal{M}}^\Pap_W (N)$  follow  from its construction as the  normalization of a proper, flat, and smooth $\co_\kk[1/N]$-stack.
It is not regular (as this is not true even of its interior).  However, its singular locus 
\[
\mathrm{Sing}_W(N) = \mathrm{Sing}_{(n-1,1)} \times_{ \mathcal{M}^\Pap_{(n-1,1) }} \mathcal{M}^\Pap_W(N)
\]
 is $0$-dimensional and proper over $\co_\kk[1/N]$, and the smoothness of
\[
\mathcal{M}^\Pap_W (N) \smallsetminus \mathrm{Sing}_W(N)
\]
implies the smoothness of 
\[
\bar{\mathcal{M}}^\Pap_W (N) \smallsetminus \mathrm{Sing}_W(N)  .
\]
This is a consequence of the following principle, found  in \S 14 of \cite{lan-ramified} and Theorem 1 of \cite{mp}:  if $U  \subset \bar{\mathcal{M}}^\Pap_W (N)$ is an  open neighborhood of the boundary $\partial\bar{\mathcal{M}}^\Pap_W (N)$, then the singularities of  $U$ are no worse than the singularities of $U \smallsetminus \partial\bar{\mathcal{M}}^\Pap_W (N)$.

By the same principle (and  its proof), the smoothness of $\bar{\mathcal{M}}^\Pap_W (N)$ near the boundary implies the smoothness of the boundary  divisor itself.  This is actually a particular feature of our $\mathrm{GU}(W)$ Shimura varieties.  For more  general  Shimura varieties the boundary has a canonical stratification by locally closed substacks (see Theorem 4.1.5 of \cite{mp} or \S 9 of \cite{lan-ramified}), and the above cited principle tells us that each \emph{stratum} is smooth.  In our special case each boundary stratum is a divisor, and is simply  a union of connected components of $\partial\bar{\mathcal{M}}^\Pap_W (N)$.  Hence the entire boundary  is  a Cartier divisor smooth over $\co_\kk[1/N]$.

Using the characterization of \eqref{basic exceptional compact} as a blow-up,  one may identify $\bar{\mathcal{M}}_W(N)$  with the blow-up of $\bar{\mathcal{M}}^\Pap_W (N)$ along its singular locus.  As the singular locus does not meet the boundary, it follows that $\bar{\mathcal{M}}_W(N)$ is itself proper and flat with boundary a smooth Cartier divisor.  Moreover, this shows that $\bar{\mathcal{M}}_W(N)$ is smooth away from the exceptional divisor of the blow-up, and in particular in an open neighborhood of its boundary.  As the interior is regular, being  \'etale over the regular stack $\mathcal{M}_W$,  we find that $\bar{\mathcal{M}}_W(N)$ is regular.

%

Finally, when $N\ge 3$ we claim that both the source and target of 
\[
\bar{\mathcal{M}}_W(N) \to \bar{\mathcal{M}}^\Pap_W(N)
\]
 are  projective schemes.  As this morphism is a blow-up, it suffices to prove this for the target.
 For this, we use the results of Chapter V.5 of \cite{faltings-chai}.  Adding principal level structure yields a finite \'etale cover 
 \[
  \mathcal{A}_n(N) \to \mathcal{A}_{n/\Z[1/N]},
 \]
  and the  compactification  $\bar{\mathcal{A}}_n$ may be chosen so that \eqref{pap-siegel level} factors as
 \[
 \mathcal{M}^\Pap_W (N) \to   \bar{\mathcal{A}}_n(N)  \to  \bar{\mathcal{A}}_{n/\co_\kk[1/N]},
 \]
where   $\bar{\mathcal{A}}_n(N)$ is the smooth projective scheme obtained by normalizing
\[
 \mathcal{A}_n(N)  \to  \bar{\mathcal{A}}_{n/\co_\kk[1/N]}.
\]
This realizes $\bar{\mathcal{M}}^\Pap_W (N)$ as the normalization of the first arrow in the composition, and hence provides us with a finite and relatively representable morphism
\[
\bar{\mathcal{M}}^\Pap_W (N) \to  \bar{\mathcal{A}}_n(N).
\]
As the target is a projective scheme, so is the source.
\end{proof}

\begin{remark}
One could avoid the constructions of \cite{howard-unitary-II} and \cite{lan} altogether, and define
$\bar{\mathcal{M}}_{(n-1,1)}$ from the start as a blow-up of the normalization of \eqref{pap-siegel}.
All properties that we need to know about this compactification could then be deduced from \cite{lan-ramified} and \cite{mp}, except for the extension \eqref{hyperplane extension} of the universal hyperplane across the boundary (which is needed in the proof of Proposition \ref{prop:taut-hodge1}).
For this one  needs  detailed information about the local  charts near the boundary,  as found  in \cite{howard-unitary-II}.
\end{remark}

\section{Arithmetic Chow groups}
\label{ss:chow}


We summarize the arithmetic intersection theory of 
Gillet-Soul\'e \cite{gillet-soule90, SABK} and  Burgos-Kramer-K\"uhn \cite{BBK,BKK}
for the  Shimura varieties  $\mathcal{M}_W$ of \eqref{relevant decomp},
 mostly for the purposes of fixing normalizations of arithmetic degrees, heights, and volumes.

As usual,  denote by $\widehat{\mathrm{Pic}}(\mathcal{M}_W)$ the group of isomorphism classes of hermitian line bundles on $\mathcal{M}_W$, and similarly for the toroidal compactification $\bar{\mathcal{M}}_W$.
    In order to account for metrics that do not extend smoothly across the boundary, when we work on the compactification $\bar{\mathcal{M}}_W$ we must relax the notion of a hermitian metric.

    Accordingly, denote by $\widehat{\Pic}( \bar{\mathcal{M}}_W, \mathscr{D}_\BKK)$ the group of isomorphism classes of pre-log singular hermitian line bundles on $\bar{\mathcal{M}}_W$, in the sense of  Definition 1.20 of \cite{BBK}.
Exactly as in Proposition 5.2.1 of \cite{howard-volumes-I}, the natural  maps
\begin{equation}\label{pre-log injection}
\widehat{\Pic}( \bar{\mathcal{M}}_W ) \to \widehat{\Pic}( \bar{\mathcal{M}}_W, \mathscr{D}_\BKK ) \to \widehat{\Pic}( \mathcal{M}_W)
\end{equation}
are injective.

Fix an integer $N\ge 3$, so that the toroidal compactification  $\bar{\mathcal{M}}_W(N)$ of Proposition \ref{prop:full compactification} is a regular scheme of dimension $n$, projective and flat over $\co_\kk[1/N]$.
For $0\le d \le n$,  we have a homomorphism of   arithmetic Chow groups 
\begin{equation}\label{chow morphism}
\widehat{\CH}^d_\GS ( \bar{\mathcal{M}}_W(N) ) 
\to \widehat{\CH}^d (\bar{\mathcal{M}}_W(N) , \mathscr{D}_\BKK )
\end{equation}
from Theorem 6.23 and  (7.49) of \cite{BKK}.

The domain of \eqref{chow morphism} is the arithmetic Chow group of Gillet-Soul\'e \cite{gillet-soule90, SABK}, defined as the rational equivalence classes of arithmetic cycles $(\mathcal{Z}, g )$ in which $\mathcal{Z}$ is a codimension $d$ cycle on $\bar{\mathcal{M}}_W(N)$, and $g$ is a $(d-1,d-1)$ current on $\bar{\mathcal{M}}_W(N)(\C)$ satisfying the Green equation 
\[
dd^c g + \delta_\mathcal{Z} = [\omega ]
\]
of currents   for some smooth $(d,d)$ form $\omega$.

The codomain of \eqref{chow morphism} is the pre-log-log arithmetic Chow group of Burgos-Kramer-K\"uhn \cite{BBK,BKK}.  It is defined as rational equivalence classes of arithmetic cycles $(\mathcal{Z}, \mathfrak{g})$
where $\mathcal{Z}$ is as above, and $\mathfrak{g}$ is as in Definition 1.15 of \cite{BBK}.
For our purposes it is enough to know that  $\mathfrak{g}=(\omega,g)$ is a pair consisting of a  
$(d-1,d-1)$ current on $\bar{\mathcal{M}}_W(N)(\C)$ and a smooth $(d,d)$ form $\omega$ on the interior 
$\mathcal{M}_W(N)(\C)$, again related by the above Green equation (by Proposition 7.44 of \cite{BKK}).  
There are further constraints: The form $\omega$ is required to have pre-log-log growth along the boundary of the toroidal compactification (which implies local integrability on the compactification, so that the Green equation makes sense).
The current $g$ is required to be represented by a smooth form defined on the complement of some codimension $d$ cycle on $\mathcal{M}_W(N)(\C)$, again satisfying certain growth conditions.  
Moreover, the relation imposed between $g$ and $\omega$ is stronger than the Green equation.  Nevertheless, the object $\mathfrak{g}$ is completely determined by the current $g$, so we usually just write $(\mathcal{Z},g)$ for the arithmetic cycle.

Whenever $N \mid N'$ the  map
$
\bar{\mathcal{M}}_W(N') \to \bar{\mathcal{M}}_W(N)
$
induces a pullback 
\[
\widehat{\CH}^d (\bar{\mathcal{M}}_W(N) , \mathscr{D}_\BKK ) 
\to \widehat{\CH}^d (\bar{\mathcal{M}}_W(N') , \mathscr{D}_\BKK ),
\]
allowing us to define, exactly as in \S 6.3 of \cite{BKK}, 
\[
\widehat{\CH}^d (\bar{\mathcal{M}}_W , \mathscr{D}_\BKK )
= \mil_{N\ge 3}  \widehat{\CH}^d (\bar{\mathcal{M}}_W(N) , \mathscr{D}_\BKK ).
\]
There is an \emph{arithmetic intersection pairing}
\[
\widehat{\CH}^{d_1} (\bar{\mathcal{M}}_W , \mathscr{D}_\BKK ) \otimes \widehat{\CH}^{d_2} (\bar{\mathcal{M}}_W , \mathscr{D}_\BKK )
 \to \widehat{\CH}^{d_1+d_2} (\bar{\mathcal{M}}_W , \mathscr{D}_\BKK )_\Q,
\]
an  \emph{arithmetic Chern class map}  
\[
\widehat{\Pic}( \bar{\mathcal{M}}_W, \mathscr{D}_\BKK )  \to  \widehat{\CH}^1 (\bar{\mathcal{M}}_W , \mathscr{D}_\BKK ),
\]
and an  \emph{arithmetic degree}
\[
\widehat{\deg} : \widehat{\CH}^n(\bar{\mathcal{M}}_W , \mathscr{D}_\BKK ) \to \R 
\]
induced by the analogous structures on the arithmetic Chow group of each scheme $\bar{\mathcal{M}}_W(N)$ with $N\ge 3$.

\begin{remark}
The conventions for delta currents, currents associated to  smooth forms,  and Green forms used  in  \cite{BKK} and \cite{BBK}   differ from those of Gillet-Soul\'e \cite{gillet-soule90, SABK} by powers of $2$ and  $2\pi i$.
We will always use the conventions of Gillet-Soul\'e.
 For example, the arithmetic Chern class map sends a pre-log singular hermitian line bundle $\widehat{\taut}$ to the arithmetic divisor
\[
\widehat{\mathrm{div}}(s) = ( \mathrm{div}(s) ,  - \log\| s\|^2  ) 
\]
for any nonzero rational section $s$.
\end{remark}

To make explicit the normalization of the arithmetic degree, first note that for each $N \ge 3$, 
the  finite \'etale cover $\mathcal{M}_W(N) \to \mathcal{M}_{W/\co_\kk[1/N]}$ has constant fiber degree, in the sense that there is a $d_N\in \Z$ satisfying
\begin{equation}\label{level degree}
   \frac{d_N}{ \# \Aut(x)    }  
  =  \#  \{  \mbox{geometric points}\ y \to \mathcal{M}_W(N)\ \mbox{above}\ x\} 
  \end{equation}
for every geometric point    $x\to \mathcal{M}_{W/\co_\kk[1/N]}$.
Now suppose we have a class 
\[
\widehat{\mathcal{Z}}=( \mathcal{Z}_N , g_N )_{N\ge 3} \in
  \widehat{\CH}^n (\bar{\mathcal{M}}_W  , \mathscr{D}_\BKK ) .
\]
If we write  each  
 $
 \mathcal{Z}_N = \sum_i m_{N,i} \mathcal{Z}_{N,i}
 $
  as a $\Z$-linear combination of $0$-dimensional irreducible closed subschemes on $\bar{\mathcal{M}}_W(N)$, 
 then the real number
 \[
  \widehat{\deg}_N ( \mathcal{Z}_N, g_N ) = 
 \sum_i  m_{N,i}   \cdot   \# \mathcal{Z}_{N,i}( \F^\alg_{i})   \cdot   \log (  \#  \F_{i} )  
 +   \int_{ \mathcal{M}_W(N)(\C) } g_N,
 \]
 where  $\F_{i}$   is the residue field of the unique prime of $\co_\kk[1/N]$ below  $\mathcal{Z}_{N,i}$,
 is well-defined up to adding a $\Q$-linear combination of $\{ \log(p) : p \mid N\}$, and 
\begin{equation}\label{degree normalization}
 \widehat{\deg} (\widehat{\mathcal{Z}}) = \frac{1}{d_N}   \cdot \widehat{\deg}_N ( \mathcal{Z}_N, g_N )
\end{equation}
up to the same ambiguity.

\begin{remark}\label{rem:no half}
There is no $1/2$ in front of the integral, in apparent disagreement with \S 3.4.3 of \cite{gillet-soule90}. 
In fact there is no disagreement.  For us  $\mathcal{M}_W(N)(\C)$ means the set of all 
morphisms $\Spec(\C) \to \mathcal{M}_W(N)$ as $\kk$-schemes, whereas Gillet-Soul\'e would take morphisms as $\Q$-schemes.   
Thus for Gillet-Soul\'e the domain of integration would be our $\mathcal{M}_W(N)(\C)$  \emph{together with its complex conjugate}, yielding an integral  twice as big as  ours.
\end{remark}

We will also need the \emph{complex degree} 
\[
\deg_\C : \widehat{\CH}^{n-1}( \bar{\mathcal{M}}_W ,\mathscr{D}_\BKK )   \to \Q
\]
sending a class
\[
\widehat{\mathcal{Z}}=( \mathcal{Z}_N , g_N )_{N\ge 3} \in
  \widehat{\CH}^{n-1} (\bar{\mathcal{M}}_W  , \mathscr{D}_\BKK ) 
\]
 to
\[
\deg_\C ( \widehat{\mathcal{Z}}) =  \frac{1}{d_N} \cdot \deg( \mathcal{Z}_{N/\C}) ,
\]
where the degree on the right hand side is the usual degree of a $0$-cycle on the complex fiber $\bar{\mathcal{M}}_W(N)_{/\C}$.
Equivalently, 
\begin{equation}\label{chern degree}
\deg_\C ( \widehat{\mathcal{Z}}) = 
 \frac{1}{d_N} \cdot \int_{\mathcal{M}_W(N) (\C)  } \omega_N 
\end{equation}
is the integral of the $(n-1,n-1)$ form determined by the Green equation.
\[
dd^c g_N + \delta_{\mathcal{Z}_N} = [\omega_N ].
\]


\section{Volumes and heights}
\label{ss:volumes}


For a pre-log singular hermitian line bundle  
\[
\widehat{\taut } \in  \widehat{\Pic}(\bar{\mathcal{M}}_W , \mathscr{D}_\BKK ),
\]
define the \emph{complex volume}
\[
\vol_\C( \widehat{\taut } )  = \deg_\C(  \widehat{\taut}^{n-1} ) 
\]
by applying the complex degree of \S \ref{ss:chow} to the $(n-1)$-fold iterated arithmetic intersection of  the arithmetic Chern class 
\[
\widehat{\taut} \in  \widehat{\CH}^1 (\bar{\mathcal{M}}_W , \mathscr{D}_\BKK ) .
\]
Of course this only depends on the restriction of $\widehat{\taut}$ to a line bundle on the complex fiber $\bar{\mathcal{M}}_{W/\C}$.
Similarly,  define the  \emph{arithmetic volume}  
\[
\widehat{\vol}( \widehat{\taut} )  
=  \widehat{\deg}(   \widehat{\taut}^n )  
\]
 using the iterated arithmetic intersection and  arithmetic degree of  \S \ref{ss:chow}.

 \begin{remark}\label{rem:chern form}
The pre-log singularity of the metric on $\widehat{\taut}$ implies the integrability of its Chern form:   the smooth $(1,1)$ form on $\mathcal{M}_W(\C)$ defined locally by 
\[
\chern(\widehat{\omega}^\mathrm{Hdg}_{A / \mathcal{M}_W} ) 
=  \frac{1}{2\pi i} \partial \overline{\partial} \log \|s\|^2
= - d d^c  \log \|s\|^2 
\]
for any nonzero  holomorphic  section $s$. 
We will frequently use the equivalent characterization of the complex volume 
\[
\vol_\C( \widehat{\taut } )   =  \int_{ \mathcal{M}_W(\C) } \chern ( \widehat{\taut })^{n-1}
\]
 as  the (orbifold) integral of its top exterior power, which follows from \eqref{chern degree}.
\end{remark}

Let $\mathcal{Z}$ be a  divisor on $\bar{\mathcal{M}}_W$, and assume that $\mathcal{Z}$  intersects the boundary 
 $\partial \bar{\mathcal{M}}_W$  properly in the generic fiber.
 As in the discussion following Theorem 7.58  of  \cite{BKK}, one can then define the \emph{arithmetic height}    (or \emph{Faltings height})   \emph{of $\mathcal{Z}$ with respect to $\widehat{\taut}$}, denoted
 \[
  \mathrm{ht}_{\widehat{\taut}}(\mathcal{Z})   \in \R.
  \]
 To make it explicit,  fix $N\ge 3$.  Choose a pair $(\mathcal{Y},  g_\mathcal{Y})$ representing the iterated intersection
\[
 \widehat{\taut}^{n-1} \in  \widehat{\CH}^{n-1}(\bar{\mathcal{M}}_W(N), \mathscr{D}_\BKK ),
\]
and  do this in such a way  that  the supports of $\mathcal{Z}_N$ and  $\mathcal{Y}$ do not intersect in the generic fiber, where $\mathcal{Z}_N$ is the pullback of $\mathcal{Z}$ via 
\[
\bar{\mathcal{M}}_W(N) \to \bar{\mathcal{M}}_{W/\co_\kk[1/N]}.
\]
Now  form the intersection 
\[
 \mathcal{Z}_N \cdot \mathcal{Y} \in \CH^n_{  \mathcal{Z}_N  \cap  \mathcal{Y} } (\bar{\mathcal{M}}_W(N))_\Q 
\]
 in the Chow group with support along  the closed subset $\mathrm{Sppt}(\mathcal{Z}_N) \cap \mathrm{Sppt}(\mathcal{Y})$.
As this closed subset  is supported in finitely many nonzero characteristics, there is a natural change-of-support map
\[
\CH^n_{  \mathcal{Z}  \cap  \mathcal{Y} } (\bar{\mathcal{M}}_W(N) )  \to 
\bigoplus_{  \substack{ \mathfrak{p} \subset \co_\kk   \\   \mathfrak{p} \nmid N\co_\kk }   }
\CH^n_\mathfrak{p}  (\bar{\mathcal{M}}_W (N) )
\]
where $\CH^n_\mathfrak{p}  (\bar{\mathcal{M}}_W(N) )$ is the Chow group with support in the mod $\mathfrak{p}$ fiber of $\bar{\mathcal{M}}_W (N)$.
For each $\mathfrak{p}$ there is a degree map
\begin{equation}\label{local degree}
\deg_\mathfrak{p} : \CH^n_\mathfrak{p}  (\bar{\mathcal{M}}_W (N) ) \to \Z
\end{equation}
sending a reduced and irreducible codimension $n$ subscheme (a.k.a., a closed point)  $\mathcal{C}  \subset \bar{\mathcal{M}}_W (N)$ supported at $\mathfrak{p}$ to 
\[
\deg_\mathfrak{p} (\mathcal{C}) = \# \mathcal{C}( \F_\mathfrak{p}^\alg) 
\]
where $\F_\mathfrak{p} = \co_\kk/\mathfrak{p}$.
The arithmetic height  satisfies 
\begin{equation}\label{height normalization}
d_N\cdot \mathrm{ht}_{\widehat{\taut}}(\mathcal{Z})  = 
\sum_{  \substack{ \mathfrak{p} \subset \co_\kk   \\   \mathfrak{p} \nmid N\co_\kk }   }
\deg_\mathfrak{p} ( \mathcal{Z}_N \cdot \mathcal{Y} ) \cdot \log(\#\F_\mathfrak{p})
+  \int_{ \mathcal{Z}_N(\C) } g_\mathcal{Y} 
\end{equation}
up to a $\Q$-linear combination of $\{ \log(p) : p\mid N\}$, where $d_N$ is as in  \eqref{level degree}.

   A choice of pre-log singular hermitian metric on the line bundle $\co(\mathcal{Z})$ determines a class
\[
\widehat{\mathcal{Z}} \in \widehat{\mathrm{Pic}}(\bar{\mathcal{M}}_W, \mathscr{D}_\BKK).
\]
If we view  the constant function $1$ as a global section of  $\co(\mathcal{Z})$  and set $g_\mathcal{Z}= - \log\|1\|^2$, 
the arithmetic Chern class of this hermitian line bundle is
\[
( \mathcal{Z} , g_\mathcal{Z} ) \in \widehat{\CH}^1( \bar{\mathcal{M}}_W, \mathscr{D}_\BKK) ,
\]
and Proposition 7.56 of \cite{BKK} gives the relation
\begin{equation}\label{height degree}
\widehat{\deg}(   \widehat{\mathcal{Z}}  \cdot  \widehat{\taut}^{n-1}   )
= \mathrm{ht}_{\widehat{\taut}}(\mathcal{Z}) + \int_{ \mathcal{M}_W(\C) }  g_\mathcal{Z} \cdot \chern(\widehat{\taut})^{n-1} .
\end{equation}

From the point of view of arithmetic volumes and heights, some hermitian line bundles are essentially indistinguishable.

\begin{definition}\label{def:numerical}
We say that   
\[
\widehat{\taut}_1 , \widehat{\taut}_2   \in   \widehat{\Pic}(\mathcal{M}_W  )
\]
 are \emph{numerically equivalent} if the difference
$\widehat{\mathcal{E}} =   \widehat{\taut}_1 - \widehat{\taut}_2$ lies in the subgroup 
\begin{equation}\label{pre-log pic}
\widehat{\Pic}( \bar{\mathcal{M}}_W, \mathscr{D}_\BKK )  \subset \widehat{\Pic}(\mathcal{M}_W  )
\end{equation}
of pre-log singular hermitian line bundles from \eqref{pre-log injection},   and satisfies 
$
\widehat{\deg}( \widehat{\mathcal{Z}} \cdot \widehat{\mathcal{E}} ) =0 
$
for all $\widehat{\mathcal{Z}} \in \widehat{\CH}^{n-1}(\bar{\mathcal{M}}_W,\mathscr{D}_\BKK ).$
\end{definition}

\begin{lemma}\label{lem:numerical basics}
If $\widehat{\taut}_1$ and  $\widehat{\taut}_2$ are numerically equivalent, then their Chern forms are equal.  If we add the assumption that both lie in the subgroup \eqref{pre-log pic} then
\begin{equation}\label{trivial shift}
\widehat{\vol}( \widehat{\taut_1} ) = \widehat{\vol}( \widehat{\taut}_2 )
\quad\mbox{and}\quad
\mathrm{ht}_{ \widehat{\taut}_1  } (\mathcal{Z}) 
= \mathrm{ht}_{ \widehat{\taut}_2 }  (\mathcal{Z})
\end{equation}
for any divisor $\mathcal{Z}$ as above.
\end{lemma}

\begin{proof}
Set  $\widehat{\mathcal{E}} =  \widehat{\taut}_1 - \widehat{\taut}_2$.
If  $g$ is any smooth $(n-2,n-2)$  form on $\bar{\mathcal{M}}_W(\C)$, the arithmetic cycle class
\[
\widehat{\mathcal{Z}} = ( 0, g) \in\widehat{\CH}^{n-1}(\bar{\mathcal{M}}_W , \mathscr{D}_\BKK)
\]
satisfies
\[
0 = \widehat{\deg}( \widehat{\mathcal{Z}} \cdot \widehat{\mathcal{E}} )
=  \int_{\mathcal{M}_W(\C)}    \chern( \widehat{\mathcal{E}}  )  \wedge g.
\]
As this holds for all choices of $g$, we must have $ \chern( \widehat{\mathcal{E}}  )=0$.  
This proves the desired equality of Chern forms.
The first equality of \eqref{trivial shift} is clear from 
\[
\widehat{\vol}( \widehat{\taut_1} ) 
= \sum_{i=0}^n \widehat{\deg} ( \widehat{\taut}_2^{n-i} \cdot \widehat{\mathcal{E}}^i ), 
\]
as every term with $i>0$ vanishes. The second  equality follows from \eqref{height degree} and the equality of Chern forms already proved.
\end{proof}

If $\mathcal{Z}$ is a  Cartier divisor on $\bar{\mathcal{M}}_W$ supported in nonzero characteristics,  the line bundle
 $\co( \mathcal{Z} )$   is canonically trivialized in the generic fiber by the constant function $1$. 
  For any real number $c$, denote by
 \begin{equation}\label{constant metrics}
 (\mathcal{Z} ,c)  \in    \widehat{\Pic}( \bar{\mathcal{M}}_W )
 \end{equation}
 the line bundle $\co(\mathcal{Z})$  endowed with the constant metric characterized by  $-\log \| 1\|^2 =   c$.
 In particular, adding $(0,c)$ to a hermitian line bundle just rescales $\| \cdot\|^2$  by $e^{-c}$.  The following lemma describes the effect of this on arithmetic volumes.

\begin{lemma}\label{lem:trivial volume shift}
For  any $\widehat{\taut} \in \widehat{\Pic}( \bar{\mathcal{M}}_W, \mathscr{D}_\BKK )$ and  any constant $c\in \R$,  we have 
\[
\widehat{\vol}\big( \widehat{\taut} + (0 ,c ) \big) =
\widehat{\vol}(  \widehat{\taut} )  + n c   \int_{\mathcal{M}_W(\C) }  \chern(\widehat{\taut})^{n-1}  ,
\]
where $ \chern(\widehat{\taut})$ is the Chern form of Remark \ref{rem:chern form}.
\end{lemma}

\begin{proof}
A trivial induction argument shows that 
\[
\widehat{\taut}^{k-1} \cdot (0 , c )   =  ( 0  , c  \cdot  \chern(\widehat{\taut})^{k-1} )
\in \widehat{\CH}^k( \bar{\mathcal{M}}_W ,\mathscr{D}_\BKK)
\]
 for all $k>0$.  As $(0,c) \cdot (0,c) =0$  we  deduce 
\begin{align*}
\big( \widehat{\taut} + (0 ,c ) \big)^k 
& = \widehat{\taut} ^k +  k \widehat{\taut}^{k-1} \cdot (0 , c )   \\
& = \widehat{\taut} ^k + k ( 0  ,  c  \cdot  \chern(\widehat{\taut})^{k-1} ).
\end{align*}
Setting $k=n$ and taking the arithmetic degree  completes the proof.
\end{proof}


\chapter{A Shimura variety and its line bundles}
\label{s:special shimura bundles}


For the rest of this paper we fix a $\kk$-hermitian space $V$ of signature $(n-1,1)$, with $n\ge 1$, and assume that $V$ admits a self-dual $\co_\kk$-lattice \eqref{self-dual lattice}.  The hermitian form on $V$ is denoted $\langle-,-\rangle$.
We will associate to $V$ an $n$-dimensional open and closed substack  
\[
\mathcal{S}_V \subset \mathcal{M}_{(1,0)} \times_{\co_\kk} \mathcal{M}_{(n-1,1)},
\]
construct some interesting hermitian line bundles on it, and explain the relations between them.


\section{A special Shimura variety}
\label{ss:special shimura}


To attach a Shimura variety to our fixed  $V$,  choose  relevant (Definition \ref{def:relevant})  hermitian spaces  $(W_0,h_0)$ and $(W,h)$ of  signatures $(1,0)$ and $(n-1,1)$, respectively,  in such a way that  
\begin{equation}\label{hermitian hom}
V \iso \Hom_\kk(W_0 , W) 
\end{equation}
as $\kk$-hermitian spaces, where the hermitian form  $\langle-  , - \rangle$ on the right  is defined by the relation
\begin{equation}\label{basic hom hermitian}
\langle x,y\rangle  \cdot h_0(w_0,w'_0 ) = h( x(w_0) , y(w'_0) )
\end{equation}
for all $w_0 , w_0'\in W_0$ and $x,y\in V$.

\begin{remark}
Such $W_0$ and $W$ always exist, and in fact there is a distinguished choice: take $W_0=\kk$ with its norm form, and $W=V$.  Any other choice of $W_0$ will be strictly similar to this distinguished choice (this is a trivial case of Proposition \ref{prop:jacobowitz}), and the relation \eqref{hermitian hom} then forces $W$ to be strictly similar to $V$.  In particular, the Shimura varieties 
\[
\mathcal{M}_{W_0} \iso \mathcal{M}_{(1,0)} \quad \mbox{and} \quad 
\mathcal{M}_{W} \iso \mathcal{M}_{V}
\]
depend only on $V$, and not on the particular   $W_0$ and $W$ chosen in \eqref{hermitian hom}.
\end{remark}

 As  in \S 2.2 of   \cite{KRunitaryII}, if  $S$ is a connected $\co_\kk$-scheme  and 
\[
(A_0,A) \in   \mathcal{M}_{W_0}(S)  \times \mathcal{M}_{W}  (S), 
\]
 then $\Hom_{\co_\kk}(A_0,A)$ carries a positive definite  hermitian form 
\begin{equation}\label{KR hermitian}
\langle x,y\rangle =    \psi_0^{-1} \circ y^\vee \circ \psi\circ  x   \in \End_{\co_\kk}(A_0) \iso \co_\kk,
\end{equation}
where $\psi_0 :A_0 \to A_0^\vee$ and $\psi:A\to A^\vee$ are the principal polarizations.

As in \S 2.3 of \cite{BHKRY-1},  there is an open and closed substack 
\begin{equation}\label{moduli inclusion}
\mathcal{S}_V \subset \mathcal{M}_{W_0} \times_{\co_\kk} \mathcal{M}_{W}
\end{equation}
  characterized by  its  points valued in  algebraically closed fields,  which are those pairs
\[
(A_0,A)\in   \mathcal{M}_{W_0} (F) \times  \mathcal{M}_{W} (F)
\]
for which there exists an isometry 
\[
V\otimes \Q_\ell \iso  \Hom_{\co_\kk} \big( T_\ell(A_0)   , T_\ell(A)  \big)  \otimes \Q_\ell 
\]
of $\kk_\ell$-hermitian spaces for every $\ell\neq \mathrm{char}(F) $. 
Here the hermitian form on the right is defined by imitating the construction \eqref{KR hermitian} on the level of $\ell$-adic Tate modules.
  When $F=\C$ this is equivalent (by the Hasse principle for hermitian spaces) to the existence of an isometry of $\kk$-hermitian spaces
  \[
V \iso \Hom_{\kk} \big( H_1(A_0 , \Q)  , H_1(A  , \Q) \big) ,
\]
where the hermitian form on the right hand side is again defined by imitating the construction \eqref{KR hermitian} in Betti homology.


\begin{remark}\label{rem:projection fiber}
The projection 
$
\mathcal{S}_V \to \mathcal{M}_W
$
is a finite \'etale surjection, and the fiber  over a geometric point $x \in  \mathcal{M}_W(F)$ satisfies
\[
\sum_{  y \in \mathcal{S}_{V,x}(F)   } \frac{1}{ |\Aut(y) | }  = \frac{ |\mathrm{CL}(\kk) |} {  | \co_\kk^\times|} 
\cdot \begin{cases}
1 & \mbox{if $n$ is even} \\
2^{1-o(D)}  & \mbox{if $n$ is odd}.
\end{cases}
\]
\end{remark}

\begin{remark}\label{rem:honest shimura}
The   stack $\mathcal{S}_V$ is denoted $\mathcal{S}_\Kra$    in \cite{BHKRY-1}.  
 Later we will want to vary $V$, and so we have  included it in the notation to avoid confusion.
As explained in [loc.~cit.], the generic fiber of  $\mathcal{S}_V$ is a Shimura variety 
for the subgroup $G \subset \mathrm{GU}(W_0) \times \mathrm{GU}(W)$ of pairs $(g_0,g)$ for which the similitude factors of the two components are equal.
\end{remark}

 \begin{definition}
 Let $N$ be a positive integer, and let $S$ be an $\co_\kk$-scheme with $N\in \co_S^\times$.
 A \emph{level $N$-structure} on a pair $(A_0,A) \in \mathcal{S}_V(S)$ consists of  level $N$-structures $(\eta_0,\xi_0)$ and $(\eta,\xi)$ on $A_0$ and $A$, in the sense of Definition \ref{def:kramer-pappas level},  such   that $ \xi_0 =   \xi$.
  \end{definition}

Adding level structure to pairs $(A_0,A)$ defines  a finite \'etale cover
\[
 \mathcal{S}_V(N) \to  \mathcal{S}_{ V /\co_\kk[1/N]  }, 
\]
and  \eqref{moduli inclusion} lifts to a canonical open and closed immersion
\[
\mathcal{S}_V (N) \subset \mathcal{M}_{W_0} (N) \times_{\co_\kk[1/N]} \mathcal{M}_W (N).
\]
Define a toroidal compactification
\begin{equation}\label{compact inclusion}
\bar{\mathcal{S}}_V (N) \subset \mathcal{M}_{W_0} (N) \times_{\co_\kk[1/N]}  \bar{\mathcal{M}}_W (N)
\end{equation}
as the Zariski closure of $\mathcal{S}_V(N)$.  
When $N=1$ we abbreviate 
\[
\bar{\mathcal{S}}_V = \bar{\mathcal{S}}_V(1).
\]
Note that  $\bar{\mathcal{S}}_V (N)$ is also characterized as the normalization of 
\[
\mathcal{S}_V (N) \to \bar{\mathcal{M}}_{W/ \co_\kk[1/N]}.
\]

\begin{remark}\label{rem:special compact level}
As \eqref{compact inclusion} is an open and closed immersion, and $\mathcal{M}_{W_0} (N)$ is finite \'etale over $\co_\kk[1/N]$, the compactification $\bar{\mathcal{S}}_V (N)$ inherits all the nice properties of $\bar{\mathcal{M}}_W (N)$.  In particular, Proposition \ref{prop:full compactification} holds word-for-word with $\mathcal{M}_W$ replaced by $\mathcal{S}_V$, and the same is true of the entire discussion of arithmetic intersection theory in \S \ref{ss:chow} and \S \ref{ss:volumes}.
\end{remark}

\begin{definition}\label{def:special exceptional}
If $n\ge 2$,  define the \emph{exceptional divisor} 
\[
\mathrm{Exc}_V \subset \mathcal{S}_V
\]
 as the pullback of the exceptional divisor \eqref{basic exceptional} via $ \mathcal{S}_V  \to \mathcal{M}_{(n-1,1)}$.  Equivalently, it is defined by the cartesian diagram
\begin{equation}\label{special exceptional}
\xymatrix{
{\mathrm{Exc}_V }  \ar[r] \ar[d]  &  { \mathcal{S}_V } \ar[d] \\
{\mathrm{Sing}_{(n-1,1)} } \ar[r]  & { \mathcal{M}^\Pap_{(n-1,1)} }.
}
\end{equation}
\end{definition}

\begin{definition}[{\cite{KRunitaryII}}]
\label{def:KR}
 For any positive $m\in \Z$, the \emph{Kudla-Rapoport divisor}   $\mathcal{Z}_V(m)$  is the $\co_\kk$-stack classifying triples $(A_0,A,x)$ consisting of a pair  
\[
(A_0,A) \in \mathcal{S}_V(S)
\]
 and an 
$
x\in \Hom_{\co_\kk}(A_0,A)
$
satisfying $\langle x,x\rangle=m$.
\end{definition}

The natural forgetful morphism 
\[
\mathcal{Z}_V(m) \to \mathcal{S}_V
\]
is finite and unramified, with image of codimension $1$. 
Denote by 
\[
\bar{\mathcal{Z}}_V(m)\to \bar{\mathcal{S}}_V
\]
 the normalization of $\mathcal{Z}_V(m) \to \bar{\mathcal{S}}_V$. 
 
 \begin{remark}\label{rem:divisor closure}
 The morphism $\mathcal{Z}_V(m) \to \mathcal{S}_V$ determines a Weil divisor on $\mathcal{S}_V$ in the usual way:
each irreducible component of  $\mathcal{Z}_V(m)$ contributes an irreducible component to the divisor, counted with multiplicity equal to the length of the local ring of its generic point.  
By abuse of notation, we denote this Weil divisor again by $\mathcal{Z}_V(m)$.
Similarly there is a divisor determined by $\bar{\mathcal{Z}}_V(m)$, denoted the same way; it agrees with the Zariski closure in $\bar{\mathcal{S}}_V$ of the divisor $\mathcal{Z}_V(m)$.
\end{remark}

Loosely speaking,  each Kudla-Rapoport divisor  is a union of  unitary Shimura varieties associated to  $\kk$-hermitian spaces of signature $(n-2,1)$.  In  Chapter \ref{s:KR divisors} we will  make this more precise, at least  when $m$ is a prime split in $\kk$.


\section{Construction of hermitian line bundles}
\label{ss:hermitian bundles}


The construction \eqref{hodge metric} associates to the universal CM elliptic curve $A_0 \to \mathcal{M}_{W_0} = \mathcal{M}_{(1,0)}$ its metrized Hodge bundle
\begin{equation}\label{elliptic hodge}
\widehat{\omega}^\mathrm{Hdg}_{A_0 / \mathcal{M}_{W_0} } \in \widehat{\Pic}(\mathcal{M}_{W_0} ) .
\end{equation}
Similarly, the universal $A  \to \mathcal{M}_{W}$  determines a metrized Hodge bundle
\begin{equation}\label{basic hodge}
\widehat{\omega}^\mathrm{Hdg}_{A / \mathcal{M}_W } \in \widehat{\Pic}(\mathcal{M}_W ) .
\end{equation}

Pulling back the universal objects via projection to the two factors in  \eqref{moduli inclusion} yields a universal pair $(A_0,A)$ of polarized abelian schemes (of  relative dimensions $1$ and $n$)  over  $\mathcal{S}_V$.
As in \S 2.4 of \cite{BHKRY-1} there is a  \emph{metrized line bundle of modular forms}
\begin{equation}\label{metrized taut}
\widehat{\taut}_V \in \widehat{\Pic}(\mathcal{S}_V).
\end{equation}
The line bundle underlying \eqref{metrized taut} has  inverse
\begin{equation}\label{dual taut}
\taut_V^{-1} =   \Lie(A_0)  \otimes   \Lie(A) / \mathcal{F} ,
\end{equation}
where $\mathcal{F} \subset \Lie(A)$ is the universal hyperplane satisfying Kr\"amer's signature condition (\S \ref{ss:basic moduli}).
  The hermitian metric  is defined as in \S 7.2 of  \cite{BHKRY-1}: 
if we use  Proposition 2.4.2 of \cite{BHKRY-1} to identify
 \begin{equation}\label{taut realization}
 \taut_{V,z}  \subset  \Hom_{\kk} ( H_1(A_{0,z} ,\Q) , H_1(A_z,\Q) )  \otimes_\Q\C \iso V\otimes_\Q\C
 \end{equation}
at a complex point $z\in \mathcal{S}_V(\C)$, the line  $\taut_{V,z}$ is   isotropic  with respect to  the $\C$-bilinear extension of the $\Q$-bilinear form  $ [ x,y] = \mathrm{Tr}_{\kk/\Q} \langle x,y\rangle$
on $V$, and 
\begin{equation}\label{taut metric}
\| s \|^2 = - \frac{ [s,\overline{s}] }{ 4\pi e^\gamma } 
\end{equation}
for any   $s\in \taut_{V,z}$.
   Note that our  $\widehat{\taut}_V$ is denoted $\widehat{\bm{\omega}}$ in  \cite{BHKRY-1, BHKRY-2}.

\begin{proposition}\label{prop:taut-hodge1}
The  hermitian line bundles   \eqref{basic hodge} and  \eqref{metrized taut}  lie in the subgroups
\begin{align*}
\widehat{\Pic}( \bar{\mathcal{M}}_W, \mathscr{D}_\BKK )   \subset \widehat{\Pic}(\mathcal{M}_W  ) 
\quad \mbox{and} \quad 
\widehat{\Pic}( \bar{\mathcal{S}}_V, \mathscr{D}_\BKK )   \subset \widehat{\Pic}(\mathcal{S}_V  ),
\end{align*}
respectively, of \eqref{pre-log injection}.
\end{proposition}

\begin{proof}
The extension  to $\bar{\mathcal{M}}_W$ of the line bundle 
\[
\omega^\mathrm{Hdg}_{A / \mathcal{M}_W } \iso \det( \Lie(A) )^{-1}
\]
 underlying \eqref{basic hodge}  is part of \eqref{hyperplane extension}.
 The fact that the hermitian metric has a pre-log singularity along the boundary is a special case of 
 Theorem 6.16 of \cite{BKK2}.
 Note that this also uses   Proposition 3.2 of \cite{freixas}, which shows  that any rank one log-singular hermitian vector bundle in the sense of \cite{BKK2} is also a  pre-log-singular hermitian line bundle in the sense of Definition 1.20 of \cite{BBK}.
This proves the claim for  \eqref{basic hodge}, and the proof for \eqref{metrized taut} is the same.
\end{proof}

Pulling back the hermitian line bundles \eqref{elliptic hodge} and \eqref{basic hodge} via projection to the two  factors  in \eqref{compact inclusion}, we obtain three hermitian line bundles 
\begin{equation}\label{three bundles}
\widehat{\omega}^\mathrm{Hdg}_{A_0 / \mathcal{S}_V },\,
\widehat{\omega}^\mathrm{Hdg}_{A / \mathcal{S}_V },\,
\widehat{\taut}_V
  \in  \widehat{\Pic}( \bar{\mathcal{S}}_V, \mathscr{D}_\BKK  ).
\end{equation} 
In the remaining sections of this chapter, we will make explicit the relations between them. 
The reader may wish to skip directly to Theorem \ref{thm:taut-hodge compare} for the main results.


\section{An application of the Chowla-Selberg formula}


We will prove that, up to numerical equivalence (Definition \ref{def:numerical}), the first line bundle in \eqref{three bundles} is  just the trivial line bundle with a constant metric.  
The constant defining the metric is an interesting quantity in its own right.
Recall  the Chowla-Selberg formula: the Faltings height,  normalized as in \S 5.2 of \cite{BHKRY-2},
of any elliptic curve with CM by $\co_\kk$ is 
\begin{equation}\label{faltings}
h^\mathrm{Falt}_\kk   = 
- \frac{1}{2} \frac{L'(0,\eps)}{L(0,\eps)}   - \frac{1}{4}   \log(4\pi^2 D )  .
\end{equation}

\begin{proposition}\label{prop:easy numerical}
The metrized line bundles
\[
\widehat{\omega}^\mathrm{Hdg}_{A_0 / \mathcal{S}_V }  \in \widehat{\Pic}(\mathcal{S}_V )
\quad \mbox{and}\quad   ( 0 , C_1) \in  \widehat{\Pic}(\mathcal{S}_V )
\]
are  numerically equivalent, where $C_1 = \log(2\pi) + 2 h^\mathrm{Falt}_\kk$, and we are using the notation of  \eqref{constant metrics}. 
\end{proposition}

\begin{proof}
Fix an $N\ge 3$.  As $\mathcal{M}_{W_0}(N)$ is finite \'etale over $\co_\kk[1/N]$, there is an isomorphism 
\[
\mathcal{M}_{W_0}(N) \iso \bigsqcup_i  \mathcal{X}_i 
\]
with each $\mathcal{X}_i \iso \Spec(\co_{\kk_i} [ 1/N])$ for a finite field extension $\kk_i/\kk$ unramified outside $N$.
For each $\mathcal{X}_i$, consider the metrized Hodge bundle
\[
\widehat{\omega}^\mathrm{Hdg}_{A_0/\mathcal{X}_i} \in \widehat{\Pic}(\mathcal{X}_i ) 
\]
of the universal $A_0 \to \mathcal{X}_i$.  The underlying line bundle can be identified with an element in the ideal class group of $\co_{\kk_i}[1/N]$, and hence some power of it admits a trivializing section
\[
s_i \in H^0\big( \mathcal{X}_i , ( \omega^\mathrm{Hdg}_{A_0/\mathcal{X}_i})^{\otimes d_i} \big)
\iso 
H^0( A_0 , \Omega^{\otimes d_i} _{ A_0 / \mathcal{X}_i } ).
\]
Comparing \eqref{hodge metric} with the definition of the Faltings height (as in \S 5.2 of \cite{BHKRY-2}) shows that 
\begin{align*}
 \frac{ -1}{  \# \mathcal{X}_i(\C) }  \sum_{ x\in \mathcal{X}_i(\C) } \log\| s_{i,x}\|^2
 & = d_i  \cdot C_1,
\end{align*}
up to a $\Q$-linear combination of $\{ \log(p) : p\mid N\}$.

Letting $d$ be the least common multiple of all $d_i$'s, we see that 
\[
\big( \widehat{\omega}^\mathrm{Hdg}_{A_0/\mathcal{S}_V(N)} \big)^{\otimes d} \in \widehat{\Pic}( \bar{\mathcal{S}}_V(N)  )
\]
admits a trivializing section $s$  with  $-\log\| s\|^2$  constant on every connected component of $\bar{\mathcal{S}}_V(N)(\C)$, and such that  the average value of $-\log\| s\|^2$ over any $\Aut(\C/\kk)$-orbit of  components is $d\cdot C_1$,
 up to a $\Q$-linear combination of $\{ \log(p) : p\mid N\}$.

  If we represent 
\[
 d \cdot \widehat{\omega}^\mathrm{Hdg}_{A_0/\mathcal{S}_V(N)}  \in \widehat{\CH}^1(  \bar{\mathcal{S}}_V(N)  )
\]
 by the arithmetic divisor
$
\widehat{\mathrm{div}}(s)  = ( 0 , - \log\| s\|^2),
$
then for any 
\[
 (\mathcal{Z}_N , g_N) \in \widehat{\CH}^{n-1}( \bar{\mathcal{S}}_V(N)  ) 
\]
the vanishing of the Chern form of $-\log\| s\|^2$ implies  the $*$-product formula
\[
 [ - \log\| s\|^2] * g_N =      - \log\| s\|^2  \wedge \delta_{\mathcal{Z}_N},
 \]
which implies the intersection formula 
\[
d \cdot \widehat{\deg}_N\big(  (\mathcal{Z}_N , g_N) \cdot
 \widehat{\omega}^\mathrm{Hdg} _{A_0/\mathcal{S}_V(N)} \big) 
 = 
- \sum_{ x \in \mathcal{Z}_N(\C)  } \log\| s_x\|^2  
\]
up to a $\Q$-linear combination of $\{ \log(p) : p \mid N\}$. 

Let   $L\subset \C$ be a finite Galois extension of $\kk$ large enough that all complex points of $\mathcal{Z}_N$ are defined over $L$, and rewrite the equality above as  
\[
d \cdot \widehat{\deg}_N \big(   (\mathcal{Z}_N , g_N) \cdot
 \widehat{\omega}^\mathrm{Hdg}_{A_0/\mathcal{S}_V(N)}   \big) 
  = 
 \frac{-1}{ [ L:\kk] }  \sum_{  \substack{ x \in \mathcal{Z}_N(L)  \\  \sigma \in \Gal(L/\kk)  }  }   \log\| s_{x^\sigma} \|^2.
 \]
The  right hand side is  $d C_1 \cdot \# \mathcal{Z}_N(\C)$, and hence
 \[
 \widehat{\deg}_N \big(  (\mathcal{Z}_N , g_N) \cdot \widehat{\omega}^\mathrm{Hdg}_{A_0/\mathcal{S}_V(N)}   \big)  
 = C_1 \cdot \# \mathcal{Z}_N(\C) 
 =   \widehat{\deg}_N  \big(   (\mathcal{Z}_N , g_N) \cdot (0,C_1)  \big)  
 \]
 up to a $\Q$-linear combination of $\{ \log(p) : p \mid N\}$. 
Varying  $N$ shows that
\[
\widehat{\deg}\big(  \widehat{\mathcal{Z}} \cdot \widehat{\omega}^\mathrm{Hdg} _{A_0/\mathcal{S}_V}   \big) 
=
 \widehat{\deg}\big(  \widehat{\mathcal{Z}} \cdot (0,C_1)  \big)  
\]
for every  $\widehat{\mathcal{Z}}   \in \widehat{\CH}^{n-1}( \bar{\mathcal{S}}_V  )$,
and the claim follows.
\end{proof}


\section{Another hermitian line bundle}


In order to relate the  hermitian line bundles of \eqref{three bundles}, we recall from  \S 5.1 of  \cite{BHKRY-2} a fourth hermitian line bundle.
Denote by 
\[
H^1_\mathrm{dR}(A) = \mathbb{R}^1\pi_* \Omega^\bullet_{A/\mathcal{S}_V}
\]
  the first relative algebraic de Rham cohomology of $\pi: A\to \mathcal{S}_V$,  a rank $2n$  vector bundle on $\mathcal{S}_V$.
The action of $\co_\kk$ on $A$ induces an action on  $H_1^\mathrm{dR}(A) = H^1_\mathrm{dR}(A)^\vee$, which is locally free of rank $n$ over $\co_\kk \otimes_\Z \co_S$. 
If 
\[
H_1^\mathrm{dR}(A)  \to    \mathcal{V}  
\]
denotes the    largest quotient  on which the action of $\co_\kk$ is through the structure morphism $\co_\kk \to \co_{\mathcal{S}_V}$, then $\mathcal{V}$ is a rank $n$ vector bundle on $\mathcal{S}_V$, equipped with 
a    morphism
\[
\det(\mathcal{V})^{-1} \to \bigwedge\nolimits^n H^1_{\mathrm{dR}}(A) 
\to H^n_{\mathrm{dR}}(A) .
\]
Given a complex point $z\in \mathcal{S}_V(\C)$ and a vector
$
s_z \in \det(\mathcal{V}_z)^{-1},
$
we view $s_z$ as an element of $H^n(A_z,\C)$, and define 
\[
\| s_z \|^2 = \left|      \int_{A_z (\C)} s_z \wedge \overline{s}_z  \right| .
\]
This defines a hermitian metric on $\det(\mathcal{V})^{-1}$, and hence we obtain a hermitian line bundle
\begin{equation}\label{det bundle}
\det(\mathcal{V})    \in     \widehat{\Pic}(\mathcal{S}_V)  .
\end{equation}

\begin{proposition} \label{prop:full bundle compare}
Assume $n\ge 2$.
Recalling the exceptional divisor of Definition \ref{def:special exceptional} and the notation  \eqref{constant metrics},  we have the equality
\[
2  \widehat{\taut}_V     
 =  \widehat{\omega}^\mathrm{Hdg}_{ A / \mathcal{S}_V}     
 +  2  \widehat{\omega}^\mathrm{Hdg}_{ A_0 / \mathcal{S}_V}  
 +  \det( \mathcal{V})    +   ( \mathrm{Exc}_V , C_2)  
\]
in $\widehat{\Pic}(\mathcal{S}_V)$, where 
\[
C_2 =  2  \log\left(   \frac{2 e^\gamma}{ D}  \right)   + (2-n)  \log(2\pi).
\]
In particular,  $\det(\mathcal{V}) \in \widehat{\Pic}( \bar{\mathcal{S}}_V, \mathscr{D}_\BKK )$ by  \eqref{three bundles}.
\end{proposition}

\begin{proof}
This is a restatement of Proposition 5.1.2 of \cite{BHKRY-2}, keeping in mind that the metric on
\[
\det(\Lie(A))^{-1} = \omega^\mathrm{Hdg}_{A / \mathcal{S}_V}
\]
used in [loc.~cit.] differs from \eqref{hodge metric} by a power of $2\pi$.
\end{proof}

It is an observation of Gross \cite{gross} that the hermitian line bundle $\det(\mathcal{V})$ 
behaves in the generic fiber, for all arithmetic purposes, like the trivial bundle with a constant metric. 
 This observation was extended to integral models in 
\S 5.3 of \cite{BHKRY-2}, whose results are the  basis of the following proposition.

\begin{proposition}\label{prop:gross numerical}
The Chern form of   $\det(\mathcal{V})$ is identically $0$.
If  $n>2$ then, up to numerical equivalence, 
\[
\det(\mathcal{V})  =   ( 0 , C_3 ), 
\]
 where 
\[
C_3 = (4-2n) h_\kk^\mathrm{Falt} + \log( 4\pi^2 D).
\]
 \end{proposition}

\begin{proof}
Fix an integer  $N\ge 3$.  
As in  Theorem 1 of  \cite{gross},  the $\C$-vector space $H^0 (  \mathcal{S}_V(N)_{/\C}  , \det(\mathcal{V}) )$ has dimension $1$,  and the norm of  any  global  section  is a locally constant function on     $\mathcal{S}_V(N)(\C)$. The vanishing of the Chern form follows.
Moreover, one can choose a nonzero   global section $t$  defined over a finite  Galois extension $\kk'/\kk$, 
and then for every $\kk$-algebra embedding 
$\sigma:\kk' \to \C$ the norm   $\| t^\sigma  \|$ must again be  locally constant.

If $n>2$ then\footnote{The assumption $n>2$ is mistakenly omitted in Theorem 5.3.1 of \cite{BHKRY-2}, whose proof requires that  $\mathcal{M}^\Pap_{(n-1,1)}$ has geometrically normal fibers.
The normality is a theorem of  Pappas  when $n>2$, but is false when $n=2$.}   
Theorem 5.3.1 of \cite{BHKRY-2} allows us to  choose $t$  so that it extends to 
 a  nowhere vanishing section 
\[
t  \in H^0 \big(  \mathcal{S}_V(N)_{/\co_{\kk'}[1/N] }  , \det(\mathcal{V}) \big) .
\]
Setting $d=[\kk':\kk]$ and taking the  tensor product of all $\Gal(\kk'/\kk)$-conjugates of $t$, we obtain a section 
\[
s \in H^0( \mathcal{S}_V(N) ,  \det(\mathcal{V})^{\otimes d}) 
\]
such that $\mathrm{div}(s) =0$, and such that  $-\log\| s\|$  is  locally constant.
Let
\[
c_X =  -\log\| s\|^2
\]
denote its  value on the connected component  $X \subset \mathcal{S}_V(N)_{/\C}$.

Fix a finite Galois extension  $L/\kk$ contained $\C$ large enough that every component 
$X \subset \mathcal{S}_V(N)_{/\C}$ is defined over $L$.
We may further enlarge $L$ to assume that each $X$ admits an $L$-point 
\[
x \in X(L) \subset \mathcal{S}_V(N)(L)
\]
that extends to 
\[
\underline{x}  : \Spec(  \co_L[1/N] )  \to  \mathcal{S}_V(N).  
\]
 For example, start by fixing a complex point  $x\in X(\C)$ corresponding to a pair $(A_0,A)$ for which $A$ has complex multiplication.  Then  enlarge $L$ so that both $A_0$ and $A$ (along with their level $N$-structures)  are defined over $L$ and have everywhere good reduction.

Applying Corollary 5.3.2 and Proposition 5.3.3 of \cite{BHKRY-2}  to $\underline{x}$, we find 
\[
\frac{-1}{ [ L:\kk] }  \sum_{ \sigma \in \Gal(L/\kk) }  \log\| s_{x^\sigma} \|^2 
=  d  C_3
\]
up to a $\Q$-linear combination of $\{ \log(p) : p \mid N\}$, and hence
\[
\frac{1}{ [ L:\kk] }  \sum_{ \sigma \in \Gal(L/\kk) }  c_{X^\sigma} =  d  C_3
\]
up to the same ambiguity.

We have now shown that the average value of $-\log\|s\|^2$ over any $\Aut(\C/\kk)$ orbit of connected components  in $\mathcal{S}_V(N)(\C)$ is $d C_3$,
up to a $\Q$-linear combination of $\{ \log(p) : p \mid N\}$.  The  proposition follows from this, by the same argument used in the proof of Proposition \ref{prop:easy numerical}.
 \end{proof}


\section{Comparison of hermitian line bundles}


We now come to the main results of Chapter \ref{s:special shimura bundles}.

The exceptional divisor $\mathrm{Exc}_V \subset \bar{\mathcal{S}}_V$  of Definition \ref{def:special exceptional} is a recurring nuisance, 
 in part because it has nontrivial arithmetic intersection with  $\widehat{\taut}_V$.
 This will be explored more fully in \S \ref{ss:exceptional volume} below.  
 The following proposition allows us to avoid this nuisance by slightly modifiying $\widehat{\taut}_V$, and also clarifies the relation between the second and third line bundles in \eqref{three bundles}.

 \begin{theorem}\label{thm:taut-hodge compare}
 Assume $n\ge 2$.
 Recalling  the notation \eqref{constant metrics}, the hermitian line bundle 
\begin{equation}\label{Kbun}
\widehat{\tautmod}_V \define 2 \widehat{\taut}_V - (\mathrm{Exc}_V,0)  \in \widehat{\Pic}( \bar{\mathcal{S}}_V, \mathscr{D}_\BKK ) 
\end{equation}
enjoys the following properties.
\begin{enumerate}
\item
Every irreducible component $E\subset\mathrm{Exc}_V$   satisfies 
\[
(E,0) \cdot \widehat{\tautmod}_V =0
\]
in $ \widehat{\CH}^2( \bar{\mathcal{S}}_V , \mathscr{D}_\BKK )_\Q$, as well as the height relation
\[
 \mathrm{ht}_{\widehat{\tautmod}_V  } (E)   =0 .
\]
The same equalities hold if we replace $\widehat{\tautmod}_V$ with $\widehat{\omega}^\mathrm{Hdg}_{ A / \mathcal{S}_V}$ or  $\widehat{\omega}^\mathrm{Hdg}_{ A_0 / \mathcal{S}_V}$.
\item
We have the equality of Chern forms 
\[
\chern(   \widehat{\omega}^\mathrm{Hdg}_{ A / \mathcal{S}_V}   ) =2 \chern( \widehat{\taut}_V ) =
\chern( \widehat{\tautmod}_V ) .
\]
\item
If $n>2$ then, up to numerical equivalence,  
\[
 \widehat{\tautmod}_V
 =
\widehat{\omega}^\mathrm{Hdg}_{ A / \mathcal{S}_V}   +    ( 0  ,  C_0(n) ) 
\in \widehat{\mathrm{Pic}}(\mathcal{S}_V) ,
\]
 where $C_0(n)$ is the constant of Theorem \ref{thm:intro main}.  In particular, up to numerical equivalence,
 \[
 2 \widehat{\taut}_V
 =
 \widehat{\omega}^\mathrm{Hdg}_{ A / \mathcal{S}_V}   +    ( \mathrm{Exc}_V  ,  C_0(n) ) .
 \]
\end{enumerate}
 \end{theorem}

\begin{proof}
For (1),  the key observation is that  the abelian scheme $A \to \mathcal{S}_V$ is a pullback  via the vertical arrow on the right in \eqref{special exceptional}.
In particular,  $\omega^\mathrm{Hdg}_{A/\mathcal{S}_V}$ is isomorphic to the pullback of 
\[
 \omega^\mathrm{Hdg}_{A/\mathcal{M}^\Pap_{(n-1,1)}} \in \Pic ( \mathcal{M}^\Pap_{(n-1,1)}  ) .
\]
  If  $\mathcal{M}^\Pap_{(n-1,1)}$ were a scheme we could  trivialize this latter line bundle  over a Zariski open neighborhood of  the ($0$-dimensional) singular locus.  
  This would pull back to a trivialization of $\omega^\mathrm{Hdg}_{A/ \mathcal{S}_V}$  over  an open neighborhood of   $\mathrm{Exc}_V$, and in particular over an open neighborhood of 
any irreducible component $E\subset \mathrm{Exc}_V$.  
  
   To account for the stackiness, simply fix an  integer $N\ge 3$
  and apply the same reasoning with level $N$-structure  to see that 
 $\omega^\mathrm{Hdg}_{A/\mathcal{S}_V(N)}$ is trivial in some Zariski open neighborhood of 
\[
E(N)  = E \times_{ \mathcal{S}_V} \mathcal{S}_V(N).
\]
This implies  the arithmetic intersection formula
$
( E(N) , 0 )  \cdot  \widehat{\omega}^\mathrm{Hdg}_{A/\mathcal{S}_V(N)} =0 
$
 and varying $N$ proves 
\[
( E , 0 )  \cdot  \widehat{\omega}^\mathrm{Hdg}_{A/\mathcal{S}_V} =0 .
\]
The line bundle \eqref{det bundle}  is also a pullback via the vertical arrow on the right in \eqref{special exceptional},   hence the same argument  shows 
\[
(E,0) \cdot \det(\mathcal{V}) =0.
\]

 The proof of Proposition \ref{prop:easy numerical} shows that, for any $N\ge 3$,
there is a positive multiple of
\[
\widehat{\omega}^\mathrm{Hdg}_{A_0/\mathcal{S}_V(N)} 
\in \widehat{\CH}^1( \bar{ \mathcal{S}} _V (N) , \mathscr{D}_\BKK ) 
\]
that can be represented by a purely archimedean  arithmetic divisor $(0,g)$.  
Any such arithmetic divisor satisfies 
$
(E(N) ,0 ) \cdot (0,g) =0,
$
and  varying $N$ shows that 
\[
(E ,0 ) \cdot \widehat{\omega}^\mathrm{Hdg}_{ A_0 / \mathcal{S}_V} =0.
\]

Rewriting the  relation of Proposition \ref{prop:full bundle compare} as
\begin{equation}\label{Kalt}
  \widehat{\tautmod}_V   
 =  \widehat{\omega}^\mathrm{Hdg}_{ A / \mathcal{S}_V}     
 +  2  \widehat{\omega}^\mathrm{Hdg}_{ A_0 / \mathcal{S}_V}  
 +  \det( \mathcal{V})    +   ( 0 , C_2)  ,
\end{equation}
we have shown that the right hand side  has trivial arithmetic intersection with $(E,0)$, and hence so does the left hand side.  This proves the first equality of (1).  The second  is a formal consequence of this and  \eqref{height degree}.

 The first equality of (2) follows  from \eqref{Kalt}, 
 as the Chern forms of the final three terms on the right vanish by    Lemma \ref{lem:numerical basics}, Proposition \ref{prop:easy numerical}, and  Proposition \ref{prop:gross numerical}.
 The second equality of (2) is clear from the definitions.

Claim (3) also follows from \eqref{Kalt}, using  Proposition  \ref{prop:easy numerical},  Proposition \ref{prop:gross numerical},   and the equality
 $ C_0 (n) = 2C_1 +  C_2+ C_3$.
\end{proof}

\begin{remark}\label{rem:taut-hodge volume}
If $n>2$ then part (3) of Theorem \ref{thm:taut-hodge compare} implies
\begin{align*}
2^n \cdot  \widehat{\mathrm{vol}} (  \widehat{\tautmod}_V   )
& = \widehat{\mathrm{vol}} \big(  \widehat{\omega}^\mathrm{Hdg}_{ A / \mathcal{S}_V}   
 +    ( 0 ,  C_0(n)  ) \big)  \\
& = 
\widehat{\mathrm{vol}} \big(  \widehat{\omega}^\mathrm{Hdg}_{ A / \mathcal{S}_V}   \big) +  n C_0(n)    \int_{\mathcal{S}_V(\C) }  \chern( \widehat{\omega}^\mathrm{Hdg}_{ A / \mathcal{S}_V}    )^{n-1}
\end{align*}
where we have used  Lemma \ref{lem:numerical basics} for the first equality, and Lemma \ref{lem:trivial volume shift} for the second.
 In Proposition \ref{prop:taut-hodge strong volume} we will show that this equality also holds when $n=2$.
\end{remark}

The final claim of Theorem \ref{thm:taut-hodge compare} has the following  variant.

\begin{proposition}\label{prop:simple taut-hodge}
Fix an $N\ge 3$.
There is a positive integer $r$ such that the $r^\mathrm{th}$ power of $\mathcal{L}_V^{\otimes 2}$ becomes isomorphic to the $r^\mathrm{th}$ power of 
$
 \omega^\mathrm{Hdg}_{ A / \mathcal{S}_V}
 \otimes \co(\mathrm{Exc}_V)
$
 after pulling both line bundles back  via the morphism 
\[
\bar{\mathcal{S}}_V(N) \to \bar{\mathcal{S}}_{V/\co_\kk[1/N]}.
\]
\end{proposition}

\begin{proof}
We begin with the equality 
\[
2  \widehat{\taut}_V     
 =  \widehat{\omega}^\mathrm{Hdg}_{ A / \mathcal{S}_V}     
 +  2  \widehat{\omega}^\mathrm{Hdg}_{ A_0 / \mathcal{S}_V}  
 +  \det( \mathcal{V})    +   ( \mathrm{Exc}_V , C_2)  
\]
of metrized line bundles in $\widehat{\Pic}( \bar{\mathcal{S}}_V, \mathscr{D}_\BKK )$ from   Proposition \ref{prop:full bundle compare}.  In particular, after forgetting the metrics we obtain the same equality in $\Pic( \bar{\mathcal{S}}_V)$.

After pullback to $\bar{\mathcal{S}}_V(N)$,  the proof of Proposition \ref{prop:easy numerical} shows that some power of 
$\omega^\mathrm{Hdg}_{ A_0 / \mathcal{S}_V}$ is trivial.  Similarly, the proof of Proposition \ref{prop:gross numerical} shows that some power of $ \det( \mathcal{V})$ admits a trivializing section $s$ over the interior $\mathcal{S}_V(N)$ such that $\| s\|$ is a locally constant function on the complex fiber.  This constancy of the norm implies that $s$, when viewed as a rational section on the toroidal compactification $\bar{\mathcal{S}}_V(N)$, still has no zeros or poles; the argument for this is exactly as in the proof of Proposition 5.2.1 of \cite{howard-volumes-I}, and relies on the pre-log singularity of the metric on $\det(\mathscr{V})$ proved in Proposition \ref{prop:full bundle compare}.
Thus some power of $\det(\mathcal{V})$ is trivial, and the claim follows.
\end{proof}


\section{Volume of the exceptional divisor}
\label{ss:exceptional volume}


In this section we assume $n \ge 2$, and fix an irreducible ($=$ connected) component $E$ of the  exceptional divisor  $\mathrm{Exc}_V \subset \bar{\mathcal{S}}_V$ of Definition \ref{def:special exceptional}.  
 Recalling the notation \eqref{constant metrics}, our goal is to compute  the arithmetic volume of 
\[
(E,0) \in \widehat{\Pic}( \bar{\mathcal{S}}_V ).
\]

By definition of the exceptional divisor, there is a commutative diagram with cartesian squares
\begin{equation}\label{singular point}
\xymatrix{
{E} \ar[r] \ar[d] & { e } \ar[d] \\
{\mathrm{Exc}}_V  \ar[r] \ar[d] & {  \mathcal{M}_{(1,0)} \times_{\co_\kk} \mathrm{Sing}_{(n-1,1)} }  \ar[d] \\
{\mathcal{S}_V}  \ar[r]  & {  \mathcal{M}_{(1,0)} \times_{\co_\kk} \mathcal{M}^\Pap_{(n-1,1)}  }  
}
\end{equation}
in which $e$ is a connected component of the $0$-dimensional reduced and irreducible $\co_\kk$-stack 
$\mathcal{M}_{(1,0)} \times_{\co_\kk} \mathrm{Sing}_{(n-1,1)}$.  In particular, $e$ is supported in a single characteristic $p\mid D$,  and admits  a presentation as a stack quotient 
\[
e \iso \Delta \backslash \Spec(\F_\mathfrak{p}'),
\]
 in which $\mathfrak{p} \subset \co_\kk$ is the prime above $p$,   $\F_\mathfrak{p}'$ is a finite extension of  $\F_\mathfrak{p}=\co_\kk/\mathfrak{p}$, and $\Delta$ is a finite group acting on  $\F_\mathfrak{p}'$. 
Define a rational number
\[
m_E \define \sum_{ z\in e(\F_\mathfrak{p}^\alg) } \frac{ 1 }{ |\Aut(z)|} = \frac{[  \F_\mathfrak{p}' : \F_\mathfrak{p}   ]  }{ |\Delta| }    .
\]

\begin{proposition}\label{prop:projective intersection}
The iterated intersection
\[
\widehat{\taut}_V^{-1} \cdots  \widehat{\taut}_V^{-1}  \cdot (E,0) \in \widehat{\CH}^n(\bar{\mathcal{S}}_V , \mathscr{D}_\BKK)
\]
has arithmetic degree $\widehat{m}_E = m_E \log(p)$
\end{proposition}

\begin{proof}
After Remark \ref{rem:singular description},  one might expect that $E$ is isomorphic to the projective space of hyperplanes in an $n$-dimensional vector space.   
In particular, there should be a universal hyperplane $\mathcal{F} \subset \co_E^n$, with the property that the restriction of  $\taut_V^{-1}$ to $E$ is isomorphic to  $\co_E^n /\mathcal{F}$.

The obstruction to this being true is due  to stacky issues, which can be removed by fixing an integer $N\ge 3$ prime to $p$ and adding level $N$-structure to the universal pair $(A_0,A)$ over $e$. 
Let $e(N) \to e$ be the scheme classifying such level structures, and consider the cartesian diagram
 (this is the definition of the upper left corner)
\[
\xymatrix{ 
{ E(N) } \ar[r]\ar[d]  & {  e(N) }  \ar[d] \\
{ E } \ar[r] & { e .}
}
\]
The scheme $e(N)$, being a reduced scheme finite over $\F_\mathfrak{p}$,  is a disjoint union of finitely many spectra of finite extensions of $\F_\mathfrak{p}$.

Fix a connected component $E' \subset E(N)$, and let 
\[
\Spec(\F_\mathfrak{p}') = e' \subset e(N)
\]
 be the connected component  below it. 
As in Remark \ref{rem:singular description}, $E'$ is precisely the projective space over  $e'$ classifying hyperplanes  $ \mathcal{F} \subset \Lie(A|_{e'})$.
Moreover,  after fixing a trivialization  $\Lie(A_0|_{e'}) \iso \F_\mathfrak{p}'$, the isomorphism  \eqref{dual taut} identifies  
\[
\taut_V^{-1} |_{E'} =      \Lie(A|_{E'}) / \mathcal{F}  \in \Pic(E') \iso \CH^1(E') 
\]
where $\mathcal{F} \subset \Lie(A|_{E'})$ is the universal hyperplane. 
 A routine exercise then shows that the iterated intersection
\[
( \taut_V^{-1} |_{E'}) \cdots (\taut_V^{-1} |_{E'}) \in \CH^{n-1}(E') \iso \Z
\]
is  represented by the cycle  class of  any $\F_\mathfrak{p}'$-valued point of $E'$.
In other words, the cycle class of any section to $E' \to e'$.

Now allow  the connected component $E' \subset E(N)$ to vary.  If we fix 
any section to $E(N) \to e(N)$, and use it to view $e(N)$ as a $0$-cycle on $E(N)$, then 
\[
e(N) = ( \taut_V^{-1} |_{E(N)}) \cdots (\taut_V^{-1} |_{E(N)}) \in \CH^{n-1}(E(N)) .
\]
This implies  the arithmetic intersection formula
\[
( e(N) , 0 ) = \widehat{\taut}_V^{-1} \cdots  \widehat{\taut}_V^{-1} \cdot ( E(N) ,  0 ) 
\in \widehat{\CH}^n( \bar{\mathcal{S}}_V(N) , \mathscr{D}_\BKK )_\Q.
\]

Finally, recalling \eqref{degree normalization}, we deduce 
\[
\widehat{\deg}  \big(  \widehat{\taut}_V^{-1} \cdots  \widehat{\taut}_V^{-1} \cdot ( E ,  0 )  \big) 
= 
 \frac{ \# e(N) (\F_\mathfrak{p}^\alg) }{d_N}  \log(p) 
 = 
   \sum_{ z\in e(\F_\mathfrak{p}^\alg) } \frac{\log(p)}{ |\Aut(z)|}
\]
up to a $\Q$-linear combination of $\{ \log(p) : p\mid N\}$.  Varying $N$ completes the proof.
\end{proof}

\begin{corollary}\label{cor:exceptional volume}
The hermitian line bundle $(E,0) \in \widehat{\Pic}(\bar{\mathcal{S}}_V )$ has arithmetic volume
\[
\widehat{\vol} ( E,0) = (-2)^{n-1} \cdot  \widehat{m}_E .
\]
\end{corollary}

\begin{proof}
Part (1) of Theorem \ref{thm:taut-hodge compare} implies the second equality in  
\[
(E,0)\cdot (E,0)  = (\mathrm{Exc}_V,0) \cdot (E,0) = 2 \widehat{\taut}_V  \cdot (E,0) 
 \in \widehat{\CH}^2(\bar{\mathcal{S}}_V , \mathscr{D}_\BKK)_\Q.
\]
This implies the iterated  intersection formula
\[
 (E,0)\cdots (E,0) 
 =  2^{n-1} \widehat{\taut}_V \cdots \widehat{\taut}_V   \cdot (E,0) 
  \in \widehat{\CH}^n(\bar{\mathcal{S}}_V , \mathscr{D}_\BKK)_\Q  ,
\]
and  the claim follows using  Proposition \ref{prop:projective intersection}.
\end{proof}


\chapter{Kudla-Rapoport divisors at split primes}
\label{s:KR divisors}


Let $\mathcal{S}_V$ be the Shimura variety \eqref{moduli inclusion} associated to a $\kk$-hermitian space  $V$  of signature $(n-1,1)$ containing a self-dual lattice.  
Throughout this chapter we assume $n>2$.

Fix a prime $p$ split in $\kk$,  and factor 
\[
p\co_\kk = \mathfrak{p} \overline{\mathfrak{p}}.
\]
Our goal  is to explain how the Kudla-Rapoport divisor  $\mathcal{Z}_V(p) \to \mathcal{S}_V$ of Definition \ref{def:KR}  is  related to  a Shimura variety $\mathcal{S}_{V^\flat}$   with $V^\flat$ a hermitian space of signature $(n-2,1)$, whose isometry class depends on $p$. 
The assumption that $p$ is split in $\kk$ is needed to guarantee that  $V^\flat$, like the initial choice of $V$, contains a self-dual $\co_\kk$-lattice.  Thus all the properties we have stated for $\mathcal{S}_V$   also hold for  this   $\mathcal{S}_{V^\flat}$ of one dimension lower.  As explained at the end of \S \ref{ss:intro generating}, this will be crucial for our later inductive arguments.


\section{Statement of the results}
\label{ss:split divisors main}


Let $V^\flat$  denote the $\kk$-hermitian space of signature $(n-2,1)$ whose local invariants satisfy
\begin{equation}\label{Vflat}
\mathrm{inv}_\ell(V^\flat)  = ( p, -D)_\ell \cdot \mathrm{inv}_\ell(V)  
\end{equation}
for all places $\ell \le \infty$.  
Using standard formulas for the Hilbert symbol, this is equivalent to 
\[
\mathrm{inv}_\ell(V^\flat)  = \leg{p}{\ell} \cdot \mathrm{inv}_\ell(V)  .
\]
Equivalently,  $V^\flat$ is the orthogonal complement to the $\kk$-span of any 
 $x\in V$ of hermitian norm $\langle x,x\rangle=p$.  
 Obviously $V^\flat$ depends on the prime $p$ fixed above, but this prime will not change and we suppress the dependence from the notation.  Our assumption that $p$ is split is essential for the following lemma.

\begin{lemma}\label{lem:lower hermitian}
The hermitian space $V^\flat$  admits a self-dual $\co_\kk$-lattice (so is relevant in the sense of Definition \ref{def:relevant}).
\end{lemma}

\begin{proof}
Suppose $\ell$ is a rational prime unramified in $\kk$.  If $\ell \neq p$ then $p\in \Z_\ell^\times$, and hence is a norm from the unramified extension $\kk_\ell$.  If $\ell =p$ then $p$ is  again a norm from $\kk_\ell \iso \Q_p \times \Q_p$.    In either case, $(p,-D)_\ell=1$, and so  the local invariants of $V$ and $V^\flat$ agree at all unramified primes.

In general, a $\kk$-hermitian space  admits a self-dual $\co_\kk$-lattice if and only  if it has local invariant $1$ at every prime  unramified  in $\kk$.  
As $V$ satisfies this condition by hypothesis, so does $V^\flat$.
\end{proof}

 Lemma \ref{lem:lower hermitian} allows us to form the $\co_\kk$-stacks
\[
\mathcal{S}_{V^\flat} \subset  \bar{\mathcal{S}}_{V^\flat} \subset \mathcal{M}_{(1,0)} \times_{\co_\kk} \bar{\mathcal{M}}_{(n-2,1)} 
\]
 analogous to 
 \[
 \mathcal{S}_V \subset \bar{\mathcal{S}}_V\subset \mathcal{M}_{(1,0)} \times_{\co_\kk} \bar{\mathcal{M}}_{(n-1,1)}, 
 \]
 but in one dimension lower.  The stack $\mathcal{S}_{V^\flat}$ is  endowed with its own  hermitian line bundle $\widehat{\taut}_{V^\flat}$  as in \eqref{metrized taut},  its own exceptional divisor $\mathrm{Exc}_{V^\flat}$  as in Definition \ref{def:special exceptional}, and its own universal pair $(A_0^\flat,A^\flat)$ of polarized abelian schemes with $\co_\kk$-actions.  
 
   The  following theorems, whose proofs will occupy the remainder of Chapter \ref{s:KR divisors}, lie at the core of the inductive arguments of Chapter \ref{s:volumes}.  In both theorems,  $\bar{\mathcal{Z}}_V(p)$ is the divisor on $\bar{\mathcal{S}}_V$ of Remark \ref{rem:divisor closure}, and the arithmetic heights and volumes are those of  \S \ref{ss:volumes}.

\begin{theorem}\label{thm:height descent 1}
If we set
\[
\widehat{\tautmod}_V  = 2 \widehat{\taut}_V - ( \mathrm{Exc}_V , 0) \in \widehat{\Pic} ( \bar{\mathcal{S}}_V , \mathscr{D}_\BKK )
\]
as in Theorem \ref{thm:taut-hodge compare},  and similarly with $V$ replaced by $V^\flat$, then
\[
 \int_{ \mathcal{Z}_V(p) (\C)  }  \chern( \widehat{\tautmod}_V)^{n-2}  
   =  ( p^{n-1}+1)  \vol_\C( \widehat{\tautmod}_{V^\flat} ).
\]
Moreover, there is a rational number $a(p)$  such that 
\[
 \frac{  \mathrm{ht}_{ \widehat{\tautmod}_V }  (\bar{\mathcal{Z}}_V(p))  } {    p^{n-1}+1    } 
 =    \widehat{\vol} (\widehat{\tautmod}_{V^\flat} )  + a(p) \log(p) .
\]
\end{theorem}

\begin{remark}
The proof of Theorem \ref{thm:height descent 1}  gives no hint as to the value of the rational number $a(p)$.
It  will be determined later in Theorem \ref{thm:K volume}.
\end{remark}

\begin{theorem}\label{thm:height descent 2}
There is a  rational number $b(p)$ such that 
\begin{align*}
\frac{ \mathrm{ht}_{\widehat{\omega}^\mathrm{Hdg}_{A/\mathcal{S}_V}}  (\bar{\mathcal{Z}}_V(p))  }{  p^{n-1}+1   } 
 & =     \widehat{\vol} ( \widehat{\omega}^\mathrm{Hdg}_{A^\flat/\mathcal{S}_{V^\flat}} )   + b(p) \log(p)   \\
& \quad   +  (1-n)  \left(  \frac{L'(0,\eps)}{L(0,\eps)}   +\frac{\log(D)}{2}  \right) 
\vol_\C(   \widehat{\omega}^\mathrm{Hdg}_{A^\flat/ \mathcal{S}_{V^\flat}}  )  .
\end{align*}
\end{theorem}

The proofs are rather long, so we summarize now the key steps.
 The central idea  is to make precise the  impressionistic relation
\[
\mathcal{Z}_V(p)  ``="  (p^{n-1}+1)  \cdot \mathcal{S}_{V^\flat}   ,
\]
by  decomposing
\begin{equation}\label{flavor decomp}
\mathcal{Z}_V(p)_{/\co_\kk[1/p]}  = 
\mathcal{U}_0  \sqcup   \mathcal{U}_{\mathfrak{p}}  \sqcup   \mathcal{U}_{\overline{\mathfrak{p}}}
\end{equation}
as a disjoint union of open and closed substacks (Proposition \ref{prop:split divisor flavors} below).
The stacks on the right hand side are related to $\mathcal{S}_{V^\flat}$ by  closed immersions
\[
i_\mathfrak{p} :  \mathcal{S}_{V^\flat / \co_\kk[1/p]}   \to \mathcal{U}_{\mathfrak{p}}  ,\qquad
i_{\overline{\mathfrak{p}}} :  \mathcal{S}_{V^\flat / \co_\kk[1/p]}   \to   \mathcal{U}_{\overline{\mathfrak{p}}},
\]
and  by a diagram
\[
\xymatrix{
 {  \mathcal{S}_{V^\flat / \co_\kk[1/p]}   }   & { \mathcal{T}_{V^\flat} } \ar[l] \ar[r]^{i_0} &  { \mathcal{U}_0   }  
 }
\]
in which the leftward arrow is a finite \'etale surjection of degree $p^{n-1}-1$, and the rightward arrow is a closed immersion.
See  \eqref{tau cover} for the definition of the stack $ \mathcal{T}_{V^\flat} $.

  Recalling the exceptional locus of Definition \ref{def:special exceptional},
denote by 
\[
 \mathcal{S}^{\nonexc}_V =  \mathcal{S}_V \smallsetminus \mathrm{Exc}_V
 \]
the (open) nonexceptional locus of $\mathcal{S}_V$,  set
\[
\mathcal{Z}^{\nonexc}_{V}(p) =  \mathcal{Z}_{V}(p)  \times_{\mathcal{S}_{V}} \mathcal{S}^\nonexc_V ,
\]
and make the same definitions with $V$ replaced by $V^\flat$.
Set 
\[
\mathcal{U}^\nonexc_\square =  \mathcal{U}_\square \cap \mathcal{Z}^\nonexc_V(p)
\]
for $\square \in \{ 0, \mathfrak{p} ,\overline{\mathfrak{p}} \}$, and 
\[
\mathcal{T}^\nonexc_{V^\flat}  = \mathcal{T}_{V^\flat} \times_{\mathcal{S}_{V^\flat}} \mathcal{S}_{V^\flat}^\nonexc .
\]

We will show that the closed immersions $i_0$, $i_\mathfrak{p}$, and $i_{\overline{\mathfrak{p}}}$ are close to being isomorphisms.  More precisely,     $i_0$  restricts to an isomorphism
\[
i_0 : \mathcal{T}^\nonexc_{V^\flat}    \iso \mathcal{U}^\nonexc_0,
\]
while 
$i_\mathfrak{p}$ and $i_{\overline{\mathfrak{p}}}$ restrict to isomorphisms
\[
i_\mathfrak{p} :  \mathcal{S}^\nonexc_{V^\flat / \co_\kk[1/p]}   \iso \mathcal{U}^\nonexc_{\mathfrak{p}} 
 \quad \mbox{and}\quad
i_{\overline{\mathfrak{p}}} :  \mathcal{S}^\nonexc_{V^\flat / \co_\kk[1/p]}   \iso   \mathcal{U}^\nonexc_{\overline{\mathfrak{p}}}.
\]
After taking  compactifications into account,  both theorems above will follow easily from these isomorphisms.
Note that it suffices to work only over the nonexceptional locus, as part (1) of Theorem \ref{thm:taut-hodge compare} guarantees that the hermitian line bundles in questions have trivial arithmetic intersection with all components of the exceptional divisor.
This is why we work with $\widehat{\tautmod}_V$ in Theorem \ref{thm:height descent 1} instead of $\widehat{\taut}_V$.

%


\section{Decomposing the Kudla-Rapoport divisor}


The decomposition \eqref{flavor decomp} is a geometric reflection of a result of  linear algebra.

As in  \eqref{hermitian hom},  choose an isomorphism
 \[
 V \iso \Hom_\kk(W_0,W)
 \]
 in which $(W_0,h_0)$ and $(W,h)$ are  relevant hermitian spaces of signatures $(1,0)$ and $(n-1,1)$. 
 Fix self-dual $\co_\kk$-lattices $\mathfrak{a}_0 \subset W_0$ and $\mathfrak{a} \subset W$.
Any  vector
\[
x\in \Hom_{\co_\kk}(\mathfrak{a}_0, \mathfrak{a}) \subset V
\]
 with $\langle x,x\rangle =p$  determines an orthogonal decomposition 
\[
W = \tilde{W}_0 \oplus W^\flat 
\]
with  $\tilde{W}_0 = x(W_0)$, and a corresponding decomposition
\[
V = \kk x \oplus V^\flat
\]
with $V^\flat=  \Hom_\kk(W_0 ,W^\flat)$.

\begin{lemma}\label{lem:lattice trifurcation}
 The $\co_\kk$-lattice
 $
 \tilde{\mathfrak{a}}_0=\mathfrak{a} \cap \tilde{W}_0
 $
satisfies exactly one of
\[
\tilde{\mathfrak{a}}_0 = x(\mathfrak{a}_0) ,\qquad \tilde{\mathfrak{a}}_0 =\mathfrak{p}^{-1} x(\mathfrak{a}_0), \qquad 
\tilde{\mathfrak{a}}_0=\overline{\mathfrak{p}}^{-1} x(\mathfrak{a}_0).
\]
\end{lemma}

\begin{proof}
Use $x$ to identify $W_0 = \tilde{W}_0$, and hence $\mathfrak{a}_0 \subset \tilde{\mathfrak{a}}_0 \subset \mathfrak{a}$.
Recalling the symplectic forms  
\[
e_0: W_0 \times W_0 \to \Q ,\qquad e: W \times W \to \Q
\]  
 of   \eqref{symplectic}, the relation  \eqref{basic hom hermitian}  implies that  $e|_{W_0} = p \cdot e_0$.  
 The inclusion
\[
e_0 ( p \tilde{\mathfrak{a}}_0 , \mathfrak{a}_0) 
= e ( \tilde{\mathfrak{a}}_0 , \mathfrak{a}_0) \subset e( \mathfrak{a} , \mathfrak{a}) =\Z,
\]
together with the self-duality of $\mathfrak{a}_0$ under $e_0$, 
shows that $p \tilde{\mathfrak{a}}_0 \subset \mathfrak{a}_0$.   If equality held we would have 
\[
e_0( \mathfrak{a}_0 , \mathfrak{a}_0 )  = p^{-1}  e(  \mathfrak{a}_0 , \mathfrak{a}_0) 
 \subset p e(  \tilde{\mathfrak{a}}_0,\tilde{\mathfrak{a}}_0)\subset p e(\mathfrak{a},\mathfrak{a}) \subset  p \Z,
\]
contradicting  $\mathfrak{a}_0$ being self-dual under $e_0$.

 As $\tilde{\mathfrak{a}}_0$ is $\co_\kk$-stable with
$\mathfrak{a}_0 \subset \tilde{\mathfrak{a}}_0 \subsetneq p^{-1} \mathfrak{a}_0$, it is 
$\mathfrak{a}_0$, $\mathfrak{p}^{-1} \mathfrak{a}_0$, or $\overline{\mathfrak{p}}^{-1} \mathfrak{a}_0$.
\end{proof}

\begin{proposition}\label{prop:split divisor flavors}
Let  $(A_0,A,x)$ be the universal triple over $\mathcal{Z}_V(p)_{/\co_\kk[1/p]}$, so that 
$
x\in \Hom_{\co_\kk}(A_0,A) 
$
 satisfies $\langle x,x\rangle =p$.  
There is a decomposition \eqref{flavor decomp} into open and closed substacks in which 
\begin{itemize}
\item
$\mathcal{U}_0$ is the locus of points where $\ker(x)=0$,
\item
$\mathcal{U}_\mathfrak{p}$ is the locus of points where $\ker(x) = A_0[\mathfrak{p}]$,
\item
$ \mathcal{U}_{\overline{\mathfrak{p}}}$ is the locus of points where $\ker(x) = A_0[ \overline{\mathfrak{p}}]$.
\end{itemize}
\end{proposition}

\begin{proof}
Recalling \eqref{KR hermitian},  the relation $\langle x,x\rangle =p$ implies that  multiplication-by-$p$ factors as 
\[
A_0 \map{x} A \iso A^\vee \map{x^\vee} A_0^\vee \iso A_0,
\]
 and so  $\ker(x) \subset A_0[p]$.   Both of these group schemes are finite \'etale over $\mathcal{Z}_V(p)_{/\co_\kk[1/p]}$, which implies that each of $\mathcal{U}_0$, $\mathcal{U}_\mathfrak{p}$, and $\mathcal{U}_{\overline{\mathfrak{p}}}$ is open and closed. 
It only remains to prove that every geometric point $s \to \mathcal{Z}_V(p)_{/\co_\kk[1/p]}$ is contained in one of them.

Abbreviate $T_0= T_p(A_{0 s})$ and $T=T_p(A_s)$.
Applying the snake lemma to the diagram
\[
\xymatrix{
0 \ar[r]  & {T_0 } \ar[r] \ar[d]_{x_s} & { T_0\otimes \Q_p } \ar[r] \ar[d]_{x_s} & { A_{0 s}[p^\infty]}  \ar[r] \ar[d]_{x_s} & 0  \\
 0 \ar[r]  & {T  } \ar[r]  & { T \otimes \Q_p } \ar[r]   & { A_{s}[p^\infty]}  \ar[r] & 0  ,
}
\]
 and using the vertical arrow on the left   to identify $T_0 \subset T$, we find that 
\[
\ker(x_s) \iso \tilde{T}_0 / T_0,
\]
where $\tilde{T}_0 = T \cap  T_{0 \Q_p}$.

After fixing an isomorphism $\Z_p \iso \Z_p(1)$ of \'etale sheaves on $s$,  there are unique $\co_{\kk,p}$-valued hermitian forms $h_0$ and $h$ on $T_0$ and $T$, respectively, related to the  Weil pairings $e_0$ and $e$  by \eqref{symplectic}.  Thus we may apply  Lemma \ref{lem:lattice trifurcation} with $\mathfrak{a}_0$ and $\mathfrak{a}$ replaced by $T_0$ and $T$, to see that $\tilde{T}_0$ must be one of $T_0$, $\mathfrak{p}^{-1} T_0$, or $\overline{\mathfrak{p}}^{-1} T_0$.  These three cases correspond to $\ker(x_s)$ being  trivial,  $A_{0s}[\mathfrak{p}]$, or $A_{0s}[\overline{\mathfrak{p}}]$.
\end{proof}


\section{Analysis of $\mathcal{U}_\mathfrak{p}$}


In this section we study the structure of the substack $\mathcal{U}_\mathfrak{p}$  of \eqref{flavor decomp},
and make explicit its relation to $\mathcal{S}_{V^\flat}$. 
The analogous analysis of $\mathcal{U}_{\overline{\mathfrak{p}}}$ is obtained by replacing $\mathfrak{p}$ by $\overline{\mathfrak{p}}$ everywhere.

Return to the situation of Lemma \ref{lem:lattice trifurcation}, so that  
$x\in \Hom_{\co_\kk}(\mathfrak{a}_0,\mathfrak{a})$ with $\langle x,x\rangle =p$ determines an orthogonal decomposition
\[
W = \tilde{W}_0 \oplus W^\flat .
\]
Set $\tilde{\mathfrak{a}}_0 = \mathfrak{a}\cap \tilde{W}_0$.

\begin{lemma}\label{lem:pi linear algebra}
  If  $\tilde{\mathfrak{a}}_0 =\mathfrak{p}^{-1} x(\mathfrak{a}_0)$,  there is an orthogonal  decomposition 
\[
\mathfrak{a} = \tilde{\mathfrak{a}}_0 \oplus \mathfrak{a}^\flat \subset W
\]
in which $\mathfrak{a}^\flat = \mathfrak{a} \cap W^\flat$.
\end{lemma}

\begin{proof}
If we use $x$ to identify $W_0 = \tilde{W}_0$,  the  assumption  $\tilde{\mathfrak{a}}_0=\mathfrak{p}^{-1}\mathfrak{a}_0$ implies that  $\tilde{\mathfrak{a}}_0$ is self-dual with respect to the hermitian form $h|_{ \tilde{W}_0 }= p h_0$.  The  desired decomposition  then  follows by elementary linear algebra.
\end{proof}

The lemma suggests that if $(A_0,A,x) \in \mathcal{U}_\mathfrak{p}(S)$ for  an $\co_\kk[1/p]$-scheme $S$,  then $x:A_0 \to A$ should determine an $\co_\kk$-linear splitting
  \begin{equation}\label{pi product}
  A =  \tilde{A}_0 \times A^\flat 
  \end{equation}
of principally polarized abelian schemes.   Indeed, this is the case.  
If we set
\[
\tilde{A}_0 =  A_0 / A_0[\mathfrak{p}]
\]
and recall that $\ker(x) = A_0[\mathfrak{p}]$,  the morphism $x :A_0 \to A$ factors  as
  \[
 A_0 \to \tilde{A}_0 \map{y} A
 \]
 for some $y\in \Hom_{\co_\kk}(\tilde{A}_0, A) $ satisfying   $\langle y,y\rangle =1$.
  In other words,  the composition 
  \[
\tilde{A}_0 \map{y} A \iso A^\vee \map{y^\vee} \tilde{A}_0^\vee \iso \tilde{A}_0
  \]
  is the identity.  This implies that the composition 
  \[
  A \iso A^\vee \map{y^\vee} \tilde{A}_0^\vee \iso \tilde{A}_0 \map{y} A
  \]
  is a Rosati-fixed idempotent  in $\End_{\co_\kk}(A)$,  and $A$ admits  a unique  splitting \eqref{pi product}
  of principally polarized abelian schemes over $S$ such that this idempotent is the projection to the first factor.

Apply the above construction to the universal triple $(A_0,A,x)$ over  $\mathcal{U}_\mathfrak{p}$, and 
recall that $A$ comes equipped with an $\co_\kk$-stable  hyperplane  $\mathcal{F} \subset \Lie(A)$ satisfying Kr\"amer's signature condition.
Using the  decomposition
\[
  \Lie(A) =  \Lie(\tilde{A}_0) \oplus \Lie(A^\flat)
\]
 of vector bundles on $S$, denote by 
$
 \mathcal{U}_\mathfrak{p}^\dagger \subset \mathcal{U}_\mathfrak{p}
 $
 the largest  closed substack over which    $\Lie(\tilde{A}_0 ) \subset \mathcal{F}$.

\begin{proposition}\label{prop:i_p construction}
There is a canonical isomorphism
\[
i_\mathfrak{p} : \mathcal{S}_{V^\flat/\co_\kk[1/p]} \iso \mathcal{U}_\mathfrak{p}^\dagger.
\]
\end{proposition}

\begin{proof}
As above, let $S$ be an $\co_\kk[1/p]$-scheme.
Given a point $(A_0,A,x) \in \mathcal{U}_\mathfrak{p}^\dagger(S)$,  the splitting \eqref{pi product} determines 
\[
 A^\flat \in \mathcal{M}_{( n-2,1)}(S),
\]
where we have endowed  $A^\flat$ with the $\co_\kk$-stable hyperplane
\[
\mathcal{F}^\flat = \mathcal{F} / \Lie(\tilde{A}_0 )  \subset \Lie(A)/  \Lie(\tilde{A}_0 ) \iso \Lie(A^\flat).
\]
satisfying Kr\"amer's condition.  
The pair $(A_0,A^\flat)$ defines an $S$-point of 
\[
\mathcal{S}_{V^\flat} \subset \mathcal{M}_{(1,0)} \times_{\co_\kk} \mathcal{M}_{( n-2,1)},
\]
and we have now constructed a morphism 
 \begin{equation}\label{i_p inverse}
 \mathcal{U}^\dagger_\mathfrak{p} \map{  (A_0,A,x) \to (A_0,A^\flat)}  \mathcal{S}_{V^\flat/\co_\kk[1/p]}.
 \end{equation}
 
Conversely, start with an $S$-point 
\[
(A^\flat_0 , A^\flat) \in \mathcal{S}_{V^\flat}(S) .
\]
First define  elliptic curves $A_0 = A_0^\flat$ and  $\tilde{A}_0 = A_0 /A_0[\mathfrak{p}]$.
Then define an abelian scheme $A$ by \eqref{pi product}, and endow $A$ with  its product principal polarization and product $\co_\kk$-action.  
 Recalling that $A^\flat$ comes equipped with a hyperplane $\mathcal{F}^\flat \subset \Lie(A^\flat)$ satisfying Kr\"amer's  condition, we endow $A$ with the hyperplane
\[
  \mathcal{F} = \Lie( \tilde{A}_0 ) \oplus \mathcal{F}^\flat \subset \Lie(A).
\]
 It is easy to check that $A$, with its extra data, defines an $S$-point of $\mathcal{M}_{(n-1,1)}$, and that 
 $(A_0,A)$ defines an $S$-point of the open and closed substack
  \[
  \mathcal{S}_V \subset \mathcal{M}_{(1,0)} \times_{\co_\kk} \mathcal{M}_{(n-2,1)}.
\]
If we define  $x\in \Hom_{\co_\kk}(A_0,A)$  as the composition
\[
A_0 \to \tilde{A}_0  \hookrightarrow \tilde{A}_0 \times A^\flat =A,
\]
where the first arrow is the quotient map,  then   $\langle x ,x \rangle =p$ and $\ker(x) =A_0[\mathfrak{p}]$.
The triple   $(A_0,A,x)$ defines an  $S$-point of $\mathcal{U}^\dagger_\mathfrak{p}$, and the morphism
\[
\mathcal{S}_{V^\flat/\co_\kk[1/p]} \map{ (A^\flat_0 , A^\flat)  \to  (A_0,A,x)  }    \mathcal{U}^\dagger_\mathfrak{p} 
\]
is inverse to \eqref{i_p inverse}.
\end{proof}

Proposition \ref{prop:i_p construction} gives us a  commutative diagram
\begin{equation}\label{i_p pullback}
\xymatrix{
{ \mathcal{S}_{V^\flat/\co_\kk[1/p] } }   \ar[r]\ar[d] & { \mathcal{S}_{V/\co_\kk[1/p] } }  \ar[d] \\
{ \big(  \mathcal{M}_{(1,0)} \times_{\co_\kk} \mathcal{M}^\Pap_{ (n-2,1) }  \big)_{/\co_\kk[1/p]} } \ar[r] & { \big(    \mathcal{M}_{(1,0)} \times_{\co_\kk} \mathcal{M}^\Pap_{ (n-1,1) }  \big)_{/\co_\kk[1/p]}    }
}
\end{equation}
in which the top horizontal arrow is the composition
\[
\mathcal{S}_{V^\flat/\co_\kk[1/p]} \map{i_\mathfrak{p}} \mathcal{U}^\dagger_\mathfrak{p} \hookrightarrow   \mathcal{Z}_V(p)_{/\co_\kk[1/p]}  \to \mathcal{S}_{V/\co_\kk[1/p]} ,
\]
and the bottom horizontal arrow sends  $(A_0^\flat , A^\flat) \mapsto (A_0,A)$, where 
\begin{equation}\label{i_p pullback explicit}
A_0 = A_0^\flat , \quad \tilde{A}_0 = A_0 / A_0[\mathfrak{p}] , \quad A = \tilde{A}_0 \times A^\flat.
\end{equation}
Both horizontal arrows are finite and unramified.

\begin{proposition}\label{prop:i_p pullbacks}
The homomorphism
\[
 \widehat{\Pic}  (    \mathcal{S}_{V/\co_\kk[1/p]}  )_\Q  \to  \widehat{\Pic}  (    \mathcal{S}_{V^\flat/\co_\kk[1/p]}  )_\Q
\]
induced by  \eqref{i_p pullback} sends 
\[
\widehat{\omega}^\mathrm{Hdg}_{A/ \mathcal{S}_V } \mapsto \widehat{\omega}^\mathrm{Hdg}_{A^\flat_0/ \mathcal{S}_{V^\flat} } +  \widehat{\omega}^\mathrm{Hdg}_{A^\flat/ \mathcal{S}_{V^\flat} } ,
\]
where  $(A_0^\flat , A^\flat)$ is the universal pair over $\mathcal{S}_{V^\flat}$.  The same map also sends
\[
\widehat{\taut}_V \mapsto \widehat{\taut}_{V^\flat}  
\quad \mbox{and}\quad
 (\mathrm{Exc}_V,0) \mapsto (\mathrm{Exc}_{V^\flat} ,0 ).
\]
\end{proposition}

\begin{proof}
It follows from \eqref{i_p pullback explicit} that the top horizontal arrow in \eqref{i_p pullback}  sends
\[
\widehat{\omega}^\mathrm{Hdg}_{A/ \mathcal{S}_V } \mapsto
 \widehat{\omega}^\mathrm{Hdg}_{\tilde{A}_0/ \mathcal{S}_{V^\flat} } 
 +  \widehat{\omega}^\mathrm{Hdg}_{A^\flat/ \mathcal{S}_{V^\flat} }.
\]
As we have inverted $p$ on the base, the degree $p$ quotient map $A_0^\flat \to \tilde{A}_0$ induces an isomorphism on Lie algebras, and hence an isomorphism 
\[
\omega^\mathrm{Hdg}_{\tilde{A}_0/ \mathcal{S}_{V^\flat} }  \iso  \omega^\mathrm{Hdg}_{ A^\flat_0/ \mathcal{S}_{V^\flat} } .
\]
This isomorphism does not respect the metrics  \eqref{hodge metric}, but one can easily check that 
\[
\widehat{\omega}^\mathrm{Hdg}_{\tilde{A}_0/ \mathcal{S}_{V^\flat} }  
=   
\widehat{\omega}^\mathrm{Hdg}_{ A^\flat_0/ \mathcal{S}_{V^\flat} } + (0 , - \log(p) )
\]
as elements of $\widehat{\Pic}( \mathcal{S}_{V^\flat/\co_\kk[1/p]})$.  The correction term $(0 , - \log(p) )$ is torsion in the arithmetic Picard group, as 
\[
2 ( 0 , -\log(p) ) = (\mathrm{div}(p) , - \log(p^2) ) = 0,
\]
and the first claim follows.
 
 For the second claim, let $S$ be an $\co_\kk[1/p]$-scheme, fix an $S$-point
 $(A_0^\flat , A^\flat) \in \mathcal{S}_{V^\flat}(S)$, and let $(A_0,A) \in \mathcal{S}_V(S)$ be its image under the top horizontal arrow in \eqref{i_p pullback}.
 Tracing through the construction of $i_\mathfrak{p}$ in Proposition \ref{prop:i_p construction}, we see that 
\[
\Lie(A) / \mathcal{F} \iso  \Lie(A^\flat) / \mathcal{F}^\flat.
\]
  It follows immediately from this and \eqref{dual taut} that $\taut_V|_S \iso \taut_{V^\flat}|_S$. 
At a complex point $s \in S(\C)$ we have a decomposition 
\[
H_1(A_s , \Q) = H_1( \tilde{A}_{0s} ,\Q) \oplus H_1( A_s^\flat , \Q),
\]
which induces an isometric embedding
\[
\Hom_\kk( H_1(A_{0s} ,\Q )  , H_1( A_s^\flat , \Q) ) \subset \Hom_\kk( H_1(A_{0s} ,\Q )  , H_1( A_s , \Q) )
\]
of $\kk$-hermitian spaces.  Recalling \eqref{taut realization}, this isometric embedding restricts to the isomorphism
$\taut_{V^\flat,s} \iso \taut_{V,s}$ just constructed, which therefore preserves the metrics defined by \eqref{taut metric}. 

Finally, we show that under the top horizontal arrow in \eqref{i_p pullback}, the exceptional divisor $\mathrm{Exc}_V \subset \mathcal{S}_V$ pulls back to the exceptional divisor $\mathrm{Exc}_{V^\flat} \subset \mathcal{S}_{V^\flat}$.
These  divisors are, by Definition \ref{def:special exceptional}, the pullbacks of the singular loci 
\begin{align}
\mathcal{M}_{(1,0)} \times_{\co_\kk} \mathrm{Sing}_{ (n-1,1) } 
& \subset \mathcal{M}_{(1,0)} \times_{\co_\kk} \mathcal{M}^\Pap_{ (n-1,1) }  \label{first singular} \\
\mathcal{M}_{(1,0)} \times_{\co_\kk}  \mathrm{Sing}_{ (n-2,1) }
& \subset \mathcal{M}_{(1,0)} \times_{\co_\kk} \mathcal{M}^\Pap_{ (n-2,1) } \label{second singular}
\end{align}
of \eqref{singular} under the vertical arrows in \eqref{i_p pullback}, and so it suffices to prove that the first singular locus  \eqref{first singular} pulls back to the second singular locus  \eqref{second singular} under the bottom horizontal arrow in \eqref{i_p pullback}. 
Moreover,  as each of \eqref{first singular} and \eqref{second singular} is a reduced $\co_\kk$-stack of dimension $0$, and as   the bottom horizontal arrow in \eqref{i_p pullback} is finite and unramified, it suffices to check this on the level of geometric points.
This is an easy exercise in linear algebra, using the characterization of the singular locus found in 
Remark \ref{rem:singular description}. 
%
\end{proof}

\begin{proposition}\label{prop:i_p nonexceptional}
The nonexceptional locus $ \mathcal{U}_\mathfrak{p}^\nonexc \subset  \mathcal{U}_\mathfrak{p}$ satisfies
\begin{equation*}
 \mathcal{U}_\mathfrak{p}^\nonexc \subset \mathcal{U}_\mathfrak{p}^\dagger,
\end{equation*}
and the  isomorphism   of Proposition \ref{prop:i_p construction}  restricts to an isomorphism 
\[
\mathcal{S}^\nonexc_{V^\flat / \co_\kk[1/p] }    \iso
 \mathcal{U}^\nonexc_{\mathfrak{p}}  .
\]
\end{proposition}

\begin{proof}
Suppose $S$ is an $\co_\kk[1/p]$-scheme, 
\[
(A_0,A,x) \in \mathcal{U}_\mathfrak{p}^\nonexc(S),
\]
and let $A^\flat$ be the  abelian scheme  of \eqref{pi product}.
In particular, \[\Lie(A) = \Lie( \tilde{A}_0 ) \oplus \Lie(A^\flat).\]

The natural map  
\[
\mathcal{U}_\mathfrak{p}^\nonexc \to  \mathcal{M}_{(n-1,1)} \smallsetminus \mathrm{Exc}_{(n-1,1)}
\]
endows $A$ with a hyperplane $\mathcal{F} \subset \Lie(A)$, and results of Kr\"amer, as summarized in Theorem 2.3.3 of \cite{BHKRY-1}, show that this hyperplane is actually determined by the $\co_\kk$-action on $A$ using the following recipe: 
If we fix  a $\pi \in \co_\kk$ such that $\co_\kk= \Z[\pi]$, and let
\[
\overline{\epsilon}_S = \overline{\pi} \otimes 1 -  1\otimes i_S(\overline{\pi}) \in \co_\kk \otimes_\Z \co_S,
\]
where $i_S : \co_\kk \to \co_S$ is the structure map,  then 
\[
\mathcal{F} = \mathrm{ker}( \overline{\epsilon}_S : \Lie(A) \to \Lie(A) ).
\]

On the other hand,  the action of $\co_\kk$ on the Lie algebra of $A_0$ is through the structure morphism $i_S$, and so the same is true of  $\tilde{A}_0 = A_0 / A_0[\mathfrak{p}]$.
This implies that $\overline{\epsilon}_S$ annihilates $ \Lie( \tilde{A}_0 )$, and hence
$
 \Lie( \tilde{A}_0 ) \subset \mathcal{F}.
$
Having shown 
\[
(A_0,A,x) \in \mathcal{U}_\mathfrak{p}^\dagger(S),
\]
 the first claim of the proposition is proved.
The second claim follows from the first and Proposition \ref{prop:i_p pullbacks}.
\end{proof}

\begin{proposition}\label{prop:i_p toroidal}
The closed immersion $i_\mathfrak{p} : \mathcal{S}_{V^\flat/\co_\kk[1/p]} \to \mathcal{U}_\mathfrak{p}$ of Proposition \ref{prop:i_p construction}  extends uniquely to a proper morphism 
\[
\bar{\mathcal{S}}_{V^\flat/\co_\kk[1/p] } \to  \bar{\mathcal{U}}_\mathfrak{p},
\]
where the codomain is defined as the Zariski closure of 
\[
 \mathcal{U}_\mathfrak{p} \subset  \bar{\mathcal{Z}}_V(p)_{/\co_\kk[1/p]}.
\]
or, equivalently, as the normalization of 
$ \mathcal{U}_\mathfrak{p} \to \bar{\mathcal{S}}_{V/\co_\kk[1/p]}.$
\end{proposition}

\begin{proof}
For ease of notation, we omit all subscripts $\co_\kk[1/p]$ in the proof.

As the exceptional locus of $\bar{\mathcal{S}}_{V^\flat}$ does not meet the boundary, it suffices to construct the extension over its complement
\[
\bar{\mathcal{S}}_{V^\flat}^\nonexc = \bar{\mathcal{S}}_{V^\flat}  \smallsetminus \mathrm{Exc}_{V^\flat} .
\]
Consider the universal pair $(A_0,A^\flat)$ over $\mathcal{S}_{V^\flat}^\nonexc$,  and let
\[
x: A_0 \to A =  \tilde{A}_0 \times A^\flat
\]
be as in the proof of Proposition \ref{prop:i_p construction}.   In other words,  the triple $(A_0,A,x)$ over $\mathcal{S}_{V^\flat}^\nonexc$ determines the isomorphism
$
i_\mathfrak{p} :  \mathcal{S}_{V^\flat}^\nonexc \iso  \mathcal{U}_\mathfrak{p}^\nonexc.
$
of Proposition \ref{prop:i_p nonexceptional}.

Consider the open immersions
\[
\bar{\mathcal{S}}_{V^\flat}^\nonexc \subset 
\bar{\mathcal{S}}_{V^\flat}  \stackrel{\eqref{compact inclusion} }{\subset} \mathcal{M}_{ (1,0 )}   \times_{\co_\kk}  \bar{\mathcal{M}}_{(n-2,1)} .
\]
As in the discussion of \S \ref{ss:toroidal}, the universal abelian scheme over $\mathcal{M}_{(n-2,1)}$ extends uniquely to a semi-abelian scheme over $\bar{\mathcal{M}}_{(n-2,1)}$.  
Pulling this back via projection to the second factor in the fiber product, and then restricting to $\bar{\mathcal{S}}_{V^\flat}^\nonexc$, yields a semi-abelian scheme over  $\bar{\mathcal{S}}_{V^\flat}^\nonexc $  extending  $A^\flat$.
Similarly,  the elliptic curve  $A_0$    extends to an elliptic curve over $\bar{\mathcal{S}}_{V^\flat}^\nonexc$, obtained as a  pullback from the first factor in the fiber product above.
The quotient $\tilde{A}_0 = A_0 /A_0[\mathfrak{p}]$  therefore also extends to an elliptic curve over $\bar{\mathcal{S}}_{V^\flat}^\nonexc$.

The product $A=  \tilde{A}_0 \times A^\flat$  therefore extends to a semi-abelian scheme over $\bar{\mathcal{S}}_{V^\flat}^\nonexc$,  and  the extension property used  in the proof of Lemma \ref{lem:normalized compactification} provides us with a morphism
\[
i_\mathfrak{p} :  \bar{\mathcal{S}}_{V^\flat}^\nonexc \to  \mathcal{M}_{(1,0)} \times_{\co_\kk} \bar{\mathcal{M}}^\Pap_{(n-1,1)}
\]
taking values in the open subscheme
\[
\mathcal{M}_{(1,0)} \times_{\co_\kk} \big( \bar{\mathcal{M}}^\Pap_{(n-1,1)} \smallsetminus \mathrm{Sing}_{(n-1,1)} \big)
\iso 
\mathcal{M}_{(1,0)} \times_{\co_\kk} \big( \bar{\mathcal{M}}_{(n-1,1)} \smallsetminus \mathrm{Exc}_{(n-1,1)} \big).
\]
Thus we obtain  a morphism
\[
i_\mathfrak{p} :  \bar{\mathcal{S}}_{V^\flat}^\nonexc \to  \mathcal{M}_{(1,0)} \times_{\co_\kk}  \bar{\mathcal{M}}_{(n-1,1)} ,
\]
 taking values in the open and closed substack $\bar{\mathcal{S}}_V$.

We now have a commutative diagram of solid arrows
\[
\xymatrix{
{ \mathcal{S}_{V^\flat}^\nonexc  } \ar[r]^{i_\mathfrak{p}}  \ar[d]  & { \bar{\mathcal{U}}_\mathfrak{p}   } \ar[d]  \\ 
{ \bar{\mathcal{S}}_{V^\flat}^\nonexc  }  \ar[r] \ar@{.>}[ur]&   { \bar{\mathcal{S}}_V  } 
}
\]
in which the vertical arrow on the right is integral, and hence affine, by its construction as a normalization; see \cite[\href{https://stacks.math.columbia.edu/tag/035I}{Tag 035I}]{stacks-project}.
To complete the proof of the proposition, it suffices to show that there is a unique dotted arrow making the diagram commute.
This immediately reduces to the corresponding claim for affine schemes, in which 
 we are given homomorphisms of rings 
\[
\xymatrix{
{ R^\flat  }   & { B } \ar[l] \ar@{.>}[dl]  \\ 
{ \bar{R}^\flat } \ar[u]  &   { A  } \ar[l] \ar[u] 
}
\]
with $B$ integral over $A$,  and $\bar{R}^\flat \subset R^\flat$ is an inclusion of normal domains with the same field of fractions.
The image of any $b\in B$ under $B \to R^\flat$ is integral over  $\bar{R}^\flat$, hence contained in $\bar{R}^\flat$.
Thus there is a unique dotted arrow making the diagram commute.
\end{proof}


\section{Two lemmas on abelian schemes}


For use in the next section, we prove  two lemmas on abelian schemes.
The first is a criterion for determining when a polarization descends to a quotient.
The second is a criterion for the existence of an extension as a semi-abelian scheme.

\begin{lemma}\label{lem:polarization descent}
Suppose  $f : X\to Y$ is an isogeny of abelian schemes over a Noetherian scheme $S$, and $\psi_X : X\to X^\vee$ is a polarization whose degree $d$ is invertible in $\co_S$.    The following are equivalent.
\begin{enumerate}
\item
The kernel of $f$ is contained in $\mathrm{ker}(\psi_X)$, and is totally isotropic with respect to the 
Weil pairing
\[
e : \mathrm{ker}(\psi_X)  \times \mathrm{ker}(\psi_X) \to \mu_d.
\]
\item
There exists a (necessarily unique) polarization $\psi_Y : Y \to Y^\vee$ making the diagram 
\[
\xymatrix{
{  X } \ar[r]^{\psi_X}  \ar[d]_f &  { X^\vee }  \\
 { Y } \ar[r]_{\psi_Y} & { Y^\vee } \ar[u]_{f^\vee} 
}
\]
commute.
\end{enumerate}
\end{lemma}

\begin{proof}
This is routine.
See, for example, Proposition 10.4 of \cite{polishchuk}.
\end{proof}

\begin{lemma}\label{lem:semiextension}
Suppose $S$ is a normal Noetherian  scheme, $U\subset S$ is a dense open subscheme, and 
 $f : X \to Y$ is an isogeny of abelian schemes over $U$ whose degree $d$ is invertible in $\co_S$.
If $X$ extends to a semi-abelian scheme  over $S$ then so does $Y$, and both extensions are unique.
\end{lemma}

\begin{proof}
The uniqueness claim follows from Proposition I.2.7 of \cite{faltings-chai}, so we only need to prove the existence of a semi-abelian extension of $Y$.

Denote by $X^* \to S$  the semi-abelian extension of $X$.  The morphism
$f : X \to Y$, as well as the multiplication-by-$d$-maps on $X$ and $X^*$, are quasi-finite and \'etale by Lemma 2 of Chapter 7.3 of \cite{BLR}.  In particular, both the composition 
$
\mathrm{ker}(f) \to X[d] \to U
$ 
and  the second arrow and the composition are   \'etale morphisms.   By \cite[\href{https://stacks.math.columbia.edu/tag/02GW}{Tag 02GW}]{stacks-project}, the first arrow is also \'etale.
As the first arrow is   a closed immersion, we deduce that  $\mathrm{ker}(f) \subset X[d]$ is a union of connected components.

We have assumed that $S$ is normal, so the same is true of the \'etale $S$-scheme  $X^*[d]$.   
As the connected components of a normal scheme are the same as its irreducible components, it follows that the Zariski closure of $\mathrm{ker}(f)$ in $X^*[d]$  is  an open and closed subscheme $\mathrm{ker}(f)^* \subset X^*[d]$.  In particular  $\mathrm{ker}(f)^*\to S$ is quasi-finite and  \'etale.

By Lemma IV.7.1.2 of \cite{moret-bailly},  the fppf quotient  $Y^* = X^*/\mathrm{ker}(f)^*$  provides  a semi-abelian extension of $Y$ to $S$.
\end{proof}


\section{Analysis of $\mathcal{U}_0$}


In this section we study the substack $\mathcal{U}_0$ of \eqref{flavor decomp}, and make explicit its relation with $\mathcal{S}_{V^\flat}$.

Return to the situation of Lemma \ref{lem:lattice trifurcation}, so that  
$x\in \Hom_{\co_\kk}(\mathfrak{a}_0,\mathfrak{a})$ with $\langle x,x\rangle =p$ determines an orthogonal decomposition
\begin{equation}\label{W split}
W = \tilde{W}_0 \oplus W^\flat .
\end{equation}
Set $\tilde{\mathfrak{a}}_0 = \mathfrak{a}\cap \tilde{W}_0$.

\begin{lemma}\label{lem:0 linear algebra}
Assume that  $\tilde{\mathfrak{a}}_0 = x(\mathfrak{a}_0)$.
If we set  $\mathfrak{b}  = \mathfrak{a} \cap W^\flat$ and let $\mathfrak{b}^\vee \subset W^\flat $ be its dual lattice with respect to $h|_{ W^\flat }$,  the $\co_\kk$-lattice
\[
\mathfrak{a}^\flat=  \mathfrak{p} \cdot  \mathfrak{b}^\vee + \mathfrak{b} ,
\]
  is self-dual  with respect to $h|_{W^\flat}$.  Moreover, there are inclusions of $\co_\kk$-lattices
\begin{equation}\label{lattice correspondence}
\tilde{\mathfrak{a}}_0 \oplus \mathfrak{a}^\flat  
\stackrel{p}{\supset}  \tilde{\mathfrak{a}}_0 \oplus \mathfrak{b} 
 \stackrel{p^2}{\subset}  \mathfrak{a}
\end{equation}
in $W$ of the indicated indices, and a  canonical  $\co_\kk$-linear injection
\begin{equation}\label{lattice t}
\mathfrak{p}^{-1} \tilde{\mathfrak{a}}_0/ \tilde{\mathfrak{a}}_0 \iso  \mathfrak{p}^{-1} \mathfrak{b} / \mathfrak{b}  \map{t} \mathfrak{p}^{-1} \mathfrak{a}^\flat/ \mathfrak{a}^\flat .
\end{equation}
\end{lemma}

\begin{proof}
Throughout the proof, we use  $x$ to identify $W_0 = \tilde{W}_0$, so that $h|_{W_0} = p h_0$.
The projections to the two factors in \eqref{W split} induce  isomorphisms
\[
\mathfrak{a} / ( \mathfrak{a}_0 \oplus \mathfrak{b} )  \to \mathfrak{a}_0^\vee / \mathfrak{a}_0,
\qquad 
\mathfrak{a} /(  \mathfrak{a}_0 \oplus \mathfrak{b} )  \to \mathfrak{b}^\vee / \mathfrak{b},
\]
where  $\mathfrak{a}_0^\vee= p^{-1}  \mathfrak{a}_0$ is the dual lattice of $\mathfrak{a}_0$ with respect to $h|_{W_0}$.
 Composing these isomorphisms yields an $\co_\kk$-linear isomorphism
\begin{equation}\label{p trivialization}
p^{-1} \mathfrak{a}_0 / \mathfrak{a}_0 \iso \mathfrak{b}^\vee / \mathfrak{b}
\end{equation}
 respecting the $p^{-1}\co_\kk/\co_\kk$-valued hermitian forms induced  by $h|_{W_0}= ph_0$ and $h|_{W^\flat}$.
It follows  that $\mathfrak{p} \cdot (  \mathfrak{b}^\vee / \mathfrak{b} )  \subset \mathfrak{b}^\vee / \mathfrak{b}$
is maximal isotropic with respect to $h|_{W^\flat}$,  and so $\mathfrak{a}^\flat$
is self-dual under $h|_{W^\flat}$.

Consider the  inclusions  $\mathfrak{b} \subset  \mathfrak{a}^\flat \subset \mathfrak{b}^\vee$ of $\co_\kk$-modules.
The first  is an isomorphism everywhere locally except at $\overline{\mathfrak{p}}$, while the second is an isomorphism    everywhere locally except at $\mathfrak{p}$.  
In particular, the first  induces an isomorphism
\[
\mathfrak{p}^{-1} \mathfrak{b} / \mathfrak{b} \iso \mathfrak{p}^{-1} \mathfrak{a}^\flat/\mathfrak{a}^\flat .
\]
Restricting  \eqref{p trivialization} to an isomorphism of $\mathfrak{p}$-torsion submodules  defines  \eqref{lattice t}.
The inclusions of \eqref{lattice correspondence}, with the indicated indices,  are clear from what we have said.
\end{proof}

Just as Lemma \ref{lem:0 linear algebra} is  more complicated than Lemma \ref{lem:pi linear algebra}, the analysis of $\mathcal{U}_0$ is  more complicated than that of $\mathcal{U}_\mathfrak{p}$.

Let $S$ be an $\co_\kk[1/p]$-scheme.
Given  Lemma \ref{lem:0 linear algebra}, one expects  that   any  $S$-point $(A_0,A,x) \in \mathcal{U}_0(S)$  should determine a diagram of  abelian schemes with $\co_\kk$-actions
\begin{equation}\label{abelian correspondence}
\xymatrix{
 { \tilde{A}_0 \times A^\flat } &  &{ \tilde{A}_0 \times B}  \ar[ll]_{\deg = p} \ar[rr]^{\deg = p^2} &  & { A}
}
\end{equation}
 in which the arrows are $\co_\kk$-linear isogenies of the indicated degrees.
 There should also be a canonical $\co_\kk$-linear closed immersion
\begin{equation}\label{t construct}
\tilde{A}_0[\mathfrak{p}] \iso B[\mathfrak{p}]  \map{t} A^\flat[ \mathfrak{p} ],
\end{equation}
and the morphism $x:A_0 \to A$ should factor as
\[
A_0 \iso \tilde{A}_0 \hookrightarrow \tilde{A}_0 \times B \to A.
\]

Moreover,  these abelian schemes should come with polarizations with the following properties.
 The elliptic curve $\tilde{A}_0$ is equipped with its unique polarization of degree $p^2$, 
 $B$ is equipped with a polarization of degree $p^2$, and  $A^\flat$ is equipped with a principal polarization.
The pullback  of the product polarization on $\tilde{A}_0 \times A^\flat$ via the leftwards arrow in \eqref{abelian correspondence} is the product polarization on $\tilde{A}_0 \times B$, and similarly for the  the pullback of the principal  polarization on $A$  via the rightward arrow.

Here is how to construct the data \eqref{abelian correspondence} and \eqref{t construct}  from $(A_0,A,x)$.
 As in the proof of Proposition \ref{prop:split divisor flavors},  at any geometric point $s \to S$ we  may apply Lemma   \ref{lem:0 linear algebra} with $\mathfrak{a}_0$ and $\mathfrak{a}$ replaced by the  $p$-adic Tate modules of $A_{0s}$ and $A_s$ to obtain $\co_\kk$-lattices 
\begin{equation}\label{tate correspondence}
T_p( \tilde{A}_{0s}) \oplus T_p( A_s^\flat)  \supset T_p( \tilde{A}_{0s}) \oplus T_p(B_s)  \subset T_p(A_s).
\end{equation}
analogous to \eqref{lattice correspondence}, along with an $\co_\kk$-linear injection
\begin{equation}\label{tate-level t}
\mathfrak{p}^{-1} T_p(\tilde{A}_{0s})/T_p(\tilde{A}_{0s})  \iso \mathfrak{p}^{-1} T_p(B_s)/T_p(B_s)  
 \map{t}  \mathfrak{p}^{-1} T_p(A_s^\flat)/T_p(A_s^\flat).
\end{equation}
 To be clear, we have not yet constructed  the abelian schemes $\tilde{A}_0$, $A^\flat$, and $B$, only  lattices in $T_p(A_s)[1/p]$ that we choose to call  $T_p(\tilde{A}_s)$,  $T_p(A_s^\flat)$,  and $T_p(B_s)$.  
 However, as $p\in \co_S^\times$ and both inclusions in \eqref{tate correspondence} have finite index, there are abelian schemes $C$ and $D$ over $S$ with $\co_\kk$-actions and $\co_\kk$-linear  isogenies
\[
\xymatrix{
{  D }  &  {  C    }  \ar[r]^{p^2} \ar[l]_{p} &  { A   } 
}
\]
of the indicated degrees that identify 
\begin{equation}\label{tate-level decomp}
T_p(C_s) = T_p(\tilde{A}_{0s}) \oplus T_p(B_s) ,\qquad T_p(D_s) = T_p(\tilde{A}_{0s}) \oplus T_p(A_s^\flat).
\end{equation}

  The condition that $\langle x,x\rangle =p$ implies that 
 \[
 A_0 \map{x} A \iso A^\vee \map{x^\vee} A_0^\vee \iso A_0
 \]
 is multiplication by $p$, which implies that the composition
 \[
 A \iso A^\vee \map{x^\vee} A_0^\vee\iso A_0 \map{x} A
 \]
 is a Rosati-fixed element $ \alpha \in \End_{\co_\kk}(A)[1/p]$ such that $\alpha\circ \alpha= p\alpha$.
In particular, we obtain an idempotent
\[
p^{-1}\alpha \in  \End_{\co_\kk}(A)[1/p]\iso  \End_{\co_\kk}(C)[1/p]\iso  \End_{\co_\kk}(D)[1/p],
\]
whose induced actions on $T_p(C_s)[1/p]$ and $T_p(D_s)[1/p]$ are just the projections to the first summands in \eqref{tate-level decomp}.  It follows that we in fact have
\[
p^{-1}\alpha \in      \End_{\co_\kk}(C) \quad \mbox{and}\quad  p^{-1}\alpha \in  \End_{\co_\kk}(D).
\]
These idempotents are the projections to the first factors for unique decompositions
\[
C= \tilde{A}_0 \times B ,\qquad  D=\tilde{A}_0 \times A^\flat
\]
of abelian schemes with $\co_\kk$-actions, and the image of $x$ under
\[
 \Hom_{\co_\kk}(A_0 , A) [1/p] \iso \Hom_{\co_\kk}(A_0 , C)[1/p] \to \Hom_{\co_\kk}(A_0 , \tilde{A}_0)[1/p]
\]
is an isomorphism  $A_0 \iso \tilde{A}_0$. 
 In particular, we now have $\co_\kk$-linear isogenies \eqref{abelian correspondence}
such that taking $p$-adic Tate modules  recovers \eqref{tate correspondence}, and \eqref{tate-level t} induces 
the closed immersion \eqref{t construct}.

It remains to construct the polarizations on $\tilde{A}_0$, $B$, and $A^\flat$.
 If we pull-back the principal polarization on $A$ via $\tilde{A}_0 \times B \to A$, we obtain a  polarization on $\tilde{A}_0 \times B$ of degree $p^4$.   
As the idempotent 
\[
p^{-1}\alpha \in      \End_{\co_\kk}(C) = \End_{\co_\kk}(  \tilde{A}_0 \times B ) 
\]
 constructed above is Rosati-fixed,  this pullback polarization  splits as the product of  polarizations 
 \begin{equation}\label{split polarizations}
 \tilde{A}_0 \to \tilde{A}_0^\vee \quad  \mbox{and}\quad    B\to B^\vee .
 \end{equation}

We claim that the kernels of both polarizations in \eqref{split polarizations} are isomorphic, \'etale locally on $S$, to the constant $\co_\kk$-module scheme $\co_\kk / (p)$.  It suffices to check this at a fixed geometric point $s \to S$, where the inclusions of $p$-adic Tate modules \eqref{tate correspondence} behave in exactly the same way as the $\co_\kk$-lattices in \eqref{lattice correspondence}.
For example, in the proof of Lemma \ref{lem:0 linear algebra} we saw that there
 are canonical isomorphisms
\[
 \mathfrak{b}^\vee / \mathfrak{b} \iso  \mathfrak{a} / ( \widetilde{\mathfrak{a}}_0 \oplus\mathfrak{b})  
 \iso  \tilde{\mathfrak{a}}_0^\vee / \tilde{\mathfrak{a}}_0  \iso  
 p^{-1} \tilde{\mathfrak{a}}_0 / \tilde{\mathfrak{a}}_0
\]
(recall that in the proof of the lemma we identified $\mathfrak{a}_0=x(\mathfrak{a}_0) = \tilde{\mathfrak{a}}_0$),
and hence there are analogous isomorphisms
\[
\frac{ T_p(B_s^\vee)  }{ T_p( B_s)}  \iso  \frac{ T_p(A_s) }{ T_p( \tilde{A}_{0,s} ) \times T_p(B_s)   } 
 \iso  \frac{   T_p(A_{0_s}^\vee) }{  T_p(A_{0,s}) }  \iso  \frac{  p^{-1} T_p(A_{0,s}) }{  T_p(A_{0,s}) } .
 \]
The rightmost quotient is obviously isomorphic to $\co_\kk / (p)$, and hence the fibers at $s$ of the kernels of \eqref{split polarizations} are both isomorphic to $\co_\kk/(p)$.

Next we claim that we have  equalities 
 \[
 \ker(B\to A^\flat)  = \mathfrak{p} \cdot \ker( B\to B^\vee ) = \ker( B\to B^\vee )[\overline{\mathfrak{p}}] 
 \]
of subgroup schemes of $\ker(B\to B^\vee )$, all maximal isotropic under the Weil pairing.
   The second equality is clear, as we have just shown that $\ker(B\to B^\vee )$ is \'etale locally isomorphic to $\co_\kk/(p)$.  
 As above, the first equality and the maximal isotropy can be checked on the level of $p$-adic Tate modules at a  geometric point $s \to S$, where they follow from the fact that our abelian schemes were constructed so that their  $p$-adic Tate modules have the same linear-algebraic properties as the $\co_\kk$-lattices in 
 Lemma \ref{lem:0 linear algebra}.  More precisely, by construction the lattice
  \[
  T_p(A_s^\flat)  = \mathfrak{p} T_p( B_s^\vee) + T_p(B_s)
 \]
 is  self-dual  under the Weil pairing on $T_p(B_s) \otimes \Q_p$, and it follows that 
 \[
  \ker(B_s\to A_s^\flat)  =
 \frac{ T_p(A_s^\flat) }{ T_p(B_s) } =  \mathfrak{p} \cdot  \frac{ T_p( B_s^\vee )  }{ T_p(B_s) }
 = \mathfrak{p} \cdot \ker( B_s\to B_s^\vee ) 
 \]
 is a maximal isotropic submodule of $\ker(B_s\to B_s^\vee ) = T_p( B_s^\vee ) /  T_p(B_s)$.

By Lemma \ref{lem:polarization descent},  the maximal isotropy of the subgroup scheme
\[
\ker(B\to A^\flat)  \subset \ker( B\to B^\vee ) 
\]
under the Weil pairing implies that the polarization
$B\to B^\vee$  descends   to a principal polarization on $A^\flat$.
 This completes the construction of the abelian schemes \eqref{abelian correspondence} with all of their expected extra structure.
 
Now apply this construction to the universal triple $(A_0,A,x)$ over $\mathcal{U}_0$.
 As $p\in \co_{\mathcal{U}_0}^\times$, the $p$-power isogenies \eqref{abelian correspondence} induce an isomorphism
 \begin{equation}\label{i_0 Lie}
 \Lie(A) \iso \Lie( \tilde{A}_0) \times \Lie(A^\flat)
 \end{equation}
 of rank $n$ vector bundles on $\mathcal{U}_0$. 
Recalling that $A$ comes equipped with an $\co_\kk$-stable  hyperplane $\mathcal{F} \subset\Lie(A)$, denote by 
\[
\mathcal{U}_0^\dagger \subset  \mathcal{U}_0
\]
  the largest closed substack over which  $\Lie(\tilde{A}_0) \subset \mathcal{F}$

Define a finite \'etale cover 
\begin{equation}\label{tau cover}
 \mathcal{T}_{V^\flat}   \to \mathcal{S}_{V^\flat/\co_\kk[1/p]}
\end{equation}
of degree $p^{n-1}-1$ as follows:
for any $\co_\kk[1/p]$-scheme $S$, let  $\mathcal{T}_{V^\flat}(S)$ be the groupoid of triples $(A^\flat_0,A^\flat,t)$ in which 
\[
(A^\flat_0,A^\flat) \in \mathcal{S}_{V^\flat}(S) \quad 
\mbox{and} \quad  
t : A^\flat_0[ \mathfrak{p} ] \to A^\flat[\mathfrak{p}]
\]
 is an $\co_\kk$-linear closed immersion of  group schemes over $S$.

\begin{proposition}\label{prop:i_0 construction}
There is a canonical isomorphism $i_0 :  \mathcal{T}_{V^\flat}  \iso \mathcal{U}_0^\dagger$.
\end{proposition}

\begin{proof}
The morphism $\mathcal{U}_0^\dagger \to  \mathcal{T}_{V^\flat}$ is essentially described above.
Suppose $S$ is an $\co_\kk[1/p]$-scheme and 
\[
(A_0,A,x) \in \mathcal{U}_0^\dagger(S).
\]
We can use the isomorphism \eqref{i_0 Lie} to endow the abelian scheme 
 $A^\flat$ of \eqref{abelian correspondence} with the  rank one local direct summand
 \[
 \mathcal{F} ^\flat = \mathcal{F} / \Lie( \tilde{A}_0) \subset \Lie(A) /\Lie( \tilde{A}_0) =  \Lie(A^\flat) 
 \]
 to obtain a point $A^\flat \in \mathcal{M}_{(n-2,1)}(S)$.  Setting $A_0^\flat=\tilde{A}_0$,
 this defines a morphism 
 \[
 \mathcal{U}_0^\dagger \map{  (A_0,A,x) \mapsto (A^\flat_0,A^\flat)  }  \mathcal{S}_{V^\flat / \co_\kk[1/p] },
  \]
and  the  closed immersion  \eqref{t construct} defines the desired lift to
  $\mathcal{U}_0^\dagger \to  \mathcal{T}_{V^\flat}$.

Now we construct the inverse.
Start with  a point
 $
(A^\flat_0, A^\flat , t) \in \mathcal{T}_{V^\flat}(S).
 $
 Denote by $B$ the abelian scheme dual to 
 \[
 B^\vee = A^\flat / \mathrm{Im}(t).
 \]
Using the principal polarization to identify $A^\flat$ with its dual, the quotient map $A^\flat \to B^\vee$ dualizes to a morphism $B \to A^\flat$, and  the composition 
\[
B\to A^\flat \to B^\vee
\]
is a degree $p^2$ polarization  (it is the pullback via $B \to A^\flat$ of the principal polarization on $A^\flat$).
It has  the property that each factor on the right hand side of 
\[
\ker( B \to B^\vee ) = \ker(B \to B^\vee ) [\mathfrak{p}]  \times \ker(B \to B^\vee ) [ \overline{\mathfrak{p}} ]
\]
has order $p$, and  the second factor is the kernel of $B\to A^\flat$.  In particular, the induced map
$B[\mathfrak{p}] \to A^\flat[\mathfrak{p}]$ is an isomorphism, and the closed immersion 
\[
A^\flat_0[\mathfrak{p}] \map{t} A^\flat[\mathfrak{p}] \iso  B[\mathfrak{p}]
\]
has image $\ker(B \to B^\vee) [\mathfrak{p}] $.
Setting  $\tilde{A}_0 = A^\flat_0$, the resulting   isomorphism 
\[
\tilde{A}_0[\mathfrak{p}] \iso \ker(B \to B^\vee) [\mathfrak{p}]
\]
 admits a  unique $\co_\kk$-linear extension to an isomorphism
\[
t : \tilde{A}_0[ p ] \iso \ker(B \to B^\vee)
\]
 identifying the Weil pairings on source and target.

The antidiagonal subgroup
\[
\Delta \define \tilde{A}_0[p]\map{  a_0 \mapsto ( - a_0 , t (a_0) ) }  \tilde{A}_0 \times B.
\]
  is totally isotropic under the Weil pairing induced by the product polarization (where $\tilde{A}_0$ is given the polarization of degree $p^2$), which therefore descends, by Lemma \ref{lem:polarization descent},  to a principal polarization on the quotient 
\[
A = (\tilde{A}_0\times B)/\Delta.
\]
As  $p \in \co_S^\times$,  there is an  induced  isomorphism of Lie algebras \eqref{i_0 Lie}, 
which allows us to define a corank one local direct summand $\mathcal{F} \subset \Lie(A)$ as the product
\[
\mathcal{F}  = \Lie(\tilde{A}_0) \times \mathcal{F}^\flat.
\]
This defines  a point $A \in \mathcal{M}_{(n-1,1)}(S)$.  Setting $A_0 = \tilde{A}_0$,  the composition 
\[
A_0=\tilde{A}_0 \hookrightarrow  \tilde{A}_0 \times B \to A
\]
defines  $x \in \Hom_{\co_\kk}(A_0,A)$ with $\langle x,x\rangle=p$

The above construction  determines  a morphism
\[
\mathcal{T}_{V^\flat} \map{ (A^\flat_0,A^\flat,t) \to (A_0,A,x)    }  \mathcal{Z}_V(p)
\]
taking values in the open substack $\mathcal{U}_0^\dagger$,  which is inverse to the map $\mathcal{U}_0^\dagger \to \mathcal{T}_{V^\flat}$ constructed above.
\end{proof}

Proposition \ref{prop:i_0 construction} gives us morphisms
\begin{equation}\label{i_0 pullback}
\xymatrix{
{  \mathcal{S}_{V^\flat/\co_\kk[1/p] } }    
& & { \mathcal{T}_{V^\flat} }  \ar[rr]       \ar[ll]_{\eqref{tau cover}}
& & { \mathcal{S}_{V/\co_\kk[1/p] } }
}
\end{equation}
in which the  arrow on the right is the (finite and unramified) composition
\[
{\mathcal{T}_{V^\flat} } \map{i_0 } \mathcal{U}^\dagger_0 \hookrightarrow   \mathcal{Z}_V(p)_{/\co_\kk[1/p]}  \to \mathcal{S}_{V/\co_\kk[1/p]} ,
\]
sending  $(A_0^\flat , A^\flat , t ) \mapsto (A_0,A)$, where $A_0=A_0^\flat$, and $A$ is related to $A^\flat$ by a diagram \eqref{abelian correspondence}.

\begin{proposition}\label{prop:i_0 pullbacks}
The hermitian line bundles
\[
\widehat{\omega}^\mathrm{Hdg}_{A_0^\flat/ \mathcal{S}_{V^\flat}} 
  + \widehat{\omega}^\mathrm{Hdg}_{A^\flat/ \mathcal{S}_{V^\flat}} 
     \in   \widehat{\Pic}(  \mathcal{S}_{V^\flat } )
\]
and
$
\widehat{\omega}^\mathrm{Hdg}_{A/ \mathcal{S}_{V}}     \in   \widehat{\Pic}(  \mathcal{S}_{V } )
$
have the same images under the homomorphisms  
\[
\xymatrix{
{ \widehat{\Pic}( \mathcal{S}_{V^\flat}   ) }  \ar[r] &  {  \widehat{\Pic} (    \mathcal{T}_{V^\flat} ) _\Q }&  {  \widehat{\Pic}(  \mathcal{S}_{V } ) }  \ar[l] 
}
\]
induced by \eqref{i_0 pullback}.  
The same is true of 
$(\mathrm{Exc}_{V^\flat} ,0 )$ and $(\mathrm{Exc}_V,0)$, and of $\widehat{\taut}_{V^\flat}$ and $\widehat{\taut}_V$.
\end{proposition}

\begin{proof}
Denote by $(A^\flat_0,A^\flat)$  the pullback of the universal object via the left arrow in \eqref{i_0 pullback}, and by $(A_0,A)$  the pullback of the universal object via the right arrow in \eqref{i_0 pullback}.
The isomorphism of Proposition \ref{prop:i_0 construction}, and hence the right arrow in \eqref{i_0 pullback}, is constructed in such a way that these pullbacks are related by isogenies \eqref{abelian correspondence} of abelian schemes over $\mathcal{T}_{V^\flat}$.  In particular, there is a  quasi-isogeny
 \[
 f \in \Hom( A^\flat_0 \times A^\flat , A)[1/p]
 \]
 of degree $\deg(f)=p$.
As in the  proof of  Proposition \ref{prop:i_p pullbacks}, this implies that
\[
\widehat{\omega}^\mathrm{Hdg}_{A/ \mathcal{T}_{V^\flat}}
=
\widehat{\omega}^\mathrm{Hdg}_{ A^\flat_0 / \mathcal{T}_{V^\flat}}
+
\widehat{\omega}^\mathrm{Hdg}_{ A^\flat / \mathcal{T}_{V^\flat}}   + (0 , -\log(p))  
\]
as elements of $\widehat{\Pic}(\mathcal{T}_{V^\flat} )$, and  the term $(0,-\log(p))$ is torsion. 
The first claim follows  from this.
The remaining claims are essentially the same as in Proposition \ref{prop:i_p pullbacks}, and we leave the details to the reader.
\end{proof}

\begin{proposition}\label{prop:i_0 nonexceptional}
The nonexceptional locus $ \mathcal{U}_0^\nonexc \subset  \mathcal{U}_0$ satisfies
\[
 \mathcal{U}_0^\nonexc \subset \mathcal{U}_0^\dagger.
\]
Moreover, the  morphism of Proposition \ref{prop:i_0 construction}  restricts to an isomorphism 
\[
\mathcal{S}^\nonexc_{V^\flat / \co_\kk[1/p] }    \iso
 \mathcal{U}^\nonexc_0  .
\]
\end{proposition}

\begin{proof}
The proof is essentially the same as Proposition \ref{prop:i_p nonexceptional}, and the details are left to the reader.
\end{proof}

\begin{proposition}\label{prop:i_0 toroidal}
Denote by $\bar{\mathcal{T}}_{V^\flat}$  the normalization of  $ \mathcal{T}_{V^\flat} \to \bar{\mathcal{S}}_{V^\flat/\co_\kk[1/p] }$.
The closed immersion $i_0:\mathcal{T}_{V^\flat} \to \mathcal{U}_0$ of Proposition \ref{prop:i_0 construction} extends to a proper morphism 
\[
\bar{\mathcal{T}}_{V^\flat} \to  \bar{\mathcal{U}}_0,
\]
where the codomain is defined as the Zariski closure of 
\[
 \mathcal{U}_0 \subset \bar{\mathcal{Z}}_V(p)_{/\co_\kk[1/p]},
\]
or, equivalently, as the normalization of 
$ \mathcal{U}_0 \to \bar{\mathcal{S}}_{V/\co_\kk[1/p]}$.
\end{proposition}

\begin{proof}
Consider the universal triple $(A_0,A^\flat,t)$ over $\mathcal{T}_{V^\flat}$, and let  $(A_0, A)$ be the pullback of the universal pair  over $\mathcal{S}_V$ via the composition 
\[
\mathcal{T}_{V^\flat} \map{i_0} \mathcal{U}_0  \to  \mathcal{S}_{V} .
\]
As in the proof of Proposition \ref{prop:i_p toroidal},
we  know that  $A^\flat$ extends to a semi-abelian scheme over $\bar{\mathcal{T}}_{V^\flat}$, obtained as a pullback via 
\[
\bar{\mathcal{T}}_{V^\flat} \to \bar{\mathcal{S}}_{V^\flat} \to \bar{\mathcal{M}}_{(n-2,1)},
\]
and that $A_0$ extends to an elliptic curve over $\bar{\mathcal{T}}_{V^\flat}$, obtained as a pullback via 
\[
\bar{\mathcal{T}}_{V^\flat} \to \bar{\mathcal{S}}_{V^\flat} \to \mathcal{M}_{(1,0)},
\]
Using Lemma \ref{lem:semiextension} and the isogenies \eqref{abelian correspondence},   we deduce that $A$  also extends to  a semi-abelian scheme  over $\bar{\mathcal{T}}_{V^\flat}$.  With this extension in hand, the proof is essentially identical to that of Proposition \ref{prop:i_p toroidal}.
\end{proof}


\section{Proofs of Theorems 6.1.2 and 6.1.4}


In this section we prove Theorems \ref{thm:height descent 1} and  \ref{thm:height descent 2}.
Define an $\co_\kk[1/p]$-stack
\[
\mathcal{Z}^\dagger_V(p)  = \mathcal{T}_{V^\flat} \sqcup  \mathcal{S}_{V^\flat / \co_\kk[1/p] } \sqcup  \mathcal{S}_{V^\flat /\co_\kk[1/p]}  .
\]
 Propositions \ref{prop:i_p construction} and \ref{prop:i_0 construction} provide us with a canonical  closed immersion
 \[
\mathcal{Z}^\dagger_V(p)  
\map{i = i_0 \sqcup i_\mathfrak{p}  \sqcup i_{\overline{\mathfrak{p}} }}    \mathcal{U}_0 \sqcup \mathcal{U}_\mathfrak{p} \sqcup \mathcal{U}_{\overline{\mathfrak{p}}} = \mathcal{Z}_V(p)_{ / \co_\kk[1/p] } 
\]
with image the closed substack $ \mathcal{U}^\dagger_0 \sqcup \mathcal{U}^\dagger_\mathfrak{p} \sqcup \mathcal{U}^\dagger_{\overline{\mathfrak{p}}}$.  Hence we have a diagram
\begin{equation}\label{divisor correspondence}
 \xymatrix{
 { \mathcal{Z}^\dagger_V(p) } \ar[rr]^{i}   \ar[d]_\alpha  & & {      \mathcal{Z}_V(p)_{ / \co_\kk[1/p] } }  \ar[d]^\beta  \\
{ \mathcal{S}_{V^\flat/\co_\kk[1/p]} }   & & { \mathcal{S}_{V/\co_\kk[1/p] } } 
}
\end{equation}
in  which $\alpha$ is a finite \'etale surjection of degree $p^{n-1}+1$, and $\beta$ is finite.
The following lemma shows that $i$ is very close to being an isomorphism.

\begin{lemma}\label{lem:exceptional KR error}
As divisors on $ \mathcal{S}_{V/\co_\kk[1/p] }$, we have
\[
 \mathcal{Z}_V(p)_{ / \co_\kk[1/p] }  = \mathcal{Z}^\dagger_V(p)  + E,
\]
where $\mathcal{Z}^\dagger_V(p)$ has no irreducible components contained in the exceptional divisor $\mathrm{Exc}_{V/\co_\kk[1/p]} \subset \mathcal{S}_{V/\co_\kk[1/p]}$ of Definition \ref{def:special exceptional}, and   $E$ is supported entirely on the exceptional divisor.
\end{lemma}

\begin{proof}
The \emph{exceptional locus} of $  \mathcal{Z}^\dagger_V(p)$ is the preimage of the exceptional divisor under $\alpha$, and  an  irreducible component of it is \emph{exceptional}  if it is contained in the exceptional locus.
Similarly, the \emph{exceptional locus} of $  \mathcal{Z}_V(p)_{ / \co_\kk[1/p] }$ is the preimage of the exceptional divisor under $\beta$, and  an irreducible component of it  is \emph{exceptional} if it is contained in the exceptional locus.

The final claims of Propositions \ref{prop:i_p nonexceptional} and \ref{prop:i_0 nonexceptional} show that $i$  restricts to an isomorphism between the non-exceptional loci of the source and target.  In particular, it establishes a bijection between the  generic points of non-exceptional irreducible components, and corresponding non-exceptional generic have local rings of the same (finite) length.

However, the morphism $\alpha$ is finite \'etale, so no irreducible component of the source can map into the exceptional divisor of the target.  
Thus the  closed immersion $i$ actually establishes a bijection between  irreducible components of the source and  non-exceptional irreducible components of the target, in a way that matches up their multiplicities.  The claim follows immediately.
\end{proof}

We want to construct a compactification of $\mathcal{Z}_V^\dagger(p)$ with sufficiently nice properties that we can form  its pre-log-log arithmetic Chow groups  as in \S \ref{ss:chow}.  To this end, we define
\begin{equation}\label{dagger KR}
 \bar{\mathcal{Z}}^\dagger_V(p) =  \bar{\mathcal{T}}_{V^\flat} \sqcup  \bar{\mathcal{S}}_{V^\flat / \co_\kk[1/p] } \sqcup  \bar{\mathcal{S}}_{V^\flat /\co_\kk[1/p]},
\end{equation}
where $\bar{\mathcal{T}}_{V^\flat}$ was defined in Proposition \ref{prop:i_0 toroidal}.
Equivalently,  $ \bar{\mathcal{Z}}^\dagger_V(p) $ is  the  normalization of 
\[
\alpha: \mathcal{Z}^\dagger_V(p) \to \bar{\mathcal{S}}_{V^\flat/\co_\kk[1/p] }.
\]
 Propositions \ref{prop:i_p toroidal} and \ref{prop:i_0 toroidal} imply that the diagram \eqref{divisor correspondence} extends  to 
\begin{equation}\label{compact divisor correspondence}
 \xymatrix{
 {  \bar{\mathcal{Z}}^\dagger_V(p) } \ar[rr]^i  \ar[d]_\alpha   & & {  \bar{\mathcal{Z}}_V(p)_{ / \co_\kk[1/p] } }  \ar[d]^\beta   \\
{  \bar{\mathcal{S}}_{V^\flat/\co_\kk[1/p]} }  &  & { \bar{\mathcal{S}}_{V/\co_\kk[1/p] } },
}
\end{equation}

For any $N \ge 1$ prime to $p$,  we can add level structure to  $\bar{\mathcal{T}}_{V^\flat}$  by  defining
\[
\mathcal{T}_{V^\flat}(N) = \mathcal{T}_{V^\flat} \times_{\mathcal{S}_{V^\flat}} \mathcal{S}_{V^\flat}(N),
\]   
and letting   $\bar{\mathcal{T}}_{V^\flat}(N)$ be  the normalization of 
\[
\mathcal{T}_{V^\flat}(N) \to  \bar{\mathcal{T}}_{V^\flat  / \co_\kk[1/N]}.
\]
Equivalently, this is the normalization of
\[
\mathcal{T}_{V^\flat}(N) \to  \bar{\mathcal{S}}_{V^\flat}(N)_{ / \co_\kk[ \frac{1}{Np} ]}.
\]

\begin{lemma}\label{lem:T compact}
The stack $\bar{\mathcal{T}}_{V^\flat}(N)$ is regular,  flat, and proper over $\co_\kk[\frac{1}{Np} ]$, and is a projective scheme if $N\ge 3$.  It is smooth in a neighborhood of its boundary 
\[
 \partial \bar{\mathcal{T}}_{V^\flat}(N) = \bar{\mathcal{T}}_{V^\flat}(N) \smallsetminus \mathcal{T}_{V^\flat}(N) , 
\]
 which is a Cartier divisor smooth over $\co_\kk[\frac{1}{Np}]$.
\end{lemma} 

\begin{proof}
The proof is similar to that of Proposition \ref{prop:full compactification}, and is again based on the results of \cite{lan-ramified} and \cite{mp}. 

Recalling from Remark \ref{rem:honest shimura} that the generic fiber of $\mathcal{S}_{V^\flat}$ is the Shimura variety associated to a reductive group $G^\flat$ and a certain compact open subgroup of $G^\flat(\A_f)$,  one can easily check that the generic fiber of the finite \'etale cover 
\begin{equation}\label{alt level}
\mathcal{T}_{V^\flat}(N) \to \mathcal{S}_{V^\flat}(N)_{ /\co_\kk[ \frac{1}{Np} ]}
\end{equation}
is obtained by to passing to two smaller  compact open subgroups.

Recall from \S \ref{ss:special shimura} that $\mathcal{S}_{V^\flat}(N)$ is constructed as an open and closed substack
\[
\mathcal{S}_{V^\flat}(N) \subset \mathcal{M}_{W_0}(N) \times_{\co_\kk[1/N]}  \mathcal{M}_{W^\flat}(N).
\]
In this construction, we may replace $  \mathcal{M}_{W^\flat}(N)$ with the stack $\mathcal{M}^\Pap_{W^\flat}(N)$ appearing in the proof of Proposition \ref{prop:full compactification} to obtain a blow-down 
\[
\mathcal{S}^\Pap_{V^\flat}(N) \subset \mathcal{M}_{W_0}(N) \times_{\co_\kk[1/N]}  \mathcal{M}^\Pap_{W^\flat}(N).
\]
Repeating the construction of \eqref{alt level} with $\mathcal{S}_{V^\flat}(N)$ replaced by $\mathcal{S}^\Pap_{V^\flat}(N)$ yields a  finite \'etale cover 
\[
\mathcal{T}^\Pap_{V^\flat}(N) \to \mathcal{S}^\Pap_{V^\flat}(N)_{ /\co_\kk[ \frac{1}{Np} ]},
\]
and $\mathcal{T}_{V^\flat}(N)$ is the blow-up of the normal stack $\mathcal{T}^\Pap_{V^\flat}(N)$  along its  proper  and $0$-dimensional locus of nonsmooth points.

As in the proof of Proposition \ref{prop:full compactification}, we now have morphisms
\[
\mathcal{T}^\Pap_{V^\flat}(N) \to  \mathcal{S}^\Pap_{V^\flat }(N)_{/\co_\kk[ \frac{1}{Np} ]}
 \to \mathcal{M}^\Pap_{W^\flat}(N)_{/\co_\kk[ \frac{1}{Np} ]}
\to \bar{\mathcal{A}}_{/\co_\kk[ \frac{1}{Np} ]},
\]
and we may define $\bar{\mathcal{T}}^\Pap_{V^\flat}(N)$ as the normalization  of the composition. 
We may now apply the results of \cite{lan-ramified} and \cite{mp} to deduce properties of  $\bar{\mathcal{T}}^\Pap_{V^\flat}(N)$  from properties of its interior.  
Arguing exactly as  in the proof of Proposition \ref{prop:full compactification}, this compactification is normal and flat, and is a projective scheme if $N\ge 3$. 
Its boundary is a smooth Cartier divisor.  
The non-smooth locus has dimension $0$, and does not meet the boundary.  As $\bar{\mathcal{T}}_{V^\flat}(N)$ can be identified with the blow-up of $\bar{\mathcal{T}}^\Pap_{V^\flat}(N)$ along its nonsmooth locus, it has all the same properties, and is also regular (being smooth near the boundary and regular in the interior).
\end{proof}

Lemma \ref{lem:T compact} and Remark \ref{rem:special compact level} provide us with an  $\co_\kk[ \frac{1}{Np} ]$-stack
\[
 \bar{\mathcal{Z}}^\dagger_V(p,N) 
 \define  \bar{\mathcal{T}}_{V^\flat}(N) \sqcup 
  \bar{\mathcal{S}}_{V^\flat}(N)_{ / \co_\kk[\frac{1}{Np}] } \sqcup  \bar{\mathcal{S}}_{V^\flat}(N)_{/\co_\kk[ \frac{1}{Np} ] } ,
\]
 with all the nice properties enumerated in Proposition \ref{prop:full compactification}.  Thus  \eqref{dagger KR} has its own theory of pre-log singular hermitian line bundles and Burgos-Kramer-K\"uhn arithmetic Chow groups
\[
\widehat{\CH}^d (  \bar{\mathcal{Z}}^\dagger_V(p) , \mathscr{D}_\BKK )
= \mil_{ \substack{ N\ge 3 \\ p \nmid N} }  \widehat{\CH}^d (  \bar{\mathcal{Z}}^\dagger_V(p,N) , \mathscr{D}_\BKK ) , 
\]
 exactly as in \S \ref{ss:chow}.   These Chow groups include   notions of arithmetic heights and volumes as in  \S \ref{ss:volumes},  taking  values in the abelian group $\R/ \Q\log(p)$.

One can repeat the construction of the diagram \eqref{compact divisor correspondence} with level structures, and so obtain pullbacks
\[
\xymatrix{
 {  \widehat{\CH}^d(  \bar{\mathcal{S}}_{V^\flat} ,\mathscr{D}_\BKK  )   }   \ar[rr]^{\alpha^*} & &
  {   \widehat{\CH}^d(  \bar{\mathcal{Z}}^\dagger_V(p) ,\mathscr{D}_\BKK  )   }    & &
     {  \widehat{\CH}^d(  \bar{\mathcal{S}}_V ,\mathscr{D}_\BKK  )   }  \ar[ll]_{(\beta \circ i)^*} 
 }
\]
and
\begin{equation}\label{second hecke pic}
\xymatrix{
 {  \widehat{\Pic}(  \bar{\mathcal{S}}_{V^\flat} ,\mathscr{D}_\BKK  )   }   \ar[rr]^{\alpha^*} &  &  
 {   \widehat{\Pic}(  \bar{\mathcal{Z}}^\dagger_V(p) ,\mathscr{D}_\BKK  )   }  & & 
  {  \widehat{\Pic}(  \bar{\mathcal{S}}_V ,\mathscr{D}_\BKK  )   }  \ar[ll]_{(\beta \circ i)^*} 
   }
\end{equation}

\begin{lemma}\label{lem:formal height induction}
If two pre-log singular hermitian line bundles
\[
\widehat{\mathcal{P}}^\flat \in \widehat{\Pic}( \bar{\mathcal{S}}_{V^\flat} , \mathscr{D}_\BKK)
\quad \mbox{and}\quad
\widehat{\mathcal{P}} \in \widehat{\Pic}( \bar{\mathcal{S}}_{V} , \mathscr{D}_\BKK)
\]
have the same pullback to $ \widehat{\Pic}(  \bar{\mathcal{Z}}^\dagger_V(p) ,\mathscr{D}_\BKK  ) _\Q$   then 
\[
 \int_{ \mathcal{Z}_V(p) (\C)  }  \chern( \widehat{\mathcal{P}})^{n-2}  
   =  ( p^{n-1}+1)  \int_{  \mathcal{S}_{V^\flat} (\C)  }   \chern( \widehat{\mathcal{P}}^\flat  )^{n-2}.
\]
Moreover, there is an $a(p) \in \Q$ such that 
\[
\mathrm{ht}_{\widehat{\mathcal{P}}}  (\bar{\mathcal{Z}}_V(p))
=
(p^{n-1}+1 ) \cdot  \widehat{\vol}(   \widehat{\mathcal{P}}^\flat  )  +  \mathrm{ht}_{\widehat{\mathcal{P}}}  (E) + a(p) \log(p) ,
\]
where $E$ is the divisor of Lemma \ref{lem:exceptional KR error}.
\end{lemma}

\begin{proof}
As in the proof of Lemma \ref{lem:exceptional KR error}, the closed immersion $i$ in \eqref{divisor correspondence}
restricts to an isomorphism of non-exceptional loci, and hence induces an isomorphism in the generic fiber. 
Thus 
\begin{align*}
 \int_{ \mathcal{Z}_V(p) (\C)  } \beta^* \chern( \widehat{\mathcal{P}})^{n-2} 
 & = 
\int_{ \mathcal{Z}^\dagger_V(p) (\C)  }  (\beta \circ i)^*  \chern( \widehat{\mathcal{P}})^{n-2}  \\
&   =   \int_{ \mathcal{Z}^\dagger_V(p) (\C)  }   \alpha^* \chern( \widehat{\mathcal{P}}^\flat  )^{n-2} \\
& =  ( p^{n-1}+1)  \int_{  \mathcal{S}_{V^\flat} (\C)  }   \chern( \widehat{\mathcal{P}}^\flat  )^{n-2} .
\end{align*}
The second claim follows from Lemma \ref{lem:exceptional KR error} and the equalities
\[
\mathrm{ht}_{ \widehat{\mathcal{P}}}  (\bar{\mathcal{Z}}^\dagger_V(p))
=
\widehat{\vol}(   (\beta \circ i)^*\widehat{\mathcal{P}}  )
=
\widehat{\vol}(   \alpha^*\widehat{\mathcal{P}}^\flat  ) 
=
(p^{n-1}+1 ) \cdot  \widehat{\vol}(   \widehat{\mathcal{P}}^\flat  ) 
\]
in $\R/\Q\log(p)$, where the  first and last equalities are obtained directly by unpacking the definitions of pullbacks, heights, and arithmetic intersections in \cite{BKK}, and using the fact that $\alpha$ is (away from the boundary) a finite \'etale surjection of degree $p^{n-1}+1$.
%
%
%
%
\end{proof}

\begin{proof}[Proof of Theorems \ref{thm:height descent 1} and \ref{thm:height descent 2}]

The hermitian line bundles
\[
\widehat{\tautmod}_{V}    \in \widehat{\Pic}( \bar{\mathcal{S}}_{V} , \mathscr{D}_\BKK)
\subset   \widehat{\Pic}(\mathcal{S}_{V} )
\]
and
\[
\widehat{\tautmod}_{V^\flat}    \in \widehat{\Pic}( \bar{\mathcal{S}}_{V^\flat} , \mathscr{D}_\BKK)
\subset   \widehat{\Pic}(\mathcal{S}_{V^\flat} ) 
\]
of  Theorem \ref{thm:taut-hodge compare} have the same images, at least up to torsion,  under the pullback maps
\[
\xymatrix{
 {  \widehat{\Pic}(   \mathcal{S}_{V^\flat}   )   }   \ar[rr]^{\alpha^*} &  &  
 {   \widehat{\Pic}(  \mathcal{Z}^\dagger_V(p)  )   }  & & 
  {  \widehat{\Pic}(  \mathcal{S}_V   )   }  \ar[ll]_{(\beta \circ i)^*} 
   }
\]
induced by \eqref{divisor correspondence}. 
Indeed,  it  suffices to verify this on each of the open and closed substacks in the decomposition 
\[
\mathcal{Z}^\dagger_V(p)  = \mathcal{T}_{V^\flat} \sqcup  \mathcal{S}_{V^\flat / \co_\kk[1/p] } \sqcup  \mathcal{S}_{V^\flat /\co_\kk[1/p]}  .
\]
On the substack $ \mathcal{T}_{V^\flat}$ this is the content of Proposition \ref{prop:i_0 pullbacks}.  On the substacks $\mathcal{S}_{V^\flat / \co_\kk[1/p] }$ this is the content of 
Proposition \ref{prop:i_p pullbacks}.

It now follows from the injectivity of 
\[
\widehat{\Pic}(  \bar{\mathcal{Z}}^\dagger_V(p) ,\mathscr{D}_\BKK  )   
\to \widehat{\Pic}(  \mathcal{Z}^\dagger_V(p)   )  
\]
that these hermitian line bundles also have the same image, up to torsion,  under the maps of  \eqref{second hecke pic}.
In other words, the hypotheses of   Lemma \ref{lem:formal height induction} are satisfied by $\widehat{\tautmod}_{V^\flat}$ and $\widehat{\tautmod}_{V}$, proving    the equality
\[
 \int_{ \mathcal{Z}_V(p) (\C)  }  \chern( \widehat{\tautmod}_V)^{n-2}  
   =  ( p^{n-1}+1)  \int_{  \mathcal{S}_{V^\flat} (\C)  }   \chern( \widehat{\tautmod}_{V^\flat}  )^{n-2},
\]
and the equality
\[
 \mathrm{ht}_{ \widehat{\tautmod}_V }(\bar{\mathcal{Z}}_V(p)) 
 = 
( p^{n-1}+1 )  \widehat{\vol} (\widehat{\tautmod}_{V^\flat} ) 
+   \mathrm{ht}_{ \widehat{\tautmod}_V }(E)
\]
up to a rational multiple of $\log(p)$.    The second term on the right vanishes by claim (1) of Theorem \ref{thm:taut-hodge compare}, completing the proof of Theorem  \ref{thm:height descent 1}.

Similarly,  Propositions \ref{prop:i_p pullbacks} and \ref{prop:i_0 pullbacks} show  that  the hypotheses of Lemma \ref{lem:formal height induction} are satisfied by 
 \[
 \widehat{\mathcal{P}}^\flat = \widehat{\omega}^\mathrm{Hdg}_{A_0^\flat/ \mathcal{S}_{V^\flat}} 
  + \widehat{\omega}^\mathrm{Hdg}_{A^\flat/ \mathcal{S}_{V^\flat}} 
 \quad \mbox{and} \quad
  \widehat{\mathcal{P}} = \widehat{\omega}^\mathrm{Hdg}_{A/\mathcal{S}_V} ,
 \]
 and so
 \[
\mathrm{ht}_{ \widehat{\omega}^\mathrm{Hdg}_{A/\mathcal{S}_V}  }  (\bar{\mathcal{Z}}_V(p))
=
(p^{n-1}+1 ) \cdot  \widehat{\vol}(   \widehat{\omega}^\mathrm{Hdg}_{A_0^\flat/ \mathcal{S}_{V^\flat}} 
  + \widehat{\omega}^\mathrm{Hdg}_{A^\flat/ \mathcal{S}_{V^\flat}}    )  +  \mathrm{ht}_{ \widehat{\omega}^\mathrm{Hdg}_{A/\mathcal{S}_V}  }  (E) 
  \]
 up to a rational multiple of $\log(p)$.  Once again,  the second term on the right vanishes by claim (1) of Theorem \ref{thm:taut-hodge compare}.  
For the first term on the right,  Proposition \ref{prop:easy numerical} and  Lemmas  \ref{lem:numerical basics} and \ref{lem:trivial volume shift} imply
\begin{align*}
\widehat{\vol}(  \widehat{\omega}^\mathrm{Hdg}_{A_0^\flat/ \mathcal{S}_{V^\flat}} 
  + \widehat{\omega}^\mathrm{Hdg}_{A^\flat/ \mathcal{S}_{V^\flat}}    ) 
  & = 
  \widehat{\vol}(   (0 , C_1)  + 
  \widehat{\omega}^\mathrm{Hdg}_{A^\flat/ \mathcal{S}_{V^\flat}}   ) \\
&=   \widehat{\vol}(    \widehat{\omega}^\mathrm{Hdg}_{A^\flat/ \mathcal{S}_{V^\flat}}  )
+ (n-1) C_1 \int_{\mathcal{S}_{V^\flat}(\C) } \chern(  \widehat{\omega}^\mathrm{Hdg}_{A^\flat/ \mathcal{S}_{V^\flat}}  )^{n-2} ,
\end{align*}
where
\[
C_1 
= \log(2\pi) + 2h^\mathrm{Falt}_\kk   
=  -  \frac{L'(0,\eps)}{L(0,\eps)}   - \frac{\log(D)}{2}   .
\]
This proves Theorem \ref{thm:height descent 2}.
\end{proof}


\section{Comparing volumes}


Recall from Remark \ref{rem:taut-hodge volume} that when $V$ has signature $(n-1,1)$ with $n>2$ (an assumption in force throughout Chapter \ref{s:KR divisors})  there is an explicit relation between the arithmetic volumes of $\widehat{\omega}^\mathrm{Hdg}_{ A / \mathcal{S}_V}$  and 
$\widehat{\tautmod}_V$.

The restriction to $n>2$, which originated in Proposition \ref{prop:gross numerical}, can now be removed using 
Theorems \ref{thm:height descent 1} and  \ref{thm:height descent 2}.

\begin{proposition}\label{prop:taut-hodge strong volume}
The arithmetic volume relation
\[
 \widehat{\mathrm{vol}} (  \widehat{\tautmod}_V   )
= 
\widehat{\mathrm{vol}} (  \widehat{\omega}^\mathrm{Hdg}_{ A / \mathcal{S}_V}    ) 
  +  n C_0(n)   \vol_\C(   \widehat{\omega}^\mathrm{Hdg}_{ A / \mathcal{S}_V}    ),
\]
known by  Remark \ref{rem:taut-hodge volume} when $n>2$, also holds for $n=2$.
Here $C_0(n)$ is the constant of Theorem \ref{thm:intro main}.
\end{proposition}

\begin{proof}
Assume that $V$ has signature $(1,1)$, and  fix a prime $p$ split in $\kk$.

Let $V^\sharp$ be the $\kk$-hermitian space of signature $(2,1)$ whose local invariants satisfy
\[
\mathrm{inv}_\ell(V)  = ( p, -D)_\ell \cdot \mathrm{inv}_\ell(V^\sharp)  
\]
for all places $\ell \le \infty$,  so that the discussion of \S \ref{ss:split divisors main} applies  with the pair $(V^\sharp,V)$ replacing $(V,V^\flat)$.   Define  $H_p \in \R$ by the relation 
\[
(p + 1) H_p = 
 \mathrm{ht}_{ \widehat{\tautmod}_{V^\sharp} }  (\bar{\mathcal{Z}}_{V^\sharp}(p)) 
  -  \mathrm{ht}_{   \widehat{\omega}^\mathrm{Hdg}_{ A ^\sharp/ \mathcal{S}_{V^\sharp}}    } (\bar{\mathcal{Z}}_{V^\sharp}(p))  .
\]

On the one  hand, part (3) of Theorem \ref{thm:taut-hodge compare} and  Lemma \ref{lem:numerical basics}  imply 
\begin{equation}\label{some height}
\mathrm{ht}_{  \widehat{\tautmod}_{V^\sharp}   }  (\bar{\mathcal{Z}}_{V^\sharp}(p))    = 
\mathrm{ht}_{   \widehat{\omega}^\mathrm{Hdg}_{ A^\sharp / \mathcal{S}_{V^\sharp}} + (0,C_0(3))     } 
 (\bar{\mathcal{Z}}_{V^\sharp}(p))    .
\end{equation}
In the codimension $2$ arithmetic Chow group of $\bar{\mathcal{S}}_{V^\sharp}$ we have the relation
\[
\Big(\widehat{\omega}^\mathrm{Hdg}_{ A^\sharp / \mathcal{S}_{V^\sharp}}  + (0 , C_0 (3) ) \Big)^{2}
 = \Big( \widehat{\omega}^\mathrm{Hdg}_{ A^\sharp / \mathcal{S}_{V^\sharp}} \Big)^{2}  
 +      \Big( 0  ,  2 C_0(3)   \chern(  \widehat{\omega}^\mathrm{Hdg}_{ A^\sharp / \mathcal{S}_{V^\sharp}} ) \Big), 
\]
as in the proof of Lemma \ref{lem:trivial volume shift}. 
Directly from the definition of arithmetic height in \S \ref{ss:volumes},  it follows that  the right hand side of \eqref{some height} is equal to 
\begin{align*}
\mathrm{ht}_{   \widehat{\omega}^\mathrm{Hdg}_{ A^\sharp / \mathcal{S}_{V^\sharp}}      }  (\bar{\mathcal{Z}}_{V^\sharp}(p)) 
+ 2 C_0(3) \int_{\mathcal{Z}_{V^\sharp}(p)(\C) }    \chern(  \widehat{\omega}^\mathrm{Hdg}_{ A^\sharp / \mathcal{S}_{V^\sharp}} ) .
\end{align*}
Combining the  first claim of Theorem \ref{thm:height descent 1} with the equality of Chern forms of Theorem \ref{thm:taut-hodge compare},  we find  
 \[
\int_{\mathcal{Z}_{V^\sharp}(p)(\C) }    \chern(  \widehat{\omega}^\mathrm{Hdg}_{ A^\sharp / \mathcal{S}_{V^\sharp}} )
=
( p+1 ) \vol_\C(  \widehat{\omega}^\mathrm{Hdg}_{ A / \mathcal{S}_V} ) .
 \]
Putting all of this together allows us to rewrite  \eqref{some height} as
\[
\mathrm{ht}_{  \widehat{\tautmod}_{V^\sharp}   }  (\bar{\mathcal{Z}}_{V^\sharp}(p))    = 
\mathrm{ht}_{   \widehat{\omega}^\mathrm{Hdg}_{ A^\sharp / \mathcal{S}_{V^\sharp}}      }  (\bar{\mathcal{Z}}_{V^\sharp}(p)) 
+ 2 C_0(3) ( p+1 ) \vol_\C(  \widehat{\omega}^\mathrm{Hdg}_{ A / \mathcal{S}_V} ) ,
\]
which is equivalent to 
 \begin{equation}\label{second Hp}
H_p =   2 C_0(3)    \vol_\C(  \widehat{\omega}^\mathrm{Hdg}_{ A / \mathcal{S}_V} ) .
\end{equation}

On the other hand,  the height formulas of Theorems \ref{thm:height descent 1} and  \ref{thm:height descent 2}  imply 
\begin{align*}
H_p   
& =   \widehat{\vol} (\widehat{\tautmod}_{V} )  - \widehat{\vol} ( \widehat{\omega}^\mathrm{Hdg}_{A/\mathcal{S}_{V}} )  
+2 \left(  \frac{L'(0,\eps)}{L(0,\eps)}   +\frac{\log(D)}{2}  \right) 
\vol_\C(   \widehat{\omega}^\mathrm{Hdg}_{A / \mathcal{S}_{V}}  )   \\
& =  \widehat{\vol} (\widehat{\tautmod}_{V} )  - \widehat{\vol} ( \widehat{\omega}^\mathrm{Hdg}_{A/\mathcal{S}_{V}} )   +   2  (   C_0(3) -  C_0(2) ) \vol_\C(   \widehat{\omega}^\mathrm{Hdg}_{A/ \mathcal{S}_{V}}  )   
\end{align*}
up to a rational multiple of $\log(p)$, and combining this with \eqref{second Hp} proves
\[
  \widehat{\vol} (\widehat{\tautmod}_{V} )  -
   \widehat{\vol} ( \widehat{\omega}^\mathrm{Hdg}_{A/\mathcal{S}_{V}} )  
 - 2    C_0(2)  \vol_\C(  \widehat{\omega}^\mathrm{Hdg}_{A/ \mathcal{S}_{V}}  )
 \in \Q \log(p).
\] 
 The only way this  can hold for all primes $p$ split in $\kk$ is if this real number is $0$, completing the proof when $n=2$.
 \end{proof}


\chapter{The generating series of arithmetic divisors}
\label{s:borcherds}


We continue to work with the Shimura variety $\mathcal{S}_V$ of \eqref{moduli inclusion}
associated to a $\kk$-hermitian space $V$ of signature $(n-1,1)$ containing a self-dual $\co_\kk$-lattice.
Throughout this chapter we assume $n>2$.

After explaining the connection between the complex orbifold $\mathcal{S}_V(\C)$ and the Shimura variety \eqref{eq:X} associated to the unitary group $\Uni(V)$, we will use the results of \S \ref{ss:green functions}  to construct Green functions for  the Kudla-Rapoport divisors  $\mathcal{Z}_V(m) \to \mathcal{S}_V$.  We then recall from \cite{BHKRY-1} the modularity of a generating series of arithmetic divisors, in a slightly more refined form than it is stated there.


\section{Arithmetic Kudla-Rapoport divisors}
\label{ss:green construction}



We want to construct Green functions for the special divisors $\mathcal{Z}_V(m) \to \mathcal{S}_V$ using the machinery of regularized theta lifts from \S \ref{ss:green functions}.  In order to do this, we must make explicit the relation between the complex orbifold $\mathcal{S}_V(\C)$ and the  Shimura variety \eqref{eq:X}.

Recall from  \S \ref{ss:special shimura} that in the construction of $\mathcal{S}_V$ we fixed an isomorphism
\[
V\iso \Hom_\kk(W_0,W),
\]
where $W_0$ and $W$ are $\kk$-hermitian spaces of signatures $(1,0)$ and $(n-1,1)$, respectively.
The subgroup  $G\subset \mathrm{GU}(W_0) \times \mathrm{GU}(W)$ of  Remark \ref{rem:honest shimura} acts on both $W_0$ and $W$ via unitary similitudes.   The group $G$ also acts on $V$, defining a surjective homomorphism
\[
G \to H=\Uni(V),
\]
whose  kernel is the central $\mathrm{Res}_{\kk/\Q} \mathbb{G}_m$ in $G$.

Fix self-dual $\co_\kk$-lattices
$\mathfrak{a}_0 \subset W_0$ and $\mathfrak{a} \subset W$ as in \S \ref{ss:basic moduli}, and define a self-dual $\co_\kk$-lattice
\[
L = \Hom_{\co_\kk}(\mathfrak{a}_0 , \mathfrak{a}) \subset V.
\]
Denote by $K \subset G(\A_f)$ the largest compact open subgroup fixing the lattices $\mathfrak{a}_0$ and $\mathfrak{a}$.
By abuse of notation, denote in the same way  the image of $K$ under $G(\A_f) \to H(\A_f)$.
We now have a commutative diagram
\begin{equation}\label{cover}
\xymatrix{
{  \mathcal{S}_V(\C) }  \ar[r]  \ar@{=}[d] & {   \mathrm{Sh}_K(H,\mathcal{D})   } \ar@{=}[d]^{  \eqref{eq:X}  }  \\
 {G(\Q) \backslash \mathcal{D} \times G(\A_f) / K    }    \ar[r]   &    {  H(\Q)\backslash \mathcal{D}\times H(\A_f)/K  } ,
}
\end{equation}
in which the identification on the left is that of  \S 2 of \cite{BHKRY-1}.

 Let $H^\infty_{2-n}(D,\eps^n)$ denote the space of $\C$-valued harmonic
Maass forms $f$ of weight $2-n$, level  $\Gamma_0(D)$,  and character $\eps^n$ such that
\begin{itemize}
\item 
$f$ is bounded at all cusps of $\Gamma_0(D)$ different from the cusp $\infty$,
\item
$f$ has polynomial growth at $\infty$, in sense that there is a
\[
P_f = \sum_{m\leq 0} c^+(m)q^m \in \C[q^{-1}]
\]
such that $f-P_f=o(1)$ as $q$ goes to $0$.
\end{itemize}
Such a harmonic Maass form has a Fourier expansion analogous to \eqref{eq:fourierf} with Fourier coefficients $c^\pm(m)\in \C$.

Fix an  $f\in H_{2-n}^\infty(D,\eps^n)$ with Fourier coefficients $c^\pm(m)$.
This form  can be lifted to a vector valued harmonic Maass form, in the sense of \S \ref{ss:green functions},   by setting
\begin{align}
\label{eq:lift}
\tilde f=\sum_{\gamma\in \Gamma_0(D)\backslash \SL_2(\Z)} (f|_{2-n} \gamma) (\omega_L(\gamma)^{-1}\varphi_0)
\in H_{2-n}(\omega_L) ,
\end{align}
where $\varphi_0 \in S_L = \C[L'/L]$ is the characteristic function of $0 \in L'/L$.
We denote the Fourier coefficients of $\tilde f$ by $\tilde c^\pm(m,\mu)$ for $\mu\in L'/L$ and $m\in \Q$. 
The coefficients of $\tilde{f}$ can be computed in terms of the coefficients of $f$, and for $m<0$ we have 
\begin{align}
\label{eq:coeffrel}
\tilde c^{+}(m,\mu) =\begin{cases}
c^+(m) & \text{if $\mu=0$}\\
0&\text{if $\mu\neq 0$}
\end{cases}
\end{align}
as in  Proposition 6.1.2 of  \cite{BHKRY-1} or  \S 5 of \cite{Sch}.

Under the covering map \eqref{cover},  the divisors $Z(m)$ and the hermitian line bundle $\widehat{\taut}$ defined in \S \ref{ss:u(v) shimura} pull pack to the  divisors $\mathcal{Z}_V(m)$ and $\widehat{\taut}_V$ defined in \S \ref{ss:special shimura}.   
This allows us to  apply the construction \eqref{eq:AutoGreen} to the vector valued form \eqref{eq:lift} to obtain a Green function 
$\Phi(z,h,\tilde{f})$ for the analytic divisor 
\[
 \sum_{m>0} c^+(-m) Z(m)  \in \mathrm{Div}_\C(  \mathrm{Sh}_K(H,\mathcal{D})  ) 
\]
of \eqref{eq:zf}, and then pull it back  to  a Green function $\Phi_V(f)$ for the divisor
\[
 \sum_{m>0} c^+(-m) \mathcal{Z}_V(m) \in \mathrm{Div}_\C (  \mathcal{S}_V ) .
\]

We will apply this construction for particular choices of $f$.
As in \S 7.3 of \cite{BHKRY-1}, for every $m>0$ there is a unique $f_m(\tau) \in H^\infty_{2-n}(D,\eps^n)$ with $q$-expansion
\[
f_m(\tau) = q^{-m} + O(1).
\]
For any  $f\in H_{2-n}^\infty(D,\eps^n)$ with Fourier coefficients $c^\pm(m)$, we then have 
\[
f(\tau) = \sum_{m>0} c^+(-m)  f_m( \tau ).
\]
Applying the above constructions to these particular harmonic forms yield Green functions 
\[
\Phi_V(m) \define \Phi_V(f_m)
\]
for the divisors $\mathcal{Z}_V(m)$ on $\mathcal{S}_V$.

Recall  from Remark \ref{rem:divisor closure} that $\bar{\mathcal{Z}}_V(m)$ is the Zariski closure of $\mathcal{Z}_V(m)$ in the toroidal compactification $\bar{\mathcal{S}}_V$ of \S \ref{ss:special shimura}.
 The following proposition shows  that one can add boundary components to $\bar{\mathcal{Z}}_V(m)$ in order to obtain a class in the arithmetic Chow group (\S \ref{ss:chow}) of $\bar{\mathcal{S}}_V$.
 We will see in Theorem \ref{thm:no boundary heights} below that these boundary components are (for our purposes) negligible.

\begin{proposition}\label{prop:arithmetic divisor}
Fix an integer $m>0$, and  recall the divisor $\bar{\mathcal{Z}}_V(m)$ on $\bar{\mathcal{S}}_V$ from Remark \ref{rem:divisor closure}.  
There is a  unique divisor  $\mathcal{B}_V(m)$ on  $\bar{\mathcal{S}}_V$ that is supported on the boundary and satisfies
\[
\widehat{\mathcal{Z}}_V^\mathrm{tot} (m)  \define (  \bar{\mathcal{Z}}_V(m) + \mathcal{B}_V(m) , \Phi_V(m) )
 \in  \widehat{\CH}^1(\bar{\mathcal{S}}_V , \mathscr{D}_\BKK).
\]
\end{proposition}

\begin{proof}
This follows from the analysis of $\Phi_V(m)$ near the boundary of $\bar{\mathcal{S}}_V$ carried out in \cite{BHY};
see Theorem 7.2.3 of  \cite{BHKRY-1}.  
A precise description of the boundary divisor $\mathcal{B}_V(m)$ is given by  (5.3.3) of  \cite{BHKRY-1}. 
\end{proof}

\begin{remark}\label{rem:S_V green integral}
As we assume  $n>2$, Theorem \ref{thm:int} and \eqref{eq:coeffrel} imply the integral formula
\[
 \int_{  \mathcal{S}_V(\C)  } \Phi_V (m)   \chern( \widehat{\taut}_V )^{n-1}
 =  \vol_\C(\widehat{\taut}_V )  B'(m, 0 ,s_0),
\]
where $B'(m,0,s_0)$ is the derivative at $s_0=(n-1)/2$ of the Eisenstein series coefficient of \eqref{eq:coeff}, at $\mu=0$,  determined by the self-dual lattice $L\subset V$.
\end{remark}


\section{The modularity theorem}
\label{ss:BHKRYmodularity}


Proposition \ref{prop:arithmetic divisor} defines  arithmetic divisors
\[
\widehat{\mathcal{Z}}_V^\mathrm{tot} (m) \in  \widehat{\CH}^1( \bar{\mathcal{S}}_V ,\mathscr{D}_\BKK )
\]
 for all $m>0$.  We extend the definition to $m=0$ by setting
\[
\widehat{\mathcal{Z}}_V^\mathrm{tot} (0) =  \widehat{\taut}_V^{-1} + ( \mathrm{Exc}_V  , -\log(D)) ,
\]
where the second term is the exceptional divisor $\mathrm{Exc}_V \subset \bar{\mathcal{S}}_V$ of Definition \ref{def:special exceptional} endowed with the constant Green function $-\log(D)$.
The following is the main result of \cite{BHKRY-1}.

\begin{theorem}[Theorem 7.3.1 of \cite{BHKRY-1}] \label{thm:BHKRY}
The generating series 
\begin{equation}\label{BHKRYseries}
\widehat{\phi}_V(\tau) = 
\sum_{m\ge 0}  \widehat{\mathcal{Z}}_V^\mathrm{tot} (m)  \cdot  q^m  \in \widehat{\CH}^1( \bar{\mathcal{S}}_V ,\mathscr{D}_\BKK ) [[q]] 
\end{equation}
is the $q$-expansion of a modular form in $M_n(\Gamma_0(D) , \eps^n) \otimes \widehat{\CH}^1( \bar{\mathcal{S}}_V ,\mathscr{D}_\BKK )$.
\end{theorem}

A close examination of the proof of  Theorem  \ref{thm:BHKRY} reveals more information 
 about the modular form \eqref{BHKRYseries}.   It admits  a natural decomposition
\begin{equation}\label{BHKRY decomp}
\widehat{\phi}_V=\widehat{\phi}^\mathrm{eis}_V  +  \widehat{\phi}^\mathrm{exc}_V +  \widehat{\phi}^\mathrm{cusp}_V 
\end{equation}
as a sum of three forms.  This decomposition will play an important role in \S \ref{ss:arithmetic intersections}, and so the remainder of this section is devoted to describing the three terms on the right.

Recall from  Proposition \ref{prop:eisenstein formulas} the Eisenstein series 
\[
E_r(\tau) = \sum_{m\ge 0} e_r(m) q^m \in M_n(\Gamma_0(D) , \eps^n)
\]
indexed by a  positive divisor $r\mid D$.  
In Proposition \ref{prop:multi eisenstein} we defined, for every divisor $r \mid D$,  a constant 
$\gamma_r \in \{ \pm 1, \pm i\}$.  Using these,  define the \emph{Eisenstein part} of \eqref{BHKRYseries} by 
\[
\widehat{\phi}^\mathrm{eis}_V (\tau) 
= \sum_{r\mid D} \gamma_r E_r (\tau) \cdot \left( \widehat{\mathcal{Z}}_V^\mathrm{tot} (0)
 - \frac{1}{2} ( \mathrm{Exc}_V ,0) + \sum_{\ell\mid r} (  \bar{\mathcal{S}}_{V/ \F_\mathfrak{l}} ,0 )  \right) ,
\]
where  $\bar{\mathcal{S}}_{V/ \F_\mathfrak{l}}$ is the special fiber of $\bar{\mathcal{S}}_V$ at the unique prime  $\mathfrak{l} \subset \co_K$  above $\ell\mid r$.
We may view $\ell \in \kk^\times$ as a rational section of  the trivial  hermitian line bundle on $\bar{\mathcal{S}}_V$, and doing so shows that the class 
\[
( \mathrm{div}(\ell) , - \log |\ell|^2 ) \in \widehat{\Pic}( \overline{\mathcal{S}}_V)_\Q
\]
is trivial; here we are using the notation \eqref{constant metrics}.
 This implies $(  \bar{\mathcal{S}}_{V/ \F_\mathfrak{l}}   ,  0    ) =  (0  ,  \log( \ell )  )$ in the arithmetic Chow group, and hence 
\begin{equation}\label{eis part}
\widehat{\phi}^\mathrm{eis}_V (\tau) 
= \sum_{r\mid D} \gamma_r E_r (\tau)  \left(  \widehat{\taut}_V^{-1} 
 + \frac{1}{2} 
  \big( \mathrm{Exc}_V   ,  -  2 \log(D/r) \big) \right) .  
\end{equation}

The \emph{exceptional part} of \eqref{BHKRYseries}  is 
\begin{equation}\label{exceptional generating}
\widehat{\phi}^\mathrm{exc}_V  (\tau) = \frac{1}{2}  \sum_{E \subset \mathrm{Exc}_V} \vartheta_E(\tau) \cdot (E,0) ,
\end{equation}
where  the sum extends over the irreducible components $E \subset  \mathrm{Exc}_V$ of the exceptional divisor on $\mathcal{S}_V$, and  the scalar valued theta series $\vartheta_E(\tau)$   counts vectors in a positive definite rank $n$ hermitian lattice determined by the component $E$.  More precisely, recall that the top horizontal arrow in \eqref{singular point} collapses $E$ to a connected $0$-dimensional closed substack 
\[
e \subset  \mathcal{M}_{(1,0)} \times_{\co_\kk} \mathcal{M}^\Pap_{(n-1,1)} .
\]
The hermitian lattice associated to $E$ is then defined by
\[
L_E = \Hom_{\co_\kk}(A_{0,s} ,A_s) 
\]
for any geometric point $s \to e$, as in  \S 2.6 of \cite{BHKRY-1}, and the $m^\mathrm{th}$ coefficient in the $q$-expansion of $\vartheta_E(\tau)$ is the number of vector in $L_E$ of hermitian norm $m$.

Having defined all but one of the four modular forms appearing in \eqref{BHKRY decomp}, we complete the definitions by setting 
\begin{equation}\label{cuspidal generating}
\widehat{\phi}^\mathrm{cusp}_V
 = \widehat{\phi}_V-\widehat{\phi}^\mathrm{eis}_V  -  \widehat{\phi}^\mathrm{exc}_V .
\end{equation}
The following result justifies calling this the \emph{cuspidal part} of \eqref{BHKRYseries}.

\begin{proposition}
The generating series \eqref{cuspidal generating} is cuspidal, in the sense that it lies in the subspace
\[
S_n(\Gamma_0(D) , \eps^n) \otimes \widehat{\CH}^1( \bar{\mathcal{S}}_V ,\mathscr{D}_\BKK )
\subset
M_n(\Gamma_0(D) , \eps^n) \otimes \widehat{\CH}^1( \bar{\mathcal{S}}_V ,\mathscr{D}_\BKK ).
\]
\end{proposition}

\begin{proof}
The proposition is implicit in the proof of Theorem 7.3.1 of \cite{BHKRY-1}.  The essential point is that Theorem 4.2.3 of \cite{BHKRY-1},  a variant of a result of Borcherds, gives a criterion for determining not just when a formal $q$-expansion is modular, but when it lies in the subspace 
\[
M^\infty_n(\Gamma_0(D) , \eps^n)  \subset M_n(\Gamma_0(D) , \eps^n) 
\]
of forms that vanish at every cusp except possibly  $\infty$.

In the proof of  Theorem 7.3.1 of \cite{BHKRY-1}, this criterion is used to show that the formal generating series
\[
\widehat{\phi}_V - \widehat{\phi}^\mathrm{exc}_V
- 
 \sum_{   \substack{   r\mid D  \\ r >1 } } 
  \gamma_r E_r (\tau) \cdot \left( \widehat{\mathcal{Z}}_V^\mathrm{tot} (0)
 - \frac{1}{2} ( \mathrm{Exc}_V ,0) + \sum_{\ell\mid r} (  \bar{\mathcal{S}}_{V/ \F_\mathfrak{l}} ,0 )  \right) 
 \]
 lies in 
 \[
 M^\infty_n(\Gamma_0(D) , \eps^n)  \otimes \widehat{\CH}^1( \bar{\mathcal{S}}_V ,\mathscr{D}_\BKK ) 
 \]
(which immediately implies that $\widehat{\phi}_V$ is modular).
 This is spelled out more directly in  the proof of Theorem 7.1.5 of \cite{BHKRY-1}, but for the  generating series,  valued in the usual Chow group of $\bar{\mathcal{S}}_V$,  obtained from $\widehat{\phi}_V$  by forgetting the Green functions.

In other words,
\[
 \widehat{\phi}^\mathrm{cusp}_V (\tau)   +  
E_1 (\tau) \cdot \left( \widehat{\mathcal{Z}}_V^\mathrm{tot} (0)
 - \frac{1}{2} ( \mathrm{Exc}_V ,0)  \right)
\]
 vanishes at every cusp except possibly $\infty$.  
 The Eisenstein series   $E_1(\tau)$  has this same property (Remark \ref{rem:each cusp}), and therefore so also does $\widehat{\phi}^\mathrm{cusp}_V (\tau)$.
On the other hand, one can check directly (using  Proposition \ref{prop:eisenstein formulas} for the Eisenstein part) that the constant term in the  $q$-expansion of \eqref{cuspidal generating} also vanishes, proving its cuspidality.
\end{proof}


\section{Heights of arithmetic divisors}


Recall from \eqref{height degree} that if the divisor $\mathcal{Z}$ underlying 
\[
\widehat{\mathcal{Z}} = ( \mathcal{Z} , g_\mathcal{Z} ) \in \widehat{\CH}^1( \bar{\mathcal{S}}_V, \mathscr{D}_\BKK) 
\]
meets the boundary of $\bar{\mathcal{S}}_V$ properly in the generic fiber, then 
\[
\widehat{\deg}(   \widehat{\mathcal{Z}}  \cdot  \widehat{\taut}_V^{n-1}   )
= \mathrm{ht}_{\widehat{\taut}_V }(\mathcal{Z}) + \int_{ \mathcal{S}_V(\C) }  g_\mathcal{Z} \cdot \chern(\widehat{\taut}_V )^{n-1} .
\]
We cannot apply this directly to the arithmetic divisor of Proposition \ref{prop:arithmetic divisor}, because the underlying divisor  $\bar{\mathcal{Z}}_V(m)+\mathcal{B}_V(m)$ typically contains boundary components with nonzero multiplicities.
The following result, as in  \S 3.2 of \cite{BBK} or  \S 3.4 of \cite{hormann},  shows that even when $\mathcal{B}_V(m) \neq 0$ these  boundary components are negligible for our purposes.

\begin{theorem}\label{thm:no boundary heights}
The arithmetic divisor of Proposition \ref{prop:arithmetic divisor} satisfies
\[
\widehat{\deg} \big( \widehat{\mathcal{Z}}_V^\mathrm{tot} (m)  \cdot \widehat{\taut}_V^{n-1}  \big)
=
\mathrm{ht}_{\widehat{\taut}_V}  (  \bar{\mathcal{Z}}_V (m) ) 
+ \int_{\mathcal{S}_V(\C)}  \Phi_V(m)  \chern( \widehat{\taut}_V)^{n-1}   .
\]
The same relation holds if $\widehat{\taut}_V$ is replaced by $\widehat{\omega}^\mathrm{Hdg}_{A/\mathcal{S}_V}$ or the hermitian line bundle $\widehat{\tautmod}_V$ of Theorem \ref{thm:taut-hodge compare}.
\end{theorem}

The proof of Theorem \ref{thm:no boundary heights}  requires quite a bit of preparation, and will occupy the remainder of this section.

Fix an $N \ge 3$, and recall from \S \ref{ss:special shimura} the finite morphism of $\co_\kk[1/N]$-schemes 
\[
\bar{\mathcal{S}}_V(N) \to \bar{\mathcal{S}}_{V/\co_\kk[1/N]} 
\]
obtained by adding level structure and compactifying.
To ease notation, the pullback of the special divisor $\bar{\mathcal{Z}}_V(m)$ via this map is denoted the same way, and similarly for $\mathcal{L}_V$ and $\omega^\mathrm{Hdg}_{A/\mathcal{S}_V}$.

\begin{proposition}\label{prop:BBsections}
There is a positive integer $k$ and sections
\[
f_1,\dots,f_{n-1} \in H^0 \big(  \bar{\mathcal{S}}_V(N) , \omega^{\mathrm{Hdg}, \otimes k}_{A/\mathcal{S}_V} \big)
\]
with the following properties:
\begin{enumerate}
\item
All possible intersections of 
\[
\bar{\mathcal{Z}}_V(m), \dv(f_1),\dots , \dv(f_{n-1})
\]
are proper.
\item  In the generic fiber, the divisor $\dv(f_1)$ is disjoint from the boundary. 
\item The codimension 2 cycle
$
\dv(f_1)\cap \dv(f_2)
$
is disjoint from the boundary.
\end{enumerate}
\end{proposition}

\begin{proof}
This uses the theory of Baily-Borel (or minimal) compactifications of integral models developed in \S 5 of \cite{mp}.   As noted after Lemma \ref{lem:normalized compactification}, the compactifications of integral models of Hodge type Shimura varieties found in \cite{mp} are not obtained by imitating the constructions of Faltings-Chai in the Siegel case.  
Rather, they are defined by normalizing the natural morphisms from integral models of open Shimura varieties to Faltings-Chai compactifications of Siegel moduli spaces.  Because of Lemma \ref{lem:normalized compactification}, all of the results proved in \cite{mp} can be applied directly to  the integral models    
\[
\mathcal{M}^\Pap_W(N) \subset \bar{\mathcal{M}}^\Pap_W(N)
\]
used  in the proof of Proposition \ref{prop:full compactification}.

As in \S 5.2.1 of \cite{mp}, the universal abelian scheme $A \to \mathcal{M}^\Pap_W(N)$ extends uniquely to a semi-abelian scheme over its toroidal compactification, and the inverse of the determinant of its Lie algebra defines an extension of the Hodge bundle 
\[
\omega \define \omega^{\mathrm{Hdg}}_{A/ \mathcal{M}^\Pap_W(N)}
\]
 to a line bundle on the toroidal compactification.  
 This allows us to define the Baily-Borel compactification
\[
\mathcal{M}^\Pap_W(N)^{\mathrm{min}}
= \mathrm{Proj} \Big(  
\bigoplus_{r \ge 0} H^0  (  \bar{\mathcal{M}}^\Pap_W(N)   ,
  \omega^{\otimes r}   )
\Big) .
\]

Theorem 5.2.11 of \cite{mp} provides us with a morphism 
\[
\bar{\mathcal{M}}^\Pap_W(N) \to 
\mathcal{M}^\Pap_W(N)^{\mathrm{min}}
\]
of projective $\co_\kk[1/N]$-schemes such that the composition 
\[
\mathcal{M}^\Pap_W(N) \to
\bar{\mathcal{M}}^\Pap_W(N) \to 
\mathcal{M}^\Pap_W(N)^{\mathrm{min}}
\]
is an open immersion whose complement (the Baily-Borel boundary) is  flat over $\co_\kk[1/N]$.
In particular, as the classical  Baily-Borel compactification of the complex fiber of $\mathcal{M}^\Pap_W(N)$ has dimension $0$, the boundary of the  integral Baily-Borel compactification has dimension $1$.

Also by Theorem 5.2.11 of \cite{mp},  the Hodge bundle $\omega$ on $\bar{\mathcal{M}}_W^\Pap(N)$ is the pullback of a uniquely determined ample line bundle on the Baily-Borel compactification (denoted the same way), whose pullback via the composition
\[
\bar{\mathcal{S}}_V(N) \map{\eqref{compact inclusion}} 
 \bar{\mathcal{M}}_W(N)
 \to
 \bar{\mathcal{M}}^\Pap_W(N)
 \to 
\mathcal{M}^\Pap_W(N)^{\mathrm{min}}
\]
agrees with  $\omega^\mathrm{Hdg}_{A/\mathcal{S}_V}$.

Fix a positive integer $k$ large enough that $\omega^k$ defines a closed immersion of $\mathcal{S}^\Pap_W(N)^\mathrm{min}$ into projective space, and choose a hyperplane section 
\[
s_1  \in H^0 (  \mathcal{M}^\Pap_W(N)^\mathrm{min} , \omega^{\otimes k} )
\] 
whose divisor does not contain any of the (finitely many) boundary points of the generic fiber of $ \mathcal{S}^\Pap_W(N)^\mathrm{min}$, 
and does not contain the image of any generic point of the divisor 
$
\bar{\mathcal{Z}}_V(m) \subset \bar{\mathcal{S}}_V(N).
$
By construction, the divisor of the pullback 
\[
f_1 \in H^0 ( \bar{\mathcal{S}}_V(N) , \omega^{\mathrm{Hdg}, \otimes k }_{A/\mathcal{S}_V} ) 
\]
of $s_1$ does not meet the toroidal boundary in the generic fiber, and intersects $\bar{\mathcal{Z}}_V(m)$ properly.

Now choose a hyperplane section 
\[
s_2  \in H^0 (  \mathcal{M}^\Pap_W(N)^\mathrm{min} , \omega^{\otimes k} )
\] 
whose divisor does not contain any point of the intersection $\mathrm{div}(f_1)$ with the Baily-Borel boundary (this intersection is $0$-dimensional, because the boundary is $1$-dimensional), and does not contain the image of any generic point of  $\mathrm{div}(f_1)$, $\bar{\mathcal{Z}}_V(m)$, or their intersection.  Now define
\[
f_2 \in H^0 ( \bar{\mathcal{S}}_V(N) , \omega^{\mathrm{Hdg}, \otimes k}_{A/\mathcal{S}_V} ) 
\]
as the pullback of $s_2$, and continue  in this fashion.
\end{proof}

\begin{corollary}\label{cor:SBBsections}
There is a positive integer $k$ and sections
\[
F_1,\dots,F_{n-1} \in H^0 \big(  \bar{\mathcal{S}}_V(N) ,  \taut_V ^{\otimes k} \big)
\]
with the following properties:
\begin{enumerate}
\item
On the complement of the exceptional divisor
\[
\mathrm{Exc}_V(N) \define \bar{\mathcal{S}}_V(N) \times_{\bar{\mathcal{S}}_V} \mathrm{Exc}_V \subset\bar{\mathcal{S}}_V(N),
\]
  all possible intersections of 
\[
\bar{\mathcal{Z}}_V(m), \dv(F_1),\dots , \dv(F_{n-1})
\]
 are proper.  In particular, this   holds on  the generic fiber of $\bar{\mathcal{S}}_V(N)$.
 \item  
 In the generic fiber, the divisor $\dv(F_1)$ is disjoint from the boundary. 
\item 
The codimension 2 cycle
$
\dv(F_1)\cap \dv(F_2)
$
is disjoint from the boundary.
\end{enumerate}
\end{corollary}

\begin{proof}
This is clear from Propositions \ref{prop:simple taut-hodge} and  \ref{prop:BBsections}, and the fact that the exceptional divisor $\mathrm{Exc}_V \subset \bar{\mathcal{S}}_V$ has empty generic fiber and does not intersect the boundary.
\end{proof}

\begin{proof}[Proof of Theorem \ref{thm:no boundary heights}] 
Fix sections $F_1,\ldots, F_{n-1}$ as in Corollary \ref{cor:SBBsections}.  Each determines an arithmetic divisor
\[
(\mathcal{Z}_i , g_i)  \define ( \dv(F_i) , - \log \| F_i\|^2 )
\]
representing  $k\cdot \widehat{\taut}_V \in \widehat{\CH}^1( \bar{\mathcal{S}}_V(N), \mathscr{D}_\BKK)$, and satisifying the Green equation
\[
dd^c [g_i] + \delta_{\mathcal{Z}_i} = k  [  \chern( \widehat{\taut}_V)].
\]
As the divisors $\mathcal{Z}_1,\ldots, \mathcal{Z}_{n-1}$ intersect properly in the generic fiber, a direct examination of the construction of the arithmetic intersection pairing in \S 4.3 of \cite{BKK} shows that the iterated intersection
\[
k^{n-1} \widehat{\taut}_V^{n-1} = (\mathcal{Z}_1 , g_1) \cdots (\mathcal{Z}_{n-1} , g_{n-1}) \in  \widehat{\CH}^{n-1}( \bar{\mathcal{S}}_V (N), \mathscr{D}_\BKK)
\]
is represented by a codimension $n-1$ arithmetic cycle 
\[
 \widehat{\mathcal{Y}} = (\mathcal{Y} , g_\mathcal{Y}),
\]
 in which the support of $\mathcal{Y}$ is contained in the naive scheme-theoretic intersection $\mathcal{Z}_1 \cap \cdots \cap \mathcal{Z}_{n-1}$, and
\[
g_\mathcal{Y} = g_1* \cdots * g_{n-1}
\] 
is the iterated star product.
(Really, this means the iterated star product of the pre-log-log Green objects
\[
\frakg_i= (k\chern( \widehat{\taut}_V), g_i) 
\in \widehat{H}^2_{\calD_{\mathrm{pre}}, \mathcal{Z}_i}(\bar\calS_V(N)(\C),1),
\]
in the sense of \cite{BBK,BKK}, but we continue to suppress this extra notational complexity.)

Our choice of sections $F_i$ guarantees that   $\mathcal{Y}$ does not intersect the boundary of $\bar{\mathcal{S}}_V(N)$, and does not intersect  $\bar{\mathcal{Z}}_V(m)$  in the generic fiber.
Thus 
\begin{align}& 
k^{n-1}\widehat{\deg}_N \big( \widehat{\mathcal{Z}}_V^\mathrm{tot} (m)  \cdot \widehat{\taut}_V^{n-1}  \big)  \nonumber  \\
& = \widehat{\deg}_N \big( \widehat{\mathcal{Z}}_V^\mathrm{tot} (m)  \cdot \widehat \calY\big) \nonumber  \\
&=  \sum_{      \substack{ \mathfrak{p} \subset \co_\kk         \\   \mathfrak{p} \nmid N \co_\kk   } }   
\deg_\mathfrak{p} (\bar{\mathcal{Z}}_V (m) \cdot \calY) \cdot \log (\# \co_\kk/ \mathfrak{p})  
 + \int_{\bar\calS_V(N)(\C)} \Phi_V(m)  * g_\calY. \label{no boundary key}
\end{align}
Here we are working in the arithmetic Chow group of the $\co_\kk[1/N]$-scheme $\bar{\mathcal{S}}_V(N)$, as in \S \ref{ss:chow}, so these equalities are understood up to a $\Q$-linear combination of $\{ \log(p) : p\mid N\}$.   The degree $\deg_\mathfrak{p}$ is that of \eqref{local degree}, applied to the $0$-cycle
\[
\bar{\mathcal{Z}}_V (m) \cdot \calY \in \mathrm{CH}^n_{ \bar{\mathcal{Z}}_V(m) \cap \mathcal{Y} } ( \bar{\mathcal{S}}_V(N) )
\]
as in the discussion of \S \ref{ss:volumes}. 

The integral on the right hand side of \eqref{no boundary key} is examined in Proposition~\ref{prop:star} below.  Substituting the expression found there into \eqref{no boundary key}, and comparing with \eqref{height normalization}, we find the equality stated in Theorem \ref{thm:no boundary heights} holds up to a $\Q$-linear combination of $\{ \log(p) : p\mid N\}$.  The theorem follows by allowing $N\ge 3$ to vary.
\end{proof}

It remains to prove the formula for the integral in \eqref{no boundary key} used above.

\begin{proposition}
\label{prop:star}
Continuing with the notation in the above proof, we have
\[
\int_{\bar\calS_V(N)(\C)} \Phi_V(m) * g_\calY = k^{n-1} \int_{\mathcal{S}_V(N)(\C)}  \Phi_V(m)  \chern( \widehat{\taut}_V)^{n-1} + \int_{\mathcal{Z}_V(m)(\C)} g_\calY.
\]
Here, continuing with our abuse of notation, the final integral is really over the pullback of $\mathcal{Z}_V(m)(\C)$ via
$\mathcal{S}_V(N)(\C) \to \mathcal{S}_V(\C)$.
\end{proposition}

\begin{proof}
We argue  as in the proof of Theorem 3.3 of \cite{BBK}.   See also Theorem 3.4.3 of \cite{hormann}. 
Abbreviate 
\[
\Omega=\chern( \widehat{\taut}_V),
\]
and denote by  $D=\partial \bar{\mathcal{S}}_V(N)(\C)$ the toroidal boundary divisor. 
For $\eps>0$ denote by $B_\eps(D)$ an $\eps$-neighborhood of $D$. 

As in the proof above, the iterated arithmetic intersection 
\[
 \widehat{\mathcal{Z}} \define (\mathcal{Z}_2 , g_2) \cdots (\mathcal{Z}_{n-1} , g_{n-1}) \in  \widehat{\CH}^{n-2}( \bar{\mathcal{S}}_V (N), \mathscr{D}_\BKK)
\]
is represented by a codimension $n-2$ arithmetic cycle 
\[
 \widehat{\mathcal{Z}} = (\mathcal{Z} , g_\mathcal{Z}),
\]
 in which the support of $\mathcal{Z}$ is contained in the naive scheme-theoretic intersection $\mathcal{Z}_2 \cap \cdots \cap \mathcal{Z}_{n-1}$, and the iterated star product
\[
g_\mathcal{Z} = g_2* \cdots * g_{n-1}
\] 
is a $(n-3,n-3)$-form with a logarithmic singularity along $\calZ(\C)$ and a pre-log-log singularity along $D$, satisfying the Green equation
\[
dd^c [ g_\mathcal{Z} ] + \delta_\mathcal{Z} = k^{n-2} [ \Omega^{n-2} ].
\]

Recall from \cite[Section 1.3]{BBK} that the star product 
\[
g_\calY =g_1 * g_\calZ
\]
can be represented as follows: Arguing as in \cite[p.~362]{Bu2} there exists an embedded resolution of singularities of $\mathcal{Z}_1(\C)\cup \calZ(\C)$,  denoted 
\[
\pi: \tilde X\to \bar\calS_V(N)(\C),
\]
which factors through embedded resolutions of $\mathcal{Z}_1(\C)$, $\calZ(\C)$, and $\mathcal{Z}_1(\C)\cup \calZ(\C)$, and such that the strict transforms of $\mathcal{Z}_1(\C)$ and $\calZ(\C)$ are disjoint. By the above assumptions on the $F_i$, the strict transform of $\mathcal{Z}_1(\C)$ is also disjoint from $\pi^{-1}(D)$. 
Let $\sigma_\calZ,1-\sigma_\calZ$ be a partition of the unity on $\tilde X$ for which $\sigma_\calZ =1$ in a neighborhood of the strict transform of $\calZ(\C)\cup D$ and $\sigma_\calZ =0$ in a neighborhood of the strict transform  of $\mathcal{Z}_1(\C)$.
Then, according to identity (1.5) of  \cite{BBK}, 
\[
g_\mathcal{Y} = 
k^{n-2}  \sigma_\calZ g_1\wedge \Omega^{n-2} + (dd^c(1-\sigma_\calZ)g_1)\wedge g_\calZ.
 \]
In particular, if $\eps$ is sufficiently small then on $B_\eps(D)$ we have
\begin{align}
\label{eq:gbound}
g_\calY = 
k^{n-2}g_1\wedge \Omega^{n-2}.
\end{align}

We may now apply 
\cite[Theorem 1.14]{BBK} with $y=\calZ_V^\mathrm{tot}(m)(\C)$, $g_y=\Phi_V(m)$, $z=\calY$, and $g_z= g_\calY$
to see that 
\begin{align*}
\int_{\bar\calS_V(N)(\C)} &\Phi_V(m) * g_\calY 
= 
\int_{\calZ_V(m)(\C)} g_\calY  + 
 k^{n-1} \int_{\bar{\mathcal{S}}_V(N)(\C)}  \Phi_V(m)  \Omega^{n-1} \\
&-\lim_{\eps\to 0} \int_{\partial B_\eps(D)} (g_\calY\wedge d^c  \Phi_V(m)- \Phi_V(m)\wedge d^c g_\calY)  .
\end{align*}
By \eqref{eq:gbound}  the limit on the right hand side is equal to 
\[
k^{n-2}\lim_{\eps\to 0}\bigg(
\int_{\partial B_\eps(D)} g_1\wedge \Omega^{n-2}\wedge d^c  \Phi_V(m)- \Phi_V(m)\wedge d^c g_1\wedge \Omega^{n-2}\bigg) ,
\]
which vanishes by Lemma \ref{lem:bound} below.
\end{proof}

\begin{lemma}
\label{lem:bound}
If  $F_1$ is  any meromorphic section of the line bundle $\taut_V^{\otimes k}$ on $\bar{\mathcal{S}}_V(N)(\C)$,  whose divisor is disjoint from the boundary, then
\begin{align*}
\lim_{\eps\to 0} &\int_{\partial B_\eps(D)} \log\|F_1\|\wedge \Omega^{n-2}\wedge d^c  \Phi_V(m)= 0, \\
\lim_{\eps\to 0} &\int_{\partial B_\eps(D)} \Phi_V(m)\wedge d^c \log\|F_1\|\wedge \Omega^{n-2}=0.
\end{align*}
\end{lemma}

Before we prove this lemma, we need to recall from  Section~4.3 of \cite{BHY} the complex analytic description of the boundary of $\bar \calS_V(\C)$.
Using \eqref{cover}, it suffices to do this for the boundary of  $\mathrm{Sh}_K(H,\mathcal{D})$. By \eqref{X components} the variety  $\mathrm{Sh}_K(H,\mathcal{D})$ is a finite disjoint union of connected Shimura varieties of the form $X_\Gamma=\Gamma\bs \calD$, where $\Gamma\subset H=\Uni(V)$ is an arithmetic subgroup stabilizing a self-dual hermitian lattice in $V$. 
Hence it suffices to consider the boundary of the canonical toroidal compactification $\bar X_\Gamma$ of $X_\Gamma$. 

The boundary of $\bar X_\Gamma$ is a smooth divisor. Its connected components correspond to $\Gamma$-orbits of isotropic lines in $ V$.
Fix an isotropic line $I\subset V$. Let $\ell\in I$ be a generator, and let
 $\widetilde\ell\in V$ be isotropic such that $\langle \ell,\widetilde\ell\rangle=1$.
The orthogonal complement
$W= \ell^\perp \cap \widetilde\ell{}^\perp$
is a positive definite hermitian space over $\kk$ of dimension $n-2$, and we have 
\[
V=W\oplus \kk\ell\oplus \kk\widetilde\ell.
\]

Every  $z\in \calD$, viewed as a negative definite $\C$-line in $V(\R)$ as in  \eqref{hermitian model}, 
has a unique basis vector of the form
\[
\frakz+\tau \sqrt{-D}\ell + \widetilde\ell
\]
with $\frakz\in W(\R)$ and $\tau\in \C$.
 The condition that $z$ is negative definite  is equivalent to the real number
\[
\norm(\frakz,\tau) \define  -\langle(\frakz,\tau),(\frakz,\tau) \rangle
 =2\sqrt{D} \cdot \mathrm{Im} (\tau)-\langle \frakz,\frakz\rangle
\]
being positive. Consequently, $\calD$ is isomorphic to
\[
\calH_{\ell,\widetilde\ell}
\define 
\{  (\frakz,\tau)\in W(\R) \times \H : \; 2\sqrt{D} \cdot \mathrm{Im}
(\tau) >\langle \frakz,\frakz\rangle
\}.
\]
The Chern form $\Omega$ is given in these coordinates by 
\begin{align}
\label{eq:Om}
\Omega=-dd^c\log \norm(\frakz,\tau).
\end{align}

For $\eps> 0$ we define
\begin{align*}
U_\eps(\ell) &= \left\{z\in \calD:\; -\frac{\langle z,z\rangle}{|\langle z,\ell\rangle|^2}>\frac{1}{\eps} \right\} \\
&\cong \{(\frakz,\tau)\in  \calH_{\ell,\widetilde\ell}: \; \norm(\frakz,\tau)>1/\eps\}.
\end{align*}
The stabilizer $\Uni(V)_\ell$ of $I=\kk\ell$ acts on this subset.
Let $\Gamma_\ell=\Gamma\cap\Uni(V)_\ell$.
If $\eps $ is sufficiently small,
then
\[
\Gamma_\ell \bs U_\eps(\ell) \longrightarrow X_\Gamma
\]
is an open immersion.
The center of $\Uni(V)_\ell$ 
is given by the subgroup of translations $T_a$ for $a\in \Q$, where
\[
T_a(x)=x +a\langle x,\ell\rangle \sqrt{-D}\ell
\]
for $x\in V$. The action of these translations on $\calH_{\ell,\widetilde\ell}$ is given by
$T_a(\frakz,\tau)= (\frakz,\tau+a).
$
The center of $\Gamma_\ell$
is of the form
\[
\Gamma_{\ell,T}= \{ T_a: a\in r\Z\}
\]
for a unique $r\in \Q_{>0}$.
If we put $q_r=e^{2\pi i \tau/r}$, then $(\frakz,\tau)\mapsto (\frakz,q_r)$ defines an isomorphism from $\Gamma_{\ell,T}\bs U_\eps(\ell)$ to
\[
V_\eps(\ell)=\left\{ (\frakz,q_r)\in \C^{n-2}\times \C: \; 0<|q_r|<\exp\left(-\tfrac{\pi}{r\sqrt{D}}(\langle\frakz,\frakz\rangle+1/\eps)\right)\right\}.
\]
Adding the origin to this punctured disc bundle yields the disc bundle
\[
\widetilde V_\eps(\ell)= \left\{ (\frakz,q_r)\in \C^{n-2}\times \C : \; |q_r|<\exp\left(-\tfrac{\pi}{r\sqrt{D}}(\langle\frakz,\frakz\rangle+1/\eps)\right)\right\}.
\]
The action of $\Gamma_\ell/\Gamma_{\ell,T}$ on $V_\eps(\ell)$ extends to a properly discontinuous action on $\widetilde V_\eps(\ell)$, which leaves the boundary divisor $q_r=0$ invariant.
We obtain an open immersion of orbifolds
\[
\Gamma_\ell\bs U_\eps(\ell) \longrightarrow \left(\Gamma_\ell/\Gamma_{\ell,T}\right)\bs \widetilde V_\eps(\ell).
\]
It can be used to glue the right hand side to $X_\Gamma$ to obtain a partial compactification near the cusp $I$. The family of these open immersions for small $\eps>0$ defines a base of open neighborhoods of the boundary divisor $D_I$ associated with $I$.

Any boundary point $z\in D_I$ has a representative of the form $(\frakz_0,0)\in \widetilde V_{\eps}(\ell)$. The images of the sets 
\begin{align}
\label{boundary-nbhd}
B_{\delta}(\frakz_0,0)= \left\{ (\frakz,q_r)\in \widetilde V_\eps(\ell):  \langle \frakz-\frakz_0,  \frakz-\frakz_0\rangle <\delta,\;|q_r|<\delta\right\}.
\end{align}
in $\bar X_\Gamma$ for small $\delta>0$ define a base of open neighborhoods of the boundary  point.

\begin{proof}[Proof of Lemma \ref{lem:bound}]
In view of the above local description of the boundary, it suffices to show that for every isotropic line $I\subset V$ and for a small neighborhood  
\[
Y=B_{\delta}(\frakz_0,0)
\]
of any boundary point $(\frakz_0,0)$ of $D_I$ we have 
\begin{align}
\label{eq:bound1}
\lim_{\eps\to 0} &\int_{\partial B_\eps(D)\cap Y} \log\|F_1\|\wedge \Omega^{n-2}\wedge d^c  \Phi_V(m)= 0, \\
\label{eq:bound2}
\lim_{\eps\to 0} &\int_{\partial B_\eps(D)\cap Y} \Phi_V(m)\wedge d^c \log\|F_1\|\wedge \Omega^{n-2}=0.
\end{align}
Using the coordinates of \eqref{boundary-nbhd}, we have 
\[
\norm(\frakz,\tau)= -\frac{\sqrt{D} r}{\pi }\log|q_r| -\langle\frakz,\frakz\rangle.
\]
By \eqref{eq:Om} we find
\begin{align*}
\Omega 
&=  \frac{1}{\norm(\frakz,\tau)^2} (d\norm(\frakz,\tau))\wedge (d^c\norm(\frakz,\tau))
+\frac{1}{\norm(\frakz,\tau)} dd^c\langle\frakz,\frakz\rangle .
\end{align*}
If we write in addition 
$q_r=e^{2\pi i \tau/r}=t\cdot e^{i\rho}$ in polar coordinates with $t\geq 0$ and $\rho\in [0,2\pi)$, and employ the identities $d\log |q_r|= \frac{dt}{t}$ and $d^c\log |q_r|= \frac{1}{4\pi} d\rho$, we obtain 
\begin{align*}
\Omega &= \frac{1}{\norm(\frakz,\tau)^2} \left( -\frac{\sqrt{D} r}{\pi} \frac{dt}{t} -d\langle\frakz,\frakz\rangle \right)\wedge  \left( -\frac{\sqrt{D}r}{4\pi^2} d\rho  -d^c\langle\frakz,\frakz\rangle \right)\\
&\phantom{=}{}+\frac{1}{\norm(\frakz,\tau)} dd^c\langle\frakz,\frakz\rangle .
\end{align*}
Hence the restriction of $\Omega$ to $\partial B_\eps(D)\cap Y$ is given by 
\begin{align*}
\Omega \mid_{\partial B_\eps(D)\cap Y}&= 
\frac{1}{\norm(\frakz,\tau)^2} d\langle\frakz,\frakz\rangle \wedge  \left( \frac{\sqrt{D}r}{4\pi^2} d\rho  +d^c\langle\frakz,\frakz\rangle \right)+\frac{1}{\norm(\frakz,\tau)} dd^c\langle\frakz,\frakz\rangle .
\end{align*}
The differential form $d^c\log\| F_1\|$ is the sum of a smooth $1$-form on $Y$ 
and a constant multiple of $\frac{1}{\norm(\frakz,\tau)}d^c\norm(\frakz,\tau)$. This implies that 
\begin{align*}
\Omega^{n-2}\wedge d^c\log\| F_1\|\mid_{\partial B_\eps(D)\cap Y} = \frac{h(t,\rho,\frakz)}{|\log t|^{n-1}} d\rho \wedge d\frakz_1 \dots   d\frakz_{n-2}\wedge d\bar\frakz_1 \dots   d\bar\frakz_{n-2},
\end{align*}
where $h$ is a continuous function in $t\in [0,1)$, $\rho\in [0,2\pi]$ and $\frakz$, which is smooth for $t>0$.

On the other hand, according to \cite[Theorem~4.10]{BHY}, the restriction of the  Green function $\Phi_V(m)$ to $Y$ is the sum of a continuous function and constant multiples of the functions 
\[
\log t,\qquad  \log|\log t|, \qquad \sum_{\substack{x\in S_m }} \log|\langle x, \mathfrak z +\tilde \ell\rangle |,
\]
where $S_m$ denotes the finite set
\[
S_m = \{ x\in L\cap \ell^\perp : \; \langle x,\frakz_0+\tilde \ell\rangle =0, \langle x, x\rangle =m\}.
\]
This implies that 
\begin{align*}
\int_{\partial B_\eps(D)\cap Y} \Phi_V(m)\wedge d^c \log\|F_1\|\wedge \Omega^{n-2}\leq C \frac{1}{|\log \eps|^{n-2}}
\end{align*}
for some positive constant $C$. 
Since $n\geq 3$, we find that the limit in \eqref{eq:bound2} vanishes.

For the limit in \eqref{eq:bound1} we may argue similarly. Here \cite[Theorem~4.10]{BHY} implies that the restriction of $d^c \Phi_V(m)$ to $Y$ is the sum of a $1$-form with log-log growth along $D\cap Y$ and constant multiples of the differentials  
\[
d\rho ,\qquad  \frac{d\rho}{ |\log t|}, \qquad \sum_{\substack{x\in S_m }} d^c \log|\langle x, \frakz +\tilde \ell\rangle |.
\]
On the other hand, $\log\|F_1\|$ is the sum of a continuous function and $\log|\log t|$.
Putting the terms together, we obtain 
\begin{align*}
\int_{\partial B_\eps(D)\cap Y} \log\|F_1\|\wedge \Omega^{n-2}\wedge d^c  \Phi_V(m)\leq C' \frac{\log|\log \eps|}{|\log \eps|^{n-2}}
\end{align*}
for some positive constant $C'$.
Again the limit as $\eps\to 0$ vanishes.
\end{proof}

Lemma \ref{lem:bound} completes the proof of Proposition \ref{prop:star}, which completes the proof of Theorem \ref{thm:no boundary heights}.


\chapter{The volume calculations}
\label{s:volumes}


This chapter contains our main results.
We continue to work with the Shimura variety $\mathcal{S}_V$ of \eqref{moduli inclusion} defined by a $\kk$-hermitian space $V$ of signature $(n-1, 1)$  with $n\ge 1$, containing a self-dual $\co_\kk$-lattice $L\subset V$.
We maintain the assumption that $D$ is odd imposed since Chapter \ref{s:integral models}.

We will  use induction on $n$ to compute the complex and arithmetic volumes  (as defined in \S \ref{ss:volumes}) of the hermitian line bundle 
\[
\widehat{\tautmod}_V =  2 \widehat{\taut}_V - (\mathrm{Exc}_V,0)   \in  \widehat{\Pic} ( \bar{\mathcal{S}}_V , \mathscr{D}_\BKK ) 
\]
 of Theorem \ref{thm:taut-hodge compare}.  All results stated in the introduction will follow as consequences of this calculation.



\section{Volumes of Shimura curves}


First we consider the case $n=2$, so that  $V$ has signature $(1,1)$ and $\bar{\mathcal{S}}_V(\C)$ has dimension one.
The following proposition is a consequence of the main result of \cite{howard-volumes-I}, which is itself a consequence of the
calculation of arithmetic volumes of modular curves and quaternion Shimura curves due to  K\"uhn \cite{kuhn},  Bost (unpublished),  and Kudla-Rapoport-Yang \cite{KRY}.

\begin{proposition}\label{prop:base case}
Suppose $n=2$.
The metrized Hodge bundle of $A\to \mathcal{S}_V$ has complex volume
 \[
\vol_\C(  \widehat{\omega}^\mathrm{Hdg}_{ A / \mathcal{S}_V }  ) 
=  \frac{| \mathrm{CL}(\kk) |^2    }{ 2^{ o(D) -1}  \cdot  12   \cdot  | \co_\kk^\times|^2  }    \prod_{\ell \mid D   } (1 + \ell^* ) 
\]
 and arithmetic volume
\[
\widehat{\vol} (   \widehat{\omega}^\mathrm{Hdg}_{ A / \mathcal{S}_V }  ) 
 =    
\left(
- 2 -  \frac{4\zeta'(-1)}{\zeta(-1)} 
-  \sum_{  \ell \mid D    } \frac{ 1- \ell^*  }{ 1+  \ell^* } \cdot \log(\ell)
\right)  \vol_\C(  \widehat{\omega}^\mathrm{Hdg}_{ A / \mathcal{S}_V }  ) ,
\]
where $\mathrm{CL}(\kk)$ is the class group of $\kk$, $o(D)$ is the number of prime divisors of $D$, and we abbreviate 
\[
\ell^*   = \leg{-1}{\ell}   \mathrm{inv}_\ell(V)\ell.
\]
\end{proposition}

\begin{proof}
As in \eqref{moduli inclusion}, we realize
\[
\mathcal{S}_V \subset \mathcal{M}_{W_0}  \times_{\co_\kk} \mathcal{M}_{W}
\]
as a union of connected components, where $W_0$ and $W$ have signatures $(1,0)$ and  $(1,1)$, respectively.

The  $\mathrm{GU}(W)$ Shimura datum defining the generic fibers of  $\mathcal{M}^\Pap_W$ and $\mathcal{M}_W$ actually has reflex field $\Q$ rather than $\kk$, and in   \S 4 of \cite{howard-volumes-I} one finds the definition of a regular and flat  Deligne-Mumford stack 
\[
\mathcal{X}_W \to \Spec(\Z)
\]
 such that  $\mathcal{M}_W^\Pap \iso  \mathcal{X}_{W/\co_\kk}$, and such that the universal abelian surface over $\mathcal{M}_W^\Pap$ descends to  $A\to \mathcal{X}_W$.  The stack $\mathcal{X}_W$ has its own canonical toroidal compactification and its own theory of arithmetic Chow groups, and Theorem A of \cite{howard-volumes-I} implies
 \begin{equation}\label{old curve 1}
 \vol_\C(  \widehat{\omega}^\mathrm{Hdg}_{A / \mathcal{X}_W}   )=  \frac{| \mathrm{CL}(\kk) |    }{ 2^{ o(D) -1} \cdot 12   \cdot  | \co_\kk^\times| }    \prod_{\ell \mid D  } (1 + \ell^* )
 \end{equation}
 and 
 \begin{equation}\label{old curve 2}
\widehat{\vol} (  \widehat{\omega}^\mathrm{Hdg}_{A / \mathcal{X}_W}     ) 
 =    
\left(
- 1-  \frac{2\zeta'(-1)}{\zeta(-1)} 
-  \frac{1}{2} \sum_{  \ell \mid  D   } \frac{ 1- \ell^*  }{ 1+  \ell^* } \cdot \log(\ell)
\right)  \vol_\C(  \widehat{\omega}^\mathrm{Hdg}_{A / \mathcal{X}_W}   )  .
\end{equation}

The composition 
 \[
 \mathcal{S}_V \to \mathcal{M}_W \to \mathcal{M}_W^\Pap \iso  \mathcal{X}_{W/\co_\kk}
 \]
 extends to  a proper morphism
\[
\varphi : \bar{\mathcal{S}}_V \to \bar{\mathcal{X}}_{ V / \co_\kk } ,
\]
which, using Remark \ref{rem:projection fiber},   restricts to a   finite \'etale surjection of degree 
\[
\deg(\varphi) = |\mathrm{CL}(\kk)| / | \co_\kk^\times|
\]
from $\bar{\mathcal{S}}_V\smallsetminus \mathrm{Exc}_V$  to the smooth locus of  $\bar{\mathcal{X}}_{ V / \co_\kk }$.

If $\widehat{\taut}$ is any pre-log singular hermitian line bundle on $\bar{\mathcal{X}}_V$, and $\widehat{\taut}_{/\co_\kk}$ is its base change to $\bar{\mathcal{X}}_{V / \co_\kk}$, then unpacking the definitions shows that\footnote{The factor of $2$ appears in the second equality because it is the degree of the finite flat morphism $\mathcal{X}_{W/\co_\kk} \to \mathcal{X}_W$. It does not appear in the first equality because, under the conventions of Remark \ref{rem:no half}, this morphism induces an isomorphism on complex points.}
\[
 \vol_\C (   \varphi^*\widehat{\taut}_{/\co_\kk}  )
   = 
  \deg(\varphi)  \cdot \vol_\C (  \widehat{\taut}   )
\]
and
\[
   \widehat{\vol} (   \varphi^*\widehat{\taut}_{/\co_\kk}  )
   = 
 2 \cdot  \deg(\varphi)  \cdot \widehat{\vol} (  \widehat{\taut}   ).
\]
The proposition  follows by applying these  to the hermitian line bundles
\[
\widehat{\taut} =  \widehat{\omega}^\mathrm{Hdg}_{A / \mathcal{X}_V} , \qquad 
\varphi^*\widehat{\taut}_{/\co_\kk} =  \widehat{\omega}^\mathrm{Hdg}_{A / \mathcal{S}_V}
\]
and using \eqref{old curve 1} and \eqref{old curve 2}.
\end{proof}

 We now express the complex and arithmetic volumes of $\widehat{\tautmod}_V$ in terms of the  meromorphic function $\mathbf{A}_V(s)$ defined by \eqref{A_V}.  This will form the base case of the inductive proofs of Theorems \ref{thm:degree induction} and \ref{thm:K volume}  below.

\begin{corollary}\label{cor:base case}
Suppose $n=2$.
The hermitian line bundle \eqref{Kbun}  has complex volume
\[
\vol_\C(  \widehat{\tautmod}_V )   = 
\frac{ | \mathrm{CL}(\kk) |    }{ 2^{  o(D) -2  }   | \co_\kk^\times|   }    \cdot \mathbf{A}_V(0) ,
\]
and arithmetic volume
\[
 \widehat{\mathrm{vol}} (  \widehat{\tautmod}_V   )
 =
 \left(   2 \frac{\mathbf{A}'_V (0) }{\mathbf{A}_V (0)}     +   \log(D)   \right)    \vol_\C(  \widehat{\tautmod}_V )  .
 \]
\end{corollary}

\begin{proof}
Replacing $W$ by $V$ in  \eqref{signs 1}, and using  $D \equiv 3\pmod{4}$,  we 
find
\[
 \prod_{\ell \mid D}  \leg{-1}{\ell} \inv_\ell(V)   = 1.
\]
Recalling \eqref{A_V}, this  proves the second equality in 
\begin{align*}
\mathbf{A}_V(0) 
  & =   
  \frac { D^{\frac{3}{2}}   L(1,\eps) \zeta(2) }{8 \pi^3  }
   \prod_{\ell \mid D} \left(  1 +   \leg{-1}{\ell}   \mathrm{inv}_\ell(V)\ell^{-1}   \right)     \\
   & =   
  \frac { D^{\frac{1}{2}}   L(1,\eps) \zeta(2) }{8 \pi^3  }
   \prod_{\ell \mid D} \left(  1  + \ell^*   \right)     \\
   & =     \frac{|\mathrm{CL}(\kk)|}{ 24 \cdot  |\co_\kk^\times|}
 \prod_{\ell \mid D} \left(  1 +  \ell^* \right).
\end{align*}
The final equality follows from $\zeta(2) = \pi^2/6$ and Dirichlet's class number formula
\[ 
D^{\frac{1}{2}}   L(1,\eps)  =2 \pi     | \mathrm{CL}(\kk) |  /   |\co_\kk^\times|.
\]
Comparing  this with Proposition \ref{prop:base case}, and using  the equality of Chern forms of Theorem  \eqref{thm:taut-hodge compare}, shows
\[
\vol_\C( \widehat{\tautmod}_V  ) =
\vol_\C( \widehat{\omega}^\mathrm{Hdg}_{ A / \mathcal{S}_V }  ) = 
\frac{ | \mathrm{CL}(\kk) |    }{ 2^{  o(D) -2  }   | \co_\kk^\times|   }    \cdot \mathbf{A}_V(0) .
\]

For the second claim, taking the logarithmic derivative of \eqref{A_V}  yields
\begin{align*}
\frac{\mathbf{A}'_V (0) }{\mathbf{A}_V (0)}  
 & = \frac{\mathbf{a}'_1 (0) }{\mathbf{a}_1 (0)} + \frac{\mathbf{a}'_2 (0) }{\mathbf{a}_2 (0)}
 - \sum_{\ell \mid D} \frac{   \log(\ell) }{  1+ \ell^* }  \\
  & = \frac{\mathbf{a}'_1 (0) }{\mathbf{a}_1 (0)} + \frac{\mathbf{a}'_2 (0) }{\mathbf{a}_2 (0)}
 -  \frac{1}{2}   \log (D)  -  \frac{1}{2} \sum_{\ell \mid D}  \left( \frac{  1- \ell^*  }{  1+ \ell^* }  \right) \log(\ell) .
\end{align*}
One can use the functional equation of $L(s,\eps^k)$ to see that
 \begin{equation}\label{a dlogs}
\frac{\mathbf{a}'_k (0) }{\mathbf{a}_k (0)}
=   - 2 \frac{L'(1-k,\eps^k) } { L(1- k ,\eps^k ) }  - \frac{\Gamma'(k)}{\Gamma(k)}  + \log \left( \frac{4\pi}{D} \right) 
+ (-1)^k \log(D)   
\end{equation}
for all $k\ge 1$.  Recalling the constant $C_0(2)$ of Theorem \ref{thm:intro main}, and using the well-known formulas $\Gamma'(1) / \Gamma(1) = -\gamma$ and $\Gamma'(2) / \Gamma(2) = 1-\gamma$, we find
\[
\frac{\mathbf{a}'_1 (0) }{\mathbf{a}_1 (0)} + \frac{\mathbf{a}'_2 (0) }{\mathbf{a}_2 (0)}
=
-1+  C_0(2) - 2 \frac{\zeta'(-1) } { \zeta(-1  ) }   .
\]
Putting this all together gives
\begin{align*}  
  2   \frac{\mathbf{A}'_V (0) }{\mathbf{A}_V (0)} - 2 C_0(2)  +    \log \left( D  \right) 
=  -2   - 4 \frac{\zeta'(-1) } { \zeta(-1) }   
 -      \sum_{\ell \mid D}    \left(    \frac{1- \ell^*} { 1+ \ell^*}   \right)    \log(\ell)  ,
  \end{align*}
and comparing with Proposition \ref{prop:base case} shows that
\[
\widehat{\vol} (   \widehat{\omega}^\mathrm{Hdg}_{ A / \mathcal{S}_V }  ) 
 =  
 \left(   2 \frac{\mathbf{A}'_V (0) }{\mathbf{A}_V (0)}     - 2C_0(2)   +   \log(D)   \right)    \vol_\C( \widehat{\omega}^\mathrm{Hdg}_{ A / \mathcal{S}_V } )  .
\]
The second claim follows from this using Proposition \ref{prop:taut-hodge strong volume} and the equality of Chern forms of Theorem \ref{thm:taut-hodge compare}.
\end{proof}


\section{The complex volume formula}


We now use induction on $n\ge 2$ to compute the complex volumes of $\widehat{\taut}_V$ and $\widehat{\tautmod}_V$ on the Shimura variety $\bar{\mathcal{S}}_V$ associated to a relevant hermitian space $V$ of signature $(n-1,1)$.

Taking arithmetic intersection against the fixed arithmetic cycle class
\[
\widehat{\taut}_V^{n-2} \in \widehat{\CH}^{n-2}( \bar{\mathcal{S}}_V ,\mathscr{D}_\BKK )_\Q
\]
and then applying the complex degree of \S \ref{ss:volumes} defines a linear functional
\[
\widehat{\CH}^1( \bar{\mathcal{S}}_V ,\mathscr{D}_\BKK )  \to  \Q,
\]
denoted 
\begin{equation}\label{complex vol functional}
\widehat{\mathcal{Z}} \mapsto \deg_\C( \widehat{\mathcal{Z}} \cdot \widehat{\taut}_V^{n-2} )
\end{equation}
and called  \emph{complex intersection against $\widehat{\taut}_V^{n-2}$}.  
Of course this only depends on the underlying divisor $\mathcal{Z}$ and line bundle $\taut_V$ on $\bar{\mathcal{S}}_V$, and in fact only on their restrictions to the complex fiber.

Assuming that $n>2$ (an assumption imposed throughout Chapter \ref{s:borcherds}), we may apply \eqref{complex vol functional} coefficient-by-coefficient to the generating series of special divisors \eqref{BHKRYseries}.  This yields a modular form
\[
 \deg_\C(   \widehat{\phi}_V (\tau) \cdot\widehat{\taut}_V^{n-2} )
=\sum_{m\ge 0}  \deg_\C(   \widehat{\mathcal{Z}}_V^\mathrm{tot} (m)\cdot \widehat{\taut}_V^{n-2} )  \cdot  q^m 
\in M_n(\Gamma_0(D) , \eps^n)
\]
with rational coefficients,
and similarly for the other forms appearing in the decomposition 
\[
\widehat{\phi}_V(\tau)
=\widehat{\phi}^\mathrm{eis}_V(\tau)  +  \widehat{\phi}^\mathrm{exc}_V (\tau)+  \widehat{\phi}^\mathrm{cusp}_V (\tau)
\]
of  \eqref{BHKRY decomp}.  Note that we have the obvious equality
\begin{equation}\label{complex exc vanish}
  \deg_\C(   \widehat{\phi}^\mathrm{exc}_V (\tau) \cdot \widehat{\taut}_V^{n-2} ) =0, 
\end{equation}
as the coefficients of the generating series  \eqref{exceptional generating} lie in the kernel of 
\[
\widehat{\CH}^{1}( \bar{\mathcal{S}}_V ,\mathscr{D}_\BKK )  \to 
\CH^{1}( \bar{\mathcal{S}}_{V/\C}  )  .
\]

\begin{proposition}\label{prop:complex expansions}
For $n> 2$, we have the  equalities of modular forms
\begin{align*}
 \deg_\C(   \widehat{\phi}_V (\tau) \cdot \widehat{\taut}_V^{n-2} )
& =
- \vol_\C( \widehat{\taut}_{V}  )+
\sum_{m>0}
 \left(   \int_{ \mathcal{Z}_V(m)(\C) }    \chern(\widehat{\taut}_{V} )^{n-2}  \right)\cdot q^m \\
\deg_\C (   \widehat{\phi}^\mathrm{eis}_V (\tau) \cdot \widehat{\taut}_V^{n-2} )
& = - \mathrm{vol}_\C(\widehat{\taut}_V )  \sum_{r\mid D} \gamma_r E_r(\tau) .
\end{align*}
\end{proposition}

\begin{proof}
The second claim follows easily from \eqref{eis part} and the definitions.

For the first claim, we compare the two modular forms coefficient-by-coefficient.  
As the equality of constant terms is obvious,  it suffices to prove that 
\begin{equation}\label{complex coefficient match}
\deg_\C ( \widehat{\mathcal{Z}}_V^\mathrm{tot} (m)  \cdot  \widehat{\taut}_V^{n-2}  ) 
= \int_{ \mathcal{Z}_V(m)(\C) }    \chern(\widehat{\taut}_{V} )^{n-2}
\end{equation}
for all $m>0$.  The left hand side is equal to the image of 
\[
 ( \bar{\mathcal{Z}}_V  (m) + \mathcal{B}_V(m) )    \cdot \taut_V^{n-2} 
 \in \mathrm{CH}^{n-1}( \bar{\mathcal{S}}_V ) 
\]
under the composition
\begin{equation}\label{generic chow degree}
 \mathrm{CH}^{n-1}( \bar{\mathcal{S}}_V ) 
 \to  \mathrm{CH}^{n-1}( \bar{\mathcal{S}}_{V/\C} )  \map{\deg} \Q
\end{equation}
of the restriction to the complex fiber and the usual degree of a $0$-cycle.
As in the proof of Theorem \ref{thm:no boundary heights}, one may represent the cycle class $\taut_V^{n-2} \in \mathrm{CH}^{n-2}(\bar{\mathcal{S}}_V)$ by a $1$-cycle  that does not meet the boundary in the generic fiber, and hence the left hand side of \eqref{complex coefficient match} is equal to the image of 
\[
\bar{\mathcal{Z}}_V  (m)  \cdot \taut_V^{n-2} 
 \in \mathrm{CH}^{n-1}( \bar{\mathcal{S}}_V ) 
\]
under the composition \eqref{generic chow degree}.

Fix a smooth (on all of $\bar{\mathcal{S}}_V(\C)$, not just the interior) metric on $\widehat{\taut}_V$, and let $\varphi$ be the Chern form of this metric.  The exterior power 
\[
  \varphi^{n-2}  \in H^{2n-4}(\bar{\mathcal{S}}_V(\C) , \C)
\]
agrees with  the cycle class of $\mathcal{L}_V^{n-2}$.  
As the second arrow in \eqref{generic chow degree} factors through the cycle class map to cohomology,  it follows that 
\[
\deg_\C ( \widehat{\mathcal{Z}}_V^\mathrm{tot} (m)  \cdot  \widehat{\taut}_V^{n-2}  ) 
= \int_{ \mathcal{Z}_V(m)(\C) }    \varphi^{n-2}.
\]

We cannot apply this directly to $\varphi =   \chern(\widehat{\taut}_{V} )$ because the metric on
$\widehat{\taut}_{V}$ does not extend smoothly across the boundary, but it is easy to see that one can write
\[
\varphi^{n-2} = \chern(\widehat{\taut}_{V} )^{n-2} + d\alpha
\]
for a pre-log-log form $\alpha$ on $\mathcal{S}_V(\C)$.  
 Propositions 7.6 and 7.12 of \cite{BKK} imply that the pullback of $\alpha$ to $\mathcal{Z}_V(m)(\C)$ is locally integrable on the closure $\bar{\mathcal{Z}}_V(m)(\C)$, and satisfies the equality of currents $[d\alpha]=d[\alpha]$.  Evaluating  both sides of this equation at the constant function $1$ yields
\[
\int_{ \mathcal{Z}_V(m)(\C) }  d\alpha =0,
\]
completing the proof.
 \end{proof}

Recall from \S \ref{ss:self-dual coefficients} and \S \ref{ss:eis decomp}  the  holomorphic Eisenstein series
\[
\sum_{ m \ge 0 } B(m, 0 ,s_0) \cdot  q^m  \in M_n(\Gamma_0(D) , \eps^n )
\]
 associated  to the self-dual hermitian lattice $L\subset V$.

\begin{proposition}\label{prop:complex divisor-eisenstein}
Assume $n>2$. 
We have the equality of modular forms
\[
 \deg_\C(   \widehat{\phi}_V (\tau) \cdot \widehat{\taut}_V^{n-2} )
=  \deg_\C (   \widehat{\phi}^\mathrm{eis}_V (\tau) \cdot \widehat{\taut}_V^{n-2} ).
\]
Equivalently, given \eqref{complex exc vanish},
\[
 \deg_\C (   \widehat{\phi}^\mathrm{cusp}_V (\tau) \cdot \widehat{\taut}_V^{n-2} )  =0 .
 \]
\end{proposition}

\begin{proof}
Given Proposition  \ref{prop:complex expansions} and the equality
\[
\sum_{r\mid D} \gamma_r E_r(\tau)   = \sum_{m\ge 0} B(m,0,s_0)  \cdot q^m 
\]
of Proposition \ref{prop:multi eisenstein}, the first claim is just a restatement of the equality
\[
 \int_{ \mathcal{Z}_V(m)(\C) }    \chern(\widehat{\taut}_{V} )^{n-2}  
  = -  \mathrm{vol}_\C(\widehat{\taut}_V )    B(m,0,s_0)
\]
of Remark \ref{rem:strong complex volume}.
Here we are using the discussion of \S \ref{ss:green construction}, and especially the diagram \eqref{cover}, to relate the  Shimura variety for the unitary group $H=\mathrm{U}(V)$ used throughout \S \ref{ss:u(v) shimura}  to the complex orbifold $\mathcal{S}_V(\C)$ currently under consideration.
\end{proof}

\begin{theorem}\label{thm:degree induction}
If $n\ge 2$,  the complex volumes of \eqref{metrized taut}  and \eqref{Kbun}  are
\[
\vol_\C(  \widehat{\taut}_V)  = 
\frac{|\mathrm{CL} (\kk)|}{   2^{ o(D) -1 }  |\co_\kk^\times|  } \cdot    \mathbf{A}_V (0)
\]
and
\[
\vol_\C(  \widehat{\tautmod}_V )  = 
\frac{|\mathrm{CL} (\kk)|}{    2^{o(D) - n }    |\co_\kk^\times|     } \cdot  \mathbf{A}_V (0) .
\]
\end{theorem}

\begin{proof}
The two claims are equivalent, by the equality of Chern forms of Theorem \ref{thm:taut-hodge compare}, 
 and both hold when $n=2$ by Corollary \ref{cor:base case}.  It therefore  it suffices to assume that  $n>2$, and that the  first equality holds for all relevant (Definition \ref{def:relevant}) hermitian spaces $V^\flat$ of signature $(n-2,1)$.

Let $p$ be a prime split in $\kk$, and let $V^\flat$ be the hermitian space of signature $(n-2,1)$ defined by   \eqref{Vflat}. 
Using the first claim of Theorem \ref{thm:height descent 1} and the equality of Chern forms of Theorem \ref{thm:taut-hodge compare}, we find 
\[
 \int_{ \mathcal{Z}_V(p) (\C)  }  \chern(\widehat{\taut}_V)^{n-2}    
 =  ( p^{n-1}+1)  \cdot  \vol_\C( \widehat{\taut}_{V^\flat} )  .
\]
On the other hand, we have the equality 
\[
 \int_{ \mathcal{Z}_V(p)(\C) }    \chern(\widehat{\taut}_{V} )^{n-2}  
  = -  B(p,0,s_0) \cdot  \mathrm{vol}_\C(\widehat{\taut}_V )   
\]
from the proof of Proposition \ref{prop:complex divisor-eisenstein}.
Directly from the   definition \eqref{A_V} and the explicit formulas of Corollaries \ref{cor:eis2} and \ref{cor:eisodd2},  we find that 
\[
\frac{ B(p,0,s_0)  }{ p^{n-1} +1 }    =     -  \frac{\mathbf{A}_{V^\flat} ( 0 )    }{ \mathbf{A}_V( 0 ) }   .
\]

Combining the three formulas in the previous paragraph, one obtains
\[
 \vol_\C( \widehat{\taut}_{V^\flat} )  
=  \frac{\mathbf{A}_{V^\flat} ( 0 )    }{ \mathbf{A}_V( 0 ) } \cdot    \mathrm{vol}_\C(\widehat{\taut}_V )  .
\]
The theorem follows from this and  our  induction hypothesis 
\[
   \vol_\C( \widehat{\taut}_{V^\flat} )  
    =  \frac{|\mathrm{CL} (\kk)|   }{   2^{ o(D) -1 }  |\co_\kk^\times|  }  \cdot   \mathbf{A}_{V^\flat}(0)  . \qedhere
\]
\end{proof}

\begin{remark}
Theorem \ref{thm:degree induction} is also true when $n=1$.
 In this case  
 \[
 \mathcal{S}_V(\C) \subset \mathcal{M}_{1,0}(\C) \times \mathcal{M}_{0,1}(\C)
 \]
  is a complex orbifold of dimension $0$ parametrizing certain pairs of elliptic curves with complex multiplication by $\co_\kk$, the Chern form of $ \widehat{\taut}_V$ is (obviously) trivial, and the theory of complex multiplication implies
\[
 \vol_\C(  \widehat{\taut}_V) =
 \int_{\mathcal{S}_V(\C)} 1 = \sum_{s\in \mathcal{S}_V(\C)} \frac{1}{ |\Aut(s)| } = \frac{|\mathrm{CL} (\kk)|^2}{ 2^{ o(D) -1 }  \cdot  |\co_\kk^\times|^2  } .
\]
This agrees with 
\[
\frac{|\mathrm{CL} (\kk)|}{   2^{ o(D) -1 }  |\co_\kk^\times|  } \cdot    \mathbf{A}_V (0)
= \frac{|\mathrm{CL} (\kk)|}{   2^{ o(D) -1 }  |\co_\kk^\times|  } \cdot  \mathbf{a}_1(0)
= \frac{|\mathrm{CL} (\kk)|^2}{ 2^{ o(D) -1 }  \cdot  |\co_\kk^\times|^2  },
\]
where the final equality is by Dirichlet's class number formula.
\end{remark}


\section{An alternate proof of the complex volume formula}
\label{ss:alt complex}


We have invoked Remark \ref{rem:strong complex volume} in the proofs of Proposition
\ref{prop:complex divisor-eisenstein} and Theorem \ref{thm:degree induction},
 but we have enough information at this point to give an independent proof of 
 Theorem \ref{thm:degree induction}, and deduce much of Remark \ref{rem:strong complex volume} and 
Proposition \ref{prop:complex divisor-eisenstein} as a consequence.
This is not logically necessary in what follows, but the same method will be used in the proof of Theorem \ref{thm:K volume} below, and so we explain the argument here.

As before, the proof of Theorem \ref{thm:degree induction} is by induction on $n\ge 2$, with the case $n=2$ already known by  Corollary \ref{cor:base case}.   Thus we assume $n>2$.

Suppose $p$ is a prime split in $\kk$, and let $V^\flat$ denote the $\kk$-hermitian space of signature $(n-2,1)$ characterized by \eqref{Vflat}.  
 It follows from Theorem \ref{thm:height descent 1} and the induction hypothesis that
 \begin{align*}
 \int_{ \mathcal{Z}_V(p)(\C) }    \chern(\widehat{\taut}_{V} )^{n-2}
&  = (p^{n-1}+1) \vol_\C ( \widehat{\taut}_{V^\flat} )   \\
& =  \frac{|\mathrm{CL} (\kk)|}{   2^{ o(D) -1 }  |\co_\kk^\times|  } \cdot  (p^{n-1}+1)  \cdot \mathbf{A}_{V^\flat} (0)  .
 \end{align*}
 Exactly as in the proof of Theorem  \ref{thm:degree induction},
Corollaries \ref{cor:eis2} and \ref{cor:eisodd2}  imply
\[
-  \frac{B(p,0,s_0)}{p^{n-1} +1}      =     \frac{  \mathbf{A}_{V^\flat} ( 0 )    }{ \mathbf{A}_V( 0 ) }   ,
\]
 and therefore
 \[
  \int_{ \mathcal{Z}_V(p)(\C) }    \chern(\widehat{\taut}_{V} )^{n-2}
  = 
   -  \frac{  |\mathrm{CL} (\kk)|}{   2^{ o(D) -1 }  |\co_\kk^\times|  } \cdot  \mathbf{A}_V( 0 )  \cdot    B(p,0,s_0) .
 \]
 Combining this with Proposition \ref{prop:complex expansions} shows that the $p^\mathrm{th}$ coefficient of the \emph{cuspidal} modular form
\begin{equation}\label{alt cusp}
 \deg_\C(   \widehat{\phi}^\mathrm{cusp}_V (\tau) \cdot \widehat{\taut}_V^{n-2} )
=
  \deg_\C(   \widehat{\phi}_V (\tau) \cdot \widehat{\taut}_V^{n-2} ) -  \deg_\C(   \widehat{\phi}^\mathrm{eis}_V (\tau) \cdot \widehat{\taut}_V^{n-2} )
\end{equation}
 is
 \begin{equation}\label{cuspidal fake coefficient}
     \left( 
      \mathrm{vol}_\C(\widehat{\taut}_V ) - 
   \frac{  |\mathrm{CL} (\kk)|}{   2^{ o(D) -1 }  |\co_\kk^\times|  } \cdot  \mathbf{A}_V( 0 ) 
   \right)
    \cdot    B(p,0,s_0) .
 \end{equation}
 
 The essential thing is that the quantity in parentheses in   \eqref{cuspidal fake coefficient} is independent of the choice of split prime $p$, which we now allow to vary.
 The Hecke bound for coefficients of cusp forms implies that the absolute value of  \eqref{cuspidal fake coefficient}  is $O(p^{n/2})$  as $p\to \infty$.  On the other hand, Corollary \ref{cor:eisgrowth} shows that the absolute value of $ B(p,0,s_0)$ is greater than a constant times $p^{n-1}$.  As $n>2$, we conclude from this that the quantity in parentheses in \eqref{cuspidal fake coefficient} vanishes, completing our alternate proof of 
  Theorem  \ref{thm:degree induction}.

Moreover, having given a new proof that the $p^\mathrm{th}$ Fourier coefficient of \eqref{alt cusp}
vanishes for all primes $p$ split in $\kk$, we deduce from Corollary \ref{cor:bel1} that the $m^\mathrm{th}$ Fourier coefficient vanishes for all $m$ with $\eps(m)=1$.  
  Again invoking  Proposition \ref{prop:complex expansions}, we deduce that 
 \[
 \int_{ \mathcal{Z}_V(m)(\C) }    \chern(\widehat{\taut}_{V} )^{n-2}  
  = -  \mathrm{vol}_\C(\widehat{\taut}_V )    B(m,0,s_0)
\]
for all such $m$.  In this way, our alternate proof recovers much of 
 Remark \ref{rem:strong complex volume} and Proposition \ref{prop:complex divisor-eisenstein} in a new way.


\section{The  arithmetic volume formula}
 \label{ss:arithmetic intersections}


The arithmetic intersection and arithmetic degree of \S \ref{ss:chow} provide us with a bilinear pairing
\[
 \widehat{\CH}^1( \bar{\mathcal{S}}_V ,\mathscr{D}_\BKK )
 \otimes
  \widehat{\CH}^{n-1}( \bar{\mathcal{S}}_V ,\mathscr{D}_\BKK )  \to  \R,
\]
sending
$
\widehat{\mathcal{Z}}_1 \otimes \widehat{\mathcal{Z}}_2
 \mapsto \widehat{\deg}( \widehat{\mathcal{Z}}_1 \cdot \widehat{\mathcal{Z}}_2  ) .
$
The  hermitian line bundle 
\[
\widehat{\tautmod}_V =  2 \widehat{\taut}_V - (\mathrm{Exc}_V,0)   \in  \widehat{\Pic} ( \bar{\mathcal{S}}_V , \mathscr{D}_\BKK ) 
\]
 of Theorem \ref{thm:taut-hodge compare} therefore determines a linear functional
 \[
 \widehat{\CH}^1( \bar{\mathcal{S}}_V ,\mathscr{D}_\BKK )  \to  \R,
 \]
called  \emph{arithmetic intersection against $\widehat{\tautmod}_V^{n-1}$}, given by
$
 \widehat{\mathcal{Z}} \mapsto    \widehat{\deg} ( \widehat{\mathcal{Z}} \cdot \widehat{\tautmod}_V^{n-1}  ) .
$

Assuming (as throughout Chapter \ref{s:borcherds}) that $n>2$, we have the modular form
\begin{equation}\label{degree height}
 \widehat{\deg} (  \widehat{\phi}_V  (\tau) \cdot \widehat{\tautmod}_V^{n-1} )  
 = \sum_{m\ge 0}  \widehat{\deg} \big( \widehat{\mathcal{Z}}^\mathrm{tot}(m)  \cdot \widehat{\tautmod}_V^{n-1} \big)  \cdot q^m
\end{equation}
in $M_n(\Gamma_0(D) , \eps^n)$ obtained
 by applying this linear functional coefficient-by-coefficient to \eqref{BHKRYseries}, 
 and similarly for the other forms appearing in the decomposition 
\[
\widehat{\phi}_V(\tau)
=\widehat{\phi}^\mathrm{eis}_V(\tau)  +  \widehat{\phi}^\mathrm{exc}_V (\tau)+  \widehat{\phi}^\mathrm{cusp}_V (\tau)
\]
of  \eqref{BHKRY decomp}.  
As $\widehat{\tautmod}_V$  has trivial arithmetic intersection with every irreducible component of the exceptional divisor $\mathrm{Exc}_V \subset \bar{\mathcal{S}}_V$ by Theorem \ref{thm:taut-hodge compare}, it has  trivial arithmetic intersection with every coefficient of \ref{exceptional generating}, and 
\begin{equation}\label{theta vanish}
 \widehat{\deg} (  \widehat{\phi}^\mathrm{exc}_V (\tau)  \cdot \widehat{\tautmod}_V^{n-1} ) =0.
\end{equation}

Motivated by Proposition \ref{prop:complex divisor-eisenstein}, we make the following conjecture.

\begin{conjecture}\label{conj:arithmetic degrees}
Assuming $n>2$,  we have the equality of modular forms
\[
\widehat{\deg} (  \widehat{\phi}_V  (\tau) \cdot \widehat{\tautmod}_V^{n-1} ) 
= \widehat{\deg} (  \widehat{\phi}^\mathrm{eis}_V (\tau) \cdot \widehat{\tautmod}_V^{n-1} ) .
\]
Equivalently, given \eqref{theta vanish}, 
\[
\widehat{\deg} (  \widehat{\phi}^\mathrm{cusp}_V (\tau) \cdot \widehat{\tautmod}_V^{n-1} )  =0.
\]
\end{conjecture}

At long last we come to the central result of this work: the calculation of the  arithmetic volume of $\widehat{\tautmod}_V$ via induction.  
The calculation will simultaneously yield evidence for Conjecture \ref{conj:arithmetic degrees}, and also provide an explicit value for the mysterious rational numbers $a(p)$ of Theorem \ref{thm:height descent 1}.

\begin{theorem}\label{thm:K volume}
The hermitian line bundle  $\widehat{\tautmod}_V$  satisfies the following properties.
\begin{enumerate}
\item
If $n\ge 2$,  its arithmetic volume (in the sense of \S \ref{ss:volumes}) is 
\[
 \widehat{\vol}( \widehat{\tautmod}_V)  
=  \left(  2\frac{\mathbf{A}'_V (0) }{\mathbf{A}_V (0)}     +   \log(D)   \right)
 \vol_\C(  \widehat{\tautmod}_V).
\]
\item
If  $n>2$, then  for any prime $p$ split in $\kk$ the rational number $a(p)$ of Theorem \ref{thm:height descent 1} is given by
\[
  a(p)  = 2  \cdot  \frac{  p^{n-1} -1 }{  p^{n-1}+1 }  \cdot   \vol_\C (\widehat{\tautmod}_{V^\flat}  ) ,
\]
where $V^\flat$ the hermitian space of signature $(n-2,1)$ with local invariants determined by \eqref{Vflat}.
\item
If  $n> 2$ then for any  positive $m\in \Z$ with $\eps(m)=1$, the  $m^\mathrm{th}$ Fourier coefficient of 
\[
 \widehat{\deg} (  \widehat{\phi}^\mathrm{cusp}_V (\tau) \cdot \widehat{\tautmod}_V^{n-1} )  \in S_n(\Gamma_0(D) ,\eps^n) 
\]
vanishes.
\end{enumerate}
\end{theorem}

\begin{proof}
The proof is by induction on $n$. As we already know the first claim for $n=2$ by Corollary \ref{cor:base case}, it suffices to assume that  $n>2$, and that the first claim holds for all relevant (Definition \ref{def:relevant}) hermitian spaces $V^\flat$ of signature $(n-2,1)$.

We define $C(V)$ by the relation
\[
\widehat{\vol}( \widehat{\tautmod}_V)  = C(V)  \vol_\C( \widehat{\tautmod}_V) ,
\]
and similarly for any $V^\flat$ as above, so that our induction hypothesis is 
\begin{equation}\label{induction constant}
C( V^\flat )    =  2\frac{\mathbf{A}'_{V^\flat} (0) }{\mathbf{A}_{V^\flat} (0)}    +   \log(D) .
\end{equation}
Recall from Proposition \ref{prop:multi eisenstein} the  Eisenstein series
\[
 \sum_{r\mid D}   \gamma_r E_r  (\tau) 
 =
\sum_{ m \ge 0 } B(m, 0 ,s_0) q^m  \in M_n(\Gamma_0(D) , \eps^n ).
\]
Recall also from Proposition \ref{prop:eisenstein twist} the rational number
\[
\beta_\ell  = (-1)^{n+1} \inv_\ell(V)  \cdot \begin{cases}
   \leg{-1}{\ell}^{\frac{n}{2} }        \ell^{ \frac{n}{2} }  &  \mbox{if $n$ is even} \\[2ex]
\leg{-1}{\ell}^{\frac{n-1}{2} }     \ell ^{ \frac{n-1}{2}  } &  \mbox{if $n$ is odd}
\end{cases}
\]
for every prime $\ell \mid D$.

\begin{lemma}\label{lem:second volume lemma}
Let $p$ be a prime split in $\kk$,  let $V^\flat$ the hermitian space of signature $(n-2,1)$ with local invariants determined by \eqref{Vflat}.  If we abbreviate
\begin{align*}
\kappa (p)  &=   2 \frac{ \mathbf{A}'_V (0)}{  \mathbf{A}_V (0)  }   + \log(D)  
  +     \sum_{\ell \mid D}    \frac{  2   \log(\ell)    }  {       1+\beta_\ell \cdot  \leg{p}{\ell}^n    } , \\
\alpha(p)  &= (p^{n-1}+1)  a(p)  - 2    ( p^{n-1} -1 )   \vol_\C (\widehat{\tautmod}_{V^\flat}  )  ,
 \end{align*}
then the $p^\mathrm{th}$ coefficient of the modular form
\[
 \widehat{\deg} (  \widehat{\phi}_V  (\tau) \cdot \widehat{\tautmod}_V^{n-1} )  \in M_n(\Gamma_0(D) , \eps^n)
\]
   is equal to 
\[
- \frac{  \kappa(p)  }{2}   \cdot B(p,0,s_0)   \cdot   \vol_\C( \widehat{\tautmod}_V)   + \alpha(p) \cdot \log(p).
\]
\end{lemma}

\begin{proof}
By Theorem \ref{thm:no boundary heights}, the $p^\mathrm{th}$ coefficient of \eqref{degree height} is
\begin{equation}\label{degree breakdown}
\widehat{\deg} \big( \widehat{\mathcal{Z}}_V^\mathrm{tot} (p)  \cdot \widehat{\tautmod}_V^{n-1}  \big)
=
\mathrm{ht}_{\widehat{\tautmod}_V}  (  \bar{\mathcal{Z}}_V (p) ) 
+ \int_{\mathcal{S}_V(\C)}  \Phi_V(p)  \chern( \widehat{\tautmod}_V)^{n-1}   .
\end{equation}
Using Theorem \ref{thm:height descent 1}, we rewrite the first term on the right hand side as
\begin{align}
\mathrm{ht}_{\widehat{\tautmod}_V}  (  \bar{\mathcal{Z}}_V (p) ) 
& = 
(p^{n-1}+1) \left( \widehat{\vol} (\widehat{\tautmod}_{V^\flat} )  + a(p) \log(p) \right)   \nonumber \\
& = 
(p^{n-1}+1) \left(C(V^\flat) \vol_\C (\widehat{\tautmod}_{V^\flat} )  + a(p) \log(p) \right) .  \label{height decomp application}
\end{align}

Now consider the second term on the right hand side of \eqref{degree breakdown}.
By Remark \ref{rem:S_V green integral} and the equality 
$
2 \chern( \widehat{\taut}_V ) = \chern( \widehat{\tautmod}_V ) 
$
 of Theorem \ref{thm:taut-hodge compare}, it is equal to
\[
\int_{ \mathcal{S}_V(\C)  } \Phi_V (p)  \chern( \widehat{\tautmod}_V) ^{n-1}  
  =    \vol_\C( \widehat{\tautmod}_V)   B'(p,0,s_0) .
  \]
  Theorem  \ref{thm:degree induction} implies the equality
\begin{equation}\label{complex vol swap}
\vol_\C(  \widehat{\tautmod}_V )   = 2   \cdot   \frac{  \mathbf{A}_V (0)  }{  \mathbf{A}_{V^\flat} (0) }   \cdot   \vol_\C(  \widehat{\tautmod}_{V^\flat} )  ,
\end{equation}
which we use to rewrite the integral above as
\begin{equation}\label{pheight C}
 \int_{\mathcal{S}_V(\C)}  \Phi_V(p)  \chern( \widehat{\tautmod}_V)^{n-1}     
    =  C_1 \cdot  (p^{n-1}+1) \cdot \vol_\C (\widehat{\tautmod}_{V^\flat} ) ,
 \end{equation}
in which we have set
\[
C_1 = 2 \cdot     \frac{    B'(p, 0 ,s_0)  }{   p^{n-1}+1   }    \cdot   \frac{  \mathbf{A}_V (0)  }{  \mathbf{A}_{V^\flat} (0) }. 
\]

Plugging  \eqref{height decomp application} and \eqref{pheight C}  into \eqref{degree breakdown}, we obtain 
  \begin{align*}
\widehat{\deg} \big( \widehat{\mathcal{Z}}_V^\mathrm{tot} (p)  \cdot \widehat{\tautmod}_V^{n-1}  \big)
 & =(p^{n-1}+1)  ( C_1+ C(V^\flat)   )  \vol_\C (\widehat{\tautmod}_{V^\flat}  )   \\
 &\quad +(p^{n-1}+1) a(p) \log(p) .
 \end{align*}
 By direct calculation, Corollaries \ref{cor:eis2} and \ref{cor:eisodd2} (along with the definition \eqref{A_V} of $\mathbf{A}_V(s)$ and $\mathbf{A}_{V^\flat}(s)$)   give the explicit formulas
\begin{equation}\label{explicit p}
\frac{ B(p,0,s_0)  }{ p^{n-1} +1 }    =     -\frac{\mathbf{A}_{V^\flat} ( 0 )    }{ \mathbf{A}_V( 0 ) }   
\end{equation}
and
\begin{align*}
\frac{  B'(p,0,s_0) }{  B(p,0,s_0)  }
       & =    
    \frac{   \mathbf{A}'_{V^\flat} (0) }{   \mathbf{A}_{V^\flat} (0)  }    -  \frac{ \mathbf{A}'_V (0)}{  \mathbf{A}_V (0)  }   
    +   \frac{  p^{n-1} -1   }{   p^{n-1}  +1    }   \log(p)  
    -   \sum_{\ell \mid D}    \frac{    \log(\ell)    }  {      1+\beta_\ell \cdot  \leg{p}{\ell}^n   }  .
\end{align*}
Using these equalities we see that the constant $C_1$ can be rewritten as
\[
C_1
=      - 2  \frac{    \mathbf{A}'_{V^\flat} (0) }{   \mathbf{A}_{V^\flat} (0)  }    + 2    \frac{ \mathbf{A}'_V (0)}{  \mathbf{A}_V (0)  }   
   - 2     \frac{  p^{n-1} -1   }{   p^{n-1}  +1    }   \log(p)  
    +     \sum_{\ell \mid D}    \frac{ 2   \log(\ell)    }  {      1+\beta_\ell \cdot  \leg{p}{\ell}^n    }  .
\]
Combining this with the induction hypothesis \eqref{induction constant} shows that
\begin{align*}
C_1+ C(V^\flat)   & =  \kappa(p)   - 2     \frac{  p^{n-1} -1   }{   p^{n-1}  +1    }   \log(p)  ,
\end{align*}
  and  we obtain 
\[
\widehat{\deg} \big( \widehat{\mathcal{Z}}_V^\mathrm{tot} (p)  \cdot \widehat{\tautmod}_V^{n-1}  \big)
 =  (p^{n-1}+1) \kappa(p)  \vol_\C (\widehat{\tautmod}_{V^\flat}  )    + \alpha(p) \log(p)     .
\]
To complete the proof of the lemma, substitute the equality 
\[
 \vol_\C (\widehat{\tautmod}_{V^\flat}  ) 
   \stackrel{\eqref{complex vol swap}}{=}
 \frac{ \mathbf{A}_{V^\flat}(0)}{2\mathbf{A}_V(0)} \cdot  \vol_\C (\widehat{\tautmod}_{V}  )  \\
       \stackrel{\eqref{explicit p}}{=}
-  \frac{  B(p,0,s_0) }{ 2( p^{n-1}+1) }    \vol_\C (\widehat{\tautmod}_{V}  ) 
\]
into the first term on the right hand side.
\end{proof}

\begin{lemma}\label{lem:third volume lemma}
We have the equality of modular forms
\begin{align}
\widehat{\deg}(  \widehat{\phi}^\mathrm{eis}_V (\tau)  \cdot \widehat{\tautmod}_V^{n-1} ) 
  & = -  \frac{1}{2}  \widehat{\mathrm{vol}}( \widehat{\tautmod}_V    )
    \sum_{r\mid D}   \gamma_r E_r  (\tau) \nonumber \\
  & \quad -   \mathrm{vol}_\C (  \widehat{\tautmod}_V )   \sum_{r\mid D} \log(D/r)     \gamma_r E_r (\tau)  .
  \label{first eis volume}
  \end{align}
  Moreover, for positive $m\in \Z$ with $\gcd(m,D)=1$  the $m^\mathrm{th}$ coefficient of this modular form is
   \[
 -    \frac{1}{2}     \left(     \widehat{\mathrm{vol}}( \widehat{\tautmod}_V    )  
     +  \mathrm{vol}_\C (  \widehat{\tautmod}_V )  
 \sum_{\ell \mid D}    \frac{ 2 \log(\ell )  }{  1+\beta_\ell \cdot  \leg{m}{\ell}^n }   \right) B(m,0,s_0) .
   \] 
  \end{lemma}

\begin{proof}
Directly from the definition \eqref{eis part} of   $\widehat{\phi}^\mathrm{eis}_V  (\tau)$
we find
\begin{align*}
\widehat{\deg}(  \widehat{\phi}^\mathrm{eis}_V (\tau)  \cdot \widehat{\tautmod}_V^{n-1} ) 
  & = -   \widehat{\deg}( \widehat{\taut}_V  \cdot \widehat{\tautmod}_V^{n-1} )  \sum_{r\mid D}   \gamma_r E_r (\tau) \\
  & \quad - \sum_{r\mid D} \widehat{\deg} \big(    ( \mathrm{Exc}_V  ,     \log(D/r) )    \cdot \widehat{\tautmod}_V^{n-1} \big)  \gamma_r E_r (\tau)   .
  \end{align*}
 For the first term on the right hand side, we use Theorem \ref{thm:taut-hodge compare} to see that
\begin{align*}
\widehat{\deg}( \widehat{\taut}_V  \cdot \widehat{\tautmod}_V^{n-1} )
  = 2^{-1} \cdot \widehat{\deg} ( \widehat{\tautmod}_V   \cdot \widehat{\tautmod}_V^{n-1} )  
  =  2^{-1} \cdot  \widehat{\mathrm{vol}}( \widehat{\tautmod}_V    ) .
\end{align*}
 For the second term on the right hand side, it follows from Theorem \ref{thm:taut-hodge compare} that $ ( \mathrm{Exc}_V   ,  0 )    \cdot  \widehat{\tautmod}_V^{n-1}    =0$.
As in the proof of Lemma \ref{lem:trivial volume shift}, we have the easy arithmetic intersection formula
\[
( 0    ,     \log(D/r)  )    \cdot  \widehat{\tautmod}_V^{n-1} 
 = \big(  0 ,     \log(D/r)  \chern( \widehat{\tautmod}_V )^{n-1}  \big) 
 \in  \widehat{\CH}^n( \bar{\mathcal{S}}_V ,\mathscr{D}_\BKK ) ,
\]
and hence, directly from the definition \eqref{degree normalization} of arithmetic degree,  
\begin{align*}
\widehat{\deg}(     ( 0    ,     \log(D/r) )  \cdot  \widehat{\tautmod}_V^{n-1}  )  
&  =     \log(D/r)   \int_{ \mathcal{S}_V(\C)}   \chern( \widehat{\tautmod}_V )^{n-1}  \\
&  =     \log(D/r)  \mathrm{vol}_\C (  \widehat{\tautmod}_V ) .
\end{align*}
Thus
\[
\widehat{\deg}(     (  \mathrm{Exc}_V   ,     \log(D/r) )  \cdot  \widehat{\tautmod}_V^{n-1}  )
= \log(D/r)  \mathrm{vol}_\C (  \widehat{\tautmod}_V ).
\]
The equality \eqref{first eis volume} follows immediately from these formulas.

For the second claim, using Proposition \ref{prop:multi eisenstein} and  Proposition \ref{prop:eisenstein twist}, the $m^\mathrm{th}$ coefficient of the modular form on the right hand side of \eqref{first eis volume} is
\begin{align*}\lefteqn{
-  \frac{1}{2}  \widehat{\mathrm{vol}}( \widehat{\tautmod}_V    )
    \sum_{r\mid D}   \gamma_r e_r(m)  
     -   \mathrm{vol}_\C (  \widehat{\tautmod}_V )   \sum_{r\mid D}    \gamma_r e_r(m)  \log(D/r)   } \\
     & =
     -  \frac{1}{2}  \widehat{\mathrm{vol}}( \widehat{\tautmod}_V    )  B(m,0,s_0) 
     -   \mathrm{vol}_\C (  \widehat{\tautmod}_V )  
      \left(   \sum_{\ell \mid D}    \frac{  \log(\ell )  }{  1+\beta_\ell \cdot  \leg{m}{\ell}^n }   \right) B(m,0,s_0) ,
\end{align*}
as desired.
\end{proof}

The proof of Theorem \ref{thm:K volume} now proceeds by examining the Fourier coefficients of the modular form
\begin{equation}\label{cuspidal difference}
\widehat{\deg} ( \widehat{\phi}^\mathrm{cusp}_V(\tau) \cdot \widehat{\tautmod}_V^{n-1} ) =
\widehat{\deg} ( \widehat{\phi}_V(\tau) \cdot \widehat{\tautmod}_V^{n-1} ) - 
\widehat{\deg} ( \widehat{\phi}^\mathrm{eis}_V(\tau) \cdot \widehat{\tautmod}_V^{n-1} ).
\end{equation}
If we abbreviate
\[
c = 
   \frac{    \widehat{\vol} ( \widehat{\tautmod}_V)   }{2}  
  - \left( \frac{\mathbf{A}'_{V} (0) }{\mathbf{A}_{V} (0)}    +   \frac{ \log(D) }{2}   \right)  \vol_\C( \widehat{\tautmod}_V)   ,
\]
then comparing Lemma \ref{lem:second volume lemma} and Lemma \ref{lem:third volume lemma} shows that  the $p^\mathrm{th}$ coefficient of \eqref{cuspidal difference} is 
\[
c  B(p,0,s_0)       + \alpha(p) \log(p)
\]
 for all primes $p$ split in $\kk$.
 In particular, the modular form 
\[
\widehat{\deg} ( \widehat{\phi}^\mathrm{cusp}_V(\tau) \cdot \widehat{\tautmod}_V^{n-1} )  
- c \sum_{m\ge 0}B(m,0,s_0)
\]
has the property that its  $p^\mathrm{th}$ coefficient is $\alpha(p) \log(p) \in \Q \log(p)$ for all such primes.
It follows from Corollary \ref{cor:bel3} that $\alpha(p)=0$, proving the second claim of Theorem \ref{thm:K volume}.

We have now seen the $p^\mathrm{th}$ coefficient of  the \emph{cuspidal} modular form \eqref{cuspidal difference} is $cB(p,0,s_0)$ for all primes $p$ split in $\kk$, and may argue as in \S \ref{ss:alt complex}. 
 The Hecke bound for coefficients of cusp forms implies that the absolute value of  \eqref{cuspidal fake coefficient}  is $O(p^{n/2})$  as $p\to \infty$, while Corollary \ref{cor:eisgrowth} shows that the absolute value of $ B(p,0,s_0)$ is greater than a constant times $p^{n-1}$.  We conclude  that $c=0$,  proving the first claim of Theorem \ref{thm:K volume}.

Finally, having now shown that the $p^\mathrm{th}$ coefficient of \eqref{cuspidal difference} vanishes for all primes $p$ split in $\kk$, Corollary \ref{cor:bel1}  shows that the $m^\mathrm{th}$ coefficient vanishes whenever $\eps(m)=1$. This proves the third claim of Theorem \ref{thm:K volume}.
\end{proof}


\section{Consequences of Theorem 8.4.2}
\label{ss:volume corollaries}


We  now collect many corollaries of Theorem \ref{thm:K volume}.
The first of these are explicit formulas for the arithmetic volumes of the other (perhaps more natural) hermitian line bundles on $\mathcal{S}_V$ defined in \S \ref{ss:hermitian bundles}.

\begin{corollary}\label{cor:other volumes}
Assuming $n\ge 2$,  we have the  arithmetic volume formulas
\[
\widehat{\mathrm{vol}} (  \widehat{\omega}^\mathrm{Hdg}_{ A / \mathcal{S}_V}    ) 
=
 \left(  2\frac{\mathbf{A}'_V (0) }{\mathbf{A}_V (0)}   -n C_0(n)    +   \log(D)   \right)
 \vol_\C( \widehat{\omega}^\mathrm{Hdg}_{ A / \mathcal{S}_V}    )
\]
and
\begin{align*}
 \widehat{\vol} ( \widehat{\taut}_V ) 
&= 
    \left(  \frac{\mathbf{A}'_V (0) }{\mathbf{A}_V (0)}    +   \frac{ \log(D)  }{2}   \right)
   \vol_\C(  \widehat{\taut}_V)
+  \frac{ (-1)^{n-1}}{2}   \sum_{E \subset \mathrm{Exc}_V}  \widehat{m}_E   .
\end{align*}
In the second equality the  sum is over all irreducible components $E \subset \mathrm{Exc}_V$  of the exceptional divisor of Definition \ref{def:special exceptional} and, if $p\mid D$ denotes the prime at which $E$ is supported,    $\widehat{m}_E = m_E \log(p) $ is the constant of \S \ref{ss:exceptional volume}.
\end{corollary}

\begin{proof}
The first equality follows from Theorem \ref{thm:K volume}, using Proposition \ref{prop:taut-hodge strong volume}  and the equality of Chern forms of Theorem \ref{thm:taut-hodge compare}.

For the second equality, combining \ref{Kbun} with  the  intersection formula  
\[
  \widehat{\tautmod}_V \cdot  (\mathrm{Exc}_V,0)  =0
  \]
   of Theorem \ref{thm:taut-hodge compare} shows that
\begin{align*}
2^n \cdot \widehat{\vol} ( \widehat{\taut}_V ) 
&= \widehat{\vol} ( \widehat{\tautmod}_V + (\mathrm{Exc}_V, 0 ))   \nonumber \\
&= \widehat{\vol} ( \widehat{\tautmod}_V )+ \widehat{\vol}  (\mathrm{Exc}_V, 0 ).
\end{align*}
Combining this with Theorem \ref{thm:K volume} and the equality of Chern forms of Theorem \ref{thm:taut-hodge compare} yields
\[
 \widehat{\vol} ( \widehat{\taut}_V ) 
= 
 \left(  \frac{\mathbf{A}'_V (0) }{\mathbf{A}_V (0)}     +   \frac{ \log(D) }{2}   \right)
 \vol_\C(  \widehat{\taut}_V)
+  \frac{  \widehat{\vol}  (\mathrm{Exc}_V, 0 ) }{ 2^n } .
\]
Now use Corollary \ref{cor:exceptional volume}.
\end{proof}

The most interesting aspect of the expected equality
\[
\sum_{m\ge 0}  \widehat{\deg} \big( \widehat{\mathcal{Z}}^\mathrm{tot}(m)  \cdot \widehat{\tautmod}_V^{n-1} \big) \cdot q^m
\stackrel{?}{=} \widehat{\deg} (  \widehat{\phi}^\mathrm{eis}_V (\tau) \cdot \widehat{\tautmod}_V^{n-1} )
\]
of Conjecture \ref{conj:arithmetic degrees} is that the coefficients of the modular form on the right hand side can be made completely explicit, by combining the formulas of \S \ref{ss:eis decomp} and Theorem \ref{thm:degree induction} with the following corollary of Theorem \ref{thm:K volume}.

\begin{corollary}\label{cor:arithmetic eis}
Assuming $n> 2$, we have the equality of modular forms
\begin{align*}\lefteqn{
\widehat{\deg}(  \widehat{\phi}^\mathrm{eis}_V (\tau) \cdot \widehat{\tautmod}_V^{n-1} )  }  \nonumber \\
  & = 
  - 
   \mathrm{vol}_\C (  \widehat{\tautmod}_V )
 \sum_{r\mid D}\left(  \frac{\mathbf{A}'_V (0) }{\mathbf{A}_V (0)}     +  \frac{3}{2}  \log(D) -\log(r)   \right)    \gamma_r E_r  (\tau)  .
  \end{align*}
 \end{corollary}

\begin{proof}
Start  with the equality 
\begin{align*}
\widehat{\deg}(  \widehat{\phi}^\mathrm{eis}_V (\tau)  \cdot \widehat{\tautmod}_V^{n-1} ) 
  & = -  \frac{1}{2}  \widehat{\mathrm{vol}}( \widehat{\tautmod}_V    )
    \sum_{r\mid D}   \gamma_r E_r  (\tau) \nonumber \\
  & \quad -   \mathrm{vol}_\C (  \widehat{\tautmod}_V )   \sum_{r\mid D} \log(D/r)     \gamma_r E_r (\tau)  
  \end{align*}
  of Lemma \ref{lem:third volume lemma}, and  substitute the arithmetic  volume formula of Theorem \ref{thm:K volume} into the first term on the right hand side.
\end{proof}

  Assume now that $n>2$.
   In light of Corollary \ref{cor:arithmetic eis},  Conjecture  \ref{conj:arithmetic degrees} is equivalent to
   \begin{align}\lefteqn{
 \widehat{\deg} \big( \widehat{\mathcal{Z}}^\mathrm{tot}(m)  \cdot \widehat{\tautmod}_V^{n-1} \big)  }  \nonumber\\
 &  \stackrel{?}{=} 
  - 
   \mathrm{vol}_\C (  \widehat{\tautmod}_V )
 \sum_{r\mid D}\left(  \frac{\mathbf{A}'_V (0) }{\mathbf{A}_V (0)}     +  \frac{3}{2}  \log(D) -\log(r)   \right)    \gamma_r e_r  (m), \label{degree conjecture 1}
\end{align}
for all positive $m\in \Z$.   Of course the right hand side can be made explicit using the formulas of \S \ref{ss:eis decomp}.

For example, let us now  suppose  $\gcd(m,D)=1$, and abbreviate
\[
\kappa (m) =   2 \frac{ \mathbf{A}'_V (0)}{  \mathbf{A}_V (0)  }   + \log(D)  
  +     \sum_{\ell \mid D}    \frac{  2   \log(\ell)    }  {       1+\beta_\ell \cdot  \leg{m}{\ell}^n    } ,
\]
where
\[
\beta_\ell  = (-1)^{n+1} \inv_\ell(V)  \cdot \begin{cases}
   \leg{-1}{\ell}^{\frac{n}{2} }        \ell^{ \frac{n}{2} }  &  \mbox{if $n$ is even} \\[2ex]
\leg{-1}{\ell}^{\frac{n-1}{2} }     \ell ^{ \frac{n-1}{2}  } &  \mbox{if $n$ is odd.}
\end{cases}.
\]
We can use 
 Propositions \ref{prop:multi eisenstein} and  \ref{prop:eisenstein twist} to rewrite \eqref{degree conjecture 1} as 
\[
 \widehat{\deg} \big( \widehat{\mathcal{Z}}^\mathrm{tot}(m)  \cdot \widehat{\tautmod}_V^{n-1} \big)  
   \stackrel{?}{=} 
  - 
 \frac{    \kappa(m)   }{2}  \cdot  \mathrm{vol}_\C (  \widehat{\tautmod}_V )  \cdot    B(m,0,s_0) .
\]
As in the proof of Proposition \ref{prop:complex divisor-eisenstein}, Remark \ref{rem:strong complex volume} implies 
\begin{equation}\label{Kbun complex integral}
 \int_{ \mathcal{Z}_V(m)(\C) }    \chern(\widehat{\tautmod}_{V} )^{n-2}  
  = - \frac{1}{2}   \cdot  \mathrm{vol}_\C(\widehat{\tautmod}_V )   \cdot   B(m,0,s_0)
\end{equation}
(the  factor of $1/2$ appears because of the relation 
$2  \chern(\widehat{\taut}_{V})   =   \chern(\widehat{\tautmod}_{V} )$
of Theorem \ref{thm:taut-hodge compare}), and  we find that \eqref{degree conjecture 1} is equivalent to 
\begin{equation}
\label{degree conjecture 2}
 \widehat{\deg} \big( \widehat{\mathcal{Z}}^\mathrm{tot}(m)  \cdot \widehat{\tautmod}_V^{n-1} \big) 
   \stackrel{?}{=} 
  \kappa(m)   \int_{ \mathcal{Z}_V(m)(\C) }    \chern(\widehat{\tautmod}_{V} )^{n-2}   .
 \end{equation}
The following corollary of Theorem \ref{thm:K volume} shows that this equality holds for many values of $m$.

 \begin{corollary}\label{cor:known degree}
 Assume $n>2$.
 The equality \eqref{degree conjecture 2}   holds for all positive $m\in \Z$ with $\eps(m)=1$.
 \end{corollary}

 \begin{proof}
 The final claim of Theorem \ref{thm:K volume}  shows that the $m^\mathrm{th}$ Fourier coefficients of 
 \[
 \widehat{\deg}(  \widehat{\phi}_V (\tau) \cdot \widehat{\tautmod}_V^{n-1} ) 
 \quad \mbox{and}\quad 
 \widehat{\deg}(  \widehat{\phi}^\mathrm{eis}_V (\tau) \cdot \widehat{\tautmod}_V^{n-1} ) 
 \]
 agree  for all $m$ with $\eps(m)=1$, and hence (by Corollary \ref{cor:arithmetic eis}) the equality \eqref{degree conjecture 1} holds for all such $m$.  
 The discussion above shows that this is equivalent to \eqref{degree conjecture 2}.
 \end{proof}

Again suppose that $n>2$ and $\gcd(m,D)=1$, and recall the equality 
\[
\widehat{\deg} \big( \widehat{\mathcal{Z}}_V^\mathrm{tot} (m)  \cdot \widehat{\tautmod}_V^{n-1}  \big)
=
\mathrm{ht}_{\widehat{\tautmod}_V}  (  \bar{\mathcal{Z}}_V (m) ) 
+ \int_{\mathcal{S}_V(\C)}  \Phi_V(m)  \chern( \widehat{\tautmod}_V)^{n-1}   
\]
of Theorem \ref{thm:no boundary heights}.
As in the proof of Lemma \ref{lem:second volume lemma}, the integral on the right hand side  is
\begin{align*}
 \int_{\mathcal{S}_V(\C)}  \Phi_V(m)  \chern( \widehat{\tautmod}_V)^{n-1}    
 & = 
   B'(m, 0 ,s_0)  \cdot   \vol_\C(\widehat{\tautmod}_V )  \\
   & \stackrel{\eqref{Kbun complex integral}}{=}
 -2  \frac{B'(m, 0 ,s_0)}{B(m, 0 ,s_0)}    \int_{ \mathcal{Z}_V(m)(\C) }    \chern(\widehat{\tautmod}_{V} )^{n-2}  .
\end{align*}
Solving for the height term, we find that the conjectural equality of \eqref{degree conjecture 2} is equivalent to 
\[
\mathrm{ht}_{\widehat{\tautmod}_V}  (  \bar{\mathcal{Z}}_V (m) )   
\stackrel{?}{=}  \left(  \kappa(m)
   +  2   \frac{  B'(m, 0 ,s_0)  }{ B(m,0,s_0)   } 
 \right)   \int_{ \mathcal{Z}_V(m)(\C) }    \chern(\widehat{\tautmod}_{V} )^{n-2} .
\]
By the explicit  calculation of $B(m,0,s)$ found in Corollaries \ref{cor:eis2} and \ref{cor:eisodd2}, this last conjectural equality is equivalent to 
\begin{equation}\label{degree conjecture 3}
\mathrm{ht}_{\widehat{\tautmod}_V}  (  \bar{\mathcal{Z}}_V (m) )   
 \stackrel{?}{=}  
\left(  2 \frac{ \mathbf{A}'_{V,m}(0)  }{  \mathbf{A}_{V,m}(0 )    }    +  \log(D)  \right)
 \int_{ \mathcal{Z}_V(m)(\C) }    \chern(\widehat{\tautmod}_{V} )^{n-2} .
\end{equation}
Here  if $n$ is even  we have set
\[
\mathbf{A}_{V,m}(s) = \mathbf{a}_1(s) \cdots \mathbf{a}_{n-1}(s) \cdot  m^{s+n-1}   \sigma_{-2s - n + 1 }(m), 
\]
while if $n$ is odd we have set
\begin{align*}
\mathbf{A}_{V,m}(s)  
&= \mathbf{a}_1(s) \cdots \mathbf{a}_{n-1}(s) \cdot m^{s+  n-1  }\sigma_{-2s - n + 1  ,\eps}(m)    \\
& \quad \times 
 \prod_{ \ell \mid D } \left(  1+ \leg{-1}{\ell}^{ \frac{n-1}{2} } \leg{m}{\ell} \inv_\ell(V) \ell^{ -s - \frac{n-1}{2}}     \right).
\end{align*}

 \begin{corollary}
  Assume $n>2$.  The equality \eqref{degree conjecture 3}
holds for all positive $m\in \Z$ with $\eps(m)=1$.
 \end{corollary}

\begin{proof}
The discussion above shows that \eqref{degree conjecture 3} is equivalent to \eqref{degree conjecture 2}, which is known when $\eps(m)=1$ by Corollary \ref{cor:known degree}.
\end{proof}


\section{Volume of the Hodge bundle}


We now derive Theorem \ref{thm:intro main} of the introduction as a consequence of Theorem \ref{thm:K volume}.    Let $W$ be a $\kk$-hermitian space of signature $(n-1,1)$ with $n\ge 1$, containing a self-dual $\co_\kk$-lattice.  Let $\mathcal{M}_W$ be the moduli space of \eqref{relevant decomp}, and recall that in  \S \ref{ss:hermitian bundles} we associated a metrized Hodge bundle 
\[
\widehat{\omega}^\mathrm{Hdg}_{A/\mathcal{M}_W} \in \widehat{\Pic}( \bar{\mathcal{M}}_W, \mathscr{D}_\BKK ) 
\]
to the universal abelian scheme $A \to \mathcal{M}_W$.

\begin{theorem}\label{thm:final hodge}
The metrized Hodge bundle  has complex volume
\[
\vol_\C(  \widehat{\omega}^\mathrm{Hdg}_{A/\mathcal{M}_W}  )
 =
 \mathbf{A}_W (0) \times
 \begin{cases}
 2^{n-1  }    & \mbox{if $n$ is odd } \\
2^{n-o(D)  }   & \mbox{if $n$ is even}
\end{cases}
\]
and arithmetic volume
\[
\widehat{\vol}  ( \widehat{\omega}^\mathrm{Hdg}_{A/ \mathcal{M}_W  } )
= 
  \left(  2 \frac{  \mathbf{A}_W'(0) }{ \mathbf{A}_W(0) }    - nC_0(n)  
  + \log(D)  \right)     \vol_\C(   \widehat{\omega}^\mathrm{Hdg}_{A/ \mathcal{M}_{W}  } ),
  \]
where $C_0(n)$ is the constant defined in Theorem \ref{thm:intro main}.
\end{theorem}

\begin{proof}
First assume $n\ge 2$.
Let $W_0=\kk$ endowed with its hermitian form $h(x,y)=x\overline{y}$,
and define a hermitian form on $V=\Hom_\kk(W_0,W)$ as in \eqref{basic hom hermitian}.   
In particular $W\iso V$ as hermitian spaces.

As in Remark \ref{rem:projection fiber}, there is a  finite \'etale surjection
$
\mathcal{S}_V \to \mathcal{M}_W
$
of degree
\[
t_n =   \frac{ |\mathrm{CL}(\kk) |} {  | \co_\kk^\times|} 
\times \begin{cases}
1 & \mbox{if $n$ is even} \\
2^{1- o(D)  }  & \mbox{if $n$ is odd.}
\end{cases}
\]
Denote again by $A \to \mathcal{S}_V$ the pullback of $A\to \mathcal{M}_W$ via this morphism, 
so that 
\[
 \vol_\C(\widehat{\omega}^\mathrm{Hdg}_{ A / \mathcal{M}_W }  )
  =
  \frac{1}{t_n}  \vol_\C(\widehat{\omega}^\mathrm{Hdg}_{ A/\mathcal{S}_V}  )
\]
and
\[
 \widehat{\vol} ( \widehat{\omega}^\mathrm{Hdg}_{ A / \mathcal{M}_W }  ) 
=  
  \frac{1}{t_n}  \widehat{\vol} ( \widehat{\omega}^\mathrm{Hdg}_{ A / \mathcal{S}_V } ).
\]
Using these formulas,  the theorem  for $n>2$ follows from Theorem \ref{thm:degree induction}  and  Corollary \ref{cor:other volumes}, along with the equality of Chern forms of Theorem \ref{thm:taut-hodge compare}.

It remains to treat the case $n=1$,  so that 
\[
\mathcal{M}_{W} \iso \mathcal{M}_{(0,1)}  \iso \mathcal{M}_{(1,0)}.
\]
 In this case
 $\mathbf{A}_W(s) = \mathbf{a}_1(s)$.  Dirichlet's class number formula implies
\[
\mathbf{a}_1(0) = \frac{  | \mathrm{CL}(\kk) | }{  | \co_\kk^\times | }  ,
\]
while  \eqref{a dlogs} and \eqref{faltings} imply  
\[
 2 \frac{  \mathbf{a}_1'(0) }{ \mathbf{a}_1(0) }    - C_0(1)    + \log(D)    
 =    -  \frac{L'(0,\eps)}{L(0,\eps)}   -    \frac{    \log(D)}{2}  = \log(2\pi) + 2 h_\kk^\mathrm{Falt} .
\]
The theory of complex multiplication implies
\[
\vol_\C(   \widehat{\omega}^\mathrm{Hdg}_{A/ \mathcal{M}_{(1,0)}  } ) =
 \int_{\mathcal{M}_{(1,0)} (\C)} 1   = \sum_{ x \in \mathcal{M}_{(1,0)}(\C) } \frac{1}{ | \Aut(x) | } =  \frac{  | \mathrm{CL}(\kk) | }{  | \co_\kk^\times | } = \mathbf{A}_W(0) ,
\]
while  the argument of Proposition \ref{prop:easy numerical}, which is essentially the Chowla-Selberg formula, implies the first equality in 
\begin{align*}
\widehat{\deg}  ( \widehat{\omega}^\mathrm{Hdg}_{A/ \mathcal{M}_{(1,0)}  } )
 & =
 \widehat{\deg}  (0 ,  \log(2\pi) + 2 h_\kk^\mathrm{Falt} ) \\ 
 & =  ( \log(2\pi) + 2 h_\kk^\mathrm{Falt} )  \int_{ \mathcal{M}_{(1,0)} (\C) }  1  \\
 & =   \left( 2 \frac{  \mathbf{A}_W'(0) }{ \mathbf{A}_W(0) }    - C_0(1)    + \log(D)  \right) 
\vol_\C(   \widehat{\omega}^\mathrm{Hdg}_{A/ \mathcal{M}_{(1,0)}  } ) .
\end{align*}
This  completes the proof.
\end{proof}

\appendix


\chapter{Coefficients of modular forms}
\label{appendix:bel}


In this appendix we quote a theorem of Bella\"{i}che, and derive some consequences that play an essential role in the proof of Theorem \ref{thm:K volume}.  We thank Joel Bella\"{i}che for his help.  
Let $N$ be a positive integer and let $\chi$ be a Dirichlet character modulo $N$. 
Fix a modular form
\[
f(\tau)  = \sum_{m\ge 0} c(m) \cdot q^m \in M_k(\Gamma_0(N),\chi) 
\]
of weight $k\in \Z_{>0}$ with nebentypus character $\chi$ (we do not assume that $f$ is a Hecke eigenform).

\begin{Atheorem}[Bella\"{i}che]
\label{thm:bel}
Suppose the coefficients of $f(\tau)$ satisfy 
\begin{enumerate}
\item[(i)] $c(m)=0$ for all integers $m$ with $\gcd(m,N)>1$, and 
\item[(ii)] $c(p)=0$ for a set of primes $p$ of density $1$.
\end{enumerate}
Then $f(\tau)$ vanishes identically. 
\end{Atheorem}

\begin{proof}
If the nebentypus character $\chi$ is trivial, this is a direct consequence of \cite[Theorem~I]{Be}, which actually gives an analogous result in positive characteristic. An inspection of the proof shows that the result equally holds for general nebentypus. 

For the convenience of the reader, we also sketch a more direct proof that was provided to us by Bella\"{i}che. 
Let $L$ be be the field generated by the coefficients of $f$. 
Without loss of generality, possibly replacing $L$ by an extension, we may view $L$  as an extension of $\Q_p$ for any prime $p$. We may also assume that  $L$ contains all $N$-th roots of unity.
We first consider the case when $N>1$. Choose a prime $p$ that divides $N$.

Let $M_f$ be the subspace of $M_k(\Gamma_0(N),\chi,L)$ generated by $f$ under the action of the Hecke operators $T_\ell$ for $\ell$ prime not dividing $N$. 
Let $A_f$ be the $L$-subalgebra of
$\End_L(M_f)$ generated by the $T_\ell$ for $\ell$ prime not dividing $N$.

Notice that $A_f$ contains all the $T_n$ for $n$ prime to $N$. This is because any such $T_n$ is a polynomial in the $T_\ell$ ($\ell$ prime not dividing $N$) and the diamond
operators, which are just scalar in $L$, and hence belong to $A_f$. (Here we use the hypothesis that $f$ has a nebentypus).

Let $G=G_{\Q,N}$ be the Galois group of the maximal algebraic extension of $\Q$ unramified
outside $N$. There exists  a continuous
 pseudo-representation $(\tau,\delta) : G \to A_f$ of dimension $2$
which sends $\operatorname{Frob}_\ell$ (for $\ell$ prime not dividing $N$)
to $T_\ell$, see e.g.~\cite[Section~10]{Be} or \cite[Section~3.10]{BS}. Here 
$\delta$ is $w^{k-1} \chi$, where  $w$ denotes the $p$-adic cyclotomic character, and hence takes values in $L^\times$.

We claim that the linear subspace generated by  $\tau(G)$ in $A_f$ has to be $A_f$. This  follows from the formula $\tau(gh)+\delta(h) \tau(gh^{-1})= \tau(g) \tau(h)$, a
formula true for all pseudo-representations of dimension $2$, since it is true for $\tau =$ the trace, and $\delta =$ the determinant in the algebra $\operatorname{Mat}_2(\C)$. Thus the linear subspace
generated by $\tau(G)$ is a subalgebra containing  all the $T_\ell$ ($\ell$ not dividing $N$), hence is $A_f$.

Let $V$ be the kernel of the  linear map $A_f\to L$ given by $T \mapsto c_{Tf}(1)$. 
By the hypothesis (ii) we know that $c_f(\ell) = c_{T_\ell f}(1)=0$ for a dense set of primes $\ell$. Hence $\tau(\operatorname{Frob}_\ell)$ is in $V$ for a dense set of $\ell$.
Since $V$ is closed, this means by the Chebotarev density theorem that $\tau(G)$ is contained in $V$. By the above claim, we find that $V = A_f$. This implies that $T_n\in V$ for
all $n$ prime to $N$, and hence that $c_f(n)=0$ for all $n$ prime to $N$.
On the other hand, by hypothesis (i) we know that $c_f(n)=0$  for all $n$ not prime to $N$. Hence 
we have $f=0$, concluding the proof in the case $N>1$.

If $N=1$, we chose an auxiliary prime $q$ and replace $f$ by 
\[
f'= \sum_{\substack{n\geq 0\\ (n,q)=1}} c(n) q^n,
\]
which belongs to $M_k(\Gamma_0(q^2),L)$ by  \cite[Lemma 4.3.10]{miyake}.
Applying the above argument to $f'$ we find $c(n)=0$ for all $n$ coprime to $q$. But since $f$ has level $1$, this means that $f=0$.     
\end{proof}

\begin{Acorollary}
\label{cor:bel1}
Keep the form
\[
f(\tau)  = \sum_{m\ge 0} c(m) \cdot q^m \in M_k(\Gamma_0(N),\chi) 
\]
as above.
Let $\psi$ be a quadratic Dirichlet character modulo $N$.
Fix $\delta\in \{\pm 1\}$, and suppose that  $c(p)=0$ for all but finitely many primes $p$ with $\psi(p)=\delta$.
Then $c(m)=0$ for all integers $m$ with $\psi(m)=\delta$.
\end{Acorollary}

\begin{proof}
Recall from \cite[Lemma 4.3.10]{miyake} that if $\psi$ is any quadratic Dirichlet character modulo $N$, then the twist 
\[
f_\psi (\tau) = \sum_{m\geq 0} \psi(m) c(m) \cdot  q^m
\]
of $f(\tau)$ defines an element of $M_k(\Gamma_0(N^2),\chi)$. 
As the Fourier coefficients of 
\[
\sum_{\substack{m\geq 0\\ \psi(m)= \delta} }  c (m)  \cdot q^m 
 =
  \frac{1}{2} \cdot \big(  f_{\psi^2}(\tau) +\delta f_\psi (\tau)  \big)
   \in M_k(\Gamma_0(N^2),\chi)
\]
 satisfy the hypotheses of Theorem \ref{thm:bel},  this modular form  vanishes identically. 
The claim follows.
\end{proof}

We now apply these results to the space of modular forms $M_n(\Gamma_0(D),\eps^n)$  relevant to the present paper.  In other words,  $\eps$ and $D$ are as in the introduction.
Let us write  
\[
M  \subset M_n(\Gamma_0(D),\eps^n)
\]
for the $\Q$-subspace of forms with rational Fourier coefficients.
The $\Q$-linear dual  $M^\vee=\Hom_\Q(M,\Q)$ is generated by the coefficient extraction functionals 
\[
a_m: M \to \Q,
\]
sending a modular form $f(\tau)$ to its $m^\mathrm{th}$ Fourier coefficient.

\begin{Acorollary}
\label{cor:bel2}
For  $\delta\in \{\pm 1\}$ and any positive $t\in \Z$, the subspaces
\begin{align*}
A_{\delta}  & = \mathrm{Span}_\Q \{  a_m :  \text{$m\in \Z_{>0}$ with $\eps(m)=\delta$} \}   \\
P_{\delta,t}  & =  \mathrm{Span}_\Q \{ a_p :  \text{$p$ prime with $\eps(p)=\delta$ and $p>t$} \} 
\end{align*}
of $M^\vee$ are equal.
\end{Acorollary}

\begin{proof}
For any subspace $B\subset M^\vee$ we write $B^\perp \subset M$ for its orthogonal complement with respect to the natural non-degenerate pairing $M^\vee\times M\to \Q$.  
It suffices to show that  the inclusion
$
A_\delta^\perp \subset P_{\delta,t}^\perp
$
is an equality, so fix an 
\[
f(\tau) = \sum_{ m\ge 0} c(m) \cdot q^m \in P_{\delta,t}^\perp \subset M.
\]
By definition of $P_{\delta,t}^\perp$, we have $c(p)=0$ for all primes  $p>t$ with $\eps(p)=\delta$, 
and so Corollary \ref{cor:bel1} implies that $c(m)=0$ for all $m\in \Z_+$ with  $\eps(m)=\delta$.  In other words, $f(\tau) \in A_\delta^\perp$.
\end{proof}

\begin{Acorollary}
\label{cor:bel3}
Fix $\delta\in \{\pm 1\}$.  If 
\[
f(\tau)  = \sum_{m\ge 0} c(m) \cdot q^m \in M_n(\Gamma_0(D),\eps^n) 
\]
is a modular form with the property that 
\[
c(p) \in \Q \log(p)
\]
for all primes $p$ with $\eps(p)=\delta$, then $c(p)=0$ for all such primes.
\end{Acorollary}

\begin{proof}
Fix a prime $p$ with $\eps(p)=\delta$, and consider the coefficient extraction functional
$
a_p : M  \to \Q.
$
Taking $t=p$ in Corollary \ref{cor:bel2}, we find that there are rational numbers $\alpha_1,\ldots, \alpha_r$ and primes $p< p_1 < \ldots < p_r $ with all $\eps(p_i)=\delta$, such that 
\[
a_p =  \sum_{i=1}^r \alpha_i  \cdot a_{p_i}   
\]
holds as an equality of $\Q$-linear functionals on $M$.
We now view this as an equality of $\C$-linear functionals on 
$
M_n(\Gamma_0(D),\eps^n) = M\otimes_\Q\C, 
$
 and  apply this linear functional to $f(\tau)$.  The  result is
\[
c(p) = \sum_{i=1}^r \alpha_i  \cdot c(p_i)  \in \sum_{i=1}^r \Q \log(p_i) ,
\]
and the claim follows from  the $\Q$-linear independence of the set 
\[
\{ \log(p),\log(p_1) ,\ldots, \log(p_r)\}. \qedhere
\]
\end{proof}

\begin{Aremark}
Of course Corollary \ref{cor:bel1} allows us to strengthen the conclusion of Corollary \ref{cor:bel3}
to $c(m)=0$ whenever $\eps(m)=\delta$.
\end{Aremark}

\backmatter
\bibliographystyle{amsalpha}

\newcommand{\etalchar}[1]{$^{#1}$}
\def\cprime{$'$}
\providecommand{\bysame}{\leavevmode\hbox to3em{\hrulefill}\thinspace}
\providecommand{\MR}{\relax\ifhmode\unskip\space\fi MR }
\providecommand{\MRhref}[2]{%
  \href{http://www.ams.org/mathscinet-getitem?mr=#1}{#2}
}
\providecommand{\href}[2]{#2}

\printindex

\end{document}